\documentclass[11pt,a4paper,reqno]{book}
\usepackage{vmargin,color}

\usepackage[english]{babel}
\usepackage[utf8]{inputenc}

\usepackage{amsmath,amssymb,amsthm,mathrsfs}
\usepackage{epsfig,color}

\usepackage[numbers]{natbib}

\usepackage{bbm,amsmath,amsthm,epsfig,latexsym,marvosym}
\usepackage{amsfonts}
\usepackage{bigints}
\usepackage{amsmath}
\usepackage{amssymb}
\usepackage{parallel} 
\usepackage{subfig} 

\usepackage[colorlinks=true,linkcolor=blue,citecolor=red]{hyperref}

\def\Z{\mathbb Z}
\def\R{\mathbb R}
\def\N{\mathbb N}

\def\Cl{\text{Cl}(N,\Om)}

\def\Rn{\R^n}

\def\cal{\mathcal}

\def\E{{\cal E}}
\def\F{{\cal F}}
\def\G{{\cal G}}
\def\H{{\cal H}}

\def\L{{\cal L}}
\def\T{{\cal T}}

\def\a{\alpha}
\def\b{\beta}
\def\g{\gamma}
\def\de{\delta}
\def\e{\varepsilon}
\def\l{\lambda}
\def\om{\omega}
\def\s{\sigma}

\def\vt{\vartheta}

\def\G{\Gamma}
\def\Om{\Omega}
\def\La{\Lambda}
\def\S{\Sigma}
\def\ov{\overline}
\def\PSI{\mathbf{\Psi}}

\newenvironment{sistema}%
{\left\lbrace\begin{array}{@{}l@{}}}%
{\end{array}\right.}

\def\pa{\partial}

\def\Id{{\rm Id}\,}

\def\Div{{\rm div}\,}
\def\d{\, \mathrm{d}}

\def\dist{{\rm dist}}
\def\bd{{\rm bd}}
\def\arc{{\rm arc}}

\def\ca{\mathbbmss{1}}

\def\pared{\partial^{*}}
\def\PHI{\mathbf{\Phi}}
\def\Lip{{\rm Lip}\,}
\def\INT{{\rm int}\,}
\def\ttau{\boldsymbol{\tau}}
\def\00{{\bf 0}}
\def\dive{{\rm div}}
\def\mez{\left(\frac{1}{2}\right)}
\def\uno{(1)}
\def\zero{(0)}

\newcommand{\cc}{\subset\subset}
\newcommand{\vol}{\mathrm{vol}\,}
\newcommand{\hd}{\mathrm{hd}}

\def\k{\kappa}

\newcommand{\facciatabianca}{\newpage\shipout\null\stepcounter{page}}

\DeclareMathOperator*{\spt}{spt}
\DeclareMathOperator*{\diam}{diam}

\DeclareMathOperator*{\freccia}{\longrightarrow}
\DeclareMathOperator*{\frecciad}{\rightharpoonup}

\def\big{\bigskip}

\newtheorem{theorem}{Theorem}[section]
\newtheorem{corollary}[theorem]{Corollary}
\newtheorem{proposition}[theorem]{Proposition}
\newtheorem{lemma}[theorem]{Lemma}

\theoremstyle{definition}
\newtheorem{remark}[theorem]{Remark}
\newtheorem{definition}[theorem]{Definition}

\numberwithin{equation}{section}
\numberwithin{figure}{section}

\author{M. Caroccia}

\makeindex
\title{On the isoperimetric properties of planar $N$-clusters}

\begin{document}

\maketitle
\facciatabianca
\begin{flushright}
\begin{scriptsize}
\textit{Devenere locos ubi nunc ingentia cernes}\\ 
\textit{Moenia surgentemque novae Karthaginis arcem,} \\
\textit{Mercatique solum, facti de nomine Byrsam,} \\
\textit{Taurino quantum possent circumdare tergo.}\\
\end{scriptsize}
\small{Virgilio, Eneide, Libro I, vv 365-368.}
\end{flushright}

\begin{flushright}
\begin{scriptsize}
\textit{They came to this spot, where today you can behold the mighty\\
Battlements and the rising citadel of New Carthage,\\
And purchased a site, which was named 'Bull's Hide' after the bargain,\\
By which they should get as much land as they could enclose with a bull's hide.\\}
\end{scriptsize}
\small{Virgil, Aeneid, Book I, vv 365-368.}
\end{flushright}
\facciatabianca
\tableofcontents

\chapter*{Notations}
\addcontentsline{toc}{chapter}{Notations} 

\begin{align*}
&B_r(x)        & &\text{Ball in $\R^{n}$ centered at $x$ and with radius $r$}\\
&B_r             & &\text{$=B_r(0)$ ball in $\R^{n}$ centered at $0$ and with radius $r$}\\
&S^{n-1}      & &\text{$=\pa B_1$, $(n-1)$-dimensional sphere in $\R^{n}$}\\
&  E\Delta F   & &\text{$= (E\setminus F)\cup (F\setminus E)$, symmetric difference between $E$ and $F$}\\ 
&  |E|  & & \text{Lebesgue measure of the set $E$}\\ 
&  \H^{s}(E)  & & \text{$s$-dimensional Hausdorff measure of the set $E$}\\ 
& \displaystyle \freccia^{X }   & &\text{Convergence wrt the topology induced from the metric space $X$}\\
& \displaystyle \frecciad^{*}   & &\text{Weak-star convergence}\\
& L^{p}_{loc}(A;\R^k)  & &\text{Space of the functions $f:A\rightarrow \R^k$ with values in $\R^k$}\\
& & &\text{and which are $p$-summable on every compact set strictly}\\
& & &\text{contained in the set $A\subseteq \R^{n}$}\\
& L^{p}_{loc}(A)  & &=L^{p}_{loc}(A;\R),\text{ space of the $\R$-valued functions which are $p$-summable}\\
& & & \text{on every compact set strictly contained in the set $A\subseteq \R^{n}$}\\
& L^{p}(A;\R^k)  & &\text{Space of the functions $f:A\rightarrow \R^k$ with values in $\R^k$}\\
& & & \text{which are $p$-summable on $A\subseteq \R^{n}$}\\
& L^{p}(A)  & &=L^{p}(A;\R), \text{ space of the $\R$-valued functions which are}\\
& & &\text{$p$-summable on $A\subseteq \R^{n}$}\\
&C_c^{k}(A;\R^k)      & &\text{Space of the  of the $k$-differentiable functions $f:A\rightarrow \R^k$ with values }\\
& & & \text{in $\R^k$ and which are compactly supported on the set $A$}\\
&C_c^{k}(A) & &C_c^{1}(A;\R),\text{ space of the $k$-differentiable, $\R$-valued functions which are}\\
& & & \text{compactly supported on the set $A$}\\
& \ca_E(x) & &\text{Characteristic function of the set $E$}\\
& \dive(T) & &\text{Divergence of the vector field $T$}\\
&P(E;F) & &\text{Relative perimeter of the Borel set $E$ inside $F$}\\
&P(E) & &\text{$=P(E;\R^{n})$, global perimeter of the Borel set $E$}\\
&|\mu|(A) & &\text{Total variation of the Radon measure $\mu$ on the set $A\subseteq \R^{n}$, }\\
&\mu\llcorner E & &\text{Radon measure obtained as the restriction of the Radon}\\
& & & \text{measure $\mu$ to the Borel set $E$, $\mu\llcorner E (F)=\mu(E\cap F)$ }\\
\end{align*}   
\begin{align*}
&\vartheta_n(x,E) & &\text{$n$-dimensional density of the set $E$ at the point $x$}\\
&E^{(t)} & &\text{$=\{x\in \R^{n} \ | \ \vt_n(x,E)=t\}$}\\
&\pa^{*}E & &\text{Reduced boundary of the set $E$}\\
&\pa^{e}E & &\text{$=\R^{n}\setminus (E^{(0)}\cup E^{(1)})$, essential boundary of the set $E$}\\
&\nu_E(x) & &\text{Measure-theoretic outer unit normal to $E$ at $x\in \pa^{*} E$}\\
&\E & &\text{$=\{\E(i)\}_{i=1}^{N}$, $N$-cluster of $\R^{n}$}\\
&\E(h,k) & & \text{$=\pared \E(k)\cap \pared \E(h)$, interface between the chambers $\E(k)$ and $\E(h)$,} \\
&\pa\E & & \text{$=\bigcup_{i=1}^{N}\pa \E(i)$, topological boundary of the cluster $\E$,} \\
&\pared \E & & \text{$=\bigcup_{0\leq h<k\leq N}^{N}\E(h,k)$, reduced boundary of the cluster $\E$}, \\
& \E\Delta\F & &\text{$=\bigcup_{i=0}^{N}\E(i)\Delta\F(i),$ symmetric difference between the $N$-clusters $\E$ and $\F$},\\
& \approx & & \text{Equal up to an $\H^{n-1}$-negligible set},\\
& P(\E;F) & & \text{Relative perimeter of the cluster $\E$ inside $F$},\\
\end{align*}

\chapter*{Introduction}

Isoperimetric problems have fascinated the human being since the ancient times, starting from the legend of Dido \textit{who left the city of Tyre to escape her brother} and whose problem was to enclose \textit{as much land as could be enclosed with a bull's hide} up to the present days, with Joseph Plateau who experimented with soap films in order to figure out what is the surface with the smallest amount of area among all the surfaces that share a fixed boundary in the three dimensional space. In general a problem can be defined \textit{isoperimetric} whenever we seek for objects attaining the smallest (or the largest) amount of (a suitable notion of) area (or volume) among all those objects satisfying a given constraint. The most famous one (so famous that is called just \textit{the isoperimetric problem}) is the one that Dido solved in the planar case when she built New Carthage with a bull's hide. It can be stated, in modern mathematical language, as follows: what is the $n$-dimensional object having the smallest perimeter ($(n-1)$ dimensional area of the boundary) among all the sets with a fixed amount of volume? Or, equivalently: what is the $n$-dimensional object having the biggest amount of volume among all the sets with a fixed amount of perimeter ($(n-1)$ dimensional area of the boundary)?\\

The reasons that led scientists and mathematicians to be attracted by this kind of questions might rely on the fact that the energy needed in a number of physical processes is related with the surface area or the mass. For example it is a well known fact that the shapes of crystals are polyhedral (see \cite{FMP10}) because they solve a variant of the classical isoperimetric problem (let us recall for the sake of completeness that the solution to the classical isoperimetric problem is the $n$-dimensional euclidean ball, see \cite{DeGiorgiisoperi}). The techniques and the ideas developed in order to approach this kind of questions turn out to be a useful equipment for the treatment of various type of problems concerning geometry and optimization process, as in the case of image recovery (see \cite{mumford1989optimal} and \cite{ambrosio1990approximation}). That is another reason that explain why these issues have been so fruitfully studied in the past and why they are, still today, a central topic in Calculus of Variations.\\

This Thesis aims to highlight some isoperimetric questions involving the so-called $N$-clusters. The term \textit{cluster} has been used in many different areas of mathematics to denote ``a family of objects that share a precise property and that are combined and connected in a specific way". This points out that a \textit{cluster} is not just a \textit{set} but it is somehow an agglomerate. In our context an $N$-cluster $\E$ is a generic family $\{\E(i)\}_{i=1}^N$ of \textit{$N$ sets with disjoint interiors} (called chambers) that compete in some variational (isoperimetric) problem as a unique object. We refer to Chapter One where the main definitions and tools are recalled. The main problem leading to define these objects, is the natural extension of the classical isoperimetric problem: the multi-chamber isoperimetric problem. This problem can be easily stated as follows. Among all families of $N$ sets $\E(1),\ldots, \E(N)\subset \R^n$ with disjoint interiors and with fixed volume $|\E(i)|=v_i$, what is the family that minimizes the $(n-1)$ dimensional area of the boundary, being careful to count once every possible common boundary between two sets? Technically speaking, given a vector of positive numbers $(v_1,\ldots,v_N)\in \R^N$ we look for a family of $N$-disjoint (up to a negligible set) Borel sets $\E(1),\ldots,\E(N)$ such that $|\E(i)|=v_i$ and
	\begin{equation}\label{introduction problems}
	P(\E)=\inf\left\{ P(\F) \ | \ \text{$\F$ is an $N$-cluster with } |\F(i)|=v_i \right\}
	\end{equation}
where 
	\[
	P(\E):=\sum_{i=0}^N \frac{P(\E(i))}{2}, \ \ \ \ \ \E(0)=\R^n\setminus \bigcup_{i=1}^N \E(i)
	\]
and where $P(\cdot)$ denotes the distributional perimeter (see \cite{DeGiorgiSOFP1}), that here could be intended as the $(n-1)$ dimensional area of the boundary. The addition of the \textit{external chamber $\E(0)$} allows us to define the perimeter in a very natural way in order to count once every piece of boundary shared by two different set from the family (see Figure \ref{perimetrocluster}). 
\begin{figure}
\centering
\includegraphics[scale=0.6]{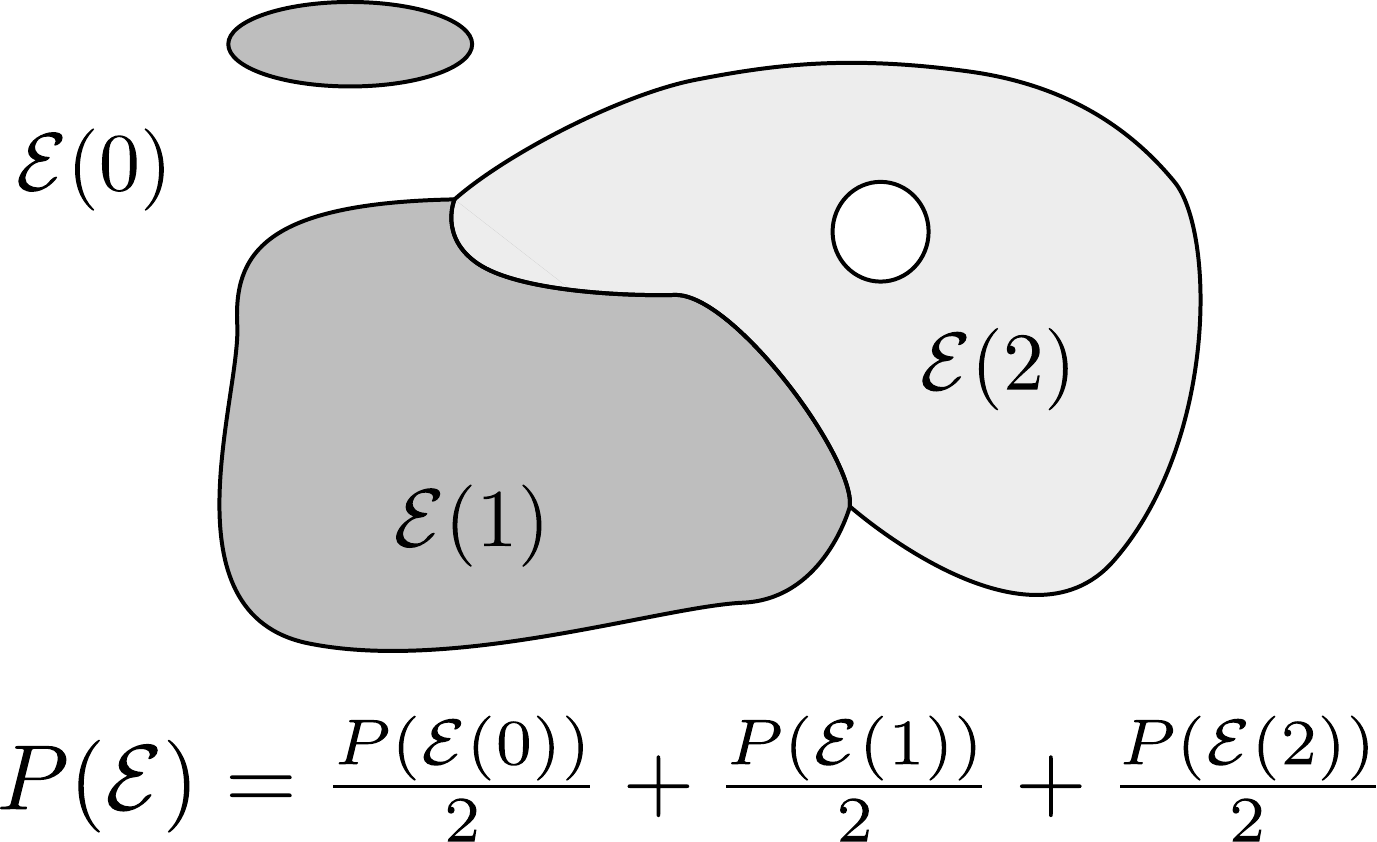}\caption{\small{An example of a $2$-cluster in the plane. The introduction of the external chamber $\E(0)$ allows us to define naturally the perimeter of the cluster $P(\E)$ as the half-sum of the perimeter of each chamber, so that each piece of boundary is counted just once.}}\label{perimetrocluster}
\end{figure}
The existence of such objects (see Figure \ref{esempiminimalcluster}), called \textit{perimeter-minimizing $N$-clusters}, was proved by \cite{almgren1975}, together with a partial regularity theorem (see Chapter One for details). 
\begin{figure}
\centering
\includegraphics[scale=0.7]{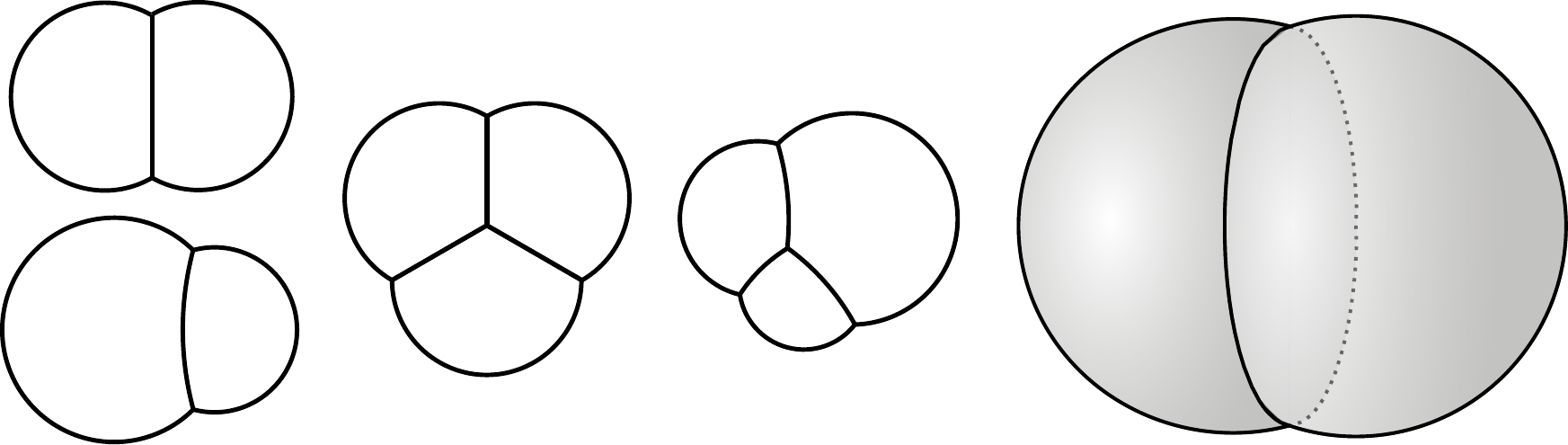}\caption{\small{Some examples of perimeter-minimizing $N$-clusters for $N=2,3$ in dimension $n=2$ and $n=3$. The $2$-clusters on the left are, respectively, the minimizer for the problem \eqref{introduction problems} with equal-volume (equal-area) chambers  $v_1=v_2$ and the minimizer for the problem \eqref{introduction problems} when different volumes $v_1\neq v_2$ have been assigned. The same situation is the central one, for $N=3$, while the right-hand picture is the perimeter-minimizing $2$-cluster for equal-volume chambers in $\R^3$.} }\label{esempiminimalcluster}
\end{figure}
Since the chambers of a perimeter-minimizing $N$-cluster will try to share as much boundary as possible these objects in general will presents some ``angle point" that we call singularity. The collection of the singularity of an $N$-cluster $\E$ is called \textit{singular set} and is usually denoted by $\S(\E)$. It is worth to remark that a complete characterization of the singular set of a perimeter-minimizing $N$-cluster, so far, is known only in dimension $n=2$ (see \cite{Morgan}) and $n=3$ (see \cite{ta76}), depicted in Figure \ref{figura singolarita}. Also a precise characterization of the minimizers is well-known only for few values of $N$. Essentially the ones depicted in Figure \ref{esempiminimalcluster} are, so far, the only perimeter-minimizing known. The case $N=2$ in the plane was solved in 1993 by Foisy,  Alfaro, Brock, Hodges, Zimba and Jasonin in \cite{foisy1993standard} while a proof for the case $N=3$ was obtained by Wichiramala and it appeared first in 2002 in \cite{wichiramala2002planar}. The case $N=2$ in the space was solved in 2002 by Hutchings, Morgan, Ritorè and Ros in \cite{hutchings2002proof}. The proof for $N=2$ in higher dimension was obtained by Reichardt in \cite{reichardt2008proof} as a generalization of the proof given by Hutchings, Morgan, Ritorè and Ros for the 3-dimensional case. In every situation listed above it has been proved that the minimizer is unique up to an isometry of the space.  In 2002 Cicalese, Leonardi and Maggi in \cite{cicalese2012sharp} proves what is called a \textit{quantitative inequality} for the case $N=2, n=2$. They showed, by exploiting that every standard double bubble $\mathcal{B}$ is the only minimizer for the problem relative to its own volumes (problem \eqref{introduction problems} with $v_1=|\mathcal{B}(1)|, v_2=|\mathcal{B}(2)|$), that every other $2$-cluster $\E$ having the same volumes of $\mathcal{B}$ and perimeter close to $P(\mathcal{B})$ must be diffeomorphic and close (in a suitable sense) to $\mathcal{B}$. \\

\begin{figure}
\centering
\includegraphics[scale=0.7]{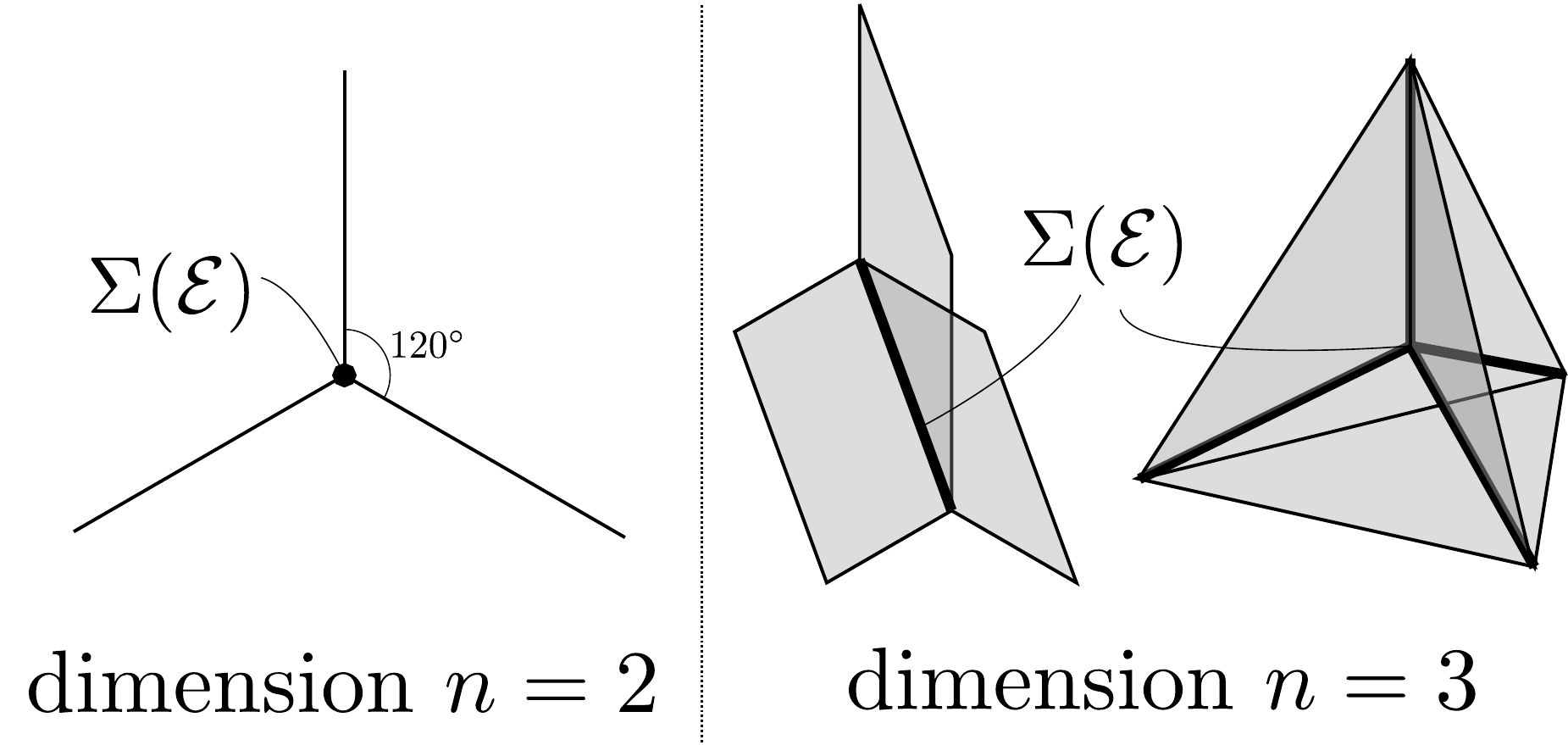}\caption{\small{The singular set of a perimeter-minimizing $N$-cluster in dimension $n=2$ and $n=3$. In dimension $n=2$ each singular point is an isolated point where three curves meet in three at equal angles. In dimension $n=3$ the singular set $\S(\E)$ consists of Hölder continuous curves along which three sheets of the surface meet in three at equal angles together with isoleted points at which four of such curves meet.}}\label{figura singolarita}
\end{figure}
Let us briefly expose what are the topics treated in each chapter. We do not focus on the details, since every chapter has its own introduction where the main questions are clarified and exposed. We limit to give a brief overview.\\

In Chapter Two we provide an asymptotic result concerning perimeter-minimizing $N$-clusters with fixed boundary. Since the detailed study of perimeter-minimizing $N$-clusters for a fixed value of $N$ seems to be a hard task, it could make sense to approach the problem from an asymptotic point of view, namely: is there some recognizable trend in the structure of these objects as $N$ approaches $+\infty$? In 2001 in \cite{hales}, Thomas Hales provided a proof of the hexagonal honeycomb conjecture: the regular hexagonal tiling (a tiling can be viewed as an $\infty$-cluster) provides the only partition of the plane in equal-area chambers having the minimum amount of localized perimeter (see Figure \ref{tilings}). His result provides an answer for the case $N=\infty$ and it turns out to be a powerful instrument for the study of the asymptotic behavior of perimeter-minimizing planar N-clusters.\\
\begin{figure}
\centering
\includegraphics[scale=0.8]{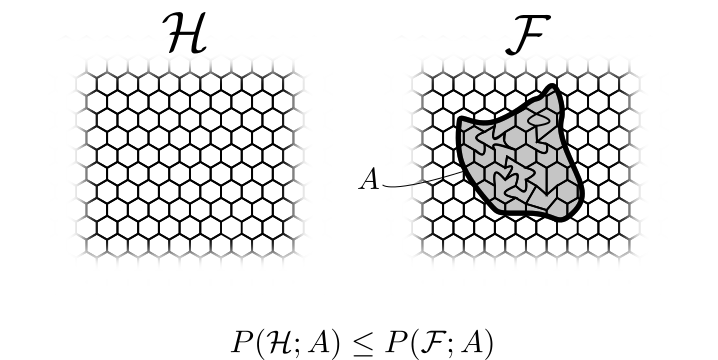}\caption{\small{The hexagonal honeycomb theorem (proved in 2001) states that the hexagonal tiling minimizes the localized perimeter among all its own compactly supported and mass preserving perturbations. }}\label{tilings}
\end{figure}

Starting from Hales's result, it is natural to expect that the interior chambers of a minimizing planar $N$-cluster with equal-volume chambers try to get closer and closer to regular hexagons as $N$ increase. However, we are still quite far from proving this fact. Another interesting question involves the external chamber of a minimizing planar $N$-cluster and it appears in \cite{HepMor05}: does the boundary of the external chamber try to look like a circle (the isoperimetric profile for the case $N=1$ in the plane) in order to save perimeter? We postpone this discussion to Chapter Two where we examine in depth questions concerning the global and local shape of perimeter-minimizing $N$-clusters. In particular, we provide a \textit{uniform distribution-type theorem}, in the spirit of the one obtained in \cite{AlChOt09}, stating that, under some reasonable assumption on the structure of these objects and far away from the boundary of the external chamber the localized perimeter is uniformly distributed. Moreover we show that the localized perimeter is equal to the localized perimeter of the hexagonal tiling, up to a remainder that is a second order term. This seems to suggest that from an energetic point of view, the interior chambers of these objects are close to regular hexagons. This result was obtained in collaboration with prof. Giovanni Alberti.\\

Chapter Three is devoted to a quantitative version of the hexagonal honeycomb theorem. We show that if $\E$ is a compactly supported and mass preserving perturbation of the hexagonal tiling $\H$ and its localized perimeter is close to the localized perimeter of $\H$ then $\E$ must be diffeomorphic and close (in a suitable sense) to $\H$. This result is obtained by exploiting the techniques developed by Cicalese, Leonardi and Maggi in \cite{CiLeMaIC1} (starting from an idea contained in \cite{CL10}) and used by the authors to prove the sharp quantitative inequality for the planar double-bubble. This result was obtained in collaboration with prof. Francesco Maggi while I was visiting the University of Texas at Austin Fall 2014. This result, if combined with the energetic estimates contained in Chapter two, seems to suggest that the interior chambers of a perimeter-minimizing planar $N$-cluster with equal-volume chambers are close to regular hexagons, providing an answer to the initial question involving the asymptotic trend of these objects. The main obstacle that arises in developing this argument is that we actually have the strong information involving the shape of the chambers only when we deal with tilings and, in order to apply the quantitative version of the hexagonal honeycomb theorem, we need to be able to ``convert" a cluster into a compactly supported perturbation of an hexagonal tiling without adding too much perimeter. \\

In Chapter Four we move to another type of isoperimetric problem concerning $N$-clusters that can be viewed as a generalization of the Cheeger problem (see \cite{Ch70}, \cite{Pa11} and \cite{Leo15} for more details about Cheeger problem). Given an open set $\Om$ we look for the solution to the variational problem
	\[
	H_N(\Om)=\inf\left\{\sum_{i=1}^N\frac{P(\E(i))}{|\E(i)|} \ \Big{|} \ \E\subseteq \Om, \ \text{$\E$ $N$-cluster} \right\}.
	\]
This variational problem turns out to be related to spectral problem of the Laplacian with Dirichlet boundary conditions. We mainly focus on the regularity of the solution to this kind of problems in order to lay the basis for future investigations (see Chapter Five where some interesting directions of research in these topics are briefly exposed). The structure of these objects is slightly different from the one of perimeter-minimizing $N$-clusters. The reason is that there is no advantage, in this variational problem, in sharing boundary and thus the chambers will try to separate as much as possible (they are constrained into $\Om$). However $\frac{P(\cdot)}{|\cdot|}$ is not scale invariant and, in particular, it makes convenient to have chambers as big as possible. Hence there are two factors that compete in opposite directions leading to non-trivial solutions. These facts imply that the boundary of each chamber is locally a $C^1$ surface inside $\Om$ and that no ``angle points" (in the planar case) are attained. After we have discussed in detail the regularity of these objects we move to study their asymptotic behavior as $N$ approaches $+\infty$. It is reasonable to expect some kind of periodicity in the asymptotic trend of these objects as $N$ increases.  \\

In Chapter Five we briefly point out how these topics could be related and we highlight some interesting issues related to these questions.\\

Let us conclude by saying that, probably, the main reason for which mathematicians are attracted by an isoperimetric problem relies in the wonderful symmetries that arise in seeking a solution. We cannot be emotionless as we become acquainted with a wonderful structure and every human being, rich or poor, quieten his incessant research of perfection standing in front of the symmetry. Probably we will never know whether God made men or not, but what we know for sure is that Symmetry made them equal.\\

{\bf\noindent Acknowledgments} The work of the author was partially supported by the project 2010A2TFX2-Calcolo delle Variazioni, funded by\textit{ the Italian Ministry of Research and University} and by NSF Grant DMS-1265910. He also acknowledges the hospitality at the UT Austin during the 2014 Fall semester where part of the present work has been done. 

\newpage
\begin{center}
{\bf\noindent {\large Special acknowledgments}}
 \end{center}
 
I wish to thank professors \textit{Giovanni Alberti} and \textit{Francesco Maggi} for the interesting discussions about this subject as well as for the support that they gave me along my path through all these delicate questions: professor \textit{Giovanni Alberti} helped me to develop the argument contained in Chapter Two while with professor \textit{Francesco Maggi} I had the pleasure to study the topic to which Chapter Three is devoted. \\ 

Thanks to \textit{Enea Parini} for the useful observations and for the interesting discussion that we have had in Levico about the topic contained in Chapter Four. \\

I am grateful to both the referees \textit{Gian Paolo Leonardi} and \textit{Aldo Pratelli} for their useful remarks and comments.\\

A special thanks goes also to \textit{Fulvio Gesmundo} that has read many different parts of this work and has helped me to edit it with a lot of interesting notes.\\

I wish to thank all the beautiful people that I have had the privilege to meet in these three years and who made my time in Pisa really enjoyable: \textit{Filippo Cerocchi}, my flatmates \textit{Augusto Gerolin}, \textit{Iga Zatorska}, \textit{Simone Bastreghi}, \textit{Fabio Tonini}, \textit{Danilo Calderini} and \textit{Ivan Martino}. \\

I would like to thank \textit{Virginia Billi}, \textit{Alex Papini}, \textit{Ilaria Viviani} and \textit{Ruzbeh Hadavandi} with whom I have had the chance to spend a wonderful time in Florence.\\

Last but not least, I wish to thank all those people that have always been there since a long time: \textit{Simone Naldi}, \textit{Michela Di Giannantonio}, \textit{Niccolò Ristori}, \textit{Niccolò Basetti Sani Vettori (Cencetti)}, \textit{Sara Mancigotti}, \textit{Martina Bellinzona} and \textit{Tommaso Poli}. And, of course, my brother \textit{Riccardo}, my father \textit{Giuseppe} and my mother \textit{Lucia}.\\

My Ph.D thesis would have never been completed without the infinite support of all these people or without their friendship.\\

\chapter{Sets of finite perimeter and $N$-clusters}\label{cpt Sets of finite perimeter and $N$-clusters}
In this chapter we define the general context of the theory of sets of finite perimeter without entering into the details of the proofs. We briefly recall the basic concepts about sets of finite perimeter we are going to use in the sequel. Here and in the sequel every set $E\subset \R^n$ will always be a Borel set. We denote, as usual, with $\overline{E}$, $\mathring{E}$ and $\pa E$ respectively the \textit{interior}, the \textit{closure} and the \textit{topological boundary} of the set $E$. We write $E\cc F$ and say $E$ is \textit{compactly contained} in $F$ if $\overline{E}\subset F$. \\ 
\text{}\\
The proof of the results that are recalled in this section, besides more details about sets of finite perimeter and Radon measures, can be found in \cite{Morgan}, \cite{maggibook}, \cite{Federerbook}. The original works \cite{DeGiorgiSOFP1} and \cite{DeGiorgiSOFP2} from Ennio De Giorgi, where the foundational part of the theory of sets of finite perimeter is developed, are in italian. The english versions of such works can be found in the book \cite{Selectedpapers} at pp. 58-78 and 112-127. 

\section{Radon measures}
The concept of Radon measures, more precisely of vector-valued Radon measures, plays a key role in the theory of sets of finite perimeter. We do not need to explain in detail what a vector-valued Radon measure is (for a complete overview on such a topic we refer to \cite{maggibook} pp. 1-62) and so we just recall  that vector measures can be represented as positive measures multiplied by a (summable) vector-valued density.
\subsection{Definition of Radon measures}\label{chpt 1 subct: Radon measures}
A measure $\mu:\mathcal{P}(\R^n)\rightarrow \R^+$ is a \textbf{positive Radon measure on $\R^n$} (or simply a \textbf{Radon measure}) if 
\begin{itemize}
\item[a)] any Borel set $E$ is a $\mu$-measurable set;
\item[b)] for any set $F\subset \R^n$ there exists a Borel set $E\supseteq F$ such that $\mu(E)=\mu(F)$;
\item[c)] $\mu(K)<+\infty$ for every $K\subset\R^n$ compact. 
\end{itemize} 
Property a) is ensuring that the family of all $\mu$-measurable sets will be not trivial. Property b) gives us some sort of regularity, since it allows us to consider just the Borel's algebra, while property c) guarantees the local finiteness of $\mu$. \\

We say that a measure $\mu:\mathcal{P}(\R^n)\rightarrow \R^m$ is an \textbf{$\R^m$-valued Radon measure} (we sometimes simply write vector-valued Radon measure) if there exists a positive Radon measure $\mu_0$ and Borel map $f:\R^n\rightarrow \R^m$ with $|f(x)|=1$ $\mu_0$-almost everywhere such that
$$\mu(E)=\int_E f(x) d\mu_0(x)$$
for every Borel set $E\subset \R^n$. 
Given a vector-valued Radon measure $\mu$, the measure $\mu_0$ associated to $\mu$ is uniquely identified and the \textbf{total variation of $\mu$}:
\begin{equation}\label{variazione totale}
|\mu|:=\mu_0,
\end{equation}
is well defined. Since also the density $f$ is unique up to a $|\mu|$-negligible set, in the sequel we are always adopting the notation 
$$\mu=f|\mu|.$$
Note that $|\mu|=0$ implies $\mu=0$. Given a Radon measure $\mu$ we define the \textbf{support of $\mu$} as the set
\begin{equation}\label{eqn: chapter intro support}
\spt(\mu)=\left\{x\in \R^n \ | \ |\mu|(B_r(x))>0 \ \ \text{for all $r>0$}\right\}.
\end{equation}

Every function $h\in L_{loc}^1(\R^n)$ induce a positive Radon measure on $\R^n$ by defining
$$\mu(E)=\int_E |h(x)|\d x,$$
for every Borel set $E$. In this case we write 
$$\mu=h\L^n.$$ 
In particular this implies that, having defined \textit{the characteristic function of a Borel set $E$} as
\begin{equation}\label{definizione funzione caratteristica}
\ca_E(x):=\left\{ 
\begin{array}{cc}
1 \ \ &\text{if $x\in E$};\\
0 \ \ &\text{otherwise},
\end{array}
\right. 
\end{equation}
if $h(x)=\ca_{E}(x)$ for some $E$ with $\L^n(E)=|E|<+\infty$, then the measure 
$$\L^n\llcorner E(F)=\L^n(E\cap F)=(\ca_E\L^n)(F)$$
is a positive Radon measure. \\

In general, if $\mu$ is a positive Radon measure and $h\in L_{loc}^1 (\R^n,\mu)$ is a function, we have that $\nu=|h|\mu$ is a positive Radon measure. In particular if $E$ is a Borel set with $\mu(E)<+\infty$ we have that \textit{the restriction of $\mu$ to $E$} defined on every Borel set $F$ as
$$\mu \llcorner E(F)=\mu(E\cap F)=(\ca_E\mu)(F)$$
is a positive Radon measure.\\

An example of vector-valued Radon measure can be obtained by setting $|\mu|=\mu_0=\L^n$ and by choosing a generic Borel vector field $f:\R^n\rightarrow \R^m$ with $|f(x)|=1$ almost everywhere. Note that if $g:\R^n\rightarrow \R^m$ is a Borel vector field, then $\mu=g\L^n$ is an $\R^m$-valued Radon measure. Indeed, by defining
\begin{eqnarray*}
|\mu|&=& |g(x)|\L^n,\\
f(x)&=&\frac{g(x)}{|g(x)|}\ca_{\{ g(x)\neq 0\}}(x),
\end{eqnarray*}
we have that
$$\mu=f|\mu|$$
with $|f(x)|=1$ for $|\mu|-$almost every $x\in \R^n$.

\subsection{Weak-star convergence of Radon measures}\label{subsection Weak-star convergence of Radon measures}
In order to speak of compactness and semi-continuity of perimeter we need to briefly introduce the weak-star convergence of Radon measures. A sequence of Radon measures $\{\mu_h\}_{h\in \N}$  on $\R^{n}$ with values in $\R^{m}$ is said to be \textbf{convergent in the weak-star sense} to a Radon measure $\mu$, and we write $\mu_h\displaystyle \frecciad^{*} \mu$, if and only if for every $\varphi\in C_c^{0}(\R^{n};\R^{m})$ it holds
$$\int_{\R^{n}} \varphi\cdot \d \mu=\lim_{h\rightarrow +\infty} \int_{\R^{n}} \varphi \cdot \d \mu_h.$$
The following equivalences about convergence of positive Radon measures are very useful, (see \cite[Proposition 4.26]{maggibook} for a detailed proof).
\begin{proposition}\label{on the weak star convergence}
If $\{\mu_h\}_{h\in\N}$ and $\mu$ are positive Radon measures on $\R^{n}$, then the following three statements are equivalent.
\begin{itemize}
\item[(i)] $\mu_h\displaystyle \frecciad^{*} \mu$.
\item[(ii)] If $K$ is compact and $A$ is open, then 

\begin{equation}
\mu(K) \geq \limsup_{h\rightarrow +\infty} \mu_h(K),
\end{equation}

\begin{equation}
\mu(A) \leq \liminf_{h\rightarrow +\infty} \mu_h(A).
\end{equation}
\item[(iii)] If $E$ is a Borel set with $\mu(\pa E)=0$, then
$$\mu(E)=\lim_{h\rightarrow +\infty}\mu_h(E).$$
\end{itemize}
\end{proposition}

\section{Sets of finite perimeter}

\subsection{Hausdorff measures and Hausdorff dimension}
For every $s,\de>0$ the $s$-dimensional Hausdorff measure of step $\de$ of a set $E\subset \R^n$ is defined as:
\begin{equation}\label{hasudorff measure delta}
\H^s_{\de}(E)=\inf_{\F_{\de}}\left\{\sum_{F\in \F_{\de}} \om_s \left(\frac{\diam(F)}{2} \right)^s \right\}
\end{equation}
where 
$$\om_s=\frac{\pi^{s/2}}{\G(1+s/2)},  \ \ \ \ \G(s)=\int_{0}^{+\infty} t^{s-1}e^{-t}\d t $$
and where the infimum in \eqref{hasudorff measure delta} is taken among all $\F_{\de}$, namely countable coverings of $E$ by Borel sets $F\subset \R^n$ with $\diam(F)\leq \de$. If $s=k$ is an integer then $\om_k$ is exactly the Lebesgue measure of a $k$-dimensional ball in $\R^k$. The \textbf{$s$-dimensional Hausdorff measure of a set $E\subset \R^n$} is then defined as:
\begin{equation}
\H^s(E):=\sup_{\de>0}\{\H^s_{\de}(E)\}=\lim_{\de\rightarrow 0^+}\H^{s}_{\de}(E).
\end{equation}
From the definition it follows that the Hausdorff measure $\H^s$ is invariant under isometries and that
$$\H^s(\l E)=\l^s\H^s(E), \ \ \ \ \forall \ \l,s>0, \ E\subset \R^n.$$
Furthermore the following properties hold:
\begin{itemize}
\item[1)] $\H^n(E)=\L^n(E)$ for every $E\subset \R^n$;
\item[2)] $\H^s(E)<+\infty$ implies $\H^t(E)=0$ for every $t>s$;
\item[3)] $\H^s(E)>0$ implies $\H^t(E)=+\infty$ for every $t<s$.
\end{itemize} 
Thanks to property 2) and 3) above it is well defined the \textbf{Hausdorff dimension of a Borel set $E$} as
\begin{eqnarray}
\dim(E)&:=&\inf\{s\in [0,+\infty] \ | \ \H^s(E)=0\}\label{Hausdorff dimension1}\\
&=&\sup\{s\in [0,+\infty] \ | \ \H^s(E)=+\infty\}\label{Hausdorff dimension2}.
\end{eqnarray}
We underline that if $1\leq k\leq n-1$ and $M$ is a $k$-dimensional $C^1$-surface in $\R^n$ then $\H^k(M)$ coincides with the classical $k$-dimensional area of $M$ and dim$(M)=k$ (we refer the reader to \cite[Chapter 3]{maggibook}). In the sequel, whenever we talk about the \textit{dimension of a set $E$} we are always meaning the Hausdorff dimension of the set $E$. \\

Let us point out that property 1) and 3) tells us that $\H^s$ is not a Radon measure in $\R^n$ unless $s\geq n$ (and in this case, for $s>n$ it is trivial thanks to property 2) ). Indeed $\H^s(A)=+\infty$ for every $s<n$ and every open set $A\subset \R^n$. Anyway, if $E$ is such that $\H^s(E)<+\infty$ the measure  $\H^s\llcorner E$, given by the restriction of  $\H^s$ to $E$, is a Radon measure on $\R^n$. 
\subsection{$L^{1}$ topology} \label{subsection $L^1$ topology}
Given a subset $\Om\subseteq \R^{n}$ we need first to specify the topology that we are considering on the Borel's algebra of $\Om$. The correct one for this framework is the one induced by the $L^{1}_{loc}$ convergence of the characteristics function. More precisely a sequence of Borel sets $\{E_h\}_{h\in \N}$ is converging to a set $E$ (in $L^{1}_{loc}$) if and only if:
\begin{equation*}\label{easy}
\ca_{E_h} \freccia^{L^{1}_{loc}} \ca_E\, 
\end{equation*}
or equivalently if for every compact set $K\subset\Om$ it holds
\begin{equation}\label{peasy}
 \lim_{h\rightarrow +\infty} |(E\Delta E_h) \cap K|\rightarrow 0.
\end{equation}
Clearly, if the convergence of the characteristic functions is stronger, say $L^{1}$, we speak of $L^{1}$ convergence instead of $L^{1}_{loc}$ and \eqref{peasy} becomes just
\begin{equation*}
 \lim_{h\rightarrow +\infty} |E\Delta E_h|\rightarrow 0.
\end{equation*}

\subsection{Sets of finite perimeter and Gauss-Green measure}\label{subsection Sets of finite perimeter and Gauss-Green measure}
A Borel set $E$ of $\R^{n}$ is said to be a \textit{set of locally finite perimeter} if there exists an $\R^n$-valued Radon measure $\mu_E$ such that:
\begin{equation}\label{sofp definizione}
\int_E \dive(T)\d x=\int_{\R^{n}} T\cdot \d \mu_E, \ \ \forall \ T\in C_c^{1}(\R^{n};\R^{n}).
\end{equation}
Notice that \eqref{sofp definizione} just means that the characteristic function of $E$ admits as distributional derivative the vector-valued Radon measure $\mu_E$. In other words $D\ca_E(x)=\mu_E$ in the sense of distributions.  
The measure $\mu_E$ is also called the \textit{Gauss-Green measure of $E$} and we define the relative perimeter of $E$ in the Borel set $F\subset \R^n$:
\begin{equation}\label{perimetro relativo}
P(E;F)=|\mu_E|(F),
\end{equation} 
where $|\mu_E|(F)$ denotes the total variation of $\mu_E$ defined in \ref{chpt 1 subct: Radon measures}, formula \eqref{variazione totale}.\\

The \textit{perimeter of a set $E$} is defined as 
$$P(E):=P(E;\R^{n}).$$ 
The reason why $\mu_E$ is called Gauss-Green measure is that whenever $E$ is a set with $C^{1}$ boundary, 
the Gauss-Green Theorem implies
$$\mu_E=\nu_E\H^{n-1}\llcorner \pa E $$
where $\nu_E$ denotes the outer unit normal of $\pa E$. 
Notice that in this case $P(E;F)=\H^{n-1}(\pa E\cap F) $, $P(E)=\H^{n-1}(\pa E)$. \\

By exploiting \eqref{sofp definizione} and \eqref{perimetro relativo} we reach also the useful alternative definition of relative perimeter 
\begin{equation}\label{perimetro definizione vera aperti}
P(E;A)= \sup\left\{\int_E \dive(T) \d x \ \Big{|} \ T\in C_c^{1}(A;B_1) \right\}, 
\end{equation}
when $A$ is open and
\begin{equation}\label{perimetro definizione vera borel}
P(E;F)= \inf\left\{P(E;A) \ | \  \text{A open and }F \subseteq A\right\}, 
\end{equation}
when $F$ is a generic Borel set.

\subsection{An equivalent definition of sets of finite perimeter}\label{subsection Equivalent definitions of sets of finite perimeter}
We sometimes make use of an equivalent definition of sets of finite perimeter introduced first by De Giorgi in \cite{DeGiorgiSOFP1} by exploiting regularizing kernels. More precisely, having defined $E_{\e}(x):=\ca_E\star \rho_{\e}(x)$, where $E$ is a given a Borel set and $\{\rho_{\e}(x)\}_{\e>0}$ is a regularizing kernel, if $E$ has locally finite perimeter then
\begin{equation}\label{perimetro approssimanti}
-(\nabla E_{\e}) \L^{n} \frecciad^{*} \mu_E\,,\ \  \ \ |\nabla E_{\e}|\L^{n}\frecciad^{*} |\mu_E|
\end{equation}
and conversely if $E$ is such that 
\begin{equation}\label{perimetro approssimanti2}
\limsup_{\e \rightarrow 0} \int_K|\nabla E_{\e}(x)| \d x<\infty \ \ \text{for all compact sets $K$}
\end{equation}
then $E$ has locally finite perimeter. 
\subsection{Compactness and semicontinuity with respect to the $L^{1}$ topology}\label{subsection Compactness and semicontinuity with respect to the $L^1$ topology}
In order to ensure existence of solutions in many variational problems we need a suitable compactness property of finite perimeter sets together with the semi-continuity of the functional perimeter (see \cite[Proposition 12.15, Theorem 12.26]{maggibook} for detailed proofs).
\begin{theorem}[Compactness theorem for sets of finite perimeter]\label{compactness theorem}
Let $\{E_h\}_{h \in \N}$ be a sequence of sets of finite perimeter such that
\begin{eqnarray*}
a)& & \sup_{h\in\N} \{P(E_h)\}<+\infty \\
b)& & \text{there exists $R>0$ such that}\ \  E_h\subset B_R \text{ for all $h\in \N$}.
\end{eqnarray*}
Then, there exists a subsequence $\{E_{h_j}\}_{j\in \N}\subseteq \{E_h\}_{h\in \N}$ and a set of finite perimeter $E\subset B_R$ such that
$$E_{h_j}\freccia^{L^{1}} E\,, \ \ \ \ \ \ \mu_{E_{h_j}}\frecciad^{*} \mu_E.$$
\end{theorem}
\begin{theorem}[Lower semicontinuity of the perimeter]\label{semicontinuity theorem}
If $\{E_h\}_{h\in \N}$ is a sequence of sets of locally finite perimeter in $\R^{n}$ such that
$$E_h\freccia^{L^{1}_{loc}} E \,, \ \ \ \ \ \limsup_{h\rightarrow +\infty} P(E_h;K)< +\infty$$
for every compact set $K$ in $\R^{n}$, then $E$ is of locally finite perimeter in $\R^{n}$, $\mu_{E_h}\frecciad^{*}\mu_E$ and, for every open set $A\subset \R^{n}$ we have
\begin{equation}\label{semicontinuity formula}
P(E;A)\leq \liminf_{h\rightarrow +\infty} P(E_h;A).
\end{equation}
\end{theorem}

\subsection{The structure of the Gauss-Green measure}\label{subsection The structure of the Gauss-Green measure and Federer's Theorem}
For every set $E$ of locally finite perimeter the \textbf{reduced boundary} $\pa^{*} E$  is defined as the set of points $x\in \spt\mu_E$ such that the limit
\begin{equation}
\lim_{r\rightarrow 0} \frac{\mu_E(B_r(x))}{|\mu_E|(B_r(x))} \ \ \ \text{exists and belongs to $S^{n-1}=\pa B_1$}.
\end{equation}
For every point $x\in \pared E$ we set:
$$\nu_E(x):=\lim_{r\rightarrow 0} \frac{\mu_E(B_r(x))}{|\mu_E|(B_r(x))}.$$
The vector field $\nu_E$ is called \textbf{measure-theoretic outer unit normal to $E$} and by the Besicovitch-Lebesgue differentiation theorem we have that
$$\mu_E=\nu_E|\mu_E| \llcorner \pared E.$$ 
Note that if $E$ is a set of finite perimeter with reduced boundary $\pared E$ then
\begin{align*}
 \pared E^c&=\pared E,\\
\nu_{E^c}(x)&=-\nu_E(x), \ \ \ \ \forall \ x\in \pared E^c.
\end{align*}
A key tool in the whole theory of sets of finite perimeter is the following theorem due to De Giorgi about the structure of the Gauss-Green measure (see \cite{DeGiorgiSOFP2}, \cite[pp. 111-127]{Selectedpapers}, \cite[Theorem 15.5, Theorem 15.9]{maggibook}).
\begin{theorem}[De Giorgi's structure Theorem] 
If $E$ is a set of locally finite perimeter in $\R^{n}$, then the following properties hold. 
\begin{itemize}
\item[1)] The Gauss-Green measure $\mu_E$ of $E$ satisfies 
\begin{equation}\label{structure of Gauss-Green measure}
 |\mu_E|=\H^{n-1}\llcorner \pared E\,, \ \ \ \ \mu_E=\nu_E \H^{n-1}\llcorner \pared E\,, 
\end{equation}
and the generalized Gauss-Green formula holds true:
\begin{equation}\label{generalized Gauss-Green formula}
\int_E \nabla \varphi \d x=\int_{\pared E} \varphi\nu_E\d \H^{n-1}\,, \ \ \ \ \forall \ \varphi\in C_c^{1}(\R^{n});
\end{equation}
\item[2)] There exists countably many $C^1$-hypersurfaces $\{M_h\}_{h\in \N}\subset \R^n$, compact sets $K_h\subset M_h$ and a Borel set $F$ with $\H^{n-1}(F)=0$ such that 
$$\pared E=F\cup \bigcup_{h\in \N} K_h\,,$$
and for every $x\in K_h$, $\nu_E(x)^{\perp}=T_x M_h$ is the tangent space of $M_h$ at $x$;
\item[3)] For every $x\in \pared E$ the sequence of sets $\left\{E_{x,r}=\frac{E-x}{r}\right\}_{r>0}$ locally converges, (as \mbox{$r\rightarrow 0^+$}), to the half space
$$H_{\nu_E(x)}:=\left\{y\in \R^n \ | \ y\cdot \nu_E(x)\leq 0\right\}$$
and it holds:
$$\mu_{E_{x,r}}\frecciad^{*} \nu_E(x) \H^{n-1}\llcorner \pa H_{\nu_E(x)}, \ \ \ |\mu_{E_{x,r}}| \frecciad^{*} \H^{n-1}\llcorner \pa H_{\nu_E(x)}.$$ 
\end{itemize}
\end{theorem}
\subsection{Essential boundary}
The $n$-dimensional density of a set $E$ at the point $x$ is the quantity
\begin{equation}\label{density of a point}
\vt_n(x,E)=\lim_{r\rightarrow 0}\frac{|E\cap B_r(x)|}{|B_r(x)|},
\end{equation}
whenever it exists. We notice that, thanks to the Besicovitch-Lebesgue differentiation theorem applied to the Radon measure $\L^{n}\llcorner E $, the limit in \eqref{density of a point} exists for almost every $x$ in $\R^{n}$. Given a set $E$ we can define the set of points of $\R^{n}$ having the same $n$-dimensional density $t\in[0,1]$:
$$E^{(t)}=\{x\in \R^{n} \ | \ \vt_n(x,E)=t\}.$$
Note that $E^{(0)}=(\R^n\setminus E)^{(1)}$.\\

By denoting with $Q_r(x)$ a cube centered at $x$ and with side-length $r$, we could have defined the $n$-dimensional density of a set $E$ at the point $x$ also as the limit
$$\bar{\vt}_n(x,E)=\lim_{r\rightarrow 0}\frac{|E\cap Q_r(x)|}{|Q_r(x)|},$$
whenever it exists. This two definitions are equivalent on the points of density $0$ and $1$. Indeed on every ball $B_r(x)$ it holds 
$$Q_{\frac{2r}{\sqrt{n}}}(x)\subset B_r(x)\subset Q_2r,$$ 
and thus
$$\vt_n(x,E)=0  \ \ \Leftrightarrow \ \ \bar{\vt}_n(x,E)=0,$$
$$\vt_n(x,E)=1 \ \ \Leftrightarrow  \ \ \bar{\vt}_n(x,E)=1.$$

However in the sequel, unless it is not specified, we are always making use of Definition \eqref{density of a point} since it is the most common one in literature.\\

With these notation the \textbf{essential boundary} $\pa^{e}E$ of a Borel set is defined as:
\begin{equation}\label{essential boundary}
\pa^{e}E:= \R^{n}\setminus\left(E^{(0)}\cup E^{(1)}\right)=\{x\in \R^{n} \ |\ 0<\vt_n(x,E)<1\}.
\end{equation}
The following theorem clarifies the relation between the essential boundary and the reduced boundary of a set of finite perimeter $E$ (see \cite[Theorem 16.2]{maggibook}).
\begin{theorem}\label{equivalence of reduced and essential boundary}
If $E$ is a set of locally finite perimeter in $\R^{n}$ then $\pared E \subset E^{(\frac{1}{2})}\subset \pa^{e}E$ and 
\begin{equation}
\H^{n-1}(\pa^{e}E\setminus \pared E)=0.
\end{equation}
\end{theorem}

A useful consequence of Theorem \ref{equivalence of reduced and essential boundary} is the following Lemma \ref{tecnico}. The proof can be obtained as a consequence of \cite[Theorem 4.1]{leonardi2002metric} or \cite[Theorem 2.4]{Leo02Partition} on the structures of the Caccioppoli partitions combined with Theorem \ref{equivalence of reduced and essential boundary}. Since Lemma \ref{tecnico} will be repeatedly used in Chapter 4 and since we have not been able to find a direct (and easy) proof of this fact in literature we provide a proof.
\begin{lemma}\label{tecnico}
If $E_1,\ldots,E_k$ are $k$ sets of locally finite perimeter such that 
$$|E_i\cap E_j|=0 \ \ \ \ \forall \ i\neq j,$$
then the following holds:
\begin{equation}\label{spaghetto}
\begin{split}
\pa^* \left(\bigcup_{i=1}^k E_i\right)&\approx \left( \bigcup_{i=1}^k \pa^*E_i\right) \setminus \left(\bigcup_{\substack{ i,j=1 \\ j\neq i}}^k \pared E_j\cap \pared E_i\right)\\
&=\left[ \bigcup_{i=1}^k \pa^*E_i \setminus \left(\bigcup_{\substack{ j=1 \\ j\neq i}}^k \pared E_j\cap \pared E_i\right)\right]
\end{split}
\end{equation}
where the symbol $\approx$ means \textit{equal up to an $\H^{n-1}$-negligible set}. In particular for every ball $B_r=B_r(x)$ it holds:
\begin{equation}\label{peri N-cluster}
P\left(\bigcup_{i=1}^k E_i ;B_r\right)=\sum_{i=1}^k P(E_i;B_r)-\sum_{\substack{i,j=1, \\ j\neq i}}^k \H^{n-1}(\pa^* E_i\cap \pa^*E_j \cap B_r). 
\end{equation}
\end{lemma}
\begin{proof}
Relation \eqref{peri N-cluster} follows straightforwardly from \eqref{spaghetto}.
We recall from Theorem \ref{equivalence of reduced and essential boundary} that $\pared E\approx E^{\left(\frac{1}{2}\right)}$ for every locally finite perimeter set $E$. Hence, by setting
$E_0=\bigcup_{i=1}^k E_i,$
it is enough to prove that there exist two $\H^{n-1}$-negligible set $M_1,M_2$ such that
\begin{equation}\label{quello che vorrei dire}
E_0^{\mez} \subseteq  M_1\cup\left[  \left(\bigcup_{i=1}^k E_i^{\left(\frac{1}{2}\right)}\right) \setminus\left(\bigcup_{\substack{i,j=1\\ j\neq i} }^k E_i^{\mez}\cap E_j^{\mez}\right)\right]\subseteq (E_0^{\mez}\cup M_2).
\end{equation}
Let us also point out that, if $E$ is a set of locally finite perimeter, Theorem \ref{equivalence of reduced and essential boundary} implies that there exists an $\H^{n-1}$-negligible set $R$ with following property
$$\R^n=E^{\zero}\cup E^{\mez}\cup E^{\uno}\cup R.$$
Thus, for every $i=0,\ldots,k$, we choose $R_i$ be the $\H^{n-1}$-negligible set such that
	\begin{equation}\label{STAR}
	 \R^n=E_i^{\zero}\cup E_i^{\mez}\cup E_i^{\uno}\cup R_i,
	 \end{equation}	
and we set 
$$M_1:=\left(E_0^{\mez}\cap \bigcup_{i=1}^k R_i\right), \ \ \ M_2:= \bigcup_{i=1}^k R_i.$$
We prove that \eqref{quello che vorrei dire} holds with this choice of $M_1,M_2$ (note that $\H^{n-1}(M_1)=\H^{n-1}(M_2)=0$ is immediate). Let us set, for the sake of brevity
$$F:=M_1\cup \left[ \left(\bigcup_{i=1}^k E_i^{\left(\frac{1}{2}\right)}\right) \setminus\left(\bigcup_{\substack{i,j=1\\ j\neq i} }^k E_i^{\mez}\cap E_j^{\mez}\right)\right],$$
and divide the proof in two steps.\\

\textit{Step one: $E_0^{\mez}\subseteq F$}. In particular we prove that if $x\notin F$ then $x\notin E_0^{\mez}$. For $x\notin F$ one of the following must be in force
\begin{itemize}
		\item[a)] $\displaystyle x\notin M_1 \text{\ and \ }x\in \left(\bigcup_{i=1}^k E_i^{\left(\frac{1}{2}\right)}\right)\cap \left(\bigcup_{\substack{i,j=1\\ j\neq i} }^k E_i^{\mez}\cap E_j^{\mez}\right)$.
		\item[b)] $\displaystyle x\notin M_1 \text{\ and \ } x\notin  \left(\bigcup_{i=1}^k E_i^{\left(\frac{1}{2}\right)}\right) $ and in this case either:
             		 \begin{itemize}
							\item[b.1)] $\displaystyle x \notin E_0^{\mez}  \text{\ and \ } x\in \bigcup_{i=1}^kR_i \text{\ and \ } x\notin  \left(\bigcup_{i=1}^k E_i^{\left(\frac{1}{2}\right)}\right) $;
							\item[b.2)] $\displaystyle x \notin E_0^{\mez}  \text{\ and \ } x\notin \bigcup_{i=1}^kR_i \text{\ and \ } x\notin  \left(\bigcup_{i=1}^k E_i^{\left(\frac{1}{2}\right)}\right) $;
							\item[b.3)] $\displaystyle x \in E_0^{\mez}  \text{\ and \ } x\notin \bigcup_{i=1}^kR_i \text{\ and \ } x\notin  \left(\bigcup_{i=1}^k E_i^{\left(\frac{1}{2}\right)}\right) $.
						\end{itemize}
\end{itemize}
If situation a) is in force we immediately have that $x\in E_i^{\mez}\cap E_j^{\mez}$ for some $i\neq j$ which leads to $x\in E_0^{\uno}$ (since $|E_i\cap E_j|=0$) and thus $x\notin E_0^{\mez}$. Since b.1) and b.2) implies straightforwardly $x\notin E_0^{\mez}$, we need just to verify that situation b.3) cannot be attained. Assume b.3) is in force and note that, for every $i=1,\ldots,k$, thanks to \eqref{STAR} it must hold $x\in E_{i}^{\uno}\cup E_{i}^{\zero}$. If $x\in  E_{i}^{\zero}$ for all $i$ we have $x\in E_0^{\zero}$. If, instead, $x\in E_i^{\uno}$ for some $i$ then $x\in E_0^{\uno}$. In both cases we reach a contradiction because of $x\in E_0^{\mez}$.\\

\textit{Step two: $F\subseteq (E_0^{\mez}\cup M_2)$}. For every $x\in F$ one of the following must be in force.
\begin{itemize}
\item[a)] $\displaystyle x\in M_1$;
\item[b)] $\displaystyle x\in\left(\bigcup_{i=1}^k E_i^{\left(\frac{1}{2}\right)}\right) \setminus\left(\bigcup_{\substack{i,j=1\\ j\neq i} }^k E_i^{\mez}\cap E_j^{\mez}\right) \text{\ and \ }  \displaystyle x\notin  M_1$ and in this case either:
             		 \begin{itemize}
							\item[b.1)]$\displaystyle x\in\left(\bigcup_{i=1}^k E_i^{\left(\frac{1}{2}\right)}\right) \setminus\left(\bigcup_{\substack{i,j=1\\ j\neq i} }^k E_i^{\mez}\cap E_j^{\mez}\right) \text{\ and \ }  \displaystyle x\notin \bigcup_{i=1}^k R_i$;
							\item[b.2)] $\displaystyle x\in \left(\bigcup_{i=1}^k E_i^{\left(\frac{1}{2}\right)}\right) \setminus\left(\bigcup_{\substack{i,j=1\\ j\neq i} }^k E_i^{\mez}\cap E_j^{\mez}\right) \text{\ and \ }  \displaystyle x\notin E_0^{\mez}$;
					\end{itemize}
\end{itemize}
If a) is the case, then $x\in M_1\subset E_0^{\mez}$ and we are done. If b.1) is in force then there exists exactly one $j$ such that $x\in E_j^{\mez}$ and $x\in E_{i}^{\zero} $ for $i\neq 0, j$ since the sets $\{E_h\}_{h=1}^k$ are disjoint up to an $\L^n$-negligible set. Thus
\begin{align*}
\frac{|(\R^n\setminus E_0) \cap B_r(x)|}{\om_n r^n}&=1-\frac{|E_j \cap B_r(x)|}{\om_n r^n}-\sum_{\substack{i=1, \\ i\neq j}}^k \frac{|E_i \cap B_r(x)|}{\om_n r^n},
\end{align*}
which, passing to the limit as $r$ goes to $0^+$ implies $x\in (\R^n\setminus E_0)^{\mez}=E_0^{\mez}$. Finally, by considering situation b.2) we deduce that there exists exactly one $j\in \{1,\ldots,k\}$ such that $x\in E_{j}^{\mez}$ and $x\in E_i^{\zero}\cup R_i$ for $i\neq j$. If $x\in E_{i}^{(0)}$ for all $i\neq 0, j$ then, as above $x\in E_0^{\mez}$ and this is a contradiction (in this situation we are assuming $x\notin E_0^{\mez}$). Hence there is an index $i\neq 0$ such that $x\in R_i$ which means $x\in M_2$. The proof is complete.
\end{proof}

\subsection{Topological boundary}\label{subsct: topological boundary}
If $A$ is an open set and $E$ and $F$ are sets of finite perimeter in $A$ with $|(E\Delta F)\cap A|=0$ then 
$$P(E;A)=P(F;A).$$
Indeed considered a generic map $T\in C_c^{1}(A;B_1)$, by exploiting definition \eqref{perimetro definizione vera aperti}, we have
\begin{eqnarray*}
P(E;A)&\geq & \int_{E} \dive(T)\d x=\int_{E\cap A} \dive(T)\d x= \int_{F\cap A} \dive(T)\d x= \int_{F} \dive(T)\d x.
\end{eqnarray*}
By taking the supremum among all $T\in C_c^1(A;B_1)$ we conclude $P(E;A)\geq P(F;A)$. The reverse inequality follows in the same way. Hence the distributional perimeter of a set $E$ depends only on its $L^{1}$ equivalence class.\\

In particular this implies that the $L^{1}$ equivalence class of a set of finite perimeter contains a lot of set with very irregular topological boundary. For example if $E$ is a set of finite perimeter, we can always find another set $F$ with $\pa F=\R^n$ and $|E\Delta F|=0$ so that $P(E)=P(F)$. The following proposition is what we need for select a ``good" representative (see \cite[Proposition 12.19]{maggibook}).
\begin{proposition}\label{choice}
If $E$ is a set of locally finite perimeter in $\R^{n}$, then 
$$\spt(\mu_E)=\left\{x\in \R^{n} \ | \ 0<|E\cap B_r(x)|<\om_n r^{n},  \  \ \forall \ r>0\right\}\subset \pa E.$$
Moreover there exists a Borel set $F$ such that
\begin{equation}\label{choice equation}
|E\Delta F|=0, \ \ \ \ \ \spt(\mu_F)=\pa F.
\end{equation}
\end{proposition}
By \eqref{choice equation} the set $F$ given in Proposition \ref{choice} has perimeter equal to $P(E)$ and has a precise characterization of its topological boundary.\\

In the sequel, whenever we speak of a set of finite perimeter $E$, we implicitly assume $E$ to be a representative of its own $L^{1}$ equivalence class satisfying $\spt(\mu_E)=\pa E$.
\subsection{Union, intersection and difference of finite perimeter sets}\label{sbst Union, intersection, differences}
Let $E$ and $F$ be sets of locally finite perimeter. Then the intersection $E\cap F$, the union $E\cup F$ and the difference $E\setminus F$, $F\setminus E$ are sets of locally finite perimeter and the following properties hold:
\begin{eqnarray*}
\mu_{E\cap F}&=&\mu_E \llcorner F^{(1)} + \mu_F \llcorner E^{(1)} + \nu_E\H^{n-1} \llcorner \{\nu_E=\nu_F\}\\
\mu_{E\cup F}&=&\mu_E \llcorner F^{(0)} + \mu_F \llcorner E^{(0)} + \nu_E\H^{n-1} \llcorner \{\nu_E=\nu_F\}\\
\mu_{E\setminus F}&=&\mu_E \llcorner F^{(0)} - \mu_F \llcorner E^{(1)} + \nu_E\H^{n-1} \llcorner \{\nu_E= -\nu_F\},
\end{eqnarray*}
where
$$\{\nu_E=\nu_F\}=\{x\in \pared E\cap \pared F \ | \ \nu_E(x)=\nu_F(x)\}\,,$$
$$\{\nu_E=-\nu_F\}=\{x\in \pared E\cap \pared F \ | \ \nu_E(x)=-\nu_F(x)\}\,.$$
Moreover the reduced boundaries satisfy
\begin{align}
\pared (E\cap F)&\approx (F^{(1)} \cap \pared E)\cup (E^{(1)} \cap \pared F) \cup \{\nu_E=\nu_F\} \label{eqn: frontiera ridotta intersezione}\\
\pared (E\cup F)&\approx (F^{(0)} \cap \pared E)\cup (E^{(0)} \cap \pared F) \cup \{\nu_E=\nu_F\}\label{eqn: frontiera ridotta unione}\\
\pared (E\setminus F)&\approx (F^{(0)} \cap \pared E)\cup (E^{(1)} \cap \pared F) \cup \{\nu_E= -\nu_F\},\label{eqn: frontiera ridotta differenza}
\end{align}
where ``$\approx$" means \textit{equal up to an $\H^{n-1}$-negligible set}. It follows that, for every Borel set $G\subseteq \R^n$,  the following hold:

\begin{align}
P(E\cap F;G)&=P(E; F^{(1)} \cap G)+ P(F;E^{(1)}\cap G)+\H^{n-1}(\{\nu_E=\nu_F\}\cap G)  \label{eqn: perimetro intersezione} \\
P(E\cup F;G)&=P(E; F^{(0)} \cap G)+ P(F;E^{(0)}\cap G)+\H^{n-1}(\{\nu_E=\nu_F\}\cap G)\label{eqn: perimetro unione}\\
P(E\setminus F;G)&=P(E; F^{(0)} \cap G)+ P(F;E^{(1)}\cap G)+\H^{n-1}(\{\nu_E= -\nu_F\}\cap G).\label{eqn: perimetro differenza}
\end{align}
We refer the reader to \cite[Theorem 16.3]{maggibook} for the proof of these assertions.

\subsection{Indecomposable sets of finite perimeter}
The notion of connectedness sets it is not relevant in the context of finite perimeter sets, since, by adding a suitable null set, we can always make a Borel set $E$ connected. The correct notion in this context is that of \textit{indecomposable set}. A set of finite perimeter $E$ is said to be \textbf{decomposable} if there exists two sets $E_1,E_2\subseteq E$ with $0<|E_1|,|E_2|$ and $|E_1\cap E_2|=0$ such that
$$P(E)=P(E_1)+P(E_2).$$

A set of finite perimeter $E$ is said to be \textbf{indecomposable} if it is not decomposable, namely if for every $E_1,E_2\subset E$ with $|E_1\cap E_2|=0$ and such that 
$$P(E)=P(E_1)+P(E_2)$$
then either $|E_1|=0$ or $|E_2|=0$.\\

The following theorem allows us to define the indecomposable components of a set of finite perimeter $E$. We refer the reader to \cite{ambrosio2001} for a detailed proof.

\begin{theorem}\label{componenti indecomponibili}
Let $E$ be a set with finite perimeter in $\R^n$. Then there exists a unique finite or countable family of pairwise disjoint indecomposable set $\{E_i\}_{i\in \N}$ such that $|E_i|>0$ and $P(E)=\sum_i P(E_i).$  Moreover
\[
\H^{n-1}\left( E^{(1)} \setminus \bigcup_i E_i^{(1)}\right)=0
\]
and the $E_i$’s are maximal indecomposable sets, i.e. any indecomposable set $F \subset E$ is contained up to an  $\L^n$-negligible set in some set $E_i$.
\end{theorem}

We say that each set $E_i$ given by Theorem \ref{componenti indecomponibili} is an \textit{indecomposable component of $E$}. In particular note that a set $E$ is indecomposable if and only if it has only one indecomposable component. \\

The set $E$ made by the union of two tangent ball $B_1$ and $B_2$ will be decomposable by setting $E_1=B_1$, $E_2=B_2$, since in this way 
$$P(E)=P(E_1)+P(E_2).$$
In this case $B_1$ and $B_2$ are the indecomposable components of $E$.\\

A cube $Q$ in $\R^n$ instead is an example of indecomposable set. \\ 

A very useful relation is attained between perimeter and diameter in the class of indecomposable planar sets of finite perimeter.

\begin{proposition}\label{indeco piano}
 If $E\subset \R^2$ is an indecomposable set of finite perimeter with $|E|<+\infty$, then
\begin{equation}\label{eqn: chapter intro peridia}
P(E)\geq 2\diam(E^{(1)}).
\end{equation}
\end{proposition}
The validity of this fact can be deduced as a consequence of \cite[Proposition 19.22]{maggibook}. Relation \eqref{eqn: chapter intro peridia} combined with Theorem \ref{compactness theorem} is very useful since gives us  the compactness of a sequence of indecomposable sets of finite perimeter $\{E_h\}_{h\in\N}$ whenever a uniform bound on $P(E_h)$ holds. 

\subsection{First variation of perimeter and of the potential energy}\label{Cpt 1 sbs: First variation of perimeter}
A  one-parameter family of diffeomorphisms $\{f_t \ | \ -\e<t<\e\}$ of $\R^n$  is a \textit{local variation in an open set $A$} if
\begin{equation}
\begin{array}{rll}
f_0(x)=x&  \ \ \ \ & \forall \ x\in \R^n,\\
\{x\in \R^n \ | \ f_t(x)\neq x\}\cc  A& \ \ \ \ &\forall \ |t|<\e.
\end{array}
\end{equation}
A map $T$ is said to be the \textit{initial velocity of a local variation $\{f_t\}_{|t|<\e}$ in $A$ } if
$$T=\frac{\partial f_t}{\pa t}\Big{|}_{t=0}.$$ 
The following theorem allows us to compute the first variation of a perimeter for a finite perimeter set $E$ (see \cite[Theorem 17.5]{maggibook})
\begin{theorem}\label{teo: chapter intro first variation of the perimeter}
Given an open set $A$, a set of finite perimeter $E$ and a local variation $\{f_t \ | \ -\e<t<\e\}$  in $A$, then
\begin{equation}\label{eqn: chapter intro teo first variation derivata del perimetro}
P(f_t(E);A)=P(E;A)+t\int_{\pared E} \dive_E T(x)\d \H^{n-1}(x)+o(t),
\end{equation}
where $T$ is the initial velocity of the local variation and 
$$\dive_E T(x)=\dive T(x)-\nu_E(x) \cdot \nabla T(x)\nu_E(x) \ \ \ \ \text{for $x\in\pared E$,} $$
is the tangential divergence of $T$ on $\pared E$.
\end{theorem}
The following theorem, instead, is what we need to compute the first variation of a functional defined as $\mathcal{G}(E)=\int_E g\d x$ for a continuous function $g$. In particular if $g=1$ the theorem provides the first variation of the Lebesgue $n$-dimensional measure (see  \cite[Theorem 17.8]{maggibook}).
\begin{theorem}\label{teo: chapter intro first variation of the potential energy}
Given an open set $A$, a set of finite perimeter $E$, $|E|<+\infty$ a continuous function $g\in C^0(\R^n)$ and a local variation $\{f_t \ | \ -\e<t<\e\}$ in $A$, then
\begin{equation}\label{eqn: chapter intro teo first variation of the potential energy}
\int_{f_t(E)}g(x)\d x=\int_E g(x)\d x+t\int_{\pared E} g(x) (T(x)\cdot \nu_E(x) )\d \H^{n-1}(x)+o(t),
\end{equation}
where $T$ is the initial velocity of the local variation.
\end{theorem}
\begin{remark}
\rm{
If $E$ is a set of finite perimeter with $|E|<+\infty$  and we apply Theorem \ref{teo: chapter intro first variation of the potential energy} by choosing $g(x)=1$ we obtain the useful formula
\begin{equation}\label{eqn: chpter intro first variaition of the lebesgue measure}
|f_t(E)|=|E|+t\int_{\pared E} ( T(x)\cdot \nu_E(x) )\d \H^{n-1}(x)+o(t).
\end{equation}
}
\end{remark}
\subsection{Distributional mean curvature of a set of finite perimeter}
Let $E$ be a finite perimeter set, $A$ an open set and $H_E$ a function in $ L^{1}(A\cap\pared E;\H^{n-1})$. We say that $H_E$ is the \textbf{distributional mean curvature of $E$ in $A$} if it holds:
\begin{equation}\label{eqn: chapter intro distributional mean curvature}
\int_{\pared E\cap A} \dive_E T \d \H^{n-1}=\int_{\pared E\cap A} H_E(x) (T\cdot \nu_E) \d \H^{n-1} \ \ \ \ \ \forall \ T\in C_c^{\infty}(A;\R^n).
\end{equation}
Note that if $E$ is a finite perimeter set with distributional mean curvature $H_E$, then the distributional mean curvature of $E^c$ is $-H_{E}$. 

\begin{remark}
\rm{Let $E$ be a set of finite perimeter with constant distributional mean curvature equal to $C$ in an open set $A$. Given an initial velocity $T(x)\in C_c(A;\R^n)$ and having defined the local variation $f_t(x)=\{x+tT(x) \ | \ -\e<t<\e\}$, by putting together \eqref{eqn: chapter intro teo first variation derivata del perimetro} and \eqref{eqn: chapter intro distributional mean curvature} we obtain the useful formula:
\begin{equation}\label{eqn: chapter intro constant mean curvature variations}
\frac{d}{dt} P( f_t(E) )\Big{|}_{t=0}=\int_{\pared E \cap A} \dive_E T(x)\d \H^{n-1}(x)=C\int_{\pared E \cap A} (T\cdot \nu_E) \d \H^{n-1}.
\end{equation}
}
\end{remark}
\section{Regularity of perimeter almost minimizing sets}

\begin{definition}[$(\Lambda,r_0)$-perimeter-minimizing inside $\Om$, \cite{maggibook} pp. 278-279]\label{Lambdarminimi}
We say that a set of finite perimeter $E$ is a $(\Lambda,r_0)$-perimeter-minimizing in 
$\Om$ if for every $B_r\subset \Om$ with $r<r_0$ and every set $F$ such that $E\Delta F\subset\subset B_r$, it holds
$$P(E;B_r)\leq P(F;B_r)+\Lambda |E\Delta F|.$$
\end{definition}
The following theorem clarifies why Definition \ref{Lambdarminimi} is so important (see \cite[Chapter 21 and pp. 354, 363-365]{maggibook}).
\begin{theorem}\label{regularity} 
If $\Om$ is an open set in $\R^n$, $n\geq 2$ and $E$ is a $(\Lambda,r_0)$-perimeter-minimizing in $\Om$, with 
$\Lambda r_0\leq 1$, then for every $\a\in(0,1)$ the set $\Om\cap \partial^*E$ is a $C^{1,\a}$ hypersurface that is relatively open in $\Om\cap \partial E$, and it is $\H^{n-1}$ equivalent to $\Om \cap \partial E$. Moreover, setting 
$$\Sigma(E;\Om):=\Om\cap (\partial E\setminus \partial^*E),$$
then the following statements hold true:
\begin{itemize}
 \item[(i)] if $2\leq n\leq 7$, then $\Sigma(E;\Om)$ is empty;
 \item[(ii)] if $n=8$, then $\Sigma(E;\Om)$ is discrete;
 \item[(iii)] if $n\geq 9$, then $\H^s(\Sigma(E;\Om))=0$ for every $s>n-8$. 
\end{itemize}
\end{theorem}
The set $\S(E;\Om)$ is called \textbf{singular set}. In every dimension greater than or equal to $8$ it is possible to exhibit an example of a $(\La, r_0)$-perimeter-minimizing set $E$ with $\H^{n-8}(\S(E))>0$ (see \cite{de2009short}, \cite{bombieri1969minimal}, \cite[Section 28.6]{maggibook}). Assertion $(iii)$ has been recently improved in \cite{naber2015singular} where the authors show that the singular sets of minimizing hypersurfaces in dimension greater than or equal to $8$ is exactly an $(n-8)$ rectifible sets with finite $(n-8)-$dimensional Hausdorff measure.

\section{Useful inequalities for sets of finite perimeter}
We here recall some useful inequalities holding on the family of sets of finite perimeter.
\subsection{Isoperimetric inequality}
For every  set of finite perimeter $E$ with $|E|<+\infty$ it holds
\begin{equation}\label{eqn: chapter intro isoperimetric}
P(E)\geq n\om_n^{1/n} |E|^{1-1/n}=P(B^E),
\end{equation} 
where $B^{E}$ is a ball such that $|B^{E}|=|E|$. Equality is attained if and only if $E$ is (equivalent to) a ball. Note that \eqref{eqn: chapter intro isoperimetric} states that among all sets of finite perimeter $E$ with the same amount of fixed volume $|E|=v$ the $n$-dimensional Euclidean ball with radius $r=\left(\frac{v}{\om_n}\right)^{\frac{1}{n}}$ is the one attaining the smallest perimeter (see \cite{DeGiorgiisoperi}, \cite[pp. 185-197]{Selectedpapers}, \cite[Chapter 14]{maggibook}). Quantitative versions of \eqref{eqn: chapter intro isoperimetric} are provided, through different methods, in \cite{CicaleseLeonardi}, \cite{FuMP08} and \cite{FMP10}.

\subsection{Isodiametric inequality}
For every Borel set $E$ with $|E|<+\infty$ it holds
\begin{equation}\label{eqn: chapter intro isodiametric}
|E|\leq \left(\frac{\diam(E^{(1)})}{2}\right)^n=|B_{\frac{\diam(E)}{2}}|
\end{equation}
where $B_{\frac{\diam(E)}{2}}$ is a ball of radius $\diam(E)/2$. Equality is attained if and only if $E$ is (equivalent to) a ball. Note that \eqref{eqn: chapter intro isodiametric} states that among all the Borel sets $E$ with the same diameter $\diam(E)=d$ the $n$-dimensional Euclidean ball with radius $r=\frac{d}{2}$ is the one attaining the biggest volume (see \cite[Theorem 3.11]{maggibook}). A quantitative version of \eqref{eqn: chapter intro isodiametric} is provided in \cite{maggi2014quantitative}.\\

\subsection{Cheeger inequality for Borel sets}
For a bounded Borel set $E\subseteq \R^n$, the \textbf{Cheeger constant of $E$} is defined as
	\begin{equation}\label{cheeger}
	h(E)=\inf\left\{\frac{P(F)}{|F|}\ \Big{|} \ F\subset E, \ \text{set of finite perimeter} \right\}
	\end{equation}
Then for every set of finite perimeter $E\subset \R^n$ and it holds 
\begin{equation}\label{eqn: chapter intro cheeger inequality}
h(E) \geq \frac{n \om_n^{1/n}}{|E|^{1/n}}=\om_n^{1/n} \frac{h(B)}{|E|^{1/n}}.
\end{equation}
where $B$ is a ball of unit-radius. Equality is attained if and only if $E$ is (equivalent to) a ball. Note that \eqref{eqn: chapter intro cheeger inequality} states that among all sets of finite perimeter $E$ with the same amount of fixed volume $|E|=v$ the $n$-dimensional Euclidean ball with radius $r=\left(\frac{v}{\om_n}\right)^{\frac{1}{n}}$ is the one attaining the smallest Cheeger constant. The proof of this fact can be obtained as a consequence of the isoperimetric property of the ball \eqref{eqn: chapter intro isoperimetric}. A quantitative version of \eqref{eqn: chapter intro cheeger inequality} is provided in \cite{FMP09}.\\

We refer to Chapter 4, Section \ref{cpt 4 sct 1} below where the main properties  of the Cheeger constant together with some brief historical notes are recalled.
 
\section{$N$-clusters and tilings}\label{cpt Ncluster of Rn}
\subsection{$N$-clusters of $\R^n$} \label{subsection $N$-clusters of $R^n$}
By $N$-cluster we mean a family of $N$ sets of finite perimeter $\E=\{\E(0),\E(1),\ldots,\E(N)\}\subset \R^n$, having positive Lebesgue measure and pairwise disjoint up to a set of measure zero. In other words a family of Borel sets $\{\E(i)\}_{i=1}^{N}$ is called an $N$-cluster if
\begin{itemize}
\item[1)] $P(\E(i))< +\infty$, for $i=1,\ldots,N$;
\item[2)] $0<|\E(i)|<+\infty$,  for $i=1,\ldots,N$;
\item[3)] $|\E(i)\cap\E(k)|=0$ for all $k\neq i$.
\end{itemize}
We allow in the previous definition also the case $N=+\infty$. In the sequel,  unless it is not specified, we are always assuming that $N<+\infty$. We define the \textbf{volume vector of $\E$}, $\vol(\E)\in \R^N$ as:
$$\vol(\E)=(|\E(1)|,\ldots,|\E(N)|).$$
 The \textbf{external chamber} $\E(0)$ of the $N$-cluster $\E$ is the set
\begin{equation}\label{esterna}
\E(0):=\Rn\setminus\left(\bigcup_{i=1^{N}}\E(i)\right).
\end{equation}
We define the $(h,k)$-interface of $\E$ as
\begin{equation}\label{(h,k)-interface}
\E(h,k):=\pa^{*}\E(h)\cap \pa^{*}\E(k).
\end{equation}
We moreover introduce the boundary of $\E$ and the reduced boundary of $\E$ as
\begin{equation}\label{boundary}
\pa\E:=\bigcup_{i=1}^{N}\pa \E(i), \ \ \ \pared \E:=\bigcup_{0\leq h<k\leq N}^{N}\E(h,k).
\end{equation}
With these notations we can easily define the {\it perimeter of an $N$-cluster $\E$} relative to a Borel set $F$ as
\begin{equation}\label{perimetro di un cluster}
P(\E;F):=\frac{1}{2}\sum_{i=0}^{N}P(\E(i);F)=\sum_{0\leq h<k\leq N}^{N} \H^{n-1}(\E(h,k)\cap F).
\end{equation}
Note that with the notation \eqref{boundary} introduced above, on every Borel set $F$ it holds:
$$P(\E;F)=\H^{n-1}(\pared \E\cap F).$$

We define the \textit{distance} between two given $N$-clusters $\E$ and $\F$ as
$$d(\E,\F)=\frac{1}{2}\sum_{i=0}^N |\E(i)\Delta \F(i)|.$$
\subsection{Planar tilings}
A planar tiling is a countable family of sets of finite perimeter in $\R^2$, $\E=\{\E(i)\}_{i=1}^{+\infty}\subset \R^2$, such that 
\begin{itemize}
\item[1)] $P(\E(i))< +\infty$, for all $i\geq 1$;
\item[2)] $0<|\E(i)|<+\infty$,  for all $i\geq 1$;
\item[3)] $|\E(i)\cap\E(k)|=0$ for all $k\neq i$;
\item[4)] $|\E(0)|=0$.
\end{itemize}
A planar tiling is substantially an $\infty$-cluster with empty external chamber. In the sequel the \textit{regular hexagonal tiling of $\R^2$} (or simply the \textit{hexagonal tiling}) is the planar tiling $\H=\left\{H+\sqrt[4]{12}\left(\frac{k}{\sqrt{3}},\frac{j}{2}\right)\right\}_{k,j\in \mathbb{Z}}$ where $H$ denotes a unit-area regular hexagon.

\subsection{The flat torus}\label{sbsct: the flat torus}
Let $v,w\in \R^2$ be two orthogonal vectors. We say that two points $x_1,x_2\in \R^2$ are equivalent, and we write $x_1\sim x_2$, if there exists two integers $k_v,k_w\in \Z$ such that
$$x_1-x_2=k_v v+k_w w.$$
We define the \textit{flat torus} to be the collection of all the equivalence classes of $\R^2$ with respect to $\sim$:
$$\T (v,w):=\R^{2}/\sim.$$
Note that
$$Q_{\T}:=\{sv \ | \ s\in (0,1]\}\times \{tw \ | \ t\in(0,1]\}\subset \R^2$$
is a fundamental domain, namely a set containing exactly one representative in each equivalence class. Moreover for any given Borel set $E\subseteq \T(v,w)$  we can always consider the \textit{periodic extension $\widehat{E}\subset \R^2$}. Thus it is well defined the relative perimeter of $E\subset \T(v,w)$ inside $F\subset \T(v,w)$ as
$$P_{\T}(E;F):=P(\widehat{E};\widehat{F}\cap Q_{\T}).$$
The total perimeter of a set $E$ is then defined as
$$P_{\T}(E):=P(\widehat{E}; Q_{\T}).$$
Note that if $E,F\subseteq Q_{\T}$, by denoting with
$$E'=E/\sim, \ \ \ \ \ F'=F/\sim$$
it must hold:
$$P_{\T}\left(E';F'\right)=P(E;F).$$
Indeed $\pared \widehat{E'}\cap \widehat{F'}\cap  Q_{\T}=\pared E \cap F$. For this reason in the sequel we avoid the subscript $\T$ and we simply write $P(E;F)$ also to denote the relative perimeter (and the perimeter) of a Borel set $E\subseteq \T(v,w)$ inside $F\subseteq \T(v,w)$. We usually write $\T$ instead of $\T(v,w)$ whenever the role of the vectors $v,w$ is clear from the context.\\

\subsection{$N$-clusters on the torus}
Given a flat torus $\T$ we define an $N$-cluster of the torus as a family of $N$ Borel sets $\{\E(1),\ldots,\E(N)\}\subset \T$ with 
\begin{itemize}
\item[1)] $P(\E(i))< +\infty$, for all $i\geq 1$;
\item[2)] $0<|\E(i)|$,  for all $i\geq 1$;
\item[3)] $|\E(i)\cap\E(k)|=0$ for all $k\neq i$;
\end{itemize}

Note that, since $|\T|<+\infty$ we do not need to add the request $|\E(i)|<+\infty$ as in the planar case. The external chamber is then defined as
$$\E(0)=\T\setminus \left(\bigcup_{i=1}^N \E(i)\right).$$

The volume of $\E$ is $\vol(\E)=(|\E(1)|,...,|\E(N)|)$, and the relative perimeter of $\E$ in $A\subset\T$ is given by
\[
P(\E;A)=\frac12\sum_{h=0}^N P(\E(h);A)\,,
\]
while the distance between two $N$-clusters $\E$ and $\F$ is defined as
\[
\d(\E,\F)=\frac12\sum_{h=0}^N|\E(h)\Delta\F(h)|\,.
\]
\subsection{$N$-tilings of the torus}

An $N$-tiling of a two-dimensional flat torus $\T$ is an $N$-cluster $\E\subseteq \T$ with the additional request 
$$|\T\setminus\bigcup_{h=1}^N\E(h)|=|\E(0)|=0.$$ The volume of $\E$ is $\vol(\E)=(|\E(1)|,...,|\E(N)|)$, and the relative perimeter of $\E$ in $A\subset\T$ is given by
\[
P(\E;A)=\frac12\sum_{h=1}^N P(\E(h);A)\,,
\]
while the distance between two tilings $\E$ and $\F$ is defined as
\[
\d(\E,\F)=\frac12\sum_{h=1}^N|\E(h)\Delta\F(h)|\,.
\]
We say that $\E$ is a {\it unit-area tiling} of $\T$ if $|\E(h)|=1$ for every $h=1,...,N$. (In particular, in that case, it must be $N=|\T|$).\\

Obviously, every $N$-cluster is an $(N+1)$-tiling and every $N$-tiling defines an $(N-1)$-cluster. Notice also that every $N$-tiling of a flat torus $\T$ can be viewed as a periodic planar tiling in $\R^2$. 

\section{Set operations on Clusters} \label{subsection Union of clusters and intersections with Borel sets.}

\subsection{Union of Clusters}
An $N$-cluster $\E$ and an $M$-cluster $\F$ are said to be \textit{disjoint} if 
$$|\E(i)\cap \F(j)|=0, \ \ \ \text{for every $i=1,\ldots, N$ and $j=1,\ldots, M$}.$$
In this case we define the $(M+N)$-Cluster $\E\cup \F$ as 
$$\E\cup \F:=\{\E(1),\ldots, \E(N),\F(1),\ldots,\F(M)\}.$$ 
Note that $(\E\cup \F)(0)=\E(0)\cap \F(0)$. \\

By exploiting formulas \eqref{eqn: perimetro intersezione}, we obtain:
\begin{equation}\label{eqn: formula per il perimetro di unioni di clusters}
P(\E\cup \F)=P(\E)+ P(\F)-\H^{n-1}(\pared \E(0)\cap \pared \F(0)).
\end{equation}

\subsection{Intersection of a Cluster with a Borel set}
Given a Borel set $F$ and an $N$-cluster $\E$ we define the cluster $\E\cap F$ as the family of sets:
$$\E\cap F:=\{\E(j)\cap F  \ \ \text{for all $j=1,\ldots,N$ such that $|\E(j)\cap F |>0$}\}.$$
Note that $(\E\cap F)(0)=\E(0)\cup F^c$. The number of chambers $k$ of $\E\cap F$ is given by
$$k=\#(\{j\in\{1,\ldots,N\} \ | \ |\E(j)\cap F |>0\}).$$

By exploiting formulas \eqref{eqn: perimetro intersezione} and \eqref{eqn: perimetro unione}, we obtain:
\begin{equation}\label{eqn: formula per il perimetro di intersezioni di clusters}
P(\E\cap F) =P(\E; F^{(1)})+P(F;\E(0)^{(0)})+\H^{n-1}(\{\nu_{\E(0)}=-\nu_F\} ). 
\end{equation}

\subsection{Difference between a set and a cluster}
Given a Borel set $F$ and an $N$-cluster $\E$ we define the Borel set
$$F\setminus\E:=\bigcap_{i=1}^N (F\setminus \E(i))=F\cap\E(0).$$
By exploiting formula \eqref{eqn: perimetro differenza}, we obtain:
\begin{equation}\label{eqn: formula per il perimetro di differenza insieme clusters}
P(F\setminus E) =P(F; \E(0)^{(1)})+P(\E(0);F^{(1)})+\H^{n-1}(\{\nu_{\E(0)}=\nu_F\} ). 
\end{equation}
We also define the cluster $\E\setminus F$ as
$$\E\setminus F:=\{\E(j)\setminus F \ | \  \text{for all $j=1,\ldots,N$ such that $|\E(j)\setminus F |>0$}\}=\E\cap F^c.$$
Note that $(\E\setminus F)(0)=\E(0)\cup F$. The number of chambers $k$ of $\E\setminus F$ is given by
$$k=\#(\{j\in\{1,\ldots,N\} \ | \ |\E(j)\setminus F |>0\}).$$

By exploiting formula \eqref{eqn: formula per il perimetro di intersezioni di clusters}, we obtain:
\begin{equation}\label{eqn: formula per il perimetro di differenza cluster insieme}
P(\E\setminus F) =P(\E; F^{(0)})+P(F;\E(0)^{(0)})+\H^{n-1}(\{\nu_{\E(0)}=\nu_F\} ). 
\end{equation}

\subsection{Symmetric difference between clusters}

Given two $N$-clusters $\E,\F$ we define the \textit{symmetric difference between $\E$ and $\F$} as the set
$$\E\Delta \F:= \bigcup_{i=1}^N \E(i)\Delta \F(i).$$

\section{$C^{k,\a}$ $N$-clusters in $\R^{2}$} \label{subsection C k, a N-clusters in R2.}
For a given a closed curve with boundary $\g:[a,b]\rightarrow \R^2$ we introduce the notations
$$ \g=\g([a,b]), \ \ \,\ \text{int}(\g)=\g((a,b)), \ \ \ \  \bd(\g)=\{\g(a),\g(b)\}.$$
We say that a family of closed connected curves with boundary $\{\g_i\}_{i\in I}$ is a \textit{network} if, having defined $\{p_j\}_{j\in J}=\{\bd(\g_i)\}_{i\in I}$, the following properties hold:
\begin{itemize}
\item[1)] $I$ and $J$ are at most countable;
\item[2)] $\{p_j\}_{j\in J}$ and $\{\g_i\}_{i\in I}$ are locally finite, in the sense that 
	\[
	\#(\{j\in J \ | \ p_j\in B_r\})+	\#(\{i\in I \ | \ \g_i \cap  B_r\neq \emptyset \})<+\infty \ \ \ \ \text{for all $r>0$};
	\]
\item[3)] $\text{int}(\g_i)\cap \text{int}(\g_h)=\emptyset$, for all $i,h\in I$, $i \neq h$;
\item[4)] Each $p_j$ is a common end-point to at least three different curves from $\{\g_i\}_{i\in I}$.
\end{itemize} 
If each $\g_i$ is also a $C^{k,\a}$-curve we say that the family  $\{\g_i\}_{i\in I}$ is a \textit{$C^{k,\a}-$network.}\\

We say that a cluster $\E\cc\R^{2}$ is of class $C^{k,\a}$ inside an open set $\Om$ if there exists a $C^{k,\a}$-network $\{\g_i\}_{i\in I}$ such that
\begin{align}
\pa \E \cap \Om&    = \bigcup_{i\in I} \g_i, \label{frontiera C k alfa}\\
\pared \E \cap \Om&= \bigcup_{i\in I} \text{int}(\g_i), \label{frontiera ridotta C k alfa}\\
\Sigma(\E;\Om)&=\Om\cap (\pa \E\setminus \pared \E) =\Om \cap \bigcup_{j\in J } \{p_j\}.  \label{insieme singolare}
\end{align}

If $\Om=\R^2$ we simply say that $\E$ is of class $C^{k,\a}$.

\section{perimeter-minimizing $N$-clusters}

\subsection{Definition of perimeter-minimizing $N$-clusters}
An $N$-cluster $\E$ is said to be a \textbf{perimeter-minimizing $N$-cluster} if for every other $N$-cluster $\F$ with 
$$\vol(\E)=\vol(\F)$$
(up to relabeling the chambers) it holds
$$P(\E)\leq P(\F).$$
The existence of perimeter-minimizing $N$-clusters was proved by Almgren in \cite{Almgren76} where also a partial regularity of these objects was discussed.  

\subsection{$(\La,r_0)$-perimeter-minimizing $N$-clusters inside $\Om$}\label{subsection perimeter-minimizing N-clusters}

We say that $\E$ is a $(\La,r_0)$-perimeter-minimizing $N$-cluster inside an open set $\Om$ if for every $B_r(x)\subset \Om$ with $r<r_0$ and every $N$-cluster $\F$ with
$$\E\Delta\F \cc B_r(x),$$
it holds
\begin{equation}\label{perimeter minimizers $N$-clusters}
P(\E;B_r(x))\leq P(\F;B_r(x))+\La d(\E,\F).
\end{equation}

It can be shown that each perimeter-minimizing $N$-cluster $\E$ is a $(\La,r_0)$-perimeter-minimizing cluster for a suitable choice of $\La$ and $r_0$ and that this fact leads to the regularity given by Theorem \ref{regularity} on $\pared \E$ and that the singular set 
	\begin{equation}\label{insieme singolare di un N-cluster}
	\S(\E;A):=(\pa \E\setminus \pared \E )\cap A,
	\end{equation}
is closed and $\H^{n-1}-$negligible. More precisely the following statement holds true (see \cite[Corollary 4.6]{CiLeMaIC1}, \cite[Chapter 30]{maggibook})
\begin{theorem}\label{rego planar cluster}
If $\E\subset \R^n$ is a $(\La,r_0)$-perimeter-minimizing cluster in an open set $\Om$, then $\pared \E\cap \Om$ is a $C^{1,\a}$-hypersurface for every $\a\in (0,1)$, it is relatively open inside $ \pa \E\cap \Om$, and $\H^{n-1}(\S(\E;\Om))= 0$. Moreover, if $n = 2$, then we can replace $C^{1,\a}$ with $C^{1,1}$.
\end{theorem} 

As pointed out in Theorem \ref{rego planar cluster} in the planar case it is possible to improve the regularity of perimeter-minimizing $N$-clusters (see for example \cite[Section 30.3]{maggibook}, \cite{Morgan}) thanks to a detailed study of the singular set $\S(\E)=\pa \E\setminus \pared \E$. In particular it can be shown that each perimeter-minimizing $N$-cluster in the plane is a cluster of class $C^{1,1}$ and 
$$\pared \E=\bigcup_{i\in I} \text{int}(\g_i)$$ 
where each curve $\g_i$ is either a segment or a circular arc. Furthermore any two arcs belonging to the same interface $\E(h,k)$ have the same curvature, the singular set $\S(\E)$ is locally finite and each singular point $p_j\in \Sigma(\E)$ is a common end-point to exactly three different curves from $\{\g_i\}_{i\in I}$, which form three 120 degree angles at $p_j$.\\

\begin{remark}\rm{
In the celebrated work of Taylor \cite{ta76} the singular set of a 3-dimensional perimeter-minimizing $N$-cluster is completely characterized. In particular is proved that the singular set $\S(\E)$ \textit{consists of H\"older continuously differentiable curves along which three sheets of the surface meet at equal angles, together with isolated points at which four such curves meet bringing together six sheets of the surface at equal angles.}\\
So far, except for the general regularity structure given by Theorem \ref{regularity}, the description of the singular set of a perimeter-minimizing $N$-cluster in dimension bigger than $3$ is still mostly unknown.}
\end{remark}
\section{Useful tools from the theory of $N$-clusters}

\subsection{Hales's Theorem and its consequences}

On every flat torus $\T$ the following Theorem due to Hales (\cite{hales}) holds true.
\begin{theorem}\label{teo: chapter intro Hales sul toro}
If $\E$ is an $N$-cluster of a torus $\T$, with $|\E(i)|\leq 1$ for every $i=1,\ldots,N$ then the following estimate holds
\begin{equation}\label{eqn: Hales toro}
P(\E)\geq \frac{P(H)}{2}\left(\min\{|\E(0)|,1\}+\sum_{i=1}^N |\E(i)|\right),
\end{equation}
where $H$ denotes a unit-area regular hexagons. Equality in \eqref{eqn: Hales toro} is attained if and only if $\E$ is an hexagonal tiling with unit-area chambers. 
\end{theorem} 
Theorem \ref{teo: chapter intro Hales piano} tells us that among all the $N$-clusters (tilings) of the torus with unit-area chambers, the hexagonal tiling is the one attaining the smallest perimeter.\\

If $\E$ is a bounded planar $N$-cluster we can always find two orthogonal vectors $v,w$ such that 
$$\E\cc \{s v \ | \ s\in(0,1]\} \times \{tw \ | \ t\in (0,1]\}.$$ 
We can consider the $\E':=\{\E(i)/\sim\}_{i=1^N}$ on the flat torus $\T=\T(v,w)$ and apply Hales's Theorem. As a consequence starting from Theorem \ref{teo: chapter intro Hales sul toro} it is possible to prove the following Theorem (also appearing in \cite{hales}) holding on planar $N$-clusters.
\begin{theorem}\label{teo: chapter intro Hales piano}
If $\E$ is a bounded planar $N$-cluster with $|\E(i)|\leq 1$ for every $i=1,\ldots,N$ then it holds
\begin{equation}\label{eqn: Hales piano}
P(\E)>\frac{P(H)}{2}\sum_{i=1}^N |\E(i)|.
\end{equation}
where $H$ denote a unit-area regular hexagons.
\end{theorem} 
It is worth noticing that inequality \eqref{eqn: Hales piano} is strict. Indeed Theorem \ref{teo: chapter intro Hales piano} is a consequence of Theorem \ref{teo: chapter intro Hales sul toro}. Since equality in \eqref{eqn: Hales toro} is attained if only if $\E'=\H$ and since there is no bounded planar $N$-cluster $\E$ such that $\E'=(\E/\sim)=\H$ equality in \eqref{eqn: Hales piano} can never be attained.

\subsection{The "improved convergence" for planar clusters}\label{section The "improved convergence" for planar Cluster}
We here recall for the sake of completeness (and clarity), the basic concepts and the main theorem we are making use in Chapter \ref{chapter sharpquantitative}. All these results can be found in \cite{CiLeMaIC1}.  \\
\text{}\\
Let $\g:[0,\H^{1}(\g)]\rightarrow \R^2$ be a simple, closed and connected $C^{1,\a}$-curve, parametrized by the arc length and with non empty boundary $\bd(\g)\neq\emptyset$. A map $f:\g \rightarrow \R^2$ is said to be of class $C^{1,\a}(\g;\R^2)$ if 
\begin{eqnarray*}
f\circ \g&\in& C^{1,\a}([0,\H^{1}(\g)];\R^2),\\
\|f\|_{C^{1,\a}(\g)}&:=&\|(f\circ \g)\|_{C^1([0,\H^{1}(\g)])}+\|(f\circ \g)'\|_{C^{0,\a}([0,\H^{1}(\g)])}<+\infty.
\end{eqnarray*}
Let $\E$ be a $C^{1,\a}$ planar $N$-cluster and let $\{\g_i\}_{i\in I}$ be the $C^{1,\a}$-network associated to $\E$. We say that $f:\pa \E\rightarrow \R^{2}$ is of class $C^{1,\a}(\pa \E; \R^{2})$ if $f$ is continuos on $\pa \E$, $f\in C^{1,\a}(\g_i;\R^{2})$ for every $i\in I$  and 
$$\|f\|_{C^{1,\a}(\pa\E)}:= \sup_{i\in I} \|f\|_{C^{1,\a}(\g_i)}<+\infty.$$
We say that $f$ is a $C^{1,\a}$-diffeomorphism between two clusters $\E$ and $\E'$ if $f$ is an homeomorphism between $\pa \E$ and $\pa \E'$ with
$$f\in C^{1,\a}(\pa \E;\R^{2}), \ \ f^{-1}\in C^{1,\a}(\pa\E'; \R^{2}) \ \ \text{and} \ \ f(\S(\E))=\S(\E'). $$
We define the \textit{tangential component} of a vector field $f:\R^{2}\rightarrow \R^{2}$ with respect to an $N$-cluster $\E$ as
$$\tau_{\E}f(x) := f(x)-(f(x)\cdot \nu_{\E}(x) )\nu_{\E}(x),$$
where $\nu_{\E}:\pared\E\rightarrow \R^{2}$ is any Borel function such that either $\nu_{\E}(x)=\nu_{\E(h)}(x)$ or $\nu_{\E(k)}(x)$ for every $x\in \E(h,k)$.
\begin{theorem}\label{improv Theorem}[Improved convergence for planar almost-minimizing clusters] Given $\Lambda \geq 0, r_0>0$ and $\E$, a bounded $C^{2,1}-$cluster in $\R^{2}$, there exist positive constant $\mu_0$ and $C_0$ (depending on $\Lambda$ and $\E$) with the following property. If $\{\E_k\}_{k\in \N}$ is a sequence of perimeter $(\Lambda, r_0)$-minimizing $N-$clusters in $\R^{2}$ such that $d(\E_k,\E) \rightarrow 0$ \mbox{(as $k\rightarrow +\infty$)}, then for every $\mu<\mu_0$ there exists $k(\mu)\in \N$ and a sequence of maps $\{f_k\}_{k\in k(\mu)}$ such that each $f_k$ is a $C^{1,1}-$diffeomorphism between $\pa \E$ and $\pa \E_k$ with 
\begin{eqnarray}
\|f_k \|_{C^{1,1}(\pa \E)} &\leq &C_0,\\
\lim_{k\rightarrow +\infty} \|f_k -\Id \|_{C^{1}(\pa \E)} &=& 0,\\
\|\tau_{\E} (f_k-\Id ) \|_{C^{1}(\pa^{*} \E)} &\leq& \frac{C}{\mu} \|f_k - \Id \|_{C^{0}(\Sigma( \E))},\\
\tau_{\E} (f_k-\Id ) &=& 0 \ \ \ \text{on $\pa \E \setminus I_{\mu}(\Sigma(\E))$,}
\end{eqnarray}
where $Id(x)=x$ and 
$$I_{\mu}(\S(\E))=\{x\in \R^n \ | \ d(x,\S(\E) )<\mu\}. $$
\end{theorem}

\chapter{Uniform distribution of the energy for an isoperimetric partition problem with fixed boundary}

\section{Introduction}
A conjecture due to Morgan and Heppes, appeared in \cite{HepMor05}, states that the global shape of perimeter-minimizing planar $N$-clusters having equal-volume chambers, suitably normalized must converge, in the $L^{1}$-sense, to a ball. The global shape should be intended as $\E(0)^c$, where $\E(0)$ is the external chamber of the cluster $\E$. So far, no progress has been made in proving this conjecture and the reason could lie in the difficulties arising when we try to understand in which sense the shape of the internal chambers has an influence on the global shape of perimeter-minimizing $N$-clusters. In 2001 Thomas Hales \cite{hales} solved the so-called \textit{hexagonal honeycomb conjecture} providing Theorems \ref{teo: chapter intro Hales sul toro} and \ref{teo: chapter intro Hales piano} that somehow give us information about the internal structure of such perimeter-minimizing clusters. Theorem \ref{teo: chapter intro Hales piano} combined with a suitable comparison argument tells us that, for $N$ approaching $+\infty$, the perimeter of a perimeter-minimizing planar $N$-cluster $\E$ with equal volume chambers is asymptotic equivalent to the perimeter of a grid of $N\times N$ hexagons:
$$\frac{P(H)}{2}N\leq P(\E) \leq \frac{P(H)}{2}N+C\sqrt{N}.$$
Hales, in its paper \cite{hales}, proves more: when we consider the partition problem on the torus (which is a way to consider a periodic tiling of $\R^2$), the hexagonal tiling is the only minimizer. A new result on this topic is the one contained in Chapter Three (that can also be found in \cite{CM14}) where a stability results of the hexagonal tiling with respect to compactly-supported and mass-preserving perturbations has been proved. Everything suggests that the internal chambers of perimeter-minimizing $N$-clusters try to get closer and closer to regular hexagons and, so far, it is not clear whether this behavior affects the global shape of perimeter-minimizing $N$-clusters (and in which sense).\\

In order to investigate the influence of the boundary on the internal structure of perimeter-minimizing $N$-clusters it makes sense to consider an isoperimetric problem  with fixed boundary on planar $N$-clusters. Namely for a fixed set $\Om$ with finite perimeter we consider the quantity
\begin{eqnarray} 
\rho(N,\Om):=\inf_{\E\in \text{Cl}(N,\Om)} \{ P(\E) \}, \label{Fixed boundary}
\end{eqnarray}
where the infimum is taken among all the \textit{$N$-clusters of $\Om$}:
\begin{equation}\label{classe di competizione}
\text{Cl}(N,\Om)= \left\{ \text{$\E$ planar $N$-cluster with $|\E(j)|=\frac{|\Om|}{N}$ for $j\neq 0$ and } \E(0)=\Om^{c} \right\}.
\end{equation}

Thanks to the compactness for sets of finite perimeter and the semi-continuity property of the functional $P(\cdot)$ with respect to the $L^1$ convergence (see Theorems  \ref{compactness theorem}, \ref{semicontinuity theorem}) we get the existence of minimizers for $\rho(N,\Om)$ for every set $\Om$ with finite perimeter and for every $N\in \N$. We call such clusters \textit{perimeter-minimizing $N$-clusters for $\Om$} or simply \textit{minimizing $N$-clusters for $\Om$}. In the following we will not use any regularity property of such clusters, however with the same techniques developed for the perimeter-minimizing $N$-clusters with free boundary it is possible to show that, if $\Om$ is open, each $\E$, minimizing $N$-cluster for $\Om$, is a $(\La,r_0)-$perimeter-minimizing $N$-cluster inside $\Om$. In particular $\E$ is of class $C^{1,1}$ inside $\Om$. This also means that each singular point $p_j\in \Sigma(\E)\cap \Om$ is a common end-point to three different curves that meet in three 120 degree angles in $p_j$ and that the singular set $\Sigma(\E)\cap \Om$ is discrete.\\
 
Our main purpose here is to better understand the behavior of the localized energy $P(\E;Q_l)$ where $Q_l$ is a square of edge-length $l$ and $\E$ is a minimizing $N$-cluster for an open set $\Om$. To describe this behavior we provide two ``equidistribution theorems'' (see Theorems \ref{Equi dia}, \ref{Equi indeco}) in the spirit of the one obtained by Alberti, Choksi e Otto in \cite{AlChOt09}. For the sake of clarity, let us state a ``heuristic version" of the theorems we are going to prove.\\
\text{}\\
\textit{\noindent There exists a universal constant $C$ such that for every open bounded set $\Om$, every  $\E\in \Cl$ minimizing $N$-cluster for $\Om$ and every closed cube $Q\cc \Om$ ``far enough" from the boundary and ``large enough with respect to the size of the chambers" the following holds:
\begin{equation}\label{equi estimate -1}
\Big{|}P(\E;Q) -|Q| \frac{P(H)}{2}  \sqrt{\frac{N}{|\Om|}}\Big{|}\leq C P(Q).
\end{equation}
where $H$ denotes a unit-area regular hexagons.
}
\begin{remark}\label{remark: spiegazione qualitativa del teorema di equidistribuzione}
\rm{
From a qualitative point of view estimate \eqref{equi estimate -1} gives us information about the average energy of $\E$ inside the cube $Q$. If we divide both member of \eqref{equi estimate -1} by $\frac{|Q|}{|\Om|}N$, which represents the expected number of chambers of $\E$ lying inside $Q$, we obtain 
$$\Big{|}\frac{P(\E;Q)}{\frac{|Q|}{|\Om|}N} -\frac{P(H)}{2} \sqrt{\frac{|\Om|}{N}} \Big{|}\leq \frac{CP(Q)}{|Q|}\frac{|\Om|}{N}.$$
Note that $\frac{P(H)}{2} \sqrt{\frac{|\Om|}{N}}$ is the average energy of a uniform grid of hexagons $H$ having volume $\frac{|\Om|}{N}$ (and thus perimeter $P(H)\sqrt{\frac{|\Om|}{N}}$). Hence we can interpret estimate \eqref{equi estimate -1} as follows: the average energy $P(\E;Q)$ of a minimizing $N$-cluster for $\Om$ computed on a fixed cube $Q\cc \Om$ approaches the average energy of a grid of hexagons with area $\frac{|\Om|}{N}$. Estimate \eqref{equi estimate -1} suggests that, no matter where we are localizing for $N$ sufficiently large the boundary $\pa \Om$ does not affect the energetic behavior of minimizing $N$-clusters for $\Om$. This also indicates that some approximate periodicity in the behavior of internal chambers, at least from an energetic point of view, is attained.
}
\end{remark}
\begin{remark}\label{remark: errore CP(Q) ottimale}
\rm{
The term $CP(Q)$ appearing in the right-hand side of \eqref{equi estimate -1} is the optimal one. Indeed, assume for a moment that $\E$ is a perfect hexagonal grid made by hexagons of area $\frac{|\Om|}{N}$. In this situation, if we compute the localized energy $P(\E;Q)$, we discover that the principal part is just the perimeter of all the hexagons compactly contained inside $Q$, that is $|Q| \frac{P(H)}{2} \sqrt{\frac{N}{|\Om|}}$. The contribution of the hexagons intersecting $\pa Q$ will be of order $CP(Q)$ for a universal constant $C$.
}
\end{remark}
\begin{remark}\label{remark: perche nel teorema di equidistribizuione serve far enough and big enough}
\rm{
Let us focus on why we need to be on a cube $Q$ ``far enough from the boundary" and ``large compared to the size of the chambers". We cannot expect the estimate \eqref{equi estimate -1} to work on every cube compactly contained in $\Om$. For example it may happen that the geometry of $\pa \Om$ can affect the internal energy at least in its proximity and so for all cubes too close to $\pa \Om$ the localized energy could be very far from the one of the hexagonal tiling. Moreover, if a cube $Q$ is very small (say for example $|Q|<\frac{|\Om|}{N}$, smaller than the size of the chambers) the theorem will probably be meaningless since the localized energy will be zero or comparable. We are going to quantify in a precise way what ``far enough from the boundary" and ``large compared to the size of the chambers " mean.}
\end{remark}

An estimate of the type of \eqref{equi estimate -1} helps us to better understand the relation between the boundary and the internal chambers in the free boundary case. Indeed, it seems that no matter what ambient space $\Om$ we choose, for $N$ sufficiently large we expect to see hexagons inside. This could mean that the behavior of the boundary does not affect the shape of interior chambers. Thus, can the shape of internal chambers affect the behavior of the boundary in perimeter-minimizing $N$-clusters? We cannot say. We point out here, that it seems that the fixed boundary does not influences the asymptotic trend of internal chambers.\\

Throughout this chapter we denote with $H$ a unit-area regular hexagon so that $\sqrt{\de}H$ will be a regular hexagon of area $\de$. We are sometimes making use of the notation $\H_{\de}$ meaning the tiling of $\R^{2}$ made by regular hexagons of area $\de$ oriented and labeled as in Figure \ref{fig reference}.
\begin{figure}
\begin{center}
 \includegraphics[scale=0.7]{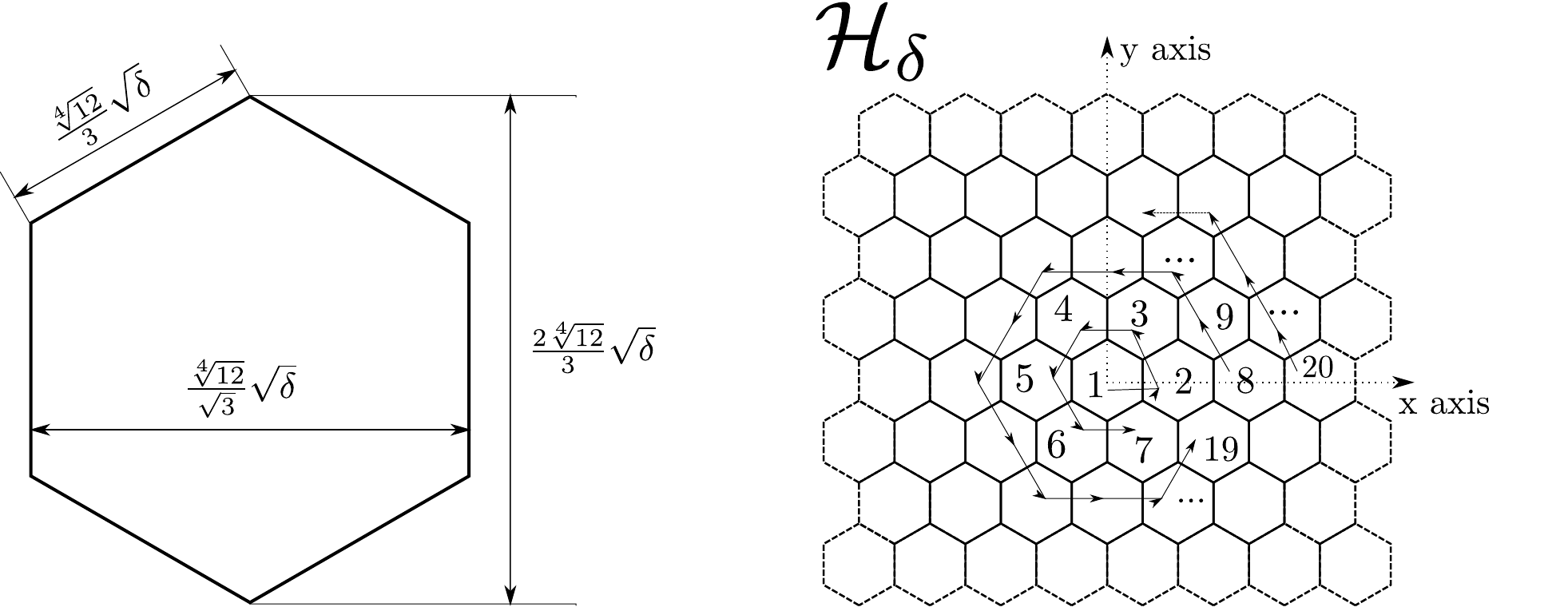}\caption{{The reference regular hexagon of area $\de$ and the correspondent reference hexagonal tiling of size $\de$.}}\label{fig reference}
\end{center}
\end{figure} 

\subsection{Brief sketch of the proof}\label{sbsct: Brief sketch of the proof}
To prove an estimate of the type \eqref{equi estimate -1} we need to provide two bounds on the localized energy $P(\E;Q)$. A lower bound is easily obtained in Lemma \ref{lower bound lemma} for a generic $N$-cluster $\E$ as a consequence of Hales's Theorem \ref{teo: chapter intro Hales piano}. On the other hand we obtain an upper bound by comparison, building a competitor through the following geometric construction. We fix a square $Q_l\cc \Om$ and we exploit the simple idea to substitute in a suitable way the existing cluster with a hexagonal grid (that we know be a heuristic minimizer) inside the square. In order to do this we first suitably enlarge $Q_l$ into $Q_{l+d}$ and then we remove all the chambers compactly contained into $Q_{l+d}$. After that, we completely cover $Q_l$ with an hexagonal grid (see Figure \ref{Skecth1}). We need to make sure that the grid that we have built does not overlap some remaining ``long chambers"  with some tentacles intersecting the boundary of the bigger square $Q_{l+d}$. To handle this phenomenon we restrict to two different classes of minimizing $N$-clusters giving us some control and allowing us to prove two different Theorems: \ref{Equi dia} and \ref{Equi indeco}. To complete the construction we need to ``stitch" with a suitable surgery the grid with the remaining parts of $\E$ (see Figure \ref{Skecth2}). The surgery will be the cluster with chambers of area exactly $\frac{|\Om|}{N}$ provided by Proposition \ref{chirurgia}. We thus obtain an $N$-cluster $\F$ differing from $\E$ only inside $Q_{l+d}$ and by comparison (and up to choose $d$ in a suitable way) we are able to reach the estimate
$$P(\E_N;Q_l)\leq P(\F;Q_l) \leq \frac{P(H)}{2}|Q_l|\sqrt{N}+\text{lower order terms}.$$
The behavior of the \textit{lower order terms} depends on the class of minimizers that we are considering. 
\begin{figure}[h]
\centering
\includegraphics[width=1\columnwidth]{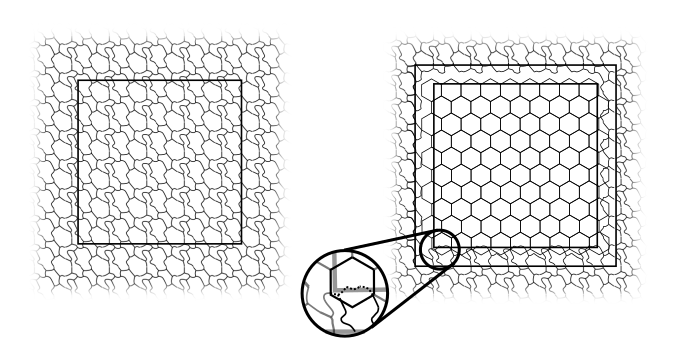}\caption{We consider a square $Q_l\cc \Om$ and we enlarge it in order to have space to complete the construction. We remove all the chambers of the cluster compactly contained into the bigger square and we cover $Q_l$ with an hexagonal tiling. We need to be sure that the tiling do not overlap some "long chamber". }\label{Skecth1} 
\end{figure} 

\begin{figure}[h]
\centering
\includegraphics[width=.5\columnwidth]{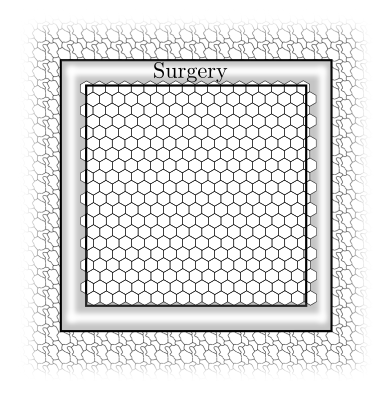}\caption{We build a suitable surgery partition of the remaining part in chambers having the right amount of area.}\label{Skecth2} 
\end{figure} 

The Chapter is organized as follows. In Section \ref{sct: technical lemmas} we prove two technical lemmas: \ref{lower bound lemma} and \ref{lemma chirurgia}. The first one is a consequence of Hales's Theorem \ref{teo: chapter intro Hales piano} and gives us a lower bound on the localized energy of planar $N$-cluster. The second one is a geometric construction for partition in equal-area chambers with a controlled amount of perimeter within a particular class of sets. This second lemma is the one that we need to perform the surgeries. In Sections \ref{sct equi bounded diameter} and \ref{sct equi indecomposable} we prove two different Equidistribution-type theorems holding on two different classes of minimizing $N$-clusters for a given open set $\Om$, namely the \textit{$\mu$-bounded minimizing $N$-clusters} (see Definition \ref{mu bounded}) and the \textit{indecomposable minimizing $N$-clusters} (see Definition \ref{indecomposable minimizers}). Both the sections are mostly devoted to the construction of a suitable competitor (through the idea explained in Subsection \ref{sbsct: Brief sketch of the proof}) in order to derive an upper bound on the localized energy.  

\section{Technical lemmas}\label{sct: technical lemmas}

\begin{lemma}\label{lower bound lemma}
Let $\Om$ be an open bounded set in $\R^2$ and let $\E\in \Cl$ be an $N$-cluster for $\Om$. Then for every open set $O \cc \Om$ it holds
\begin{equation}\label{lower bounded}
P(\E;O)\geq |O| \frac{P(H)}{2}\sqrt{\frac{N}{|\Om|}}-P(O)
\end{equation}
\end{lemma}
\begin{proof}
Since $|\E(i)\cap O|\leq \frac{|\Om|}{N}$, we can apply Theorem \ref{teo: chapter intro Hales piano} to 
$$\F=\sqrt{\frac{N}{|\Om|}}\left(\E\cap O\right)$$
and obtain 
\begin{eqnarray*}
P(\F) \geq  \frac{P(H)}{2}\sum_{i=1}^N  \frac{N}{|\Om|}|\E(i)\cap O|\geq  \frac{P(H)}{2}\frac{N}{|\Om|}|O|.
\end{eqnarray*}
But since, thanks to \eqref{eqn: formula per il perimetro di intersezioni di clusters}, 
\begin{eqnarray*}
P(\F)=\sqrt{\frac{N}{|\Om|}}P\left(\E\cap O\right) = \sqrt{\frac{N}{|\Om|}}[ P(\E; O)+P(O)]
\end{eqnarray*}
we obtain \eqref{lower bounded}.
\end{proof}

\begin{lemma}\label{lemma chirurgia}
Let $A\subset \R^2$, $0<|A|<+\infty$ be a set for which there exists two concentric cubes $Q_0\subset\subset Q_1$ such that
$$A\subseteq Q_1\setminus Q_0.$$
Then for every $M\in \N$ there exists an $M$-cluster $\E_M$ such that $|\E_M(i)|=\frac{|A|}{M}$ for all $i=1,\ldots,M$, $\E_M(0)=A^c$ and for which the following estimate holds:
\begin{equation}\label{chirurgia}
P(\E_M)\leq C|Q_1\setminus Q_0|\sqrt{\frac{M}{|A|}}+P(A),
\end{equation}
for a universal constant $C>0$.
\end{lemma}

\begin{proof}
Let $M\in \N$ be a fixed number. We want to partition $A$ in chambers of area $\frac{|A|}{M}$. To do that, we first partition $A$ in sectors $S_i$ enclosing the same (suitable) amount of area using lines starting from the baricenter  $O$ of the cubes (as in Figure \ref{costruzione chirurgia 1}). Then we divide each sector in chambers of area  $\frac{|A|}{M}$ with circular arcs centered at $O$ (as in Figure \ref{costruzione chirurgia 2}). We need to choose the amount of area that we want to allocate in each sector in a coherent way. \\

Set $d=\frac{1}{2} (\sqrt{|Q_1|}-\sqrt{|Q_0|})$ to be the thickness of the frame $Q_1\setminus Q_0$. Since we are planning to cover each sector with chambers of measure $\frac{|A|}{M}$ (and thus of diameter of order $\frac{\sqrt{|A|}}{\sqrt{M}}$), it is natural to say that the number of chambers that we expect to be the right one, allocable in each sector, would be $\frac{d}{\frac{\sqrt{|A|}}{\sqrt{M}}}=\frac{d\sqrt{M}}{\sqrt{|A|}}$. Hence we define the integer value
\begin{equation}\label{il numero s}
s=\left\lceil\frac{d\sqrt{M}}{\sqrt{|A|}} \right\rceil.
\end{equation}
We can write,
$$M=sk+r, \ \ \ \ \ r<s,$$
for some $k\in \N$. We thus divide $A$ in $k$ sectors $S_1,S_2,\ldots,S_k$ in a way that each sector $S_i$  has Lebesgue measure exactly equal to $s\frac{|A|}{M}$, plus an eventual remainder sector $R=A\setminus \left(\cup_i S_i\right)$ with measure $|R|<s \frac{|A|}{M}$ (see Figure \ref{costruzione chirurgia 1}).\\

An upper bound on the value of $k$ can be obtained by exploiting the relation:
	\[
	sk  \frac{|A|}{M}\leq  |A|,
	\]
which, thanks to the definition of $s$ in \eqref{il numero s}, implies
\begin{align*}
 |A|\geq sk  \frac{|A|}{M} \geq d k\sqrt{\frac{|A|}{M}},
\end{align*}
and hence:
\begin{equation}\label{quanti k-a}
k \leq \frac{\sqrt{M |A|}}{d}.
\end{equation}

We then divide each sector $S_i$ and the sector $R$ with circular arcs having a suitable radius and centered at $O$ in order to obtain chambers of area exactly $\frac{|A|}{M}$. In this way each sectors but the sector $R$ is containing exactly $s$ chambers (note that the sector $R$ will contain $r<s$ chambers). Of course, since each chambers has area $|A|/M$, we end up with exactly $M$ chambers. We thus define $\E_M$ to be the cluster given by this construction (see Figure \ref{costruzione chirurgia 3}).\\

\begin{figure}[h!]
\centering
\subfloat[][\emph{We radially partition $A$ in sectors $S_i$ enclosing the same (suitable) amount of area plus an eventual remaining sector $R$ enclosing a possibly smaller area.}]
{\includegraphics[width=.48\columnwidth]{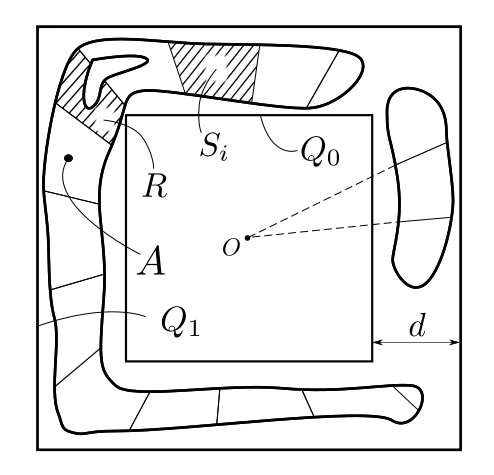}\label{costruzione chirurgia 1} }\quad
\subfloat[][\emph{We proceed to partition each sector in chambers of area $|A|/M$ with circular arcs centered at $O$. The length $\a$ of each circular arc is always less than the length $\rho$ of its radial projection onto $\pa Q_1$.}]
{\includegraphics[width=.47\columnwidth]{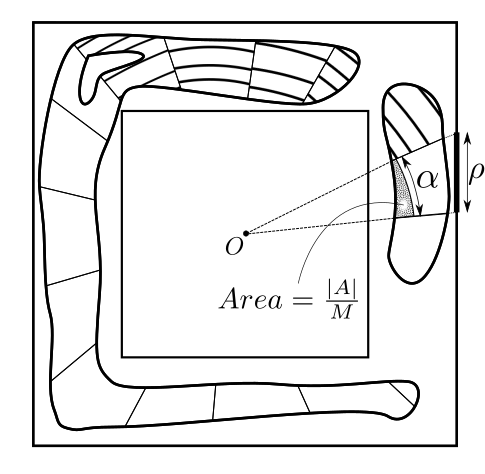}\label{costruzione chirurgia 2}}\quad
\subfloat[][\emph{The $M$-cluster $\E_M$}.]
{\includegraphics[width=.48\columnwidth]{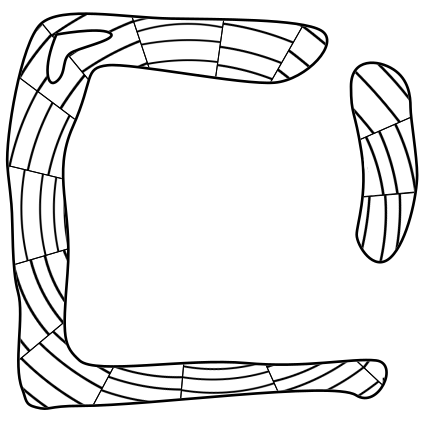}\label{costruzione chirurgia 3}}
\caption{The construction of the $M$-cluster $\E_M$ in the proof of Proposition \ref{chirurgia}.}
\end{figure}

To build the sectors $S_i$ we make use of $k$ segments of length less than $2d$. Note that two arcs from the same sector have the same radial projection on $\pa Q_1$ and each circular arc has length less than the length of its radial projection onto $\pa Q_1$ (see Figure \ref{costruzione chirurgia 2}).\\

These facts lead us to say that the global contribution of the circular arcs to the perimeter of $\E_M$ will be less than $P(Q_1)s$. Hence, thanks to \eqref{quanti k-a}, the global perimeter of $\E_M$ inside $A^{(1)}$ is easily estimated by
\begin{eqnarray*}
P(\E_M;A^{(1)})&\leq &2 k d+P(Q_1)s\\
&\leq & 2\sqrt{M |A|}+2P(Q_1)\frac{d\sqrt{M}}{\sqrt{|A|}}\\
&\leq & 2\left(|A|+P(Q_1)d\right)\frac{\sqrt{M}}{\sqrt{|A|}}.
\end{eqnarray*}
By noticing that $P(Q_1)d\leq 4|Q_1\setminus Q_0|$ and $|A|\leq |Q_1\setminus Q_0|$ we reach 
\begin{eqnarray}
P(\E_M; A^{(1)})&\leq &C |Q_1\setminus Q_0|\frac{\sqrt{M}}{\sqrt{|A|}},
\end{eqnarray}
for a universal constant $C>0$ and thus
\begin{eqnarray}
P(\E_M)&\leq &C |Q_1\setminus Q_0|\frac{\sqrt{M}}{\sqrt{|A|}}+P(A).
\end{eqnarray}
\end{proof}

\section{Uniform distribution for clusters with equi-bounded diameter.} \label{sct equi bounded diameter}
We are always considering $N$-clusters from the class $\text{Cl}(N,\Om)$ defined in \eqref{classe di competizione} where $\Om$ is an open bounded set with finite perimeter and $N\in \N$ is a natural number. 
\begin{definition}\label{mu bounded}
Let $\mu>0$ be a positive constant. An $N$-Cluster $\E\in \text{Cl}(N,\Om)$ is said to be a \textit{$\mu$-bounded $N$-Cluster for $\Om$} if
$$\diam(\E(i))\leq \mu\sqrt{\frac{|\Om|}{N}} \ \ \ \ \text{for all $i=1,\ldots,N$}.$$
If $\E$ is also a minimizing cluster for $\Om$ we call it a \textit{$\mu$-bounded minimizing $N$-cluster}.
\end{definition} 
On this class we are able to prove the following Theorem:
\begin{theorem}\label{Equi dia}
Let $\Om$ be an open and bounded set with finite perimeter. There exists a universal constant $C>0$ with the following property. For every  $\mu\geq \diam(H)$, every closed cube $Q_l\cc \Om$ such that
\begin{equation}\label{vincoli dia}
d(\pa Q_l, \pa \Om)  > 4\mu\sqrt{\frac{|\Om|}{N}}, \ \ \ \ \ \ \ \  l\geq 6\mu\sqrt{\frac{|\Om|}{N}}
\end{equation}
and every $\E\in \Cl$ $\mu$-bounded minimizing $N$-cluster the following holds:
\begin{equation}\label{tesi dia}
\Big{|}P(\E;Q_l)- |Q_l| \frac{P(H)}{2} \sqrt{\frac{N}{|\Om|}}\Big{|} \leq CP(Q_l)\mu.
\end{equation}
\end{theorem}
In this class we can exploit the advantage that the chambers cannot be ``too long". Note that what really matters in Theorem \ref{Equi dia} is how small is the size of the expected chambers $(N^{-\frac{1}{2}})$ compared to the size of the cube $Q_l$ that we are considering.
\begin{remark}\label{remark: come mai ci va il mu0 nel teorema sui mu-bounded}
\rm{
It may seems that the restriction $\mu>\diam(H)$ is a disadvantage in all the eventual situations where a very small diameter is attained. But the point is that the class of $\mu$-bounded $N$-cluster of $\Om$ is empty when $\mu$ is too small. Indeed, thanks to the planar isodiametric inequality \eqref{eqn: chapter intro isodiametric} we have that 
$$\diam(\E(i))\geq 2\sqrt{\frac{|\E(i)|}{\pi}},$$
so 
$$\mu\geq \frac{2}{\sqrt{\pi}}.$$
Thus it is not restrictive to require $\mu>\mu_0$ for some universal constant $\mu_0$ and the choice of $\mu_0=\diam(H)$ is just the most convenient one.
}
\end{remark}
\begin{remark}\label{remark: servono informazioni su mu}
\rm{
Note that each $N$-cluster is a $\mu$-bounded minimizing $N$-cluster with
$$\mu:=\sqrt{\frac{N}{|\Om|}}\max\left\{\diam(\E(i)) \ | \ i=1,\ldots,N\right\}.$$
The fact that $\mu$ appears in the right-hand side of \eqref{tesi dia} means that, without a good information about $\mu$, at an  asymptotic level the estimate is meaningless. For example if we only know that $\mu\leq N^{\frac{1}{2}}$, we get
$$\Big{|}P(\E;Q_l)- |Q_l| \frac{P(H)}{2} \sqrt{\frac{N}{|\Om|}}\Big{|} \leq CP(Q_l)N^{\frac{1}{2}}$$
which does not carry any information.  The optimal situation, when the Theorem becomes sharp, is attained when $\mu$ is of order of a constant, meaning that each chamber has diameter really of order $N^{-\frac{1}{2}}$. Let us also point out that, since $\mu\sqrt{\frac{|\Om|}{N}}$ is the size of each chamber both the restrictions on $l$ appearing in Theorem \ref{Equi dia} are sharp.  
}
\end{remark}

We premise the geometric construction of a competitor, working on every $\mu$-bounded $N$-cluster.

\subsection{Construction of a competitor}
\begin{proposition}\label{upper diametro}
Let $\Om$ be an open and bounded set with finite perimeter. There exists a universal constant $C_0$  with the following property. For every $\mu\geq \diam(H)$ , every closed cube $Q_l\cc \Om$ with 
$$d( \pa Q_l, \pa \Om)>4\mu\sqrt{\frac{|\Om|}{N}}, \ \ \ \ \ l\geq 6\mu\sqrt{\frac{|\Om|}{N}}$$
every $\mu$-bounded $N$-cluster $\E\in \Cl$ there exists an $N$-cluster $\F\ in \Cl$ for which the following estimate holds:
\begin{equation}\label{stima con diametro}
P(\F)\leq |Q_l| \frac{P(H)}{2}\sqrt{\frac{N}{|\Om|}}+P(\E;Q_{l}^c)+C_0 P(Q_l)\mu.
\end{equation}
\end{proposition}
\begin{proof}
Thanks to the assumption on $d( \pa Q_l,\pa \Om)$, setting $d:=2\mu \sqrt{\frac{|\Om|}{N}}$, we can consider the cube $Q_{l+d}$ concentric to $Q_l$  and still have $Q_{l+d}\cc \Om$ (see Figure \ref{Diam1}). With this choice, since $\E$ is a $\mu$-bounded $N$-cluster, every chamber intersecting $Q_{l+d}^c$ does not intersect $Q_{l}$. Set
\begin{equation}
\left\{
\begin{array}{ll}
I_{l,d}&=\ \left\{i\in \{1,\ldots ,N\} \ | \ \E(i) \cc Q_{l+d} \right\},\\
k(l,d)&= \#(I_{l,d}),\\
\end{array}
\right. \,
\end{equation}
and let us define 
$$\F_0=\{\E(i) \ | \ \E(i)\cap Q_{l+d}^c\neq \emptyset\}.$$
We now remove all the $k$ chambers $\E(i)$ compactly contained into $Q_{l+d}$ and we completely cover $Q_l$ with an hexagonal grid. We can be sure that this hexagonal grid does not overlap $\F_0$ since $\mu\geq \diam(H)$. 

\begin{figure}[h!]
\centering
\subfloat[][\emph{We place the cube $Q_l\cc \Om$ and we suitably enlarge it in a cube $Q_{l+d}$.}]
{\includegraphics[width=.7\columnwidth]{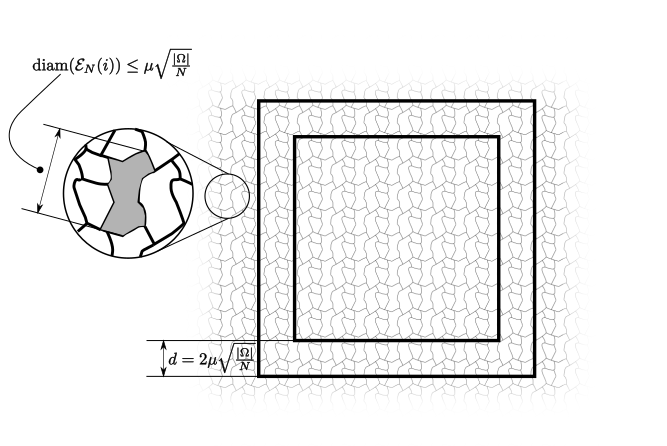}\label{Diam1} }\quad
\subfloat[][\emph{We remove all the chambers compactly contained into $Q_{l+d}$ and we cover $Q_l$ with an hexagonal grid. The request $\mu>\diam(H)$ and the choice of $d$ ensure us that the grid do not overlap the remaining chambers}]
{\includegraphics[width=.47\columnwidth]{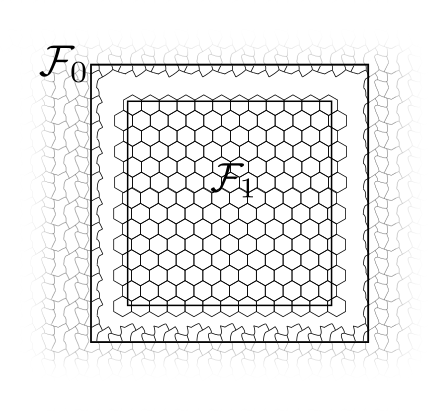}\label{Diam2}}\quad
\subfloat[][\emph{We apply Lemma \ref{lemma chirurgia} to build a partition of the remaining part of $Q_{l+d}\setminus Q_l$}.]
{\includegraphics[width=.47\columnwidth]{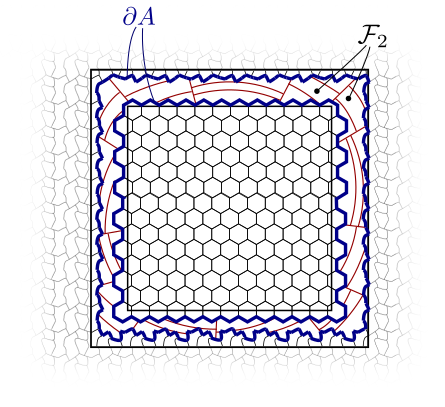}\label{Diam3}}
\caption{The construction of the $N$-cluster $\F$ in the proof of Proposition \ref{upper diametro}.}
\end{figure}
Denote with $h$ the total number of hexagons needed to completely cover $Q_l$. We build our covering in a way that each hexagon intersects $Q_l$ (see Figure \ref{Diam2}). In this way the hexagonal grid is completely contained into $Q_{l+\diam(H)\sqrt{\frac{|\Om|}{N}}}$ and so, since 
$$l\geq 6\mu\sqrt{\frac{|\Om|}{N}}\geq 6\diam(H)\sqrt{\frac{|\Om|}{N}}, $$
we obtain
\begin{eqnarray*}
h\frac{|\Om|}{N}&\leq & \left(l+\diam(H)\sqrt{\frac{|\Om|}{N}}\right)^2 \leq  l^2+Cl\sqrt{\frac{|\Om|}{N}} 
\end{eqnarray*}
where $C$ is a universal constant. Hence
\begin{equation}\label{quanti h1}
h \leq  l^2\frac{N}{|\Om|}+Cl\sqrt{\frac{N}{|\Om|}}.
\end{equation}
Denote this hexagonal grid with $\F_1$. The perimeter of $\F_1$ is estimated by
$$P(\F_1)\leq h \frac{P(H)}{2}\sqrt{\frac{|\Om|}{N}}+P(\F_1(0)).$$
It is straightforward that $P(\F_1(0))\leq Cl$ for a universal constant $C$. Thus, thanks to \eqref{quanti h1}, we reach:
\begin{equation}\label{qua}
P(\F_1)\leq l^2 \frac{P(H)}{2}\sqrt{\frac{N}{|\Om|}}+Cl,
\end{equation}
for a universal constant $C$.\\

After the construction of $\F_1$ and $\F_0$ we are left to partition the open set 
$$A=\left(\bigcup_{i\in I_{l,d}}\E(i)\right) \setminus \left(\bigcup_{i=1}^h \F_1(i)\right)$$
(evidenced in blue in Figure \ref{Diam3}). We here make use of Lemma \ref{lemma chirurgia} and we divide $A$ into $(k-h)$ chambers. Note that
$$|A|=k\frac{|\Om|}{N}-h\frac{|\Om|}{N},$$
and thus
\begin{equation}\label{key remark}
\frac{|A|}{k-h}=\frac{|\Om|}{N}.
\end{equation}
Since the set $A$ is contained into $Q_{l+d}\setminus Q_l$ we can apply Lemma \ref{lemma chirurgia} with $Q_0=Q_l$,$Q_1=Q_{l+d}$ and discover that there exists a $(k-h)$-cluster $\E_{k-h}\in \text{Cl}(k-h,A)$ such that 
$$P(\E_{h-k})\leq |Q_{l+d}\setminus Q_l| \sqrt{\frac{(k-h)}{|A|}}+P(A),$$
which, thanks to \eqref{key remark} and since $d<l$, means
\begin{equation} \label{three3-a}
P(\E_{h-k};A)\leq 3ld \sqrt{\frac{N}{|\Om|}}+P(A).
\end{equation}
Setting $\F_2=\E_{k-h}$ and
$$\F=\F_0 \cup \F_1 \cup \F_2,$$ 
clearly $\F\in \Cl$. Notice that (see Figures \ref{Diam2},\ref{Diam3})
\begin{align*}
P(\F)&=P(\F_0)+P(\F_1)+P(\F_2) -P(A).
\end{align*}
Furthermore by exploiting \eqref{qua}, \eqref{three3-a} we obtain
\begin{align*}
P(\F)&=P(\F_0)+P(\F_1)+P(\F_2) -P(A)\\
&\leq  P(\F_0)+l^2 \frac{P(H)}{2}\frac{N}{|\Om|}+3ld \sqrt{\frac{N}{|\Om|}}+Cl.
\end{align*}
Since, by construction, it holds 
$$P(\F_0) \leq P(\E;Q_l^c),$$
by recalling that $d=2\mu\sqrt{\frac{|\Om|}{N}}$ we obtain:
\begin{eqnarray*}
P(\F)&\leq &l^2\frac{P(H)}{2}\frac{N}{|\Om|}+P(\E;Q_l^c)+3ld \sqrt{\frac{N}{|\Om|}}+Cl\\
&\leq &l^2\frac{P(H)}{2}\frac{N}{|\Om|}+ P(\E;Q_l^c)+Cl(\mu+1)\\
&\leq & l^2\frac{P(H)}{2}\frac{N}{|\Om|}+P(\E;Q_l^c)+Cl\mu\left(1 +\frac{1}{\diam(H)}\right)
\end{eqnarray*}
 Setting $C_0=C\left(1+\frac{1}{\diam(H)}\right),$ we get the thesis \eqref{stima con diametro}.
\end{proof}

\subsection{Proof of Theorem \ref{Equi dia}}
\begin{proof}[Proof of Theorem \ref{Equi dia}]
Let $C_0$ be the constant given by Proposition \ref{upper diametro}. Let $\mu>\diam(H)$, $Q_l\cc \Om$ be a closed cube satisfying \eqref{vincoli dia} and $\E\in \Cl$ be a $\mu$-bounded minimizing $N$-cluster for $\Om$. Thanks to Proposition \ref{upper diametro} we can find an $N$-cluster $\F\in \Cl$ for which it holds:
$$P(\F)\leq |Q_l|\frac{P(H)}{2} \sqrt{\frac{N}{|\Om|}}+C_0P(Q_l)\mu+P(\E;Q_l^c).$$
By exploiting the minimality of $\E$ we obtain 
$$P(\E)\leq P(\F)$$
which leads to
\begin{equation}\label{fine dia su}
P(\E;Q_l)\leq |Q_l| \frac{P(H)}{2} \sqrt{\frac{N}{|\Om|}}+C_0P(Q_l)\mu.
\end{equation}
Proposition \ref{lower bounded} ensures that on $Q_l$ it holds 
$$P(\E;Q_l)\geq |Q_l|\frac{P(H)}{2}\sqrt{\frac{N}{|\Om|}}-P(Q_l),$$
and hence
\begin{equation}\label{fine dia giu}
P(\E;Q_l)\geq|Q_l|  \frac{P(H)}{2}\sqrt{\frac{N}{|\Om|}}-\frac{\mu^2}{\diam(H)^2}P(Q_l).
\end{equation}
Up to choosing $C=\max\left\{C_0, \frac{1}{\diam(H)^2}\right\}$, by combining\eqref{fine dia giu} and \eqref{fine dia su} we achieve the proof.
\end{proof}
\section{Uniform distribution for indecomposable minimizing clusters} \label{sct equi indecomposable}

Since getting information about the diameter of the chambers seems to be a very hard task we provide a second result, which applies on a possibly wider class of minimizing $N$-clusters.

\begin{definition} \label{indecomposable minimizers}
An $N$-Cluster $\E\in \Cl$ is said to be an \textit{indecomposable $N$-Cluster for $\Om$} if each chamber $\E(j)$ is an indecomposable set of finite perimeter. If $\E$ is also a minimizing $N$-cluster we call it an \textit{indecomposable minimizing $N$-cluster for $\Om$}.
\end{definition}
On this class the following result holds.
\begin{theorem}\label{Equi indeco}
Let $\Om$ be an open bounded set with Lipschitz boundary and $0\leq \b<\frac{1}{2}$ be a positive real number. Then there exist three positive constant $\eta,\l, C$ depending only on $\b$ and on the shape of $\Om$ with the following property. For every cube $Q_l\cc \Om$ with
\begin{equation}\label{vincoli indeco main thm}
d( \pa Q_l, \pa \Om)  > \eta\sqrt{|\Om|} N^{-\frac{1}{6}}, \ \ \ \ \ \ \ \  l\geq \l\sqrt{|\Om|} N^{-\b}
\end{equation}
and for every indecomposable minimizing $N$-cluster $\E\in \Cl$ the following holds
\begin{equation}\label{tesi indeco main thm}
\Big{|}P(\E;Q_l)- \frac{P(H)}{2}|Q_l| \sqrt{\frac{N}{|\Om|}}\Big{|} \leq C P(Q_l)^{\frac{3}{2}} \left(\frac{N}{|\Om|}\right)^{\frac{1}{4}}.
\end{equation}
\end{theorem}
In this case we follow the simple idea that, the longer is the chamber the bigger will be its contribution to the global perimeter. An a priori estimate (Proposition \ref{trivial estimate}) on the global energy $P(\E)$ allows us to control the number of the bad chambers and leads us to the sought upper bound on $P(\E;Q_l)$. 

\begin{remark}
\rm{
We need $\Om$ to have Lipschitz boundary in order to achieve the proof of Proposition \ref{trivial estimate} (a key step in the proof of Theorem \ref{Equi indeco}).
}
\end{remark}

\begin{remark}\label{remark3}
\rm{We are going to explain in Remark \ref{remark su un sesto} below where the exponent $\frac{1}{6}$ in the hypothesis on the distance between $\pa Q_l$ and $\pa \Om$ comes from.
}
\end{remark}

\begin{remark}
\rm{
Let us remark that the existence of indecomposable minimizing $N$-cluster is actually an open problem though, intuitively, there are good reasons that underline that the class defined in \eqref{indecomposable minimizers} is not empty. In many situations it could happen that the chambers in the proximity of $\pa \Om$ decide to split in order to compensate the effect of a possibly irregular boundary. As an example consider the case when $\Om$ is an open square $Q$ (with $|Q|=N$) union two disjoint thin open rectangles $R_1,R_2$ of area $\frac{1}{2}$, height $t$, length $\frac{1}{2t}$ (see Figure \ref{controesempioconnessione}). If we consider $\E$ a minimizing $(N+1)$-cluster for $\Om$ it is not clear whether it is convenient for $\E$ to be indecomposable or to have a chamber with two big indecomposable components given by $R_1$ and $R_2$.\\

It is reasonable to expect that on every fixed ambient space $\Om$, for $N$ sufficiently large this behavior is avoided at least for every chambers far enough from the boundary. We could have enlarged our class a bit more by requiring the indecomposability only for those chambers lying at a distance $d(N)$ (decaying in $N$) from the boundary of the ambient space, but since our arguments will work in the same way we prefer, for the sake of clarity, not to add this more technical restriction.
\begin{figure}
\begin{center}
 \includegraphics[scale=0.7]{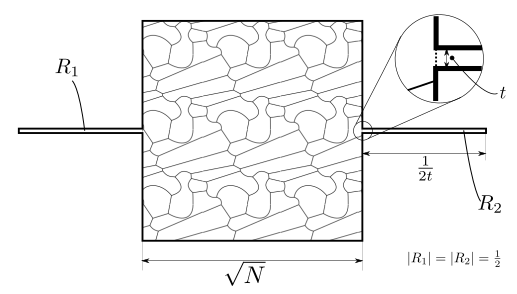}\caption{{For this ambient space $\Om$ and for $t$ small enough the $(N+1)$-minimizers could have at least an indecomposable chamber.}}\label{controesempioconnessione}
\end{center}
\end{figure}
}
\end{remark}

\subsection{Construction of a competitor.}\label{subsct Equi: indecomposable}
The construction of the competitor in the case of indecomposable $N$-cluster is a slight modification of the one developed in the proof of Proposition \ref{upper diametro}. We first state and prove Proposition \ref{bootstrap} which, for a fixed open set $\Om'\subseteq \Om$ and for a fixed (suitable) cube $Q_l\cc \Om'$, starting from a generic $\E\in \Cl$ gives us an $N$-cluster $\F\in \Cl$ having perimeter 
$$P(\F)\leq |Q_l|\frac{P(H)}{2}\sqrt{N} + P(\E; Q_l^c)+ C_0 \sqrt{P(\E;\Om')}\sqrt{l},$$
for some constant $C_0$. \\

The presence of the localized perimeter $P(\E;\Om')$ in the right-hand side of the previous estimate requires some kind of weak control on the perimeter of $\E$ and thus we need to exploit minimality to complete our construction. This is done in Proposition \ref{upper indeco} where a first rough estimate 
\begin{equation}\label{rough estimate}
P(\E;Q_l)\leq P_0|Q_l|\sqrt{N}
\end{equation}
for a universal constant $P_0$,  is obtained for every closed cube $Q_l\cc \Om$ far enough from $\pa \Om$.  The proof of \eqref{rough estimate} is achieved by combining Proposition \ref{bootstrap} with an estimate on the global energy proved in Proposition \ref{trivial estimate}. \\

We choose a cube $Q_l$ satisfying the hypothesis of Theorem \ref{Equi indeco} and we carefully enlarge it into $Q_{l+d}$. By setting $\Om'=\mathring{Q_{l+d}}$ and by applying again Proposition \ref{bootstrap}, thanks to the rough estimate \eqref{rough estimate} on $P(\E;Q_{l+d})$ (provided a suitable $d$), we build the competitor with the desired energy in Proposition \ref{upper indeco quello che serve davvero}.\\

In the following we are always considering an open set $\Om$ with $|\Om|=1$. In the end, with a scaling argument, we achieve the proof of Theorem \ref{Equi indeco} for a generic set $\Om$.

\begin{proposition}\label{bootstrap}
Let $\Om$ be an open, bounded set with Lipschitz boundary and $|\Om|=1$. There exist universal constants $\eta_0,\lambda_0, C_0$ with the following property. Let $\Om ' \subseteq  \Om$ and let $\E\in \Cl$ be a generic indecomposable $N$-cluster. Set:
\begin{equation}\label{passo di boot}
P:=P(\E;\Om ' )
\end{equation} 
Then for every closed cube $Q_l \cc \Om' $ with 
\begin{equation}\label{1 vincoli su l}
 d(\pa Q_l, \pa \Om ' )  > \eta_0 \sqrt{\frac{P}{l N}}, \ \ \ \ \ \ \ \ l\geq  \l_0\left[ \left(\frac{P}{N}\right)^{\frac{1}{3}}+\frac{1}{\sqrt{N}}\right]
 \end{equation} 
%
there exists $\F\in \Cl$ for which the following estimate holds:
\begin{eqnarray}
P(\F)\leq |Q_l|\frac{P(H)}{2}\sqrt{N} + P(\E; Q_l^c)+ C_0 \sqrt{Pl}. \label{upper only}
\end{eqnarray}
\end{proposition}
\begin{remark}
\rm{
Note that assumptions \eqref{1 vincoli su l} implies that Proposition \ref{bootstrap} is meaningless whenever the energy of the indecomposable $N$-cluster for $\Om$ is too much. In particular the restriction on the size of the cube implies that $P(\E;\Om')$ must be less than or equal to $\frac{l^3}{\l_0^3}N$. In particular it could happen that for some ``wrong"choice of $\E$ there are no cubes satisfying restrictions \eqref{1 vincoli su l}. However, we are going to apply Proposition \ref{bootstrap} on the indecomposable minimizing $N$-clusters for $\Om$ where an upper bound on the global energy is always attained (see Proposition \ref{trivial estimate}).
}
\end{remark}
\begin{remark}
\rm{
Note that the exponent $1/3$ in the hypothesis \eqref{1 vincoli su l} on $l$ cannot be improved. Indeed, assume that we are able to prove Proposition \ref{bootstrap} with
 	\begin{equation}\label{non so se e il caso ma comqunque}
  	l\geq  \l_0\left[ \left(\frac{P}{N}\right)^{\a}+\frac{1}{\sqrt{N}}\right],
 	\end{equation}

 for some $\a>1/3$. As we are going to show below in Proposition \ref{upper indeco}, if $\E$ is a perimeter-minimizing $N$-cluster for $\Om$ then $P=P(\E;Q_l)\approx l^2\sqrt{N}$ which means that \eqref{non so se e il caso ma comqunque} on $l$ can be wrote as
	\[
	l\geq l^{2\a } N^{-\a/2}, \ \ \ \ l \geq N^{\frac{-\a}{2-4\a}}.
	\]
Since $l\geq N^{-1/2}$ (it does not make sense to consider cubes smaller than the expected size of the chambers) we are lead to
	\[
	 N^{\frac{1}{2}- \frac{\a}{2-4\a}} \geq 1 \ \Rightarrow \  \a \leq 1/3.
	\]
}
\end{remark}
\begin{proof}[Proof of Proposition \ref{bootstrap}]
Fix $Q_l\cc \Om'$ satisfying \eqref{1 vincoli su l}. Note that, thanks to Proposition \eqref{lower bounded}, on $Q_l$ it holds
\begin{align*}
P(\E;Q_l)&\geq  |Q_l|\frac{P(H)}{2}\sqrt{N}-P(Q_l)=l^2\frac{P(H)}{2}\sqrt{N}-4l ,
\end{align*}
and since $Q_l\cc \Om'$:
\begin{eqnarray*}
P=P(\E;\Om')&\geq &P(\E;Q_l) \\
&\geq &  l^2\frac{P(H)}{2}\sqrt{N}-4l.
\end{eqnarray*}
Hence, by using \eqref{1 vincoli su l} and observing that $l\sqrt{N}\geq \l_0$,
\begin{eqnarray*}
\frac{P}{l}&\geq & \left(\frac{P(H)}{2}\l_0 -4\right).
\end{eqnarray*}
Thus, up to taking $\l_0$ bigger than a universal constant we can always assume
\begin{equation}\label{comoda}
\frac{P}{l}\geq 1.
\end{equation}
Let $d\in \R$ be defined as
\begin{equation}\label{1 vincolo su d}
d:=\eta \sqrt{\frac{P}{l N}},
\end{equation}
for some $\eta$ to be chosen (we postpone the choice of $\l_0, \eta_0, \eta$ to the end of the proof).  Let us set the restriction
$$\frac{l}{100}\geq d$$
in order to be sure that $d$ is much smaller than $l$. This leads to the restriction:
\begin{equation*}
 l\geq (100\eta)^{\frac{2}{3}} \left(\frac{P}{N}\right)^{\frac{1}{3}},
\end{equation*}
which becomes immediately a restriction on $\l_0$
\begin{equation}\label{su lambda 0}
\l_0\geq (100\eta)^{\frac{2}{3}}.
\end{equation}
In this way the concentric closed boxes
$$Q_l\cc Q_{l+\frac{1}{4}d}\cc Q_{l+\frac{3}{4}d} \cc Q_{l+d},$$ 
are all compactly contained into $\Om'$ providing $\eta_0>2\eta$ (see Figure \ref{figura1}).\\

\begin{figure}[h!]
\centering
\subfloat[][\emph{We consider a square $Q_{l}\cc Q_{l+d}\cc \Om' $. The square $Q_l$ is where we perform our construction. We remove all the chambers $\E(i)$ for $i\in I_{l,d}$ (defined in \eqref{number of}) completely lying inside $Q_{l+d}$. We save a small substripe $Q_{l+\frac{3}{4}d}\setminus Q_{l+\frac{1}{4}d}$ where we are going to allocate some small squares $Q_j$ needed to recover the area loss during the construction.}]
{\includegraphics[width=.47\columnwidth]{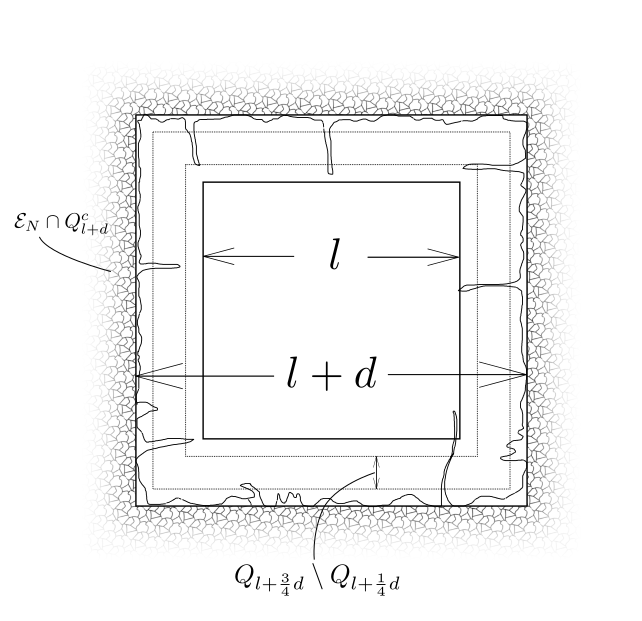}\label{figura1} }\quad
\subfloat[][\emph{We cut away all the tentacles $\E(j)\cap Q_{l+d}$ for $j\in J_{l,d}$ (defined in \eqref{number of}) and we substitute them with squares $Q_j$ inside the rectangle $Q'\subset Q_{l+\frac{3}{4}d}\setminus Q_{l+\frac{1}{4}d}$. We allocate each $Q_j$ inside a square of the pre-allocate grid $\mathcal{Q}_m$. The number of such long chambers $m(l,d)$ (and thus of the chambers of the grid $\mathcal{Q}_m)$ is controlled by the starting upper bound on the localized energy \eqref{passo di boot}.}]
{\includegraphics[width=.47\columnwidth]{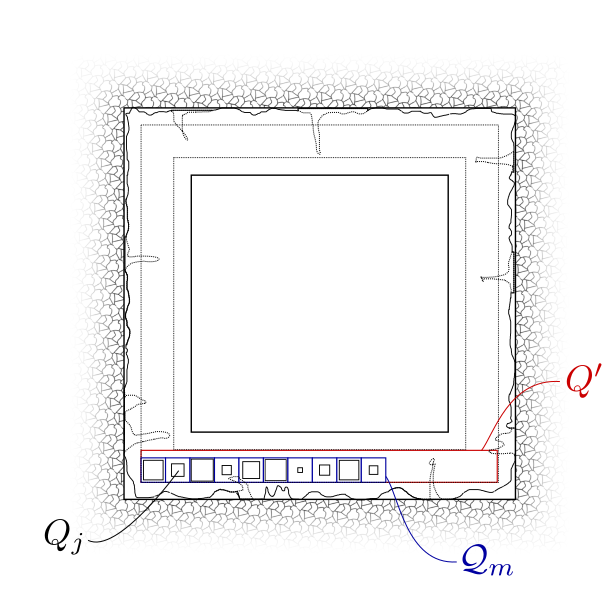}\label{figura2}}\quad
\subfloat[][\emph{We can now cover the square $Q_l$ with an hexagonal grid, without overlapping the other part of the construction.}]
{\includegraphics[width=.47\columnwidth]{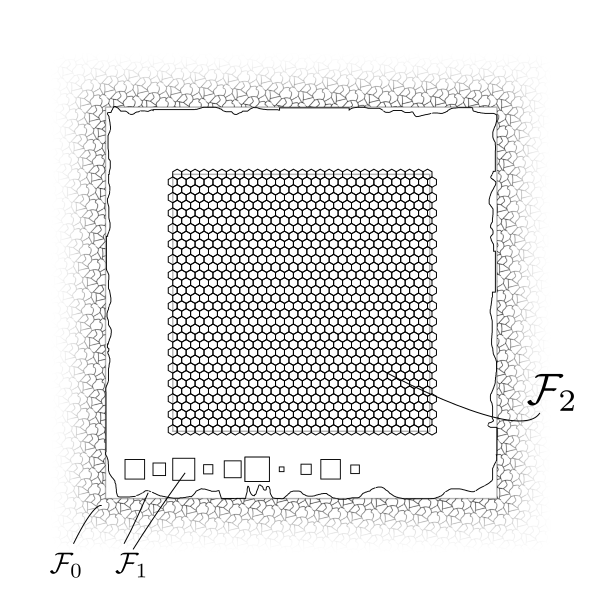}\label{figura3}}\quad
\subfloat[][\emph{We finally get the competitor $\F$, an $N$-cluster of equal area chambers agreeing with $\E$ outside of $Q_{l+d}$ and satisfying \eqref{upper only} by using Lemma \ref{lemma chirurgia} to cover the remaining part of $Q_{l+d}\setminus Q_l$ with chambers having the right amount of area.}]
{\includegraphics[width=.47\columnwidth]{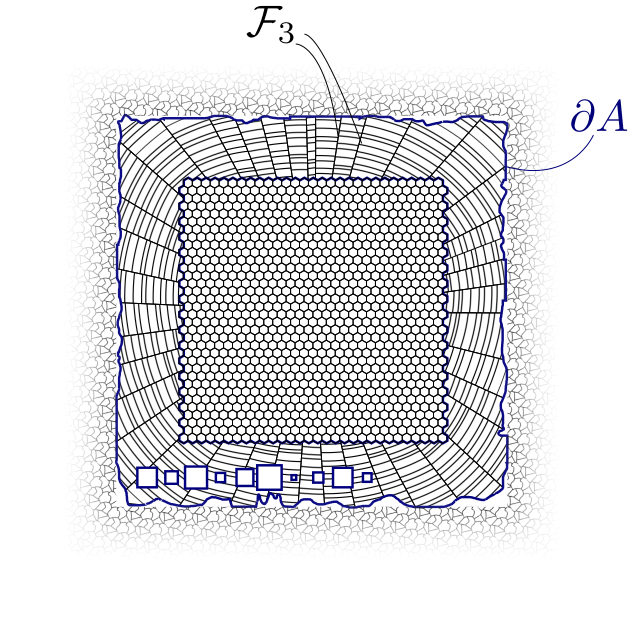}\label{figura6}}
\caption{The construction of the $N$-cluster $\F$ in the proof of Proposition \ref{bootstrap}.}
\end{figure}
Define the sets:
\begin{equation}\label{number of}
\left\{
\begin{array}{ll}
I_{l,d}&=\ \left\{i\in \{1,\ldots ,N\} \ | \ \E(i) \cc Q_{l+d} \right\},\\
k(l,d)&= \#(I_{l,d}),\\
J_{l,d}&:=\left\{j\in\{1,\ldots,N \} \ | \ \E(j)^{(1)} \cap \pa Q_{l+d}\neq  \emptyset \ \text{and} \ \E(j)^{(1)} \cap \pa Q_{l+\frac{3}{4}d}\neq   \emptyset \right\}, \\
m(l,d)&:=\#(J_{l,d}).
\end{array}
\right.
\end{equation}

We now divide the proof into four steps, for the sake of clarity. In the sequel we always adopt the same letter (namely $C$, except for $\eta, \eta_0, \l_0$) for the constants though the value of the constants can change from line to line. Let us set $\de=\frac{1}{N}$.\\
\text{}\\
\textit{Step one: Figure \ref{figura2}. The cluster $\F_0$ and $\F_1$: replacement of the long chambers.} In this step we provide a suitable adjustment of all the $m$ chambers $\E(j)$ for $j\in J_{l,d}$ that are too long. \\

We cut the part of the chambers $\{\E(j)\}_{j\in J_{l,d}}$ lying inside $Q_{l+d}$. After this operation we need to recover the loss of area. Our aim is to recover the area by placing $m$ small cubes with the right amount of area inside $Q'$ (evidenced in red in Figure \ref{figura2}) the lower rectangle of the stripe $Q_{l+\frac{3}{4}d}\setminus Q_{l+\frac{1}{4}d}$. To do that we first place a big grid $\mathcal{Q}_m$ (evidenced in blue in Figure \ref{figura2}) of $m$ boxes of area $\de$ (suitably arranged as in Figure \ref{griglia}) that we are using as skeleton. Inside each box we place a cube $Q_j$ having the right amount of area ($|Q_j|=|\E(j)\cap Q_{l+d}|$ for $j\in J_{l,d}$) and we complete the construction. Clearly, to perform this construction we need to show that there is enough space inside $Q'$. By making use of an estimate on the number $m(l,d)$ and provided $\l_0$ and $\eta$ are big enough, we show that this is the case.\\

Note that  since $\E$ is an indecomposable $N$-cluster, for every $j\in J_{l,d}$ it holds 
$$\diam(\E(j)^{(1)})\geq \frac{d}{4}$$ 
and thus thanks to Proposition \ref{indeco piano} we must have $P(\E(j);Q_{l+d})\geq\frac{d}{2}$. By the trivial upper bound 
$$\sum_{j\in J_{l,d}} P(\E(j);Q_{l+d})\leq 2P(\E; Q_{l+d})\leq 2P$$ 
we obtain
\begin{eqnarray*}
m \frac{d}{2}&\leq & \sum_{j\in J_{l,d}} P(\E(j);Q_{l+d})\leq 2P
\end{eqnarray*}
and thus
\begin{equation}\label{bound m(l,d)}
m\leq 4\frac{P}{d}.
\end{equation}
The total area of the union of the long chambers inside $Q_{l+d}$ is easily estimated from above by 
$$\sum_{j\in J_{l,d}} | \E(j) \cap Q_{l+d} |\leq m \de,$$
which, combined with \eqref{bound m(l,d)}, implies:
\begin{eqnarray}
\sum_{j\in J_{l,d}} | \E(j) \cap Q_l | &\leq & m\de\nonumber\\
 &\leq & 4\frac{P}{d}\de.\ \ \ \label{miley cirus}
\end{eqnarray}
Since we want to cut the long chambers and rebuild them into $Q'$ (evidenced in red in Figure \ref{figura2}) where $|Q'|=\left(l-\frac{d}{2}\right)\frac{d}{2}>\frac{ld}{4}$ (because of $d<\frac{l}{100}<\frac{l}{2}$), it is enough to ensure that
\begin{equation}\label{goal}
\sum_{j\in J_{l,d}} | \E(j) \cap Q_{l+d} |\leq \frac{ld}{4}
\end{equation}
which, thanks to \eqref{miley cirus}, can be obtained as a consequence of
$$4\frac{P}{d}\de \leq  \frac{ld}{4},$$
or equivalently
\begin{equation}\label{correction1}
 d\geq 4\sqrt{\frac{P}{l}\de}.
\end{equation}  
Thanks to the definition of $d$ \eqref{1 vincolo su d}, up to taking $\eta$ bigger than a universal constant (as well as  $\l_0$ according to \eqref{su lambda 0} ) we can always ensure the validity of \eqref{correction1} and thus of \eqref{goal}. \\

We now show how to place the grid $\mathcal{Q}_m$ (see Figure \ref{griglia}). 
\begin{figure}[t!]
\centering
\includegraphics[width=.7\columnwidth]{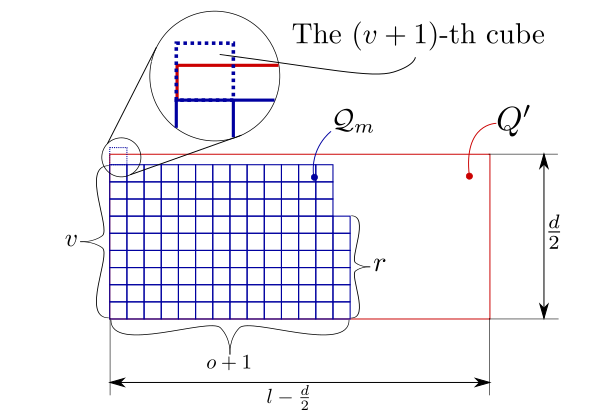}
\caption{We place the skeleton grid $\mathcal{Q}_m$ by choosing $v\in \N$ in a way that exactly $v$ vertical cubes of area $\de$ (and no one more) are contained into $Q'$. Then we choose $o,r$ to be the integer such that $m=vo+r$.}\label{griglia} 
\end{figure}
Choose $v\in \N$ such that
$$v\sqrt{\de}\leq \frac{d}{2}, \ \ \ \frac{d}{2}\leq (v+1)\sqrt{\de}.$$ 
The number $v$ represent the maximum number of cubes of area $\de$ that we can place "vertically" inside $Q'$ (for example, $v=1$ in Figure \ref{figura2}). In particular
\begin{equation}
\frac{d}{2\sqrt{\de}} - 1 \leq v\leq \frac{d}{2\sqrt{\de}}.
\end{equation} 
Let $o,r\in \N$ be such that:
$$m=vo+r, \ \ r\leq v.$$
We choose $\mathcal{Q}_m$ to be a grid of $o+1$ columns of cubes where the first $o$ columns are made by $v$ cubes and the $(o+1)-$th column contains exactly $r$ cubes of area $\de$ (see Figure \ref{griglia} where a generic situation is represented, or \ref{figura2} where $v=1$). 
Clearly 
$$ vo \leq m \leq( v+1)o,$$
and so 
\begin{equation}\label{pappapapa}
\frac{2m\sqrt{\de} }{d +2\sqrt{\de}} \leq o\leq \frac{2m\sqrt{\de} }{d-2\sqrt{\de} }.
\end{equation}
In order to be sure that we have enough space inside $Q'$ (to insert the grid $\mathcal{Q}_m$) we need to check that
 $$(o+1)\sqrt{\de}\leq l-\frac{d}{2}.$$
Since, thanks to \eqref{pappapapa} and to \eqref{bound m(l,d)},
 \begin{eqnarray*}
 (o+1)\sqrt{\de} &\leq & \left(\frac{2(m-1)\sqrt{\de} +d }{d-2\sqrt{\de} }\right) \sqrt{\de}\\
 &\leq &\left(\frac{2m\sqrt{\de} +d }{d-2\sqrt{\de} }\right) \sqrt{\de}\\
 &\leq &\left(\frac{8P\sqrt{\de} +d^2 }{d^2-2d\sqrt{\de} }\right) \sqrt{\de},
 \end{eqnarray*}
it is enough to check
$$ \left(\frac{8P\sqrt{\de} +d^2 }{d^2-2d\sqrt{\de} }\right) \sqrt{\de}\leq \frac{l}{2},$$
which means
\begin{eqnarray*}
d^2\left(\sqrt\de -\frac{l}{2}\right)+dl\sqrt\de + 8P\de\leq 0,
\end{eqnarray*}
that is satisfied when
$$d\geq \sqrt\de \left(\frac{1+\sqrt{1+\frac{32P}{l}\left(\frac{1}{2}-\frac{\sqrt{\de}}{l}\right)}}{1-2\frac{\sqrt\de}{l}}\right).$$
By exploiting $P/l\geq 1$ and up to taking $\l_0$ and $\eta$ bigger than a universal constant, by exploiting \eqref{1 vincoli su l}, we can always ensure that the previous condition is satisfied by $d$.\\

Thus we have space to place the grid $\mathcal{Q}_m\subset Q'$. For every $j\in J_{l,d}$ we consider a cube $Q_j$ with the property $|Q_j|=|\E(j)\cap Q_{l+d}|$ and we place it into an empty box of $\mathcal{Q}_m$. We define the following clusters $\F_0$ and $\F_1$ as:
\begin{eqnarray*}
\F_{0}(j)&=&\E(j) \ \ \ \text{$j\notin J_{l,d}\cup I_{l,d}$}\\
\F_{1}(j)&=&(\E(j)\cap \overline{Q_{l+d}^{c}}) \cup Q_j \ \ \ \text{$j\in J_{l,d}$}.
\end{eqnarray*}
By construction each chamber of $\F_0$ and $\F_1$ has area $\frac{1}{N}$. Moreover
\begin{eqnarray}
P(\F_0\cup \F_1) &\leq & P(\E ; Q_{l+d}^{c})+4l+\sum_{j\in J}P(Q_j)\nonumber\\
&\leq & P(\E ; Q_{l+d}^{c})+C\left(l+\frac{ P}{d}\sqrt{\de}\right) \label{one1}
\end{eqnarray}
for a universal constant $C$, because of \eqref{bound m(l,d)}.\\
\text{}\\
\textit{Step two: Figure \ref{figura3}. The $h$-cluster $\F_2$: the hexagonal tiling.}
We completely cover $Q_l$ with an hexagonal grid (see Figure \ref{figura3}). As in the proof of Proposition \ref{upper diametro} we do not consider the hexagons that do not intersect $Q_l$. Up to taking $\l_0$ and $\eta$ bigger than a universal constant the total number of hexagons $h$ (as in the proof of Proposition \ref{upper diametro}) is estimated by:
\begin{eqnarray*}
h\de &\leq & \left(l+\diam(H)\sqrt{\de}\right)^2 \leq l^2+Cl\sqrt{\de}
\end{eqnarray*}
where $C$ is a universal  constant. Hence
\begin{equation}\label{quanti h}
h \leq \frac{l^2}{\de}+C\frac{l}{\sqrt{\de}},
\end{equation}
If we denote with $\F_2$ such a cluster we obtain:
\begin{equation}\label{two2}
P(\F_2)\leq \frac{P(H)}{2}\frac{l^{2}}{\sqrt{\de}}+Cl,
\end{equation}
for a universal constant $C$. Note that up to choose a universal $\eta$ big enough we can ensure that the cluster $\F_1$ and $\F_2$ do not overlap.\\
\text{}\\
\textit{Step three: Figure \ref{figura6}. The $(k-h)$-Cluster $\F_3$: a link between $\F_0$ and $\F_2$.}
After the construction of $\F_2$ and $\F_1$ in the first two steps we are left to partition the set $A=(Q_{l+d}\setminus [\F_0\cup \F_1\cup \F_2])$ (evidenced in blue in Figure \ref{figura6}). 
We use Lemma \ref{lemma chirurgia} to build an $(k-h)$-cluster $\E_{k-h}$ with
$$P(\E_{k-h})\leq C|Q_{l+d}\setminus Q_l|\sqrt{\frac{k-h}{|A|}}+P(A),$$
for a universal constant $C$. By construction and thanks to the choice of $d$ the following hold
$$\frac{|A|}{k-h}=\frac{1}{N}=\de, \ \ \ \ \ \ |Q_{l+d}\setminus Q_l|\leq 3ld.$$
We thus set $\F_3=\E_{k-h}$:
\begin{equation}\label{quasi tre} 
P(\F_3)\leq C \frac{ld}{\sqrt{\de}}+P(A).
\end{equation}

\textit{Step four: Figure \ref{figura6}. The $N$-cluster $\F$ and estimate \eqref{upper only}}. We now consider the $N$-cluster 
$$\F=\F_0 \cup \F_1\cup\F_2\cup \F_3.$$ 
 Notice that (see Figures \ref{figura3},\ref{figura6} ) 
\begin{align*}
P(\F)&=P(\F_0\cup \F_1)+P(\F_2)+P(\F_3)-P(A).
\end{align*}
Furthermore, by exploiting \eqref{one1}, \eqref{two2}, \eqref{quasi tre}
\begin{align}
P(\F)&= P(\F_0\cup \F_1)+P(\F_2)+P(\F_3)-P(A)\nonumber\\ 
&\leq \frac{P(H)}{2}\frac{l^{2}}{\sqrt{\de}}+P(\E ; Q_{l+d}^{c})+C\left(l+ \frac{ P}{d}\sqrt{\de}+ \frac{ld}{\sqrt{\de}}\right)\label{daidai4} 
\end{align}
Notice that 
\begin{eqnarray*}
l+\frac{P}{d}\sqrt{\de} \leq \frac{ld}{\sqrt{\de}},
\end{eqnarray*}
when
\begin{equation}\label{d satisfies}
 d\geq \sqrt{\de} \left(1+\sqrt{1+4\frac{P}{l}}\right).
\end{equation} 
Up to taking $\eta$ and $\l_0$ bigger than a universal constant, thanks to \eqref{comoda}, we can always guarantee that $d$ satisfies \eqref{d satisfies}. Hence, \eqref{daidai4} leads us to
\begin{equation}\label{ho finito i nomi delle label}
P(\F) \leq\frac{P(H)}{2}\frac{l^2}{\sqrt{\de}}+P(\E;Q_{l}^{c})+C\frac{ld}{\sqrt{\de}}. 
\end{equation} 
We now fix a universal $\eta$ big enough. After we set $\eta_0=4\eta$ and we choose $\l_0$ big enough in dependence on $\eta$ (and thus universal) satisfying \eqref{su lambda 0}. Since $d=\eta\sqrt{\frac{P}{lN}}$ and $\de=\frac{1}{N}$ we obtain from \eqref{ho finito i nomi delle label}:
\begin{equation}\label{upper only 0}
P(\F) \leq \frac{P(H)}{2}l^2\sqrt{N}+P(\E;Q_{l}^{c})+C\eta \sqrt{Pl}.
\end{equation} 
By setting $C_0=C\eta$ we get \eqref{upper only} from \eqref{upper only 0}.
\end{proof}
Condition \eqref{1 vincoli su l} needs a starting energy estimate that is provided in the next Proposition. It is obtained by comparison, with a competitor (see Figure \ref{figura partizione aperto}) constructed by simply intersecting $\Om$ with the hexagonal tiling $\H_{\de}$ (in Figure \ref{fig reference}) for $\de=\frac{1}{N}$.
\begin{proposition}\label{trivial estimate}
For every $\Om\subset \R^{2}$ open bounded set with Lipschitz boundary having $|\Om|=1$ there exists a natural number $M_0>0$ depending only on the shape of $\Om$, and a universal constant $C_0$ such that
\begin{equation}\label{Energia totale}
\frac{P(H)}{2}\sqrt{ N}\leq \rho(N,\Om) \leq \frac{P(H)}{2}\sqrt{N}+ C_0 P(\Om) \ \ \ \text{for} \ N\geq M_0.
\end{equation}
\end{proposition}
\begin{proof}
By applying Proposition \ref{lower bounded} with $\Om=A$ we immediately get that for every $\E$ minimizing $N$-cluster for $\Om$ it holds
\begin{eqnarray}
\rho(N,\Om)=P(\E)\geq P(\E;\Om) \geq \frac{P(H)}{2}\sqrt{ N}. \label{eqn: lower bound nella stima globale di energia rho}
\end{eqnarray}
Set $\de=\frac{1}{N}$, consider the planar tiling $\H_{\de}$ as in Figure \ref{fig reference} and the sets of indexes
\begin{eqnarray*}
& I(\Om) &=\{i\in \N \ | \ \H_{\de}(i)\cc \Om  \}, \\
& I( \pa \Om) &=\{i\in \N \ | \ \H_{\de}(i)\cap \pa \Om \neq \emptyset \}
\end{eqnarray*}
with $k=\#(I(\Om) ), h=\#(I(\pa \Om) )$. Clearly 
\begin{equation}
\de k=\sum_{i\in I(\Om)}|\E(i)|\leq |\Om|=1,
\end{equation}
thus $k\leq\frac{1}{\de}=N$. Moreover, if we introduce 
$$(\pa \Om)_{2\sqrt{\de}}=\pa\Om+B_{2\sqrt{\de}},$$
we notice that $\H_{\de}(i)\cc (\pa \Om)_{2\sqrt{\de}}$ for $i\in I(\pa\Om)$ and so
\begin{eqnarray}\label{donadoni}
\de h&=&\sum_{i\in I(\pa \Om)}|\H_{\de}(i)|\leq |(\pa \Om)_{2\sqrt{\de}}|.
\end{eqnarray}
Since $\Om$ has Lipschitz boundary there exists a $\de_0>0$ depending only on the shape of $\Om$ and a universal constant $C$ such that
\begin{equation}\label{soglia}
|(\pa \Om)_{2\sqrt{\de}}|\leq C\sqrt{\de}P(\Om),  \ \ \ \ \forall \ \de<\de_0.
\end{equation}
By asking $\de$ small enough and by plugging \eqref{soglia} into \eqref{donadoni} we reach
$$ h\leq \frac{CP(\Om)}{\sqrt{\de}} . $$
Summarizing, for a suitable $\de_0$ depending only on the shape of $\Om$ only the following bounds hold:
\begin{equation}\label{bounds on h e k}
k\leq \frac{1}{\de} \ \ \ \text{and} \ \ \ h \leq \frac{CP(\Om)}{\sqrt{\de}}, \ \ \ \text{for \ $\de\leq \de_0$},
\end{equation}
for a universal constant $C$. Define the cluster 
$$\F_{\Om}=\{ \H_{\de}(i) \ | \ i\in I(\Om) \}.$$
We want to re-organize the chambers $\{\H_{\de}(i)\cap \Om\}_{i \in I(\pa \Om)}$ into a new cluster $\F_{\pa \Om}$ in a way that every chamber $\F_{\pa \Om}(j)$ encloses exactly an area $\de$.
\begin{figure}
\centering
\includegraphics[scale=1.1]{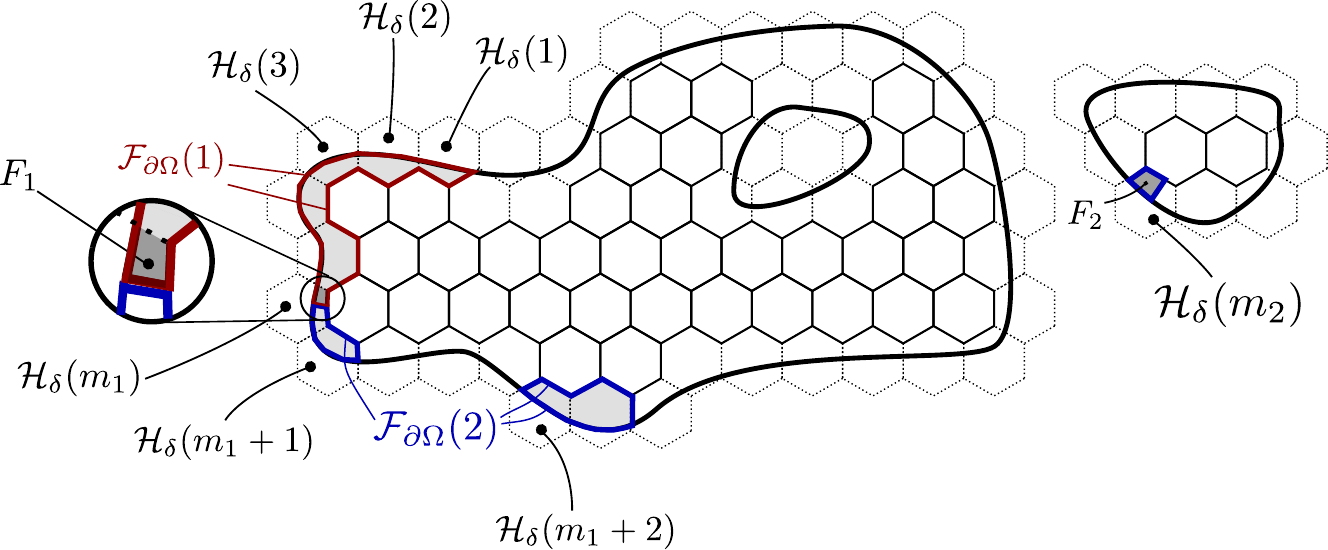}\caption{We consider all the hexagons from the tiling $\H_{\de}$ that are compactly contained into $\Om$. Then we re-organize the hexagons intersecting $\pa \Om$ in order to build chambers of area exactly $\frac{1}{N}$. To do that we suitably put together pieces of those chambers.}\label{figura partizione aperto}
\end{figure}
 To do that consider a relabeling of $\H_{\de}$ such that $I(\pa \Om)= \{1,2,\ldots, h\}$ and 
$$|\H_{\de}(1)\cap \Om|\leq |\H_{\de}(2)\cap \Om|\leq \ldots\leq |\H_{\de}(h)\cap \Om|.$$
Let $1\leq m_1\leq  h$ be the first index such that 
\begin{equation}\label{alba}
\sum_{j=1}^{m_1-1}|\H_{\de}(j)\cap \Om| \leq \de \leq \sum_{j=1}^{m_1}|\H_{\de}(j)\cap \Om|.
\end{equation}
If \eqref{alba} is in force we can select $F_1\subset \H_{\de}(m_1)\cap \Om$ in a way that  
\begin{equation}\label{kiara}
\sum_{j=1}^{m_1-1}|\H_{\de}(j)\cap \Om| +|F_1| =\de,
\end{equation}
for example by simply tracing a segment (see Figure \ref{figura partizione aperto}). We define the first chamber of the cluster $\F_{\pa \Om}$ as 
$$\F_{\pa \Om}(1)=F_1\cup \bigcup_{i=1}^{m_1-1 }\H_{\de}(j)\cap \Om.$$ 
We consider now $m_1\leq m_2\leq h$, the first index after $m_1$ such that
\begin{equation}\label{aria}
\sum_{j=m_1+1}^{m_2-1}|\H_{\de}(j)\cap \Om| \leq \de-\Big{|}(\H_{\de}(m_1)\cap \Om)\setminus F_1\Big{|} \leq\sum_{j=m_1+1}^{m_2}|\H_{\de}(j)\cap \Om|.
\end{equation}
Again, the validity of \eqref{aria} ensures that it is possible to cut away a subset $F_2\subset \H_{\de}(m_2)\cap \Om$ such that
$$\Big{|}(\H_{\de}(m_1)\cap \Om
)\setminus F_1\Big{|}+\sum_{j=m_1+1}^{m_2-1}|\H_{\de}(j)\cap \Om| +|F_2| =\de,$$
and define 
$$\F_{\pa \Om}(2)=F_2\cup [(\H_{\de}(m_1)\cap \Om
)\setminus F_1]\cup \bigcup_{i=m_1+1}^{m_2-1 }\H_{\de}(j)\cap \Om,$$
as clarified from Figure \ref{figura partizione aperto}. Since $\de=\frac{1}{N}$ 
\begin{eqnarray*}
h\de \geq  \sum_{j=1}^{h}|\H_{\de}(j)\cap \Om |=1-k\de = (N-k)\de.
\end{eqnarray*}
By iterating the previous argument we end up, in exactly $N-k\leq h$ steps, with an $(N-k)$-cluster $\F_{\pa \Om}=\{\F_{\pa \Om}(j)\}_{j=1}^{N-k}$ having chambers of area $\de$ and partitioning $\Om\setminus \F_{\Om}.$  
Moreover, by construction we have added almost $h$ hexagons and $h$ segment of length smaller than $2\sqrt{\de}$, so
\begin{equation}\label{sul perimetro di F}
P(\F_{\pa \Om}) \leq P(\Om) +  h P(H) \sqrt{\de}+h 2 \sqrt{\de}=P(\Om)+C h\sqrt{\de}
\end{equation}
for a universal constant $C$. If we now combine the cluster $\F_{\Om}=\{\H_{\de}(i)\}_{i\in I(\Om)}$ with $\F_{\pa \Om}$ we obtain a competitor for $\rho(N,\Om)$ and by construction and \eqref{bounds on h e k},\eqref{sul perimetro di F} we get
\begin{eqnarray}
\rho(N,\Om)&\leq & P(\F_{\pa \Om} \cup \F_{\Om})\nonumber\\
&\leq &k\frac{P(H)}{2}\sqrt{\de}+P(\F_{\pa \Om})\nonumber\\
&\leq &k\frac{P(H)}{2}\sqrt{\de}+Ch\sqrt{\de}+P(\Om)\nonumber\\
&\leq &\frac{P(H)}{2}\frac{1}{\sqrt{\de}}+CP(\Om) \ \ \ \text{for every $\de<\de_0$}.\label{lady D}
\end{eqnarray}
By recalling that $\de=\frac{1}{N}$ and combining \eqref{eqn: lower bound nella stima globale di energia rho} and \eqref{lady D} we get \eqref{Energia totale} with $C_0=C$ a universal constant and for every $N>\frac{1}{\de_0}=:M_0(\Om)$.
\end{proof}
\begin{remark}\label{remark}
\rm{ With Proposition \ref{trivial estimate} in mind we note that Theorems \ref{Equi indeco}, \ref{Equi dia} imply that at every scale $L>0$ and for $N$ big enough, the solution $\E$ attaining $\rho(N,\Om)$ has the main part of the localized energy close to the main part of $\rho(N,Q_L)$, namely dependent only on $L$ (as in the case of periodic pattern). Therefore:
$$\lim_{N\rightarrow +\infty}\frac{P(\E;Q_L)}{\sqrt{|Q_L| N}}\sqrt{\frac{|\Om|}{|Q_L|}}=\frac{P(H)}{2}=\lim_{N\rightarrow +\infty} \frac{\rho(N,Q_L)}{\sqrt{|Q_L|N}}.$$}
\end{remark}
By combining Proposition \ref{bootstrap} with Proposition \ref{trivial estimate} we obtain a first (rough) estimate on the energy in a cube $Q_l$ that we are going to refine in the sequel.
\begin{proposition}\label{upper indeco}
Let $\Om$ be an open bounded set with Lipschitz boundary and $|\Om|=1$. Let $0\leq \b<\frac{1}{2}$ be a real number. Then there exist three positive constants $\overline{\eta},\overline{\l},\overline{M}$ depending only on $\b$ and on the shape of $\Om$ and a universal constant $P_0\geq 1$ with the following property. For every $N>\overline{M}$, every closed cube $Q_l\cc \Om$ satisfying
\begin{equation}\label{vincoli upper indeco}
d(Q_l, \pa \Om)  > \overline{\eta} N^{-\frac{1}{6}}, \ \ \ \ \ \ \ \  l\geq \overline{\l} N^{-\b}
\end{equation}
and every indecomposable minimizing $N$-cluster $\E\in \Cl$ the following estimate holds:
\begin{equation}\label{tesi upper indeco}
P(\E;Q_l)\leq P_0 |Q_l| \sqrt{N}.
\end{equation}
\end{proposition}
\begin{remark}\label{remark su un sesto}
\rm{
We are now in the position for explain where the exponent $\frac{1}{6}$, appearing in the hypothesis on the distance between boundaries  \eqref{vincoli indeco main thm} in Theorem \ref{Equi indeco}, comes from. To prove that estimate \eqref{tesi upper indeco} is in force on every scale $l>>N^{-\frac{1}{2}}$ we repeatedly apply Proposition \ref{bootstrap}. We  argue, essentially, as follows. We start by applying Proposition \ref{bootstrap} with $\Om'=\Om$ and by exploiting the global energy estimate given by Proposition \ref{trivial estimate}. This leads to the validity of \eqref{tesi upper indeco} for every cube from a certain family $Q_L\in \mathcal{Q}_1$ of edges bigger than a certain $N^{-\a_1}$ and with $d(\pa Q_L,\pa \Om)>N^{-\frac{1}{6}}$. Now we select a suitable cube $Q_L\in \mathcal{Q}_1$ with $L=N^{-\a_1}$ on which \eqref{tesi upper indeco} is in force thanks to the previous step and we apply again Proposition \ref{bootstrap} by setting $\Om'=\mathring{Q_L}$. By exploiting $P\leq P_0L^2\sqrt{N}$ we are lead to the validity of \eqref{tesi upper indeco} for every cube from a family $Q_{l}\in \mathcal{Q}_2\supset \mathcal{Q}_1$ of size bigger than a certain $N^{-\a_2}$ but still lying at a distance $d(\pa Q_l,\pa \Om)>N^{-\frac{1}{6}}$. The iteration of this argument leads to the proof of \eqref{tesi upper indeco} for every closed cube $Q_l$ with $l>N^{-\frac{1}{2}}$. At each application we gain the validity of \eqref{tesi upper indeco} at a smaller scale but the restriction on $d(\pa Q_l, \pa \Om)$ cannot be improved by this argument, since we need always to have at least the space for run the first iteration (which is $N^{-\frac{1}{6}}$) .
}
\end{remark}
\begin{proof}[Proof of Proposition \ref{upper indeco}]
Set 
\begin{equation*}
\left\{
\begin{array}{ll}
\a_{k+1}&=\frac{2}{3}\a_k+\frac{1}{6} \\
& \\
\a_0&=0,
\end{array}
\right.
\end{equation*} 
and let us divide the proof in two steps. Note that
\begin{equation*}
\left\{
\begin{array}{ll}
\displaystyle \a_k < \frac{1}{2}  & \ \ \text{for all $k\geq 0$}, \\
\displaystyle \a_k < \a_{k+1} & \ \ \text{for all $k\geq 0$}, \\
\displaystyle \lim_{k\rightarrow +\infty} \a_k=\frac{1}{2}. &
\end{array}
\right.
\end{equation*} 
\text{}\\
\indent \textit{Step one.} We prove that for every $k\in \N $ there exist positive constants $\tau_{k-1}$, $\tau_{k}$, $\eta_k,M_k$ depending only on the shape of $\Om$ and a universal positive constant $P_0\geq 1$ such that 
$$\tau_0=\diam(\Om), \  \ \ \ \ \ \eta_k\geq \eta_{k-1} ,\  \ M_{k}\geq M_{k-1} \ \ \forall \ k$$
and with the following property. For every $N\geq M_k$, for every  closed cube $Q_l\cc \Om$ with 
\begin{equation}\label{hyp passo k}
d( \pa Q_l, \pa \Om)>\eta_k N^{-\frac{1}{6}} \ \ \ \ \ \ \tau_{k-1} N^{-\a_{k-1}}\geq l \geq \tau_{k} N^{-\a_{k}}
\end{equation}
and for every $\E\in \Cl$ indecomposable minimizing $N$-cluster for $\Om$ the following estimate holds
\begin{equation}
P(\E;Q_l)\leq P_0 |Q_l|\sqrt{N}.
\end{equation}
We argue by induction on $k$. We start by proving the validity of our assertion for $k=1$ by setting  $\tau_0=\diam(\Om)$. Thanks to Proposition \ref{trivial estimate} there exists an $M_0$ depending on the shape of $\Om$ only such that if $N>M_0$ and $\E\in \Cl$ is a minimizing $N$-cluster for $\Om$ the following estimate holds:
\begin{equation}\label{stima a priori induzione 0}
P(\E) \leq \frac{P(H)}{2}\sqrt{ N}+C_0P(\Om),
\end{equation}
for a universal constant $C_0$.
We now want to apply Proposition \ref{bootstrap} with $\Om'=\Om$. Let us choose $\eta_1,\tau_1$ such that  
$$\tau_1\geq 2\l_0 P(H)^{\frac{1}{3}}, \ \ \  \ \ \ \ \eta_1\geq \frac{\eta_0 }{\sqrt{\tau_1}}  \sqrt{P(H)}, $$
where $\l_0,\eta_0$ are the constants given by Proposition \ref{bootstrap}. We choose $M_1\geq M_0$ big enough so that
\begin{equation}\label{stima a priori induzione}
P(\E) \leq \frac{P(H)}{2}\sqrt{ N}+C_0P(\Om)\leq P(H)\sqrt{N},
\end{equation}
and
$$\tau_0 > \tau_1 N^{-\frac{1}{6}} $$
for all  $N\geq M_1$. If $N\geq M_1$ and $Q_l\cc \Om$ satisfies
\begin{equation}\label{livorno}
 d(\pa Q_l,\pa \Om)\geq \eta_1N^{-\frac{1}{6}}, \ \ \ \ \ \tau_0 \geq  l\geq \tau_1 N^{-\frac{1}{6}},
 \end{equation}
 with the previous choice of $\eta_1,\tau_1,M_1$,  then it satisfies also
\begin{equation}\label{assaggio12}
 d(\pa Q_l,\pa \Om)\geq \eta_0 \sqrt{\frac{P(\E)}{N l }}, \ \ \ \ l\geq \l_0\left( \left(\frac{P(\E)}{N}\right)^{\frac{1}{3}}+N^{-\frac{1}{2}}\right).
\end{equation}
Indeed if $N\geq M_1$ and $Q_l\cc \Om$ is a closed cube for which \eqref{livorno} holds we see that
\begin{eqnarray*}
\frac{N^{-\frac{1}{4}}}{\sqrt{l}} &\leq &  \frac{1}{\sqrt{\tau_1}} N^{-\frac{1}{6}},
\end{eqnarray*}
and since $M_1\geq M_0$, thanks to \eqref{stima a priori induzione} and to the choice of $\eta_1,\tau_1$ we obtain \eqref{assaggio12}:
\begin{align}
d( Q_l, \pa \Om) &>\eta_1 N^{-\frac{1}{6}} \nonumber\\
&>\frac{\eta_0 \sqrt{P(H)}}{\sqrt{\tau_1}} N^{-\frac{1}{6}}\nonumber\\
&>\eta_0 \sqrt{P(H)}\frac{N^{-\frac{1}{4}}}{\sqrt{l}}\nonumber \\
&\geq \eta_0 \sqrt{\frac{P(\E)}{N l }},\\
\text{}\nonumber\\
 l \ \ &\geq \tau_1 N^{-\frac{1}{6}}\nonumber\\
 &\geq \frac{\tau_1}{2} N^{-\frac{1}{6}}+\frac{\tau_1}{2} N^{-\frac{1}{6}}\nonumber\\
 &\geq  \l_0 P(H)^{\frac{1}{3}}N^{-\frac{1}{6}}+\l_0 N^{-\frac{1}{2}}\nonumber\\
 &\geq  \l_0 \left(\left(\frac{P(\E)}{N}\right)^{\frac{1}{3}}+N^{-\frac{1}{2}}\right).
 \end{align}
In particular we can apply Proposition \ref{bootstrap} by setting $\Om'=\Om$ and find an $N$-cluster $\F\in \Cl$ such that 
\begin{equation}\label{viareggio}
P(\F)\leq l^2 \frac{P(H)}{2}\sqrt{N}+P(\E;Q_l^c)+C_0\sqrt{P(\E)l}.
\end{equation}
Estimate on $l$ in \eqref{assaggio12} implies also that 
$$P(\E)\leq \frac{l^3}{\l_0^3}N,$$
which inserted in \eqref{viareggio} leads to
\begin{equation}\label{livorno3}
P(\F)\leq \left(\frac{P(H)}{2}+\frac{C_0}{\l_0^{3/2} }\right) |Q_l| \sqrt{N}+P(\E;Q_l^c).
\end{equation}
By minimality we have
$$P(\E)\leq P(\F),$$
which combined with \eqref {livorno3} leads to
\begin{equation}
P(\E;Q_l)\leq P_0 |Q_l| \sqrt{N},
\end{equation}
where we have set $P_0=\left(\frac{P(H)}{2}+\frac{C_0}{\l_0^{3/2} }\right)\geq 1 $. This proves the claim for $k=1$ with the previous choice of universal $\tau_0,\tau_1,\eta_1, M_1$ depending only on the shape of $\Om$.\\

Assume now that our assertion holds on $\a_k$ for some constants $\tau_{k-1},\tau_k,\eta_k,M_k$  ($k\geq 2$) and with the same constant $P_0$ defined above. We can conclude the validity of our assertion on $\a_{k+1}$ as a consequence of the following \textit{claim}. \\

\textbf{Claim.} \textit{There exist constants $\tau_{k+1},\eta_{k+1},M_{k+1}$ such that for every $N\geq M_{k+1}$ and every cube $Q_l\cc \Om$ satisfying
\begin{equation}\label{boot k+1}
d(\pa Q_l, \pa \Om)  > \eta_{k+1} N^{-\frac{1}{6}}\ \ \ \ \ \ \ \  \tau_{k} N^{-\a_{k}}\geq l \geq \tau_{k+1} N^{-\a_{k+1}}
\end{equation}
there exists a cube $Q_L$ concentric to $Q_l$ satisfying $Q_l\cc Q_L \cc \Om$ and
\begin{equation}\label{boot k}
d(\pa Q_L, \pa \Om)  > \eta_k N^{-\frac{1}{6}}, \ \ \ \ \ \ \ \  \tau_{k-1} N^{-\a_{k-1}}\geq L \geq \tau_k N^{-\a_k}
\end{equation}
for which the following holds on every indecomposable minimizing $N$-cluster $\E\in \Cl$:
\begin{equation}\label{eccheppalleeee}
d(\pa Q_l,\pa Q_L)\geq \eta_0 \sqrt{ \frac{P(\E;\mathring{Q_L})}{Nl} }, \ \ \ \ l\geq \l_0\left(\left(\frac{P(\E;\mathring{Q_L})}{N}\right)^{\frac{1}{3}}+N^{-\frac{1}{2}}\right).
\end{equation}
}
Indeed, if we assume for a moment the validity of the claim, on every cube $Q_l$ satisfying \eqref{boot k+1} we can apply Proposition \ref{bootstrap} with $\Om'=\mathring{Q_L}$ and find an $N$-cluster $\F\in \Cl$ such that
\begin{equation}\label{pappa e ciccia}
P(\F)\leq \frac{P(H)}{2}|Q_l| \sqrt{N}+P(\E;Q_l^c)+ C_0\sqrt{P(\E;\mathring{Q_L})l}.
\end{equation}
Since, thanks to \eqref{eccheppalleeee}, it holds
$$P(\E;\mathring{Q_L})\leq\frac{l^3}{\l_0^3}  N$$
by comparison \eqref{pappa e ciccia} is leading to \eqref{tesi upper indeco} with the same constant $P_0=\frac{P(H)}{2}+\frac{C_0}{\l_0^{3/2}}$. Hence we can achieve the proof of step one by induction.\\

Let us focus on the proof of the claim. Set $L=2\tau_{k}N^{-\a_{k}}$ and let $Q_L$ be the cube concentric to $Q_l$. Choose 
\begin{equation}\label{eqn: vincolo su eta k+1 dipendente da eta k e xi, e vincolo su xi dipendente da tau k nel passo induttivo}
\eta_{k+1}>\eta_k + 2\tau_k, \ \ \ \tau_{k+1}\geq \left\{ 6\l_0 P_0^{\frac{1}{3}} \tau_k^{\frac{2}{3}}\, ,8(\eta_0^2)P_0\, ,2\l_0 \right\}
\end{equation}
 and $M_{k+1}\geq M_k$ to be such that 
\begin{equation}\label{bello caldo}
\tau_{k-1} N^{-\a_{k-1}}\geq 2\tau_k N^{-\a_{k}}
\end{equation}
for every $N\geq M_{k+1}$. The cube $Q_L$ and the cube $Q_l$ share the center, hence we easily have 
\begin{equation*}
d(\pa Q_l, \pa Q_L)=\frac{L-l}{\sqrt{2}}
\end{equation*}\label{distanze fra cubi}
that, combined with the triangular inequality, implies
\begin{eqnarray}
d(\pa Q_L,\pa \Om)&>& d(\pa Q_l,\pa \Om)-(L-l).
\end{eqnarray}
Since $\a_k\geq \frac{1}{6}$ $(k\geq 1)$and $d(\pa Q_l,\pa \Om)>\eta_{k+1} N^{-\frac{1}{6}}$, relation \eqref{distanze fra cubi} leads to
\begin{align*}
d(\pa Q_L,\pa \Om)&> d(\pa Q_l,\pa \Om)-(L-l)\\
& \geq (\eta_{k+1}-2\tau_k ) N^{-\frac{1}{6}},
\end{align*}
which implies \eqref{boot k} because of \eqref{eqn: vincolo su eta k+1 dipendente da eta k e xi, e vincolo su xi dipendente da tau k nel passo induttivo}. 
Moreover by exploiting hypothesis \eqref{boot k+1} on $Q_l$ and the fact that on $Q_L$, by induction (thanks to \eqref{boot k}), it holds 
\begin{equation*}
P(\E;Q_L)\leq P_0|Q_L|\sqrt{N}
\end{equation*}
for every $\E\in \Cl$ indecomposable minimizing $N$-cluster for $\Om$ (with $N\geq M_{k+1}$) we obtain \eqref{eccheppalleeee}:
\begin{align*}
d(\pa Q_l,\pa Q_L)&=\frac{L-l}{\sqrt{2}}\nonumber\\
&\geq \frac{ \tau_k N^{-\a_k}}{\sqrt{2}}\sqrt{ \frac{l N }{P_0|Q_L| \sqrt{N}}}\sqrt{ \frac{P_0 |Q_L| \sqrt{N}}{N l}}\nonumber\\
&\geq \frac{  \sqrt{\tau_{k+1}}}{2\sqrt{2P_0}}N^{\frac{1}{4}-\frac{\a_{k+1}}{2}} \sqrt{ \frac{P(\E;\mathring{Q_L})}{N l}}\nonumber\\
&\geq \eta_0 \sqrt{ \frac{P(\E;\mathring{Q_L})}{N l}},\nonumber\\
\text{}\\
l \ \ &\geq \frac{\tau_{k+1}}{2}N^{-\a_{k+1}}+\frac{\tau_{k+1}}{2}N^{-\a_{k+1}}\\
&\geq  \frac{\tau_{k+1}}{6\tau_k^{\frac{2}{3}}P_0^{\frac{1}{3}}}N^{\frac{1}{6}+\frac{2}{3}\a_k-\a_{k+1}}\left(\frac{P_0|Q_L|\sqrt{N}}{N}\right)^{\frac{1}{3}}+\l_0 N^{-\frac{1}{2}}\\
&\geq  \l_0 \left( \left(\frac{P(\E;\mathring{Q_L})}{N}\right)^{\frac{1}{3}}+ N^{-\frac{1}{2}}\right)
\end{align*}
and we achieve the proof of the claim and thus of step one.\\

\textit{Step two.} Let $\b\in \left[0,\frac{1}{2}\right)$ be a positive number and let $k$ be such that $\a_{k-1}\leq \b\leq\a_{k}$. Let $\tau_{k-1},\tau_k,\eta_k, M_k$ be the constants given by the step one. We now set $\overline{\l}=\tau_k$, $\overline{\eta}=\eta_k$ and $\overline{M}=M_k$. Then we argue as follows. Let $Q_l\cc \Om$ be a closed cube with
$$d(\pa Q_l,\pa \Om)> \overline{\eta} N^{-\frac{1}{6}}, \ \ \ \ \ \ \ \ l> \overline{\l} N^{-\b}.$$
By construction, $l\geq\tau_k N^{-\a_{k}}$. If it holds also $l\leq \tau_{k-1} N^{-\a_{k-1}}$ and thus 
$$d(\pa Q_l,\pa \Om)> \overline{\eta} N^{-\frac{1}{6}}=\eta_k N^{-\frac{1}{6}} \ \ \ \  \ \tau_{k-1} N^{-\a_{k-1}}\geq  l\geq\tau_k N^{-\a_{k}},$$
thanks to step one we immediately have that  \eqref{tesi upper indeco} is in force on $Q_l$. If it does not hold  $l\leq \tau_{k-1} N^{-\a_{k-1}}$, then it must hold $l\geq \tau_{k-1} N^{-\a_{k-1}}$. As before, if it is true that $l\leq \tau_{k-2} N^{-\a_{k-2}}$, since we have chosen $\overline{\eta}=\eta_k\geq \eta_{k-1}$, we get 
\begin{equation}
d(\pa Q_l,\pa \Om) \geq \overline{\eta} N^{-\frac{1}{6}}\geq \eta_{k-1} N^{-\frac{1}{6}} \ \ \ \ \ \tau_{k-2} N^{-\a_{k-2}} \geq l\geq \tau_{k-1} N^{-\a_{k-1}}.
\end{equation}
Thus, thanks to step one, \eqref{tesi upper indeco} must be in force for every $N\geq M_{k-1}$ hence for every $N\geq \overline{M}=M_k\geq M_{k-1}$. If this is not the case, and thus $l\geq \tau_{k-2} N^{-\a_{k-2}}$ we iterate the previous argument and we move to the interval $l\in [\tau_{k-2}N^{-\a_{k-2}}, \tau_{k-3}N^{-\a_{k-3}}]$ by exploiting $\eta_{k-1}\geq \eta_{k-2}$ and $M_{k-1}\geq M_{k-2}$. This argument will end in exactly $k$ steps since, for sure $l\leq \diam(\Om)=\tau_0 N^{\a_0}.$ As a consequence \eqref{tesi upper indeco} must be in force for every $N\geq \overline{M}$, $l\geq \overline{\l} N^{-\b}$ and for every cube $Q_l$ lying at a distance bigger than $\overline{\eta} N^{-\frac{1}{6}}$ from $\pa \Om$ where $\overline{\eta},\overline{\l}, \overline{M}$ depend only on the shape of $\Omega$ and on the first $\a_k$ bigger than $\b$ (and thus on $\b$).
\end{proof}
\begin{remark}\rm{
The proof of Proposition \eqref{upper indeco} points out that we cannot reach the value $\b=\frac{1}{2}$ because the constants $\overline{\eta}(\b),\overline{\tau}(\b),\overline{M}(\b)$ are approaching $+\infty$ when $\b$ gets closer to $\frac{1}{2}$.
}
\end{remark}
Finally we put together Propositions \ref{bootstrap} and \ref{upper indeco} to get our competitor at every scale $l>>N^{-\frac{1}{2}}$.
\begin{proposition}\label{upper indeco quello che serve davvero}
Let $\Om$ be an open bounded set with Lipschitz boundary and $|\Om|=1$. Let $0\leq \b<\frac{1}{2}$ be a fixed number. Then there exists three positive constants $\eta,\l,M$ depending only on $\b$ and on the shape of $\Om$, and a universal constant $C$ with the following property. For every $N>M$, every closed cube $Q_l\cc \Om$ satisfying
\begin{equation}\label{vincoli upper indeco finali}
d(\pa Q_l, \pa \Om)  > \eta N^{-\frac{1}{6}}, \ \ \ \ \ \ \ \  l\geq \l N^{-\b}
\end{equation}
and every indecomposable minimizing $N$-cluster $\E\in \Cl$ there exists an $N$-cluster $\F\in \Cl$ such that:
\begin{equation}\label{tesi upper indeco finale}
P(\F)\leq \frac{P(H)}{2}|Q_l| \sqrt{N}+P(\E;Q_l^c)+C P(Q_l)^{\frac{3}{2} }N^\frac{1}{4}.
\end{equation}
\end{proposition}

\begin{proof}
The proof follows easily as a consequence of the following \textit{claim}. \\

\textbf{Claim.} \textit{Let $0\leq \b<\frac{1}{2}$ be a fixed number. We can find positive constants $\eta,\lambda,M$ depending only on $\b$ and on the shape of $\Om$ with the following property. For every $N\geq M$, if $Q_l\cc \Om$ is a closed cube satisfying \eqref{vincoli upper indeco finali} there exists a concentric cube $Q_{l+d}$ with $d<l$ such that $Q_l\cc Q_{l+d}\cc \Om$ and
\begin{equation}\label{e allora ancora}
P(\E;Q_{l+d})\leq P_0|Q_{l+d}|\sqrt{N}
\end{equation}
where $P_0$ is the universal constant appearing in \eqref{upper indeco}. Moreover for every $\E$ indecomposable minimizing $N$-cluster for $\Om$ it holds
\begin{equation}\label{eccheppalleeee2}
d(\pa Q_l,\pa Q_{l+d})\geq \eta_0 \sqrt{ \frac{P(\E;\mathring{Q_{l+d}})}{N l} }, \ \ \ \ l\geq \l_0\left(\left(\frac{P(\E;\mathring{Q_{l+d}})}{N}\right)^{\frac{1}{3}}+N^{-\frac{1}{2}}\right),
\end{equation}
where $\l_0,\eta_0$ are the constants appearing in Proposition \ref{bootstrap}.}\\ 

Indeed, assume for a moment the validity of the previous fact and choose a cube $Q_l\cc \Om$ satisfying \eqref{vincoli upper indeco finali} with the constants given by the claim. For any given indecomposable minimizing $N$-cluster $\E\in \Cl$, estimate \eqref{eccheppalleeee2} ensures us that we can apply Proposition \ref{bootstrap} with $\Om'=\mathring{Q_{l+d}}$ and thus find an $N$-cluster $\F\in \Cl$ such that
$$P(\F)\leq \frac{P(H)}{2} |Q_l|\sqrt{N}+P(\E;Q_l^c)+C_{0}\sqrt{l P(\E;Q_{l+d})},$$
for a universal constant $C_0$. By combining \eqref{e allora ancora} with the previous estimate and by recalling that $d<l$ we reach
$$P(\F)\leq \frac{P(H)}{2} |Q_l|\sqrt{N}+P(\E;Q_l^c)+2C_{0}\sqrt{P_0} l^{\frac{3}{2}}N^{\frac{1}{4}},$$
which is \eqref{tesi upper indeco finale} with $C=2C_{0}\sqrt{P_0}$.\\

Let us focus on the proof of the claim.  Let $\overline{\eta}, \overline{\l}, \overline{M}$ be the constants given by Proposition \ref{upper indeco}. We show that by choosing $\l,\eta$ such that
\begin{equation}\label{scelta di lambda, eta}
\eta \geq \max \{ 2\overline{\eta}\, , 4\eta_0\sqrt{2 P_0}\}, \ \ \ \ \l \geq \max\left\{\overline{\l}\, , 32 \l_0^3 P_0\, , \frac{\eta^2}{4}\right\}
\end{equation}
and $M$ such that
\begin{equation}\label{scelta di M}
N^{\frac{1}{12}}\geq 2\sqrt{\diam(\Om)}+\overline{M} , \ \ \ \text{for all $N\geq M$},
\end{equation}
the claim holds. Choose a cube $Q_l\cc \Om$ satisfying \eqref{vincoli upper indeco finali} with the previous choice of $\l,\eta,M$. Set 
$$d=\frac{\eta}{2} \sqrt{l}N^{-\frac{1}{4}},$$
let $Q_{l+d}$ be a cube concentric to $Q_l$ and note that $d<l$ since:
\begin{align*}
\frac{\eta}{2} \sqrt{l}N^{-\frac{1}{4}}\leq \frac{\eta}{2\sqrt{\l}} \sqrt{l}\sqrt{\l N^{-\b}}\leq \frac{\eta}{2\sqrt{\l}} l\leq l. 
\end{align*}
Thanks to the choice of $M$, for $N\geq M$  it holds:
\begin{align*}
d(\pa Q_{l+d},\pa \Om)&\geq d(\pa Q_l,\pa \Om)-2d\\
&\geq \eta N^{-\frac{1}{6}}-\eta \sqrt{l} N^{-\frac{1}{4}}\\
&\geq \eta N^{-\frac{1}{6}}(1- \sqrt{\diam(\Om)} N^{-\frac{1}{12}})\\
&\geq \frac{\eta}{2} N^{-\frac{1}{6}} \geq \overline{\eta} N^{-\frac{1}{6}}\\
\text{}\\
l+d &\geq l \geq \l N^{-\b}\geq \overline{\l} N^{-\b}.
\end{align*}
But then $Q_{l+d}$ satisfies hypothesis \eqref{vincoli upper indeco} and thus \eqref{e allora ancora} must be in force thanks to Proposition \ref{upper indeco}. Moreover for $N\geq M$ ($\geq \overline{M}$), thanks to Proposition \ref{upper indeco} and to the validity of \eqref{e allora ancora}, we have that on $Q_{l+d}$ it holds
$$P(\E;Q_{l+d}) \leq P_0 (l+d)^2 \sqrt{N}\leq 4P_0 l^2 \sqrt{N},$$
for every indecomposable minimizing $N$-cluster $\E\in \Cl$. Thus 
\begin{align}
d(\pa Q_l, \pa Q_{l+d})&=\frac{d}{\sqrt{2}} \nonumber\\
&=\frac{\eta}{2}\sqrt{l} N^{-\frac{1}{4}} \sqrt{\frac{Nl}{4 P_0 l^2 \sqrt{N}}} \sqrt{\frac{4 P_0 l^2 \sqrt{N}}{Nl}}\nonumber\\
&\geq \frac{\eta}{4\sqrt{2P_0}} \sqrt{\frac{P(\E;\mathring{Q_{l+d}})}{Nl}}\nonumber\\
&\geq \eta_0 \sqrt{\frac{P(\E;\mathring{Q_{l+d}})}{Nl}}. \label{assaggio13}\\
\text{}\nonumber
\end{align}
From $\b<\frac{1}{2}$ it follows
$$\l N^{-\frac{1}{2}}\leq l ,\ \ \ \ \Rightarrow \ \ \ \ \frac{1}{\sqrt{N}}\leq \frac{l}{\l}$$
that leads to
\begin{align*}
\frac{P(\E;Q_{l+d})}{N} \l_0^3 &\leq \frac{P_0}{\sqrt{N}} 4l^2  \l_0^3  \leq \frac{4 P_0  \l_0^3 }{\l } l^3  \leq \frac{l^3}{8}
\end{align*}
and so 
\begin{equation}\label{assaggio14}
l\geq 2\l_0  \left(\frac{ P(\E; \mathring{Q_{l+d}}) } { N }\right)^{\frac{1}{3}}.
\end{equation}
Clearly, 
\begin{equation}\label{assaggio15}
l\geq \l N^{-\b}\geq 2\l_0 N^{-\frac{1}{2}},
\end{equation} 
and thus by combining \eqref{assaggio13},\eqref{assaggio14} and \eqref{assaggio15} we obtain the validity of \eqref{eccheppalleeee2} and we achieve the proof.
\end{proof}

\subsection{Proof of Theorem \ref{Equi indeco}}

\begin{proof}[Proof of Theorem \ref{Equi indeco}]
Let $\Om$ be a generic open set with Lipschitz boundary and let $\b<\frac{1}{2}$ be fixed. We choose $\eta,\l,M,C$ to be the constants given by Proposition \ref{upper indeco quello che serve davvero}. Let $Q_l\cc \Om$ be a cube satisfying \eqref{vincoli indeco main thm}. Up to a roto-translation we can assume that the center of $Q_l$ is the origin. Then we perform the scaling $\Om_0=\frac{\Om}{\sqrt{|\Om|}}$, $Q=\frac{Q_l}{\sqrt{|\Om|}}$. This immediately implies that
$$d(\pa Q, \pa \Om)>\eta N^{-\frac{1}{6}}, \ \ \ \ \sqrt{|Q|}> \l N^{-\b},$$
since $Q_l$ was satisfying \eqref{vincoli indeco main thm}. Thus $Q$ is compactly contained into $\Om_0$ and thanks to the choice of the constants it satisfies \eqref{vincoli upper indeco finali}. Moreover, if $\E\in \Cl$ indecomposable minimizing $N$-cluster for $\Om$ then $\frac{\E}{\sqrt{|\Om|}}$ is an indecomposable minimizing $N$-cluster for $\Om_0$. Thus for every $N\geq M$ and for every $\E\in \Cl$ indecomposable minimizing $N$-cluster for $\Om$ we can apply Proposition \ref{upper indeco quello che serve davvero} and find an $N$-cluster $\F\in \text{Cl}\left(N,\Om_0\right)$ such that:
$$P(\F)\leq \frac{P(H)}{2}|Q| \sqrt{N}+C|Q|^{\frac{3}{4}}N^{\frac{1}{4}}+P\left(\frac{\E}{\sqrt{|\Om|}}; Q^c\right).$$
By comparison, this leads to
\begin{eqnarray}
P\left(\frac{\E}{\sqrt{|\Om|}};Q\right)&\leq & \frac{P(H)}{2}|Q| \sqrt{N}+C|Q|^{\frac{3}{4}}N^{\frac{1}{4}}\nonumber\\
\frac{1}{\sqrt{|\Om|}}P(\E;Q_l)&\leq & \frac{P(H)}{2}\frac{|Q_l|}{|\Om|} \sqrt{N}+C\frac{|Q_l|^{\frac{3}{4}}}{|\Om|^{\frac{3}{4}}}N^{\frac{1}{4}}\nonumber\\
P(\E;Q_l)&\leq& \frac{P(H)}{2}|Q_l| \sqrt{\frac{N}{|\Om|}}+CP(Q_l)^{\frac{3}{2}} \left(\frac{N}{|\Om|}\right)^{\frac{1}{4}}.\label{fine indeco su}
\end{eqnarray}
We have obtained \eqref{fine indeco su} for every $N>M$ where $M$ is depending only on the shape of $\Om$. Notice that if $N\leq M$ the restriction on $l$ implies that $l\geq \l \sqrt{|\Om|} M^{-\b}$. Moreover, we can easily build a competitor with $N$ chambers of the right measure (for example by dividing with $N$ parallel segments the set $\Om$ in chambers of the right amount of area) and obtain the estimate
$$P(\E;Q_l)\leq P(\E)\leq N\diam(\Om)+P(\Om)\leq N(\diam(\Om)+P(\Om)).$$
For all $N\leq M$ we have:
\begin{eqnarray*}
\frac{P(\E;Q_l)-\frac{P(H)}{2}|Q_l| \sqrt{\frac{N}{|\Om|}}}{P(Q_l)^{\frac{3}{2}}\left(\frac{N}{|\Om|}\right)^{\frac{1}{4}}}&\leq &
\frac{N(\diam(\Om)+P(\Om))|\Om|^{\frac{1}{4}}}{P(Q_l)^{\frac{3}{2}}}\\
&\leq & \frac{M(\diam(\Om)+P(\Om))|\Om|^{\frac{1}{4}}}{\l^{\frac{3}{2}} |\Om|^{\frac{3}{4}} M^{-\frac{3}{2}\b}}\\
&\leq & \frac{(\diam(\Om)+P(\Om))}{\sqrt{|\Om|}}\frac{ M^{1+\frac{3}{2}\b}}{\l^{\frac{3}{2}}}.
\end{eqnarray*}
Note that the quantity 
$$\frac{(\diam(\Om)+P(\Om))}{\sqrt{|\Om|}}$$
is invariant under scaling. So up to increasing the constant $C$ in dependence only on $\b$ and the shape of $\Om$ we can provide estimate \eqref{fine indeco su} for every $N\in \N$. \\

As in the proof of Theorem \ref{Equi indeco}, Lemma \ref{lower bound lemma} gives us
\begin{equation}\label{alloraaaa}
 P(\E;Q_l)\geq \frac{P(H)}{2}|Q_l| \sqrt{\frac{N}{|\Om|}}-P(Q_l).
 \end{equation}
Since $l> \frac{\sqrt{|\Om|}}{\sqrt{N}}$ we immediately have 
$$\left(\frac{N}{|\Om|}\right)^{\frac{1}{4}}|Q_l|^{\frac{3}{4}}\geq l^{-\frac{1}{2}} l^{\frac{3}{2}}=l $$
which, together with \eqref{alloraaaa} implies
\begin{equation}\label{fine indeco giu}
P(\E;Q_l)\geq \frac{P(H)}{2}|Q_l| \sqrt{\frac{N}{|\Om|}}-4\left(\frac{N}{|\Om|}\right)^{\frac{1}{4}}|Q_l|^{\frac{3}{4}}.
\end{equation}
By combining \eqref{fine indeco su} and \eqref{fine indeco giu} we achieve the proof.
\end{proof}

\chapter{A sharp quantitative version of Hales’ isoperimetric Honeycomb Theorem}\label{chapter sharpquantitative}

\section{Introduction} The isoperimetric nature of the planar ``honeycomb tiling''  has been apparent since antiquity. Referring to \cite[Section 15.1]{Morgan} for a brief historical account on this problem, we just recall here that Hales's isoperimetric Theorem, see inequality \eqref{hexagonal honeycomb thm torus} below or Theorem \ref{teo: chapter intro Hales sul toro}, gives a precise formulation of this intuitive idea. Our goal here is to strengthen Hales's theorem into a quantitative statement, similarly to what has been done with other isoperimetric theorems in recent years (see, for example, \cite{FuMP08,FMP10}).\\

Let $\hat H$ denote the reference unit-area hexagon in $\R^2$ depicted in Figure \ref{fig hexagon}, so that $\ell=(12)^{1/4}/3=P(H)/6$ is the side-length of $\hat H$. Given $\a,\b\in\N$, let us consider the torus $\T=\T(v_{\b},w_{\a})$ where
$$v_{\b}:=(\sqrt{3} \b\ell,0 ), \ \ \ w_{\a}:=\left( 0, \frac{3}{2} \a  \ell \right).$$
We recall that $v_{\b},w_{\a}$ defines an equivalence relation $\sim$ on $\R^2$ (see Subsection \ref{sbsct: the flat torus} where the flat torus is defined) thus we are allowed to define $H=\hat H/_{\sim} \subset\T$.

\begin{figure}
\begin{center}
 \includegraphics[scale=0.7]{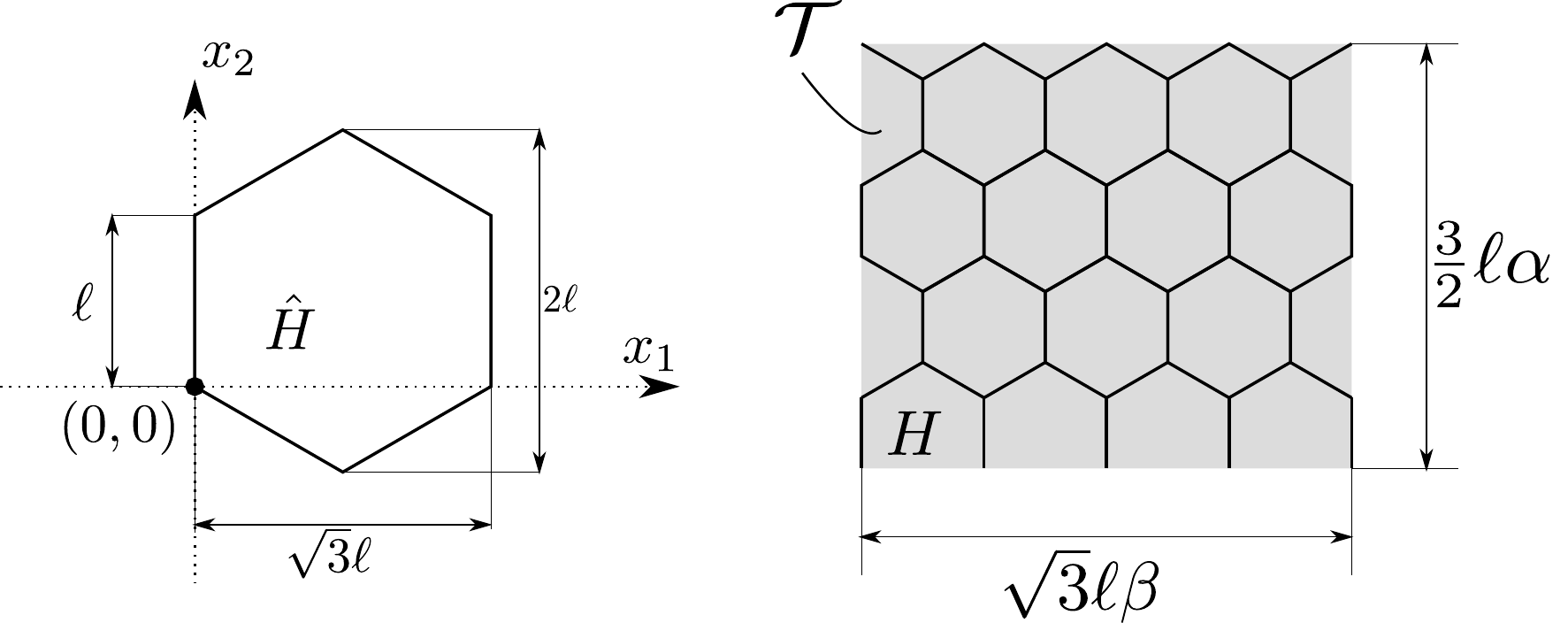}\caption{{\small Thoroughout the chapter $\hat H$ denotes the unit-area regular hexagon in $\R^2$ depicted on the left and we set $H=\hat H/_{\sim}$. Since $|H|=1$, one has $P(H)=2(12)^{1/4}$, and the side-length of $H$ is thus $\ell=(12)^{1/4}/3$. On the right, the torus $\T$ (depicted in gray) and the reference unit-area tiling $\H$ of $\T$ (with $\a=\b=4$). Notice that $N=|\T|=\a\,\b$. The chambers of $\H$ are enumerated so that $\H(1)=H$, $\{\H(h)\}_{h=1}^\b$ is the bottom row of hexagons in $\T$, and, more generally, if $0\le k\le\a-1$, then $\{\H(h)\}_{h=1+k\beta}^{(k+1)\beta}$ is the $(k+1)$th row of hexagons in $\T$.}}\label{fig hexagon}
\end{center}
\end{figure}

In order to avoid degenerate situations, {\it we shall always assume that}
\begin{equation}
  \label{occhio}
  \mbox{$\a$ is even and $\beta\ge 2$}\,.
\end{equation}
 In this way, $H$ is a regular unit-area hexagon (i.e., the vertices of $\hat{H}$ belong to six different equivalence classes) and one obtains a reference unit-area tiling $\H=\{\H(h)\}_{h=1}^N$ of $\T$ consisting of $\a$ rows and $\beta$ columns of regular hexagons by considering translations of $H$ by $(h\,\sqrt{3}\ell,3\ell\,k/2)$ ($h,k\in\Z$); see again Figure \ref{fig hexagon}. Under this assumption, {\it Hales's isoperimetric honeycomb theorem} asserts that
\begin{equation}
  \label{hexagonal honeycomb thm torus}
  P(\E)\ge P(\H)\,,
\end{equation}
whenever $\E$ is a unit-area tiling of $\T$, and that $P(\E)=P(\H)$ if and only if (up to a relabeling of the chambers of $\E$) one has $\E(h)=v+\H(h)$ for every $h=1,...,N$ and for some $v=(t\sqrt{3}\ell,s\ell)$ with $s,t\in[0,1]$. Our first main result strengthens this isoperimetric theorem in a sharp quantitative way.

\begin{theorem}
  \label{thm main periodic}
  There exists a positive constant $\k$ depending on $\T$ such that
  \begin{equation}
  \label{hexagonal honeycomb thm torus quantitative}
  P(\E)\ge P(\H)\,\Big\{1+\k\,\a(\E)^2\Big\}\,,
  \end{equation}
  whenever $\E$ is a unit-area tiling of $\T$ and
  \[
  \a(\E)=\inf \d(\hat{\E},v+\H)
  \]
  where the minimization takes place among all $v=(t\sqrt{3}\ell,s\ell)$, $s,t\in[0,1]$, and among all tilings $\hat\E$ obtained by setting $\hat\E(h)=\E(\s(h))$ for a permutation $\s$ of $\{1,...,N\}$. (Recall that the chambers of the reference honeycomb $\H$ are enumerated in a specific way, see Figure \ref{fig hexagon}.)
\end{theorem}

\begin{remark}
  {\rm We notice that \eqref{hexagonal honeycomb thm torus quantitative} is sharp in the decay rate of $\a(\E)$ in terms of $P(\E)-P(\H)$. Indeed, if $\om:(0,\infty)\to(0,\infty)$ is such that $P(\E)\ge P(\H)(1+\om(\a(\E)))$ for every unit-area tiling $\E$, then, for some $s_0>0$, one must have $\om(s)\le C\,s^2$ for $s\in(0,s_0)$. Indeed, one can explicitly construct a one-parameter family $\{\E_t\}_{0<t<\e}$ of unit-area tilings of $\T$ (by gently pushing three edges of the grid around a singular point by maintaining the area constraints) such that $P(\E_t)\le P(\H)(1+C\,\a(\E_t)^2)$ and $\{\a(\E_t):t\in(0,\e)\}=(0,s_0)$, so that $\om(s)\le C\,s^2$ for every $s\in(0,s_0)$.}
\end{remark}

In Theorem \ref{hexagonal honeycomb thm torus quantitative small} below, inequality \eqref{hexagonal honeycomb thm torus quantitative} is proven in much stronger form for $\pa\E$ in a special class of $C^1$-small $C^{1,1}$-diffeomorphic images of $\pa\H$, see \eqref{stabilita hd} and \eqref{stabilita f C1}.  The two main ingredients in the proof of Theorem \ref{hexagonal honeycomb thm torus quantitative small} are: a quantitative version of the hexagonal isoperimetric inequality, which we deduce from \cite{shilleto,indreinurbekyan}, see Lemma \ref{lemma indrei}; and a quantitative version of Hales's hexagonal isoperimetric inequality (the key tool behind Hales's proof of \eqref{hexagonal honeycomb thm torus}), proved in Lemma \ref{lemma hales}. These inequalities allow one to prove that each chamber of the unit-area tiling $\E$ is actually close, in terms of the size of $P(\E)-P(\H)$, to some regular unit-area hexagon in $\T$. These hexagons have no reason to fit nicely into an hexagonal honeycomb of $\T$ (that is, a translation of $\H$), therefore we need an additional argument to show that, up to translations and rotations of order $P(\E)-P(\H)$, one can achieve this. Having completed the proof of Theorem \ref{hexagonal honeycomb thm torus quantitative small}, we deduce Theorem \ref{thm main periodic} by a contradiction argument based on an improved convergence theorem for planar bubble clusters that was recently established in \cite{CiLeMaIC1}, and along the lines of the selection principle method proposed in \cite{CicaleseLeonardi}. Another consequence of Theorem \ref{hexagonal honeycomb thm torus quantitative small}, obtained in a similar vein, is the following result, which gives a precise description of isoperimetric tilings of $\T$ subject to an ``almost unit-area'' constraint.

\begin{theorem}\label{thm pertub volumes}
There exist positive constants $C_0,\de_0$ depending on $\T$ with the following property. If $\sum_{h=1}^Nm_h=N$ with $m_h>0$ and $|m_h-1|<\de_0$ for every $h=1,...,N$, and if $\E_m$ is an $N$-tiling of $\T$ which is a minimizer in
\begin{equation} \label{variational problem volumes}
\inf\big\{P(\E):|\E(h)|=m_h\quad\forall h=1,...,N\big\}
\end{equation}
then, up to a relabeling of the chambers of $\E_m$, there exists a $C^{1,1}$-diffeomorphism $f_m:\pa\H\to\pa\E_m$ such that
\begin{equation}\label{spirit}
\|f_m-(v+\Id)\|_{C^0(\pa\H)}^2+\|f_m-(v+\Id)\|_{C^1(\pa\H)}^4\le C_0\,\sum_{h=1}^N\,|m_h-1|\,,
\end{equation} 
for some $v=(t\sqrt{3}\ell,s\ell)$, $s,t\in[0,1]$.
\end{theorem}

Next, let us consider the family $X$ of those $\Phi\in C^0(\T\times S^{1};(0,\infty))$ such that the positive one-homogeneous extension of $\Phi(x,\cdot)$ to $\R^2$ is convex, fix $\psi\in C^0(\T;(0,\infty))$, and consider the isoperimetric problem
\begin{equation}
  \label{finsler}
  \l(\Phi,\psi)=\inf\Big\{\PHI(\E)=\frac12\sum_{h=1}^N\PHI(\E(h)):\int_{\E(h)}\psi=\frac1{N}\int_\T\psi\quad \forall h=1,...,N\Big\}\,,
\end{equation}
where for a set of finite perimeter $E\subset\T$ we have set
\[
\PHI(E;A)=\int_{A\cap\pa^*E}\Phi(x,\nu_E(x))\,d\H^1(x)\,,\qquad \PHI(E)=\PHI(E;\T)\,,
\]
provided $\pa^*E$ and $\nu_E:\pa^*E\to S^1$ denote, respectively, the reduced boundary and the measure-theoretic outer unit normal of $E$. Notice that although we do not assume $\Phi$ to be even, we have nevertheless that $\l(\Phi,\psi)=\l(\hat\Phi,\psi)$ where $\hat\Phi(x,\nu)=(\Phi(x,\nu)+\Phi(x,-\nu))/2$. An interesting example is obtained when $g$ is a Riemannian metric on $\T$ and
\[
\Phi(x,\nu)=\sqrt{g(x)[\nu^\perp,\nu^\perp]}\,,\qquad \psi=\sqrt{\det(g(x))}\,,
\]
where $\nu^\perp=(\nu_2,-\nu_1)$ if $\nu=(\nu_1,\nu_2)$. In this case, \eqref{finsler} boils down to minimizing the total Riemannian perimeter of a partition of $\T$ into $N$-regions of equal Riemannian area.


\begin{theorem}\label{thm pertub metric}
  Given $L>0$ and $\g\in(0,1]$, there exist $C_0,\de_0>0$ (depending on $\T$, $L$ and $\g$) with the following property. If $\E$ is a minimizer in \eqref{finsler} for $\Phi\in X\cap\Lip(\T\times S^1)$ and $\psi\in C^{1,\g}(\T)$ such that
  \begin{gather}\label{hp finsler 1}
  \Lip\,\Phi+\|\psi\|_{C^{1,\g}(\T)}\le L\,,
  \\\nonumber
  \|\Phi-1\|_{C^0(\T\times S^1)}+\|\psi-1\|_{C^0(\T)}<\de_0\,,
  \end{gather}
  then
  \begin{equation}
    \label{venerdi}  \inf_{s,t\in[0,1]}\hd(\pa\E,v+\pa\H)^4\le C_0\,\Big(\|\Phi-\Id\|_{C^0(\T\times S^1)}+\|\psi-1\|_{C^0(\T)}\Big)\,,
  \end{equation}
  where $v=(t\sqrt{3}\ell,s\ell)$ and $\hd(S,T)$ denote the Hausdorff distance between the closed sets $S$ and $T$ in $\T$.
\end{theorem}

We deduce Theorem \ref{thm pertub metric} from Theorem \ref{thm main periodic} by some comparison arguments and density estimates. Since we are assuming that $\nabla\Phi$ is merely bounded, we do not expect $\pa\E$ to be a $C^1$-diffeomorphic image of $\pa\H$. From this point of view,  \eqref{venerdi} seems to express a qualitatively sharp control on $\pa\E$. At the same time, when more regular integrands $\Phi$ are considered (see, e.g., \cite{DuzaarSteffen} for the kind of assumption one may impose here) one would expect to be able to obtain a control in the spirit of \eqref{spirit}. However a description of singularities of isoperimetric clusters in this kind of setting, although arguably achievable at least in some special cases, is missing at present. In turn, understanding singularities would be essential in order to adapt the improved convergence theorem from \cite{CiLeMaIC1} to this context (see Theorem \ref{improv Theorem} above), and thus to be able to strengthen \eqref{venerdi} into an estimate analogous to \eqref{spirit}.

The chapter is organized as follows. In Section \ref{section indrei} we deduce from \cite{shilleto,indreinurbekyan} a quantitative isoperimetric inequality for polygons of possible independent interest. In Section \ref{section small def} we prove Theorem \ref{hexagonal honeycomb thm torus quantitative small} on small $C^1$-deformations of $\pa\H$ (actually with the Hausdorff distance between $\pa\E$ and $\pa\H$ in place of $\d(\E,\H)$ on the right-hand side of \eqref{hexagonal honeycomb thm torus quantitative}). In Section \ref{section fine} we exploit the improved convergence theorem from \cite{CiLeMaIC1} to deduce Theorem \ref{thm main periodic} and Theorem \ref{thm pertub volumes} from Theorem \ref{hexagonal honeycomb thm torus quantitative small}, and, finally, to deduce Theorem \ref{thm pertub metric} from Theorem \ref{thm main periodic}.

\section{A quantitative isoperimetric inequality for polygons}\label{section indrei} Thorough this section we fix $n\ge 3$. We denote by $\Pi$ a convex unit-area $n$-gon, and by $\Pi_0$ a reference unit-area regular $n$-gon. If $\ell$ and $r$ denote, respectively, the side-length and radius of $\Pi_0$, then one easily finds that
\[
P(\Pi_0)=n\,\ell=2\sqrt{n\,\tan\Big(\frac{\pi}n\Big)}\,,\qquad r^{-1}=\sqrt{n\,\sin\Big(\frac{\pi}n\Big)\,\cos\Big(\frac{\pi}n\Big)}\,.
\]
(Notice that in the other sections of the chapter we always assume $n=6$, so that $\ell=(12)^{1/4}/3$ according to the convention set in the introduction.) The isoperimetric theorem for $n$-gons asserts that
\begin{equation}
  \label{isoperimetric inequality ngon}
  P(\Pi)\ge n\,\ell\,,
\end{equation}
with equality if and only if $\Pi=\rho(\Pi_0)$ for a rigid motion $\rho$ of $\R^2$. A sharp quantitative version of \eqref{isoperimetric inequality ngon} is proved in \cite{indreinurbekyan} starting from the main result in \cite{shilleto}. Precisely, let us now denote by $\ell_i$ and $r_i$ the lengths of the $i$th edge and the $i$th radius of $\Pi$ (labeled so that $\ell_i=\ell_j$ and $r_i=r_j$ if $i=j$ modulo $n$), and set
\[
\bar{\ell}=\frac1n\sum_{i=1}^n\ell_i\,,\qquad\bar{r}=\frac1n\sum_{i=1}^n r_i\,.
\]
Then \cite[Corollary 1.3]{indreinurbekyan} asserts that
\begin{equation}
  \label{isoperimetric inequality ngon indrei}
  C(n)\,\left(P(\Pi)^2-(n\ell)^2\right)\ge \sum_{i=1}^n (r_i-\bar{r})^2+\sum_{i=1}^n (\ell_i-\bar{\ell})^2\,.
\end{equation}
The right-hand side of inequality \eqref{isoperimetric inequality ngon indrei} measures the distance of $\Pi$ from being a unit-area regular $n$-gon in the sense that if $r_i=\bar r$ and $\ell_i=\bar\ell$, then it must be $\bar r=r$ and $\bar\ell=\ell$ by the area constraint, and thus $\Pi$ is a regular unit-area $n$-gon. However, in addressing our problem we shall need (in the case $n=6$) to control the distance of $\Pi$ from a specific regular unit-area $n$-gon by means of $P(\Pi)^2-(n\ell)^2$. Passing from \eqref{isoperimetric inequality ngon indrei} to this kind of control is the subject of the following proposition.

\begin{proposition}\label{lemma indrei}
There exists a positive constant $C(n)$ with the following property: for every convex unit-area $n$-gon $\Pi$ there exists a rigid motion $\rho$ of $\R^2$ such that
   \begin{equation}
       \label{quantitative indrei inq}
      C(n)\,\left(P(\Pi)^2-(n\ell)^2\right)\ge \hd(\pa \Pi,\pa \rho\Pi_0 )^2\,.
  \end{equation}
\end{proposition}

\begin{proof}
  Up to a translation, we can assume that $\Pi$ has barycenter at $0$. Next, if $P(\Pi)\ge n\ell+\eta\,P(\Pi)$ for some $\eta>0$, then $P(\Pi)^2-(n\ell)^2\ge \eta\,P(\Pi)^2$. Since $\hd(\pa\Pi,\pa\rho\Pi_0)<\diam(\Pi)+\diam(\Pi_0)\le (P(\Pi)+P(\Pi_0))/2\le P(\Pi)$ whenever $\pa\rho\Pi_0$ intersects $\pa\Pi$, we conclude that \eqref{quantitative indrei inq} holds with $C(n)=\eta^{-1}$. In other words, in proving \eqref{quantitative indrei inq}, one can assume without loss of generality that
  \begin{equation}\label{piccolezza}
  P(\Pi)-n\,\ell< \eta\,P(\Pi)\,
  \end{equation}
  for an arbitrarily small constant $\eta=\eta(n)$. By a trivial compactness argument (on the class of convex $n$-gons with barycenter at $0$), one sees that given $\e>0$ there exists $\eta>0$ such that if \eqref{piccolezza} holds, then, up to rigid motions,
  \begin{equation}
    \label{piccolezza hd}
    \hd(\pa \Pi, \pa \Pi_0)< \e\,,
  \end{equation}
  where the reference regular unit-area $n$-gon $\Pi_0$ is assumed to have barycenter at $0$.

  Now let $v_i$ and $w_i$ denote the positions of the vertices of $\Pi$ and $\Pi_0$ respectively: by \eqref{piccolezza hd} and up to a rotation, one can entail that
  \[
  |v_i-w_i|<\e\,,\qquad\forall i=1,...,n\,,\qquad v_1=\l\,w_1\qquad\mbox{for some $\l>0$}\,.
  \]
  Let $\rho_i$ denote the rotation around the origin such that $\rho_i(v_i)=\l_i\,w_i$ for some $\l_i>0$ (so that $\rho_1=\Id$ by $v_1=\l\,w_1$), and let $\theta_i$ denote the angle identifying $\rho_i$ as a counterclockwise rotation; since $\|\rho_i-\Id\|\le|\theta_i|$ and $|\rho_i(v_i)-w_i|=|r_i-r|$, one has
  \begin{equation}
    \label{bene -1}
    \hd(\pa\Pi,\pa\Pi_0)\le C\,\sum_{i=1}^n\,|v_i-w_i|\le C\,\sum_{i=1}^n\,r_i|\theta_i|+|r_i-r|\,.
  \end{equation}
  Let us now set $\de=P(\Pi)-n\ell$: by \eqref{isoperimetric inequality ngon indrei} and \eqref{piccolezza} one finds
  \begin{equation}
    \label{bene 0}
  \max_{1\le i\le n}|r_i-\bar{r}|+|\ell_i-\bar\ell|\le C\,\sqrt{\de}\,.
  \end{equation}
  Since $\bar\ell=n^{-1}P(\Pi)$ gives $|\bar\ell-\ell|=n^{-1}\de$, we deduce from $|\ell_i-\bar\ell|\le C\sqrt\de$ that
  \begin{equation}
    \label{bene 1}
      \max_{1\le i\le n}|\ell_i-\ell|\le C\sqrt{\de}\,.
  \end{equation}
  Let now $A(a,b,c)$ denote the area of a triangle with sides of length $a$, $b$ and $c$. Since $A$ is a Lipschitz function in an $\e$-neighborhood of $(r,r,\ell)$ (where both $(\bar r,\bar r,\ell)$ and $(r_i,r_{i+1},\ell_i)$ lie by \eqref{piccolezza hd}), by \eqref{bene 0}, \eqref{bene 1} and by $|\Pi_0|=|\Pi|$ we find
  \[
  \Big|n\,A(r,r,\ell)-n\, A(\bar{r},\bar{r},\ell)\Big|
  =\Big|\sum_{i=1}^n A(r_i,r_{i+1},\ell_i)-n\, A(\bar{r},\bar{r},\ell)\Big|\le C\,\sqrt{\de}.
  \]
  Since $A(a,a,\ell)=(\ell/4)\,\sqrt{4a^2-\ell^2}$ we immediately see that $|A(r,r,\ell)-A(a,a,\ell)|\ge c\,|a-r|$ whenever $|a-r|<\e$ and where $c=c(\ell)=c(n)>0$. Thus, $|r-\bar{r}|\le C\,\sqrt{\de}$, and  \eqref{bene 0} and \eqref{bene 1} give
  \begin{equation}
    \label{bene 2}
    \max_{1\le i\le n}|r_i-r|+|\ell_i-\ell|\le C\,\sqrt{\de}\,.
  \end{equation}
  If $\a_i$ denotes the interior angle between $v_i$ and $v_{i+1}$ (so that $|\a_i-2\pi/n|=O(\e)$ by \eqref{piccolezza hd}), then
  \[
  \a_i=f(r_i,r_{i+1},\ell_i)\,,\qquad\mbox{where}\quad f(a,b,c)=\arccos\Big(\frac{a^2+b^2-c^2}{2ab}\Big)\,.
  \]
  Since $f$ is a Lipschitz function in an $\e$-neighborhood of $(r,r,\ell)$, we conclude from \eqref{bene 2} that
\begin{equation*}
\max_{1\le i\le n}\Big{|}\a_i-\frac{2\pi}n\Big{|}=\max_{1\le i\le n}\Big{|}f(r_i,r_{i+1},\ell_i)-f(r,r,\ell)\Big{|}\le C\,\sqrt\de.
\end{equation*}
In particular, since $\theta_1=0$ (as $\rho_1=\Id$), we deduce from this last estimate that $|\theta_i|\le C\,\sqrt{\de}$ for $i=1,...,n$. We plug this inequality and \eqref{bene 2} in \eqref{bene -1} to conclude the proof.
\end{proof}

Coming to the torus $\T$, we shall use the following corollary of Proposition \ref{lemma indrei}.

\begin{corollary}\label{corollary indrei}
  There exist positive constants $\eta$ and $c$, independent from $\T$, with the following property. If $\Pi$ is a convex hexagon in $\T$ such that $\hd(\pa\Pi,\pa H)\le \eta$, then there exists a regular hexagon $H_*$ in $\T$ with $|\Pi|=|H_*|$ and for which the following holds:
  \begin{equation}
    \label{precisi0}
      P(\Pi)-P(H)\sqrt{|\Pi|}\ge c\,\hd(\pa\Pi,\pa H_*)^2\,.
  \end{equation}
\end{corollary}

\begin{proof}
  We first notice that by Proposition \ref{lemma indrei} and by scaling, if $\hat\Pi$ is a convex hexagon in $\R^2$, then there exists a regular hexagon $\hat H_*$ with $|\hat H_*|=|\hat\Pi|$ and
  \begin{equation}
    \label{precisi1}
      P(\hat\Pi)^2-P(\hat H)^2 |\hat\Pi|\ge c\,\hd(\pa\hat\Pi,\pa \hat H_*)^2\,.
  \end{equation}
  Since $\Pi$ is a convex hexagon in $\T$ with $\hd(\pa\Pi,\pa H)\le\eta$, then there exists a convex hexagon $\hat\Pi$ in $\R^2$ isometric to $\Pi$ with $\hd(\pa\hat\Pi,\pa\hat H)\le\eta$. In particular, for some constant $C$ independent from $\T$, one has
  \[
  P(\hat\Pi)-P(\hat H)\sqrt{|\hat\Pi|}\le C\,\eta\,,\qquad P(\hat\Pi)+P(\hat H)\sqrt{|\hat\Pi|}\le C\,,
  \]
  and thus \eqref{precisi1} gives, up to further decrease the value of $c$,
  \begin{equation}
    \label{precisi2}
      C\eta\ge P(\hat\Pi)-P(\hat H)\sqrt{|\hat\Pi|}\ge c\,\hd(\pa\hat\Pi,\pa \hat H_*)^2\,.
  \end{equation}
  By \eqref{precisi2} and $\hd(\pa\hat\Pi,\pa\hat H)\le\eta$ we have $\hd(\pa\hat H,\pa\hat H_*)\le C\sqrt\eta$. Now, since $\beta\ge 2$ and $\a$ is even one can find $\eta_*>0$ (independent of $\a$ and $\beta$) such that $I_{\eta_*}(\hat H)=\{x\in\R^2:\dist(x,\hat{H})\le\eta_*\}$ is compactly contained into a rectangular box of height $3\ell\a/2$ and width $\sqrt{3}\ell\beta$. As a consequence, if $\hat{J}$ is a polygon contained in $I_{\eta_*}(\hat H)$, then $J=\hat{J}/\sim \subset\T$ is isometric to $\hat J$. Thus, if $C\sqrt\eta<\eta_*$, then $H_*=\hat H_*/\sim $ is a regular hexagon in $\T$ with $|H_*|=|\Pi|$ and $\hd(\pa\hat\Pi,\pa\hat H_*)=\hd(\pa\Pi,\pa H_*)$, and \eqref{precisi0} follows from \eqref{precisi2}.
\end{proof}

\section{Small deformations of the reference honeycomb}\label{section small def} The main result of this section is Theorem \ref{hexagonal honeycomb thm torus quantitative small}, which provides us, on a restricted class of unit-area tilings, with a stronger stability estimate than the one in Theorem \ref{thm main periodic}. Before stating this result we briefly recall the following terminology.

\medskip

\noindent {\bf Regular and singular sets:} Given a $N$-tiling $\E$ of $\T$ one sets
\[
\pa\E=\bigcup_{h=1}^N\pa\E(h)\,,\qquad\pa^*\E=\bigcup_{h=1}^N\pa^*\E(h)\,,
\]
\[
\S(\E)=\pa\E\setminus\pa^*\E\,,
\qquad[\pa\E]_\mu=\{x\in\pa\E:\dist(x,\S(\E))>\mu\}\,,\quad\mu>0\,,
\]
where $\pa^*E$ denotes the reduced boundary of a set of finite perimeter $E$ in $\T$, and where the normalization convention $\pa E=\ov{\pa^*E}=\spt{\mu_E}$ for sets of finite perimeter is always assumed to be in force, see Subsection \ref{subsct: topological boundary} above. We call $\pa^*\E$ and $\S(\E)$ the {\it regular set} and the {\it singular set} of $\pa\E$ respectively. In this way, $\pa^*\H$ and $\S(\H)$ are, respectively, the union of the open edges and the union of the vertices of the hexagons $\H(h)$ for $h=1,...,N$.

\medskip

\noindent {\bf Tilings and maps of class $C^{k,\a}$:} Given $k\in\N$ and $\a\in[0,1]$, one says that a tiling $\E$ of $\T$ is of {\it class $C^{k,\a}$} if there exist a finite family $\{\g_i\}_{i\in I}$ of compact $C^{k,\a}$-curves with boundary and a finite family $\{p_j\}_{j\in J}$ of points such that
\begin{equation}
  \label{class C1}
  \pa\E=\bigcup_{i\in I}\g_i\,,\qquad\pa^*\E=\bigcup_{i\in I}\INT(\g_i)\,,\qquad\Sigma(\E)=\bigcup_{i\in I}\bd(\g_i)=\bigcup_{j\in J}\{p_j\}\,,
\end{equation}
where $\INT(\g_i)$ and $\bd(\g_i)$ denote the interior and the boundary of $\g_i$ respectively. Moreover, given a function $f:\pa\E\to\T$, one says that $f\in C^{k,\a}(\pa\E;\T)$ if $f$ is continuous on $\pa\E$ and
\[
\|f\|_{C^{k,\a}(\pa\E)}:=\sup_{i\in I}\|f\|_{C^{k,\a}(\g_i)}<\infty\,.
\]
Finally, given two $C^{k,\a}$-tilings $\E$ and $\F$ of $\T$, one says that $f$ is a $C^{k,\a}$-diffeomorphism between $\pa\E$ and $\pa\F$ if $f$ is an homeomorphism between $\pa\E$ and $\pa\F$ with $f(\S(\E))=\S(\F)$, $f(\pa\E(h))=\pa\F(h)$ for every $h=1,...,N$, $f\in C^{k,\a}(\pa\E;\T)$ and $f^{-1}\in C^{k,\a}(\pa\E;\T)$.

\medskip

\noindent {\bf Tangential component of a map and $(\e,\mu,L)$-perturbations of $\H$:} Given a tiling $\E$ of $\T$ of class $C^1$, by taking \eqref{class C1} into account one can define $\nu_\E\in C^{0}(\pa^*\E;S^1)$ in such a way that $\nu_\E$ is a unit normal vector to $\g_i$ for every $i$. Correspondingly, given a map $f:\pa\E\to\T$, we define $\ttau_\E f:\pa^*\E\to\T$, the tangential component of $f$ with respect to $\pa\E$, as
\[
\ttau_\E f(x)=f(x)-(f(x)\cdot\nu_\E(x))\,\nu_\E(x)\,,\qquad x\in\pa^*\E\,.
\]
Finally, one says that $\E$ is an {\it $(\e,\mu,L)$-perturbation of $\H$} if $\E$ is of class $C^{1,1}$ and there exists an homeomorphism $f$ between $\pa\H$ and $\pa\E$ with
\begin{eqnarray}\label{eL perturbation 2}
  \|f\|_{C^{1,1}(\pa\H)}&\le&L\,,
  \\\nonumber
  \|f-\Id\|_{C^1(\pa\H)}&\le&\e\,,
  \\\nonumber
  \|\ttau_{\H}(f-\Id)\|_{C^1(\pa^*\H)}&\le&\frac{L}\mu\,\sup_{\S(\H)}|f-\Id|\,,
    \\\nonumber
  \ttau_{\H}(f-\Id)=0\,,&&\qquad\mbox{on $[\pa\H]_\mu$}\,.
\end{eqnarray}

\medskip

Note that the definition of $\pa \E$, $\pared \E$ as well as the definition of singular set and tilings of class $C^{k,\a}$ are the suitable translations in the framework of \textit{tilings on the torus} of the definition given in Section \ref{cpt Ncluster of Rn} above for clusters and planar tilings.

\begin{theorem}
  \label{hexagonal honeycomb thm torus quantitative small}
  For every $L>0$ there exist positive constants $\mu_0$, $\e_0$ and $c_0$ (depending on $L$ and $|\T|$), $C$ depending on $|\T|$ only, and $C'$ depending on $L$ only, with the following property. If $\E$ is a unit-area $(\e_0,\mu_0,L)$-perturbation of $\H$, then there exists $v\in\R^2$ such that
  \begin{equation}
    \label{stabilita hd}
      P(\E)-P(\H)\ge c_0\,\hd(\pa\E,v+\pa\H)^2\,,\qquad |v|\le C\,\e_0\,.
  \end{equation}
  Moreover, there exists a $C^{1,1}$-diffeomorphism $f_0$ between $v+\pa\H$ and $\pa\E$ such that
  \begin{equation}
    \label{stabilita f C1}
      P(\E)-P(\H)\ge c_0\,\Big(\|f_0-\Id\|_{C^0(v+\pa\H)}^2+\|f_0-\Id\|_{C^1(v+\pa\H)}^4\Big)\,,
  \end{equation}
  and $\|f_0\|_{C^{1,1}(v+\pa\H)}\le C'$.
\end{theorem}

We premise two lemmas to the proof of Theorem \ref{hexagonal honeycomb thm torus quantitative small}. The first one, Lemma \ref{lemma hales} below, provide a quantitative version of (a particular case of) Hales's hexagonal isoperimetric inequality, the key step in the proof of \eqref{hexagonal honeycomb thm torus} in \cite{hales}. 

\begin{lemma}\label{lemma hales}
  There exist positive constants $\e_1$ and $c_1$ with the following property. If $\E$ is a unit-area tiling of $\T$ such that there exists an homeomorphism $f$ between $\pa\H$ and $\pa\E$ with $\|f-\Id\|_{C^0(\pa\H)}\le\e_1$, if $E=\E(h)$ for some $h\in\{1,...,N\}$ and $\Pi$ is the convex envelope of $\S(\E)\cap\pa E$ (so that $\Pi$ is convex hexagon with set of vertices $\S(\E)\cap\pa E$ provided $\e_1$ is small enough), then there exists a regular hexagon $H_*$ with $|H_*|=|\Pi|$ such that
  \begin{equation}
    \label{utile}
    P(E)\ge P(H)+\frac{P(H)}2\,(|\Pi|-|E|)+c_1\,\Big(|E\Delta\Pi|^2+\hd(\pa\Pi,\pa H_*)^2\Big)\,.
  \end{equation}
\end{lemma}

\begin{remark}
  {\rm The constants $\e_1$ and $c_1$ will just depend on the metric properties of the unit-area hexagon. In particular they do not depend on $\T$.}
\end{remark}

\begin{proof}
  [Proof of Lemma \ref{lemma hales}]
  Let $\arc_t(a)$ denote the length of a circular arc that bounds an area $a\ge0$ and whose chord length is $t>0$, and let us set $\arc(a)=\arc_1(a)$. In this way, $\arc:[0,\infty)\to[1,\infty)$ is an increasing function. Since the derivative of $\arc$ at $a$ is the curvature of any circular arc bounding an area $a$ above a unit length chord, and since this curvature is increasing as $a$ ranges from $0$ to $\pi/8$ (the value $a=\pi/8$ corresponds to the case of an half-disk with unit diameter), we conclude that $\arc$ is convex on $[0,\pi/8]$ (and, in fact, also concave on $[\pi/8,\infty)$). Moreover, a Taylor expansion gives that $\arc''(0^+)>0$: hence there exists $\eta>0$ such that
  \begin{equation}
    \label{arc coercivo}
    \arc(a)\ge 1+\eta\,a^2\,,\qquad\forall a\in[0,\eta)\,.
  \end{equation}
  Let $\ell_i$ denote the length of the $i$th side of $\Pi$, and let $a_i$ denote the total area enclosed between the $i$th side of $\Pi$ and the $i$th side of $E$;
  \begin{figure}
  \begin{center}
\includegraphics[scale=0.7]{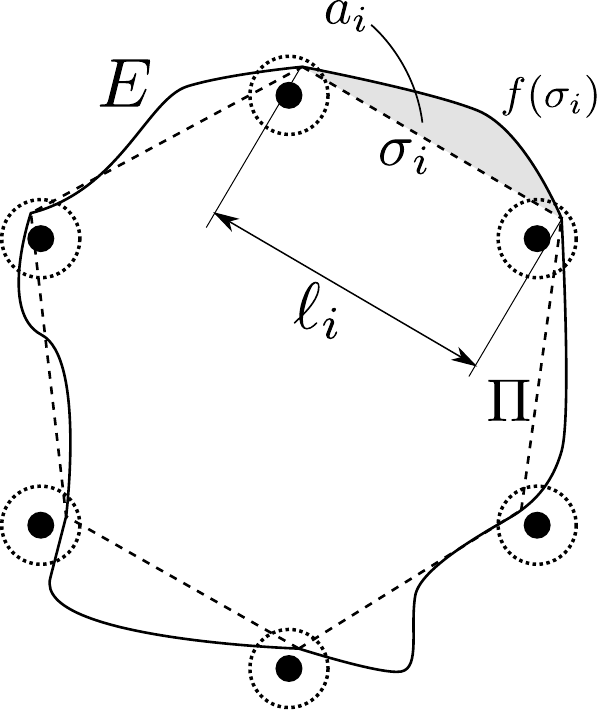}\caption{{\small The convex hexagon $\Pi$ spanned by $\S(\E)\cap\pa E$. The vertices of $\Pi$ are $\e_1$-close to the vertices of the unit-area regular hexagon $\H(h)$ (as $E=\E(h)$ and $f(\pa\H(h))=\pa\E(h)$) which are depicted as black dots. The boundaries of $\Pi$ and $E$ are depicted, respectively, by a dashed line and by a continuous line.}}\label{fig pi}
  \end{center}
  \end{figure}
  see Figure \ref{fig pi}. (If $\s_i$ is the $i$th side of $\Pi$, then the $i$th side of $E$ is a small $C^0$-deformation of $\s_i$ with fixed end-points). Noticing that $\arc_t(a)=t\,\arc(a/t^2)$, by Dido's inequality we find that
  \[
  P(E)\ge\sum_{i=1}^6\arc_{\ell_i}(a_i)=\sum_{i=1}^6\,\ell_i\,\arc\Big(\frac{a_i}{\ell_i^2}\Big)\,.
  \]
  By $\|f-\Id\|_{C^0(\pa\H)}\le \e_1$ and provided $\e_1\le 1$, one has
    \begin{equation}
    \label{omega}
    \hd(\pa\Pi,\pa\H(h))\le \e_1\,,\qquad \max_{1\le i\le 6}\Big\{a_i,\Big|\ell_i-\frac{P(H)}6\Big|\Big\}\le C\,\e_1\,,
  \end{equation}
  where a possible value for $C$ in \eqref{omega} is $2(\pi+\ell)$. By \eqref{omega}, by further decreasing $\e_1$, we can assume that $a_i/\ell_i^2\in[0,\pi/8]$ for every $i=1,...,6$. We thus apply Jensen inequality to find that
  \[
  P(E)\ge \sum_{i=1}^6\,\ell_i\,\arc\bigg(\frac1{\sum_{i=1}^6\ell_i}\sum_{i=1}^6\frac{a_i}{\ell_i^2}\bigg)\,.
  \]
  Since $P(H)/6=(12)^{1/4}/3<1$, by \eqref{omega} we may further assume that $\ell_i\le 1$ for every $i=1,...,6$, and thus conclude by $P(\Pi)=\sum_{i=1}^6\ell_i$, $|E\Delta\Pi|=\sum_{i=1}^6a_i$, and the monotonicity of $\arc$ that
  \begin{equation}
    \label{chordal isoperimetric inequality}
    P(E)\ge P(\Pi)\,\arc\Big(\frac{|E\Delta \Pi|}{P(\Pi)}\Big)\,.
  \end{equation}
  (Inequality \eqref{chordal isoperimetric inequality} is clearly related to the {\it chordal isoperimetric inequality} \cite[Proposition 6.1-A]{hales}, see also \cite[15.5]{Morgan}.) By \eqref{arc coercivo}, \eqref{omega} and \eqref{chordal isoperimetric inequality},
  \begin{equation}
    \label{chordal isoperimetric inequality x}
    P(E)\ge P(\Pi)+\eta\,\frac{|E\Delta\Pi|^2}{P(\Pi)^2}\ge P(\Pi)+c_1\,|E\Delta\Pi|^2\,,
  \end{equation}
  where $c_1>0$. Provided $\e_1$ is small enough, by \eqref{omega} we can apply Corollary \ref{corollary indrei} to find a regular hexagon $H_*$ with $|H_*|=|\Pi|$ and
  \[
    P(\Pi)- P(H)\sqrt{|\Pi|}\ge c\,\hd(\pa\Pi,\pa H_*)^2\,.
  \]
  Thus, up to further decrease the value of $c_1$,  \eqref{chordal isoperimetric inequality x} gives
  \begin{equation}
    \label{proof1}
    P(E)\ge P(H)\sqrt{|\Pi|}+c_1\,\Big(\hd(\pa\Pi,\pa H_*)^2+|E\Delta\Pi|^2\Big)\,.
  \end{equation}
  Finally, given $\tau>0$ let $\l>0$ be such that $\sqrt{1-s}\ge1-(s/2)-\tau\,s^2$ for $|s|<\l$: up to further decrease $\e_1$, by $\|f-\Id\|_{C^0(\pa\H)}\le\e_1$ we entail $|\s|<\l$ for $\s=|E|-|\Pi|$, and thus deduce with the aid of \eqref{proof1}  and $|E|=1$ that
  \begin{equation}
    \label{proof3}
    P(E)\ge P(H)-\frac{P(H)}2\,\s-P(H)\tau\,\s^2
    +c_1\,\Big(\hd(\pa\Pi,\pa H_*)^2+|E\Delta\Pi|^2\Big)\,.
  \end{equation}
  Since $|\s|=||E|-|\Pi||\le|E\Delta\Pi|$, for $\tau$ small enough depending from $c_1$, we prove \eqref{utile}.
\end{proof}

Given $\E$ an $(\e_0,\mu_0,L)-$perturbation of $\H$, thanks to Lemma \ref{lemma hales}, we are able to perform the construction of a suitable translation $\H_0=\H+v$ of $\H$ having the singular set $\S(\H_0)$ close to $\S(\E)$ in terms of the perimeter deficit. We thus try to show that the distance between $\pa \E$ and $\pa \H_0$ is estimated by the perimeter deficit but, in this situation, we cannot use just the information provided by the diffeomorphism $f:\pa \H\rightarrow \pa \E$. To achieve the proof of Theorem \ref{hexagonal honeycomb thm torus quantitative small} we need to have some more information about the relation between the new tilings $\H_0$ and the tiling $\E$. In particular we are going to build a new diffeomorphism $f_0:\pa \H_0\rightarrow \pa \E$ having the tangential component small in terms of the distance between $\S(\E)$ and $\S(\H_0)$. In order to do that we make use of the following lemma combined with the existence of the diffeomorphism $f$ between $\pa \E$ and $\pa \H$ given by our definition of $(\e_0,\mu_0,L)$-perturbation. In particular Lemma \ref{lemma di riparametrizzazione sui segmenti} establishes the existence of a diffeomorphism between a given segment $\s_0$ and a curve $\g$ close enough to $\s_0$. We set 
$$[\s_0]_t=\{x\in\s_0:\dist(x,\bd(\s_0))>t\},$$ 
for $t>0$ and we state the Lemma as follows.

\begin{lemma}\label{lemma di riparametrizzazione sui segmenti}
For every $M,\l>0,\a\in(0,1]$ there exist two constants $C_1$ and $\bar{\mu}$ depending on $M$,$\l$ and $\a$ only with the following property.  Let $\s_0$ be a segment of length $\l$ and $\g$ be a $C^{1,\a}$ curve in the plane having $\bd(\g)\neq\emptyset$ and 
\begin{equation}\label{controllo sulle normali}
\left\{
\begin{array}{ll}
|\nu(x)-\nu(y)|&\leq M |x-y|^{\a},\\
|\nu(x)\cdot(y-x)|&\leq M |y-x|^{1+\a},\\
\end{array}
\right.
\ \ \ \ \ \ \text{for all} \ \ \ x,y\in \g
\end{equation}
where $\nu(x)$ denotes the normal unit-vector to $\g$ at $x$. Assume also that for some $\rho<\bar{\mu}^2$, the curve $\g$ satisfies the following hypothesis
\begin{itemize}
\item[(a)]  $\hd(\s_0,\g)\leq \rho$;
\item[(b)] setting $\bd(\s_0)=\{p_0,q_0\}$ and $\bd(\g)=\{p,q\}$ it holds
$$|\tau(p)-\tau_0|+|\tau(q)-\tau_0|\leq \rho,$$
where $\tau_0=\frac{p_0-q_0}{|p_0-q_0|}$ is the tangent unit-vector to $\s_0$ and $\tau(x)=-\nu(x)^{\perp}$ denotes the tangent unit-vector to $\g$ at $x$;
 \item[(c)] there exists a map $\psi_0\in C^{1,1}([\s_0]_{\rho})$ such that
  \begin{equation*}
  [\g]_{3\rho}\subset(\Id+\psi_0\nu_0)\left([\s_0]_{\rho}\right)\subset\g\,,
  \end{equation*}
  \begin{equation*}
  \|\psi_0\|_{C^{1,1}([\s_0]_{\rho})}\le M\,,\qquad \|\psi_0\|_{C^1([\s_0]_{\rho})}\le \rho\,;
  \end{equation*}
\end{itemize}
Then, for every $\mu\in (\sqrt{\rho},\bar{\mu})$, there exists a $C^{1,\a}$-diffeomorphism $f_0$ between $\s_0$ and $\g$ such that $f_0(\bd(\s_0))=\bd(\g)$ and
\[
  \|f_0\|_{C^{1,\a}(\s_0)}\le C_1\,,\qquad\|f_0-\Id\|_{C^1(\s_0)}\le \frac{C_1}{\mu}\,\rho^{\a}\,,
\]
\[
\|(f_0-\Id)\cdot\tau_0\|_{C^1(\s_0)}\le \frac{C_1}{\mu}\,\sup_{\bd(\s_0)}|f_0-\Id|\,.
\]
\end{lemma}
\begin{proof}
The proof follows by applying \cite[Theorem 3.1]{CiLeMaIC1} in the case $n=2$, $k=1$ and by setting $S_0=\s_0$, $S=\g$.
\end{proof}
\begin{proof}[Proof of Theorem \ref{hexagonal honeycomb thm torus quantitative small}] {\it Step one}: The reflection of $\R^2$ with respect to a generic line does not induce a map on $\T$. However, by \eqref{occhio}, one has that if $R_\theta\hat H$ denotes the counterclockwise rotation of $\hat H$ by an angle $\theta$ around the origin, then $R_\theta\hat H$ is compactly contained in a box of height $3\ell\a/2\ge3\ell$ and width $\sqrt{3}\ell\beta\ge2\sqrt{3}\ell$ for every $\theta$. As a consequence, given a unit-area regular hexagon $K$ in $\T$, all the rotations of $K$ are well-defined as unit-area regular hexagons in $\T$; in particular, it always makes sense to define the reflection $g_\s(K)$ of $K$ with respect to an edge $\s$ of $K$. Taking this into account, we notice that there exist positive constants $\eta$ and $C$ (independent of $\T$) such that, if $K$ and $K'$ are unit-area regular hexagons in $\T$, and if $\s$ and $\s'$ are edges of $K$ and $K'$ respectively, then
\[
\left\{
\begin{array}
  {l l}
  \hd(\s,\s')\le\eta\,,
  \\
  |K\Delta K'|\ge 2-\eta\,,
\end{array}
\right .
\qquad\Rightarrow\qquad
\hd(\pa g_\s(K),\pa K')\le C\,\hd(\s,\s')\,.
\]
This geometric remark is going to be repeatedly used in the following arguments, where we shall denote by $\e_1$ and $c_1$ the constants of  Lemma \ref{lemma hales} and set $\de=P(\E)-P(\H)$. We notice that, by the area formula and since $\|f-\Id\|_{C^1(\pa\H)}\le\e_0$, one has
\begin{equation}
  \label{delta eps0}
  \de\le C\,P(\H)\,\e_0^2\,,
\end{equation}
where $C$ is independent from $\T$ and where $P(\H)=|\T|\,P(H)/2$.

\medskip

\noindent  {\it Step two}: We claim that, if $\e_0$ is small enough depending only from $|\T|$, and if $\Pi_h$ denotes the convex envelope of $\pa\E(h)\cap\S(\E)$ (so that $\Pi_h$ is a convex hexagon, not necessarily with unit-area), then for every $h=1,...,N$ there exists a regular unit-area hexagon $K_h$ such that
\begin{eqnarray}
  \label{trombone}
  \hd(\pa\Pi_h,\pa K_h)\le C\,\sqrt\de\,,&&
  \\
  \label{trombone 2}
  |K_h\Delta K_{h+1}|\ge 2-C\,\sqrt\de\,,&&
\end{eqnarray}
where here and in the rest of this step, $C$ denotes a constant depending from $|\T|$ only. Indeed, since $\{\Pi_h\}_{h=1}^N$ is a partition of $\T$, one has $\sum_{h=1}^N|\Pi_h|=|\T|=\sum_{h=1}^N|\E(h)|$. By requiring $\e_0\le\e_1$ we can apply Lemma \ref{lemma hales} to each $\E(h)$ in order to find regular hexagons $H^*_h$ with $|H^*_h|=|\Pi_h|$ such that, by adding up \eqref{utile} on $h$, one finds
\begin{eqnarray}\label{pino}
  2\,\de=\sum_{h=1}^N (P(\E(h))-P(H))\ge c_1\,\sum_{h=1}^N\Big(|\E(h)\Delta\Pi_h|^2+\hd(\pa\Pi_h,\pa H^*_h)^2\Big)\,.
\end{eqnarray}
By \eqref{pino},
\begin{equation}
  \label{esagoniiii}
  ||\Pi_h|-1|\le|\E(h)\Delta \Pi_h|\le \sqrt{\frac{2\de}{c_1}}\,.
\end{equation}
By \eqref{occhio}, we may further decrease the value of $\eta$ introduced in step one so to have that if $J$ is a regular hexagon in $\T$ with $||J|-1|\le\eta$, then it makes sense to scale $J$ with respect to its barycenter in order to obtain a unit-area regular hexagon $J'$ with $\hd(\pa J,\pa J')\le C\,||J|-1|$. In particular, by \eqref{delta eps0} and \eqref{esagoniiii}, up to decrease the value of $\e_0$ we can define unit-area hexagons $K_h$ in $\T$ with the property that
\[
\hd(\pa K_h,\pa H_h^*)\le C\,||H_h^*|-1|=C\,||\Pi_h|-1|\le C\,\sqrt\de\,.
\]
By combining this estimate with \eqref{pino} we prove \eqref{trombone}. By \eqref{trombone}, $|K_j\Delta\Pi_j|\le C\sqrt\de$ for every $j$, and thus
\[
|K_h\Delta K_{h+1}|\ge|\E(h)\Delta\E(h+1)|-\sum_{j=h}^{h+1}|\E(j)\Delta K_j|\ge 2-C\,\sqrt\de-\sum_{j=h}^{h+1}|\E(j)\Delta \Pi_j|\,.
\]
In particular, \eqref{trombone 2} follows from \eqref{pino}.

\medskip

\noindent {\it Step three}: We claim the existence of a tiling $\H_0=v+\H$ of $\T$ such that
\begin{equation}
  \label{bella li}
  \hd(\S(\E),\S(\H_0))\le C\,\sqrt\de\,,\qquad |v|\le C\,\e_0\,,
\end{equation}
where here and in the rest of this step, $C$ denotes a constant depending from $|\T|$ only. Let us recall from Figure \ref{fig hexagon} that the chambers of $\H$ are ordered so that $\{\H(h)\}_{h=1}^\b$ is the ``bottom row'' in the grid defined by $\H$ and that $\H(1)=H$. Since $\E$ is an $(\e_0,\mu_0,L)$-perturbation of $\H$ one has
\begin{equation}
  \label{vicine}
  \max\Big\{\hd(\pa\E(h),\pa\H(h)),\hd(\pa\Pi_h,\pa\H(h))\Big\}\le \e_0\,,\qquad\forall h=1,...,N\,,
\end{equation}
so that \eqref{trombone} implies $\hd(\pa H,\pa K_1)\le C\,\e_0$. In particular, there exists $|\theta|,|s|,|t|\le C\e_0$ such that
\[
K_1=(t\sqrt{3}\ell,s\ell)+R_\theta H\,,
\]
where, with a slight abuse of notation, $R_\theta H$ denotes the counterclockwise rotation of $H$ by an angle $\theta$ around its left-bottom vertex (see step one). Of course, there is no reason to get a better estimate than $|s|,|t|\le C\,\e_0$ here (indeed, $\E$ itself could just be an $\e_0$-size translation of $\H$). Nevertheless, if $\theta\ne 0$, then we cannot fit $K_1$ into an hexagonal honeycomb of $\T$: therefore one expects
\begin{equation}
  \label{expectations}
  |\theta|\le C\sqrt\de\,.
\end{equation}
We prove \eqref{expectations}: set $J_1=K_1$, let $\tau_1$ be the common edge between $\Pi_1$ and $\Pi_2$, and let $\s_1$ and $\s_1'$ be the edges of $K_1$ and $K_2$ respectively such that, thanks to \eqref{trombone}, $\hd(\tau_1,\s_1)+\hd(\tau_1,\s'_1)\le C\,\sqrt\de$. In this way $\hd(\s_1,\s_1')\le C\,\sqrt\de$, and by \eqref{trombone 2} we can apply step one to deduce
\begin{equation}
  \label{hey}
  \hd(\pa J_2,\pa K_2)\le C\,\hd(\s_1,\s_1')\le C\sqrt\de\,,\qquad |J_2\Delta K_2|\le C\,\sqrt\de\,,
\end{equation}
where $J_2$ is the reflection of $J_1$ with respect to $\s_1$. Let now $\tau_2$ be common side between $\Pi_2$ and $\Pi_3$. By \eqref{trombone} and \eqref{hey} we have $\hd(\pa J_2,\pa\Pi_2)+\hd(\pa K_3,\pa\Pi_3)\le C\sqrt\de$, thus there exist edges $\s_2$ and $\s_2'$ of $J_2$ and $K_3$ respectively such that $\hd(\tau_2,\s)+\hd(\tau_2,\s')\le C\,\sqrt\de$. By \eqref{trombone 2} and \eqref{hey} one has $|J_2\Delta K_3|\ge 2-C\sqrt\de$, so that by step one $\hd(\pa J_3,\pa K_3)\le C\sqrt\de$ where $J_3$ is the reflection of $J_2$ with respect to $\s_2$. If we repeat this argument $\beta$-times, then we find regular unit-area hexagons $J_1,...,J_\beta$ such that $J_1=K_1$, $J_h$ is obtained by reflecting $J_{h-1}$ with respect to its ``vertical'' right edge, and $\hd(\pa J_h,\pa K_h)\le C\,\sqrt\de$ for $h=1,...,\beta$. By construction, $\Pi_\beta$ and $\Pi_1$ also share a common edge $\tau$, and correspondingly $J_\beta$ and $K_1$ have edges $\s$ and $\s'$ respectively with $\hd(\tau,\s)+\hd(\tau,\s')\le C\,\sqrt\de$. By reflecting $J_\beta$ with respect to $\s$ we thus find a regular unit area hexagon $J_*$ with
\[
\hd(\pa J_*,\pa K_1)\le C\,\sqrt\de\,.
\]
At the same time, since $J_*$ has been obtained by iteratively reflecting $J_1=K_1$ with respect to its ``vertical'' right edge, we find that
\[
\hd(\pa J_*,\pa J_1)\ge \frac{|\theta|}C\,.
\]
Thus \eqref{expectations} holds. As a consequence, up to apply to $K_1$ a rotation of size $C\,\sqrt\de$, one can assume that
\begin{equation}
  \label{richiuditi}
  K_1=(t\sqrt{3}\ell,s\ell)+H\,,\qquad\mbox{for some $|t|,|s|\le C\,\e_0$}\,.
\end{equation}
In particular, if we set $\H_0(h)=(t\sqrt{3}\ell,s\ell)+\H(h)$, then $\H_0$ defines a unit-area tiling of $\T$ by regular hexagons. By arguing as in the proof of \eqref{expectations}, one easily sees that
\begin{equation}
  \label{rigidita}
  \hd(\pa\Pi_h,\pa \H_0(h))\le C\,\sqrt{\de}\,,\qquad\forall h=1,...,N\,.
\end{equation}
In particular, the set of vertices of $\Pi_h$ and $\H_0(h)$ lie at distance $C\,\sqrt\de$. Since $\S(\E)$ is the set of all the vertices of the $\Pi_h$s, we complete the proof of \eqref{bella li}.

%
%
%

\medskip

\noindent {\it Step four}: We show that if $\mu_0$ is small enough with respect to $L$, and $\e_0$ is small enough with respect to $\mu_0$ and $|\T|$, then there exists a $C^{1,1}$-diffeomorphism $f_0$ between $\pa\H_0$ and $\pa\E$ such that
\begin{equation}
  \label{friday00}
  \|f_0\|_{C^{1,1}(\pa\H_0)}\le C\,,\qquad\|f_0-\Id\|_{C^1(\pa\H_0)}\le C\,\mu_0\,,
\end{equation}
\begin{equation}
  \label{friday01}
  \|(f_0-\Id)\cdot\tau_0\|_{C^1(\pa\H_0)}\le C\,\sup_{\S(\H_0)}|f_0-\Id|\,.
\end{equation}
where $C$ depends on $L$ only.  The map $f_0$ is built starting from Lemma \ref{lemma di riparametrizzazione sui segmenti} and is more useful than the map $f$ appearing in \eqref{eL perturbation 2} because the best estimate for $f-\Id$ on $\S(\H)$ is of order $\e_0$, while, thanks to \eqref{bella li}, we have a much more precise information about $f_0-\Id$ on $\S(\H_0)$, namely
\begin{equation}
  \label{deficit 0}
    \sup_{\S(\H_0)}|f_0-\Id|\le C\,\sqrt\de\,.
\end{equation}
(In \eqref{deficit 0}, $C$ depends on $|\T|$.) Let us also notice that we cannot just define $f_0$ by composing $f$ with the translation bringing $\pa\H_0$ onto $\pa\H$, because this translation is $O(\e_0)$, and thus the resulting map $f_0$ would still have tangential displacement $O(\e_0)$. We thus need a more precise construction, directly relating $\pa\H_0$ and $\pa\E$.

In order to apply Lemma \ref{lemma di riparametrizzazione sui segmenti}, we fix an edge $\s$ of $\H$, and set $\s_0=v+\s$, so that $\s_0$ is an edge of $\H_0$. We denote by $\tau_0$ and $\nu_0=\tau_0^\perp$ the constant tangent and normal unit-vector fields to $\s_0$ (and, obviously, to $\s$). We let $\gamma=f(\s)$ and set $\tau(x)=\nabla^\s f(f^{-1}(x))[\tau_0]$ and $\nu(x)=\tau(x)^\perp$, where $\nabla^\s f$ denotes the tangential gradient of $f$ with respect to $\s$. We start by noticing that $\bd(\g)\neq \emptyset$ and we argue as follows.  \\

 By applying Lemma \ref{lemma di riparametrizzazione sui segmenti} with $\l=\ell$ and $\a=1$ we discover that for any given $M>0$ for which $\g$ satisfies \eqref{controllo sulle normali} there exist positive constants $C_1$ and $\bar{\mu}$ (depending on $M$) such that if $\g$ satisfies also hypothesis (a),(b) and (c) for some $\rho<\bar{\mu}^2$ then, for every $\mu\in (\sqrt{\rho},\bar{\mu})$, there exists a $C^{1,1}$-diffeomorphism $f_0$ between $\s_0$ and $\g$ such that $f_0(\bd(\s_0))=\bd(\g)$ and
\[
  \|f_0\|_{C^{1,1}(\s_0)}\le C_1\,,\qquad\|f_0-\Id\|_{C^1(\s_0)}\le \frac{C_1}{\mu}\,\rho\,,
\]
\[
\|(f_0-\Id)\cdot\tau_0\|_{C^1(\s_0)}\le \frac{C_1}{\mu}\,\sup_{\bd(\s_0)}|f_0-\Id|\,.
\]
Thus it is enough to show for $\mu_0,\e_0$ small enough depending on $L$ then properties (a),(b), (c) and condition \eqref{controllo sulle normali} hold on $\g$ with a suitable choice of $M=M(L)$ and $\rho$. Clearly $\g$ satisfies \eqref{controllo sulle normali} for some $M=M(L)$, since $\|f\|_{C^{1,1}(\s)}\le L$ and $\|f-\Id\|_{C^1(\s)}\le\e_0$. We notice that property (a) holds provided $\rho\ge C\e_0$ for some $C$ depending on $|\T|$ only: indeed, by $\|f-\Id\|_{C^0(\s)}\le\e_0$ one finds $\hd(\s,\g)\le\e_0$, while $|v|\le C\,\e_0$ (recall \eqref{bella li}) gives $\hd(\s,\s_0)\le C\,\e_0$. Similarly, property (b) holds if $\rho\ge\e_0$, as $\tau(x)=\nabla^\s f(f^{-1}(x))[\tau_0]$ and $\|f-\Id\|_{C^1(\s)}\le \e_0$.  Finally, concerning property (c), we notice that by exploiting the fact that $\E$ is an $(\e_0,\mu_0,L)$-perturbation of $\H$ and setting $\psi=(f-\Id)\cdot\nu_0$, one has $\psi\in C^{1,1}([\s]_{\mu_0})$ with
\begin{equation}\label{friday1}
[\g]_{\mu_0+2\e_0}\subset(\Id+\psi\nu_0)\left([\s]_{\mu_0}\right)\subset\g\,,
\end{equation}
\begin{equation}\label{friday2}
\|\psi\|_{C^{1,1}([\s]_{\mu_0})}\le L\,,\qquad \|\psi\|_{C^1([\s]_{\mu_0})}\le \e_0\,,
\end{equation}
where the first inclusion in \eqref{friday1} follows from $\|f-\Id\|_{C^0(\s_0)}\le\e_0$ and $\g=f(\s)$.
By exploiting \eqref{friday1}, \eqref{friday2}, and the fact that $\s_0=v+\s$ with $|v|\le C\,\e_0$ by \eqref{bella li}, one can find two constants $C_2\le C_3$ (both depending just on $|\T|$) and $\psi_0\in C^{1,1}([\s]_{\mu_0+C_2\,\e_0})$ such that properties (a), (b) and (c) hold with $\rho=\mu_0+C_2\,\e_0$, and
  \begin{equation}\label{friday3}
  [\g]_{\mu_0+C_3\,\e_0}\subset(\Id+\psi_0\nu_0)\left([\s_0]_{\mu_0+C_2\,\e_0}\right)\subset\g\,,
  \end{equation}
  \begin{equation}\label{friday4}
  \|\psi_0\|_{C^{1,1}([\s_0]_{\mu_0+C_2\,\e_0})}\le L\,,\qquad \|\psi_0\|_{C^1([\s_0]_{\mu_0+C_2\,\e_0})}\le \e_0\,,
  \end{equation}
see
\begin{figure}
\begin{center}
  \includegraphics[scale=0.7]{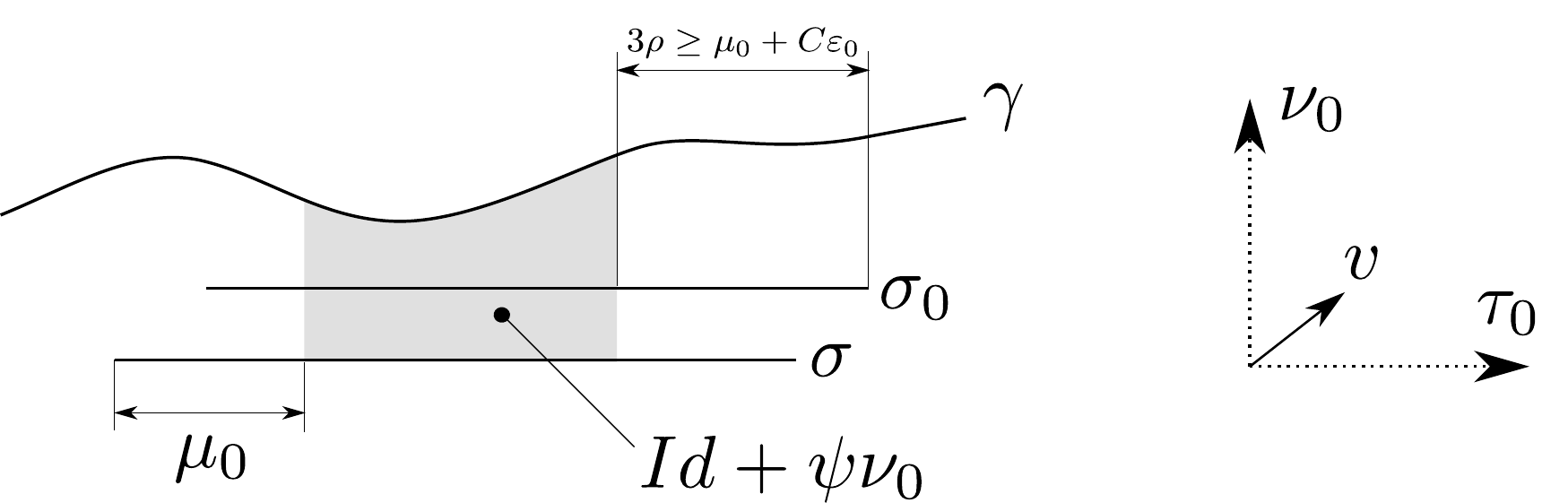}\caption{{\small The function $\psi_0$ is defined by computing the values of $\psi$ after a projection of $\s_0$ onto $\s$.}}\label{fig psi0}
  \end{center}
\end{figure}
Figure \ref{fig psi0}. Of course one can entail $3\rho> \mu_0+C_3\,\e_0$ by requiring $\e_0$ small enough with respect to $\mu_0$: in this way, property (c) follows from \eqref{friday3} and \eqref{friday4}. Summarizing, we have shown that if $\mu_0$ is small enough depending on $L$ (that is, depending on $M=M(L)$), and if $\e_0$ is small enough with respect to $\mu_0$ and $|\T|$, then properties (a),(b), (c) hold with $\rho=\mu_0+C_2\,\e_0$ and condition \eqref{controllo sulle normali} holds with a suitable $M=M(L)$. Up to further decrease the values of $\mu_0$ and $\e_0$ we may entail $\rho<\bar{\mu}^2$ and fix $\mu\in(\sqrt{\rho},\bar{\mu})$ depending on $L$ only. Thus, thanks to Lemma \ref{lemma di riparametrizzazione sui segmenti}, we find a $C^{1,1}$-diffeomorphism $f_0$ between $\s_0$ and $\g$ such that $f_0(\bd(\s_0))=\bd(\g)$ and
\[
\|f_0\|_{C^{1,1}(\s_0)}\le C\,,\qquad\|f_0-\Id\|_{C^1(\s_0)}\le C\,\mu_0\,,
\]
\[
\|(f_0-\Id)\cdot\tau_0\|_{C^{1,1}(\s_0)}\le C\,\sup_{\bd(\s_0)}|f_0-\Id|\,,
\]
where $C$ depends on $L$ only. By repeating this construction on every edge $\s_0$ of $\pa\H_0$ we complete the proof of \eqref{friday00} and \eqref{friday01}.

\medskip

\noindent {\it Step four}: With a little abuse of notation, let us denote by $\{\s_i\}_{i=1}^{3N}$ the family of segments such that $\pa\H_0=\bigcup_{i=1}^{3N}\s_i$. For every $i$ let $\tau_i$ denote a constant tangent unit vector to $\s_i$. If we set $g=f_0-\Id$, then we have
\begin{eqnarray*}
  P(\E)-P(\H)=\sum_{i=1}^{3N}\int_{\s_i}\big(|\nabla^{\s_i}g[\tau_i]+\tau_i|-1\big)\,d\H^1\,,
\end{eqnarray*}
where, by $\|g\|_{C^1(\pa\H_0)}\le\mu_0$, $\sqrt{1+t}\ge 1+t/2-t^2/8-C\,|t|^3$ ($t\ge -1$), and provided $\mu_0$ is small enough,
\begin{eqnarray*}
  |\nabla^{\s_i}g[\tau_i]+\tau_i|-1
  &=&\sqrt{1+2 \tau_i\cdot\nabla^{\s_i}g[\tau_i]+|\nabla^{\s_i}g[\tau_i]|^2}-1
  \\
  &\ge& \tau_i\cdot\nabla^{\s_i}g[\tau_i]+\frac{|\nabla^{\s_i}g[\tau_i]|^2}2-\frac{|2\,\tau_i\cdot\nabla^{\s_i}g[\tau_i]|^2}8
  -C\,\mu_0\,|\nabla^{\s_i}g[\tau_i]|^2\,.
\end{eqnarray*}
Let $\S(\H_0)=\{p_j\}_{j=1}^{2N}$, and for $p_j\in\bd(\s_i)$ denote by $v_j^i$ the tangent unit vector to $\s_i$ at $p_j$ pointing outside $\s_i$. In this way,
\[
\sum_{i=1}^{3N}\int_{\s_i}\,\tau_i\cdot\nabla^{\s_i}g[\tau_i]\,d\H^1=\sum_{j=1}^{2N}\sum_{\{i:p_j\in\bd(\s_i)\}}\,v_j^i\,g(p_j)=0\,,
\]
since $\{i:p_j\in\bd(\s_i)\}=\{i_1,i_2,i_3\}$ with $v_j^{i_2}$ and $v_j^{i_3}$ obtained from $v_j^{i_1}$ by counterclockwise rotations of $2\pi/3$ and $4\pi/3$ respectively. Hence, if we set $\nu_i=\tau_i^\perp$, then
  \begin{equation}\label{deficit 1}
    P(\E)-P(\H)\ge \sum_{i=1}^{3N}\int_{\s_i}\frac{|\nu_i\cdot\nabla^{\s_i}g[\tau_i]|^2}2\,d\H^1-C\,\mu_0\,\int_{\s_i}|\nabla^{\s_i}g[\tau_i]|^2\,d\H^1\,.
  \end{equation}
  By \eqref{friday01} and \eqref{deficit 0} we find that
  \[
  \sup_{1\le i\le 3N}\|\tau_i\cdot\nabla^{\s_i}g[\tau_i]\|_{C^0(\s_i)}\le C\,\sqrt\de\,,
  \]
  where $C$ depends on $L$ and $|\T|$. By combining this last inequality with \eqref{deficit 1}, and provided $\mu_0$ is small enough with respect to $L$ and $|\T|$, we find
  \begin{equation}
    \label{combina}
      C\,\sqrt{\de}\ge \sum_{i=1}^{3N}\int_{\s_i}|\nabla^{\s_i}g[\tau_i]|\ge\sum_{i=1}^{3N}\|g-g(p_{j(i)})\|_{C^0(\s_i)}\,,
  \end{equation}
  where for each $i=1,...,3N$ we have picked $p_{j(i)}\in \bd(\s_i)$. By \eqref{deficit 0} we have $|g(p_{j(i)})|\le C\sqrt\de$, so that \eqref{combina} implies
  \begin{equation}
    \label{combina C0}
      C\,\sqrt{\de}\ge\sum_{i=1}^{3N}\|g\|_{C^0(\s_i)}=\|f_0-\Id\|_{C^0(\pa\H_0)}\,.
  \end{equation}
  Since $f_0$ is a bijection between $\pa\H_0$ and $\pa\E$, we find that $\|f_0-\Id\|_{C^0(\pa\H_0)}\ge\hd(\pa\H_0,\pa\E)$ and thus prove \eqref{stabilita hd}. We now notice that if $u:(a,b)\to\R$ is a Lipschitz function with $\|u\|_{C^0(a,b)}\leq 1$, then
  \begin{equation}
    \label{interpol}
    \|u\|_{C^0(a,b)}^2\le 8\,\max\Big\{\Lip(u),\frac1{b-a}\Big\}\,\|u\|_{L^1(a,b)}\,.
  \end{equation}
  Indeed, let $x_0\in[a,b]$ be such that $u(x_0)=\|u\|_{C^0(a,b)}$ and set $L=\Lip(u)$, $r=|u(x_0)|/4L$. If $(x_0,x_0+r)\subset(a,b)$ or $(x_0-r,x_0)\subset(a,b)$, then by integrating $|u(y)|\ge |u(x_0)|-L|x_0-y|$ in $y$ over $(x_0,x_0+r)$ or over $(x_0-r,x_0)$ respectively, we find
  \[
  \int_{(a,b)}|u|\ge r\,|u(x_0)|-L\frac{r^2}2\ge \frac{|u(x_0)|^2}{8L}\,;
  \]
  otherwise one has $b-a\le 2r$ and thus $|u(y)|\ge |u(x_0)|/2$ for every $y\in(a,b)$. In order to complete the proof of \eqref{stabilita f C1} we just need to use \eqref{combina C0} and to combine the first inequality in \eqref{combina} with $\|f_0\|_{C^{1,1}(\pa\H)}\le C$ and with \eqref{interpol} (applied  to the components of $\nabla^{\pa^*\H_0}(f_0-\Id)$).
\end{proof}

\section{Proof of Theorem \ref{thm main periodic}, Theorem \ref{thm pertub volumes} and Theorem \ref{thm pertub metric}}\label{section fine} We start by introducing the following fundamental tool in the study of isoperimetric problems with multiple volume constraints. This kind of construction is originally found in \cite{Almgren76}, and it is fully detailed in our setting in \cite[Sections 29.5-29.6]{maggibook}, see also \cite[Theorem B.1]{CiLeMaIC1}. Since the version of this lemma needed here does not seem to appear elsewhere, we give some details of the proof.

\begin{lemma}[Volume-fixing variations]\label{lemma volume fixing}
  If $\E_0$ is a $N$-tiling of $\T$, $\g\in(0,1]$ and $L>0$, then there exist positive constants $r_0$, $\s_0$, $\e_0$, and $C_0$ (depending on $\E_0$, $L$ and $\g$ only) with the following property: if $\eta\in\R^N$ with $\sum_{h=1}^N\eta_h=0$, $\Phi\in\Lip(\T\times S^1;(0,\infty))$, $\psi\in C^{1,\g}(\T;(0,\infty))$, $x\in\T$, and $\E$ and $\F$ are $N$-tilings of $\T$ with
  \begin{gather}\label{volumefix hp 0}
    \|\Phi\|_{C^{0,1}(\T\times S^1)}+\|\psi\|_{C^{1,\g}(\T)}\le L\,,
    \\
    \label{volumefix hp 1}
    \d(\E,\E_0)\le \e_0\,,
    \\
    \label{volumefix hp 2}
    \F\Delta\E\cc B_{r_0}(x)\,,\qquad |\eta|<\s_0\,,
  \end{gather}
  then there exists a $N$-cluster $\F'$ such that
  \begin{eqnarray}\label{volumefix thesis 1}
    \F'\Delta\F&\cc& \T\setminus\ov{B}_{r_0}(x)\,,
    \\\label{volumefix thesis 2}
    \int_{\F'(h)}\psi&=&\eta_h+\int_{\E(h)}\psi\,,
    \\\label{volumefix thesis 3}
    |\PHI(\F')-\PHI(\F)|&\le& C_0\,P(\E)\,\Big(\sum_{h=1}^N\Big|\int_{\F(h)}\psi-\int_{\E(h)}\psi\Big|+|\eta|\Big)\,,
    \\\label{volumefix thesis 4}
    |\d(\F',\E)- \d(\F,\E)|&\le&C_0\,P(\E)\,\Big(\sum_{h=1}^N\Big|\int_{\F(h)}\psi-\int_{\E(h)}\psi\Big|+|\eta|\Big)\,.
  \end{eqnarray}
\end{lemma}

\begin{remark}
  {\rm In practice we are going to apply this lemma either with $\eta=0$ and $\F\Delta\E\ne\emptyset$, or with $\eta\ne 0$ and $\F=\E$. In the first case, we are given a compactly supported variation $\F$ of $\E$, and we want to modify $\F$ outside of $B_{r_0}(x)$ into a new $N$-tiling $\F'$ so that $\int_{\F'(h)}\psi=\int_{\E(h)}\psi$ for every $h=1,...,N$. In the second case we want to modify $\E$ so that  $\int_{\E(h)}\psi$ is changed into $\eta_h+\int_{\E(h)}\psi$ for every $h=1,...,N$. In both cases, we want to control the change in $\PHI$-energy and the change in distance from $\E$ needed to pass from $\F$ to $\F'$. The name attached to the lemma is motivated by the fact that one usually takes $\psi\equiv1$.}
\end{remark}

\begin{proof}[Proof of Lemma \ref{lemma volume fixing}]
  The basic step consists in picking up a ball $B_{z,\e}$ and notice that if $T\in C^\infty_c(B_{z,\e};\R^2)$ and $f_t(x)=x+t\,T(x)$ for $x\in\T$, then for every Borel set $E\subset\T$ the function $\PSI_E(t)=\int_{f_t(E)}\psi=\int_E\psi(f_t)Jf_t$ is of class $C^{1,\g}(-t_0,t_0)$ with
  \begin{equation}
    \label{derivata psi}
    \|\PSI_E\|_{C^{1,\g}(-t_0,t_0)}\le C\,,\qquad  \Big|\int_{f_t(E)}\psi-\int_E\psi-t\,\int_E\,\Div(\psi\,T)\Big|\le C\,|t|^{1+\g}\,,
  \end{equation}
  where $t_0$ and $C$ denote positive constants depending only on $\g$, $L$, $|\T|$, and $\|T\|_{C^1(\T)}$. Next, one considers two families of balls $\{B_{z_i,\e}\}_{i=1}^M$  and $\{B_{y_i,\e}\}_{i=1}^M$ with $z_i\,,y_i\in\pa^*\E_0(h(i))\cap \pa^*\E_0(k(i))$ (for $1\le h(i)\ne k(i)\le N$ to be properly chosen -- see condition \eqref{rank} below) and with $|z_i-z_j|>2\e$ and $|y_i-y_j|>2\e$ for $1\le i<j\le M$ and $|y_i-z_j|>2\e$ for $1\le i\le j\le M$. For each $i$ we can find $T_i\in C^\infty_c(B_{z_i,\e};\R^2)$ such that
  \begin{gather}\label{fix 1}
  \int_{\E_0(h(i))}\Div(\psi\,T_i)=1=-\int_{\E_0(k(i))}\Div(\psi\,T_i)\,,
  \\\label{fix 2}
  \int_{\E_0(j)}\Div(\psi\,T_i)=0\,,\qquad j\ne h(i),k(i)\,.
  \end{gather}
  Let us consider the smooth map $f:(-t_0,t_0)^M\times\T\to\T$ defined by $f(\mathbf{t},x)=x+\sum_{i=1}^M\,t_i\,T_i(x)$, $\mathbf{t}=(t_1,...,t_M)$, so that for $t_0>0$ small enough $f(\mathbf{t},\cdot)$ is a smooth diffeomorphism of $\T$ with
  \begin{equation}
    \label{spt 1}
      \spt(f(\mathbf{t},\cdot)-\Id)\cc \bigcup_{i=1}^MB_{z_i,\e}\,.
  \end{equation}
  If we let $\a=(\a_1,...,\a_N)\in C^{1,\g}((-t_0,t_0)^M;\R^N)$ be defined by
  \[
  \a_h(\mathbf{t})=\int_{f(\mathbf{t},\E(h))}\psi-\int_{\E(h)}\psi\,,\qquad h=1,...,N\,,
  \]
  then $\a((-t_0,t_0)^M)\subset V=\{\eta\in\R^N:\sum_{h=1}^N\eta_h=0\}$, $\|\a\|_{C^{1,\g}((-t_0,t_0)^M)}\le C$, and, by
  \eqref{volumefix hp 0}, \eqref{volumefix hp 1}, \eqref{derivata psi}, \eqref{fix 1} and \eqref{fix 2}, one finds
  \begin{eqnarray}\label{vicino}
  \Big|\frac{\pa\a_{h(i)}}{\pa t_i}(\mathbf{t})-1\Big|+\Big|\frac{\pa\a_{k(i)}}{\pa t_i}(\mathbf{t})+1\Big|+\max_{j\ne h(i),k(i)}\Big|\frac{\pa\a_{j}}{\pa t_i}(\mathbf{t})\Big|\le C\,\e_0\,,
  \end{eqnarray}
  where, from now on, $C$ denotes a constant depending only on $L$, $\g$, $|\T|$, and $\E_0$ (through $\|T_i\|_{C^1(\T)}$). Provided $h(i)$ and $k(i)$ are suitable defined (see \cite[Step one, Proof of Theorem 29.14]{maggibook}) one can entail from \eqref{vicino} that
  \begin{equation}
    \label{rank}
    {\rm dim}\nabla\a(\00)=N-1\,.
  \end{equation}
  By the implicit function theorem there exists $\s_1>0$ and an open neighborhood $U$ of $\00\in\R^M$ such that $\a^{-1}\in C^{1,\g}(V_{\s_1};U)$ with $V_{\s_1}=\{\eta\in V:|\eta|<\s_1\}$, and
  \begin{equation}
    \label{sotto}
    |\a^{-1}(\eta)|\le C\,|\eta|\,,\qquad \forall \eta\in V_{\s_1}\,.
  \end{equation}
  Similarly, we may construct functions $g$ and $\beta$, analogous to $f$ and $\a$, starting from the family of balls $\{B_{y_i,\e}\}_{i=1}^M$. Now let $\F$ be as in \eqref{volumefix hp 2}, and assume that
  \begin{equation}
    \label{s0 r0}
    \s_0+\|\psi\|_{C^0(\T)}\pi\,r_0^2<\s_1\,.
  \end{equation}
  Up to further decrease the value of $r_0$ with respect to $\e$, we may also assume that $\ov{B}_{r_0}(x)\cap \ov{B}_{\e}(z_i)=\emptyset$ for every $i=1,...,M$, or that $\ov{B}_{r_0}(z)\cap \ov{B}_{\e}(y_i)=\emptyset$ for every $i=1,...,M$. Without loss of generality we may assume to be in the former case, and set
  \[
  \F'(h)=(\F(h)\cap B_{r_0}(x))\cup(f(\a^{-1}(w),\E(h))\setminus B_{r_0}(x))\,,\qquad 1\le h\le N\,,
  \]
  where $w_h$ is defined by the identity
  \[
  \int_{\F(h)\cap B_{r_0}(x)}\psi=\eta_h-w_h-\int_{\E(h)\cap B_{r_0}(x)}\psi\,,\qquad 1\le h\le N\,.
  \]
  By construction one has \eqref{volumefix thesis 1}. Moreover, by definition of $w_h$, by \eqref{spt 1} and since $\ov{B}_{r_0}(x)\cap \ov{B}_{\e}(z_i)=\emptyset$ for every $i=1,...,M$, one has
  \begin{eqnarray*}
    \int_{\F'(h)}\psi-\int_{\E(h)}\psi&=&\int_{\F(h)\cap B_{r_0}(x)}\psi+\int_{f(\a^{-1}(w),\E(h))\setminus B_{r_0}(x)}\psi-\int_{\E(h)}\psi
    \\
    &=&\eta_h-w_h+\int_{f(\a^{-1}(w),\E(h))\setminus B_{r_0}(x)}\psi-\int_{\E(h)\setminus B_{r_0}(x)}\psi
    \\
    &=&\eta_h-w_h+\int_{f(\a^{-1}(w),\E(h))}\psi-\int_{\E(h)}\psi=\eta_h-w_h+\a_h(\a^{-1}(w))\,.
  \end{eqnarray*}
  By  \eqref{volumefix hp 2} and \eqref{s0 r0} one has $|w|<\s_1$, so that \eqref{volumefix thesis 2} is proved. We now notice that by \cite[Equation (2.9)]{dephilippismaggiARMA}
  \[
  \PHI(f(\mathbf{t},E))=\int_{f(\mathbf{t},\pa^*E)}\Phi(y,\nu_{f_t(E)}(y))\,d\H^1(y)
  =
  \int_{\pa^*E}\Phi\Big(f_t(x),{\rm cof}\nabla f_t(x)[\nu_E(x)]\Big)\,d\H^1(x)\,,
  \]
  so that, by \eqref{volumefix hp 0}, $|\PHI(f(\mathbf{t},E))-\PHI(E)|\le C\,|t|\,P(E)$. By \eqref{sotto} we immediately deduce \eqref{volumefix thesis 3}. Finally \eqref{volumefix thesis 4} is obtained by exploiting \cite[Lemma B.2]{CiLeMaIC1}.
\end{proof}

We now translate the improved convergence Theorem for planar bubble clusters \ref{improv Theorem} in the case of tilings of $\T$. One says that a $N$-tiling $\E$ of $\T$ is {\it $(\Lambda,r_0)$-minimizer} if
\[
P(\E)\le P(\F)+\Lambda\,\d(\E,\F)\,,
\]
whenever $\F$ is a $N$-tiling of $\T$ and  $\E\Delta\F\cc B_{r_0}(x)$ for some $x\in\T$. If $\E$ is a $(\Lambda,r_0)$-minimizing tiling of $\T$, then (by a trivial adaptation of Theorem \ref{rego planar cluster} above) $\E$ is of class $C^{1,1}$. Moreover, the curves $\g_i$ and the points $p_j$ in \eqref{class C1} are such that each $\g_i$ has distributional curvature bounded by $\Lambda$, and for every $p_j$ there exists exactly three curves from $\{\g_i\}_{i\in I}$ which share $p_j$ as a common boundary point, and meet at $p_j$ by forming three 120 degrees angles.

We notice that, by \eqref{hexagonal honeycomb thm torus}, the reference honeycomb $\H$ is a $(0,\infty)$-minimizing unit-area tiling of $\T$. The following result is what we call an {\it improved convergence theorem}.  The proof comes as a simple variant of Theorem \ref{improv Theorem} (or \cite[Theorem 1.5]{CiLeMaIC1}) and therefore we omit the details.

\begin{theorem}\label{thm improved convergence}
  Given $\Lambda\ge0$, there exist positive constants $L$ and $\mu_*>0$ (depending on $\Lambda$ and $\H$) with the following property. If $N=|\T|$, $\mu<\mu_*$ and $\{\E_k\}_{k\in\N}$ is a sequence of  $(\Lambda,r_0)$-minimizing $N$-tilings of $\T$ (for some $r_0>0$) with $\d(\E_k,\H)\to 0$ as $k\to\infty$, then there exist $k(\mu)\in\N$ and, for every $k\ge k(\mu)$, a $C^{1,1}$-diffeomorphism $f_k$ with
  \begin{equation}
    \label{ic C11 e C1}
      \sup_{k\ge k(\mu)}\|f_k\|_{C^{1,1}(\pa\H)}\le L\,,\qquad \lim_{k\to\infty}\|f_k-\Id\|_{C^1(\pa\H)}=0\,,
  \end{equation}
  \begin{equation}
    \label{ic tangenziale}
      \ttau_{\H}(f_k-\Id)=0\quad\mbox{on $[\pa\H]_\mu$}\,,\qquad \|\ttau_{\H}(f_k-\Id)\|_{C^1(\pa^*\H)}\le\frac{L}\mu\,\sup_{\S(\H)}|f_k-\Id|\,.
  \end{equation}
  In particular, $\E_k$ is a $(\e_k,\mu,L)$-perturbation of $\H$ whenever $k\ge k(\mu)$.
\end{theorem}


Let us now set
\begin{equation}
  \label{kappa R}
\k=\k(\T)=\inf\,\liminf_{k\to\infty}\frac{P(\F_k)-P(\H)}{\a(\F_k)^2}\,,
\end{equation}
where the infimum is taken among all sequences $\{\F_k\}_{k\in\N}$ of unit-area tilings of $\T$ such that $\a(\F_k)>0$ for every $k\in\N$ and $\a(\F_k)\to0$ as $k\to\infty$. By a compactness argument, Theorem \ref{thm main periodic} is equivalent in saying that $\k>0$.

\begin{lemma}\label{thm selection}
  If $\k=0$, then there exists a sequence of $(\Lambda,r_0)$-minimizing unit-area tilings $\{\E_k\}_{k\in\N}$ such that $\a(\E_k)>0$ for every $k\in\N$, $\a(\E_k)\to 0$ as $k\to\infty$, and
  \begin{equation}\label{behavior di Ek}
  P(\E_k)=P(\H)+o(\a(\E_k)^2)\,,\qquad\mbox{as $k\to\infty$}\,.
  \end{equation}
\end{lemma}

\begin{proof}
  By definition of $\k$, and since we are assuming $\k=0$, there exist unit-area tilings $\{\F_k\}_{k\in\N}$ of $\T$ such that $\a(\F_k)>0$ for every $k\in\N$, and
  \begin{equation}
    \label{Fk}
    \a(\F_k)\to 0\,,\qquad P(\F_k)=P(\H)+o(\a(\F_k)^2)\,,\qquad\mbox{as $k\to\infty$}\,.
  \end{equation}
  For every $k\in\N$, let $\E_k$ be a minimizer in the variational problem
  \[
  \inf\,\Big\{P(\E)+\d(\E,\F_k)^2 \ | \ \text{$\E$ unit-area tiling of $\T$ with $\a(\E)>0$}\Big\}\,.
  \]
  By comparing $\E_k$ with $\F_k$ and then subtracting $P(\H)$ one has
  \begin{equation}
    \label{comparison Fk}
      P(\E_k)-P(\H)+\d(\E_k,\F_k)^2\leq P(\F_k)-P(\H)=o(\a(\F_k)^2)\,.
  \end{equation}
  Since $|\a(\E_k)-\a(\F_k)|\le\d(\E_k,\F_k)$ and $P(\E_k)\ge P(\H)$, we conclude that
  \begin{equation}
    \label{take}
      \lim_{k\to\infty}\frac{\a(\E_k)}{\a(\F_k)}=1\,,
  \end{equation}
  so that, in particular, $\a(\E_k)\to 0$ as $k\to\infty$. Dividing by $\a(\E_k)^2$ in \eqref{comparison Fk} and using \eqref{take}, we complete the proof of \eqref{behavior di Ek}. We now show that each $\E_k$ is $(\Lambda,r_0)$-minimizer in $\T$. Indeed, let $r_0$, $\e_0$, $\s_0$ and $C_0$ be the constants associated by Lemma \ref{lemma volume fixing} to $\E_0=\H$, $\PHI=P$ and $\psi\equiv 1$. Since $\a(\E_k)\to0$, up to translations we have $\d(\E_k,\H)\le\e_0$ for $k$ large. We apply Lemma \ref{lemma volume fixing} with $\E=\E_k$, $\F$ a $N$-tiling with $\E_k\Delta\F\cc B_{r_0}(x)$ for some $x\in\T$, and $\eta=0$, to find a unit-area tiling $\F'$ such that
  \begin{eqnarray*}
    &&P(\E_k)\le P(\E_k)+\d(\E_k,\F_k)^2\le P(\F')+\d(\F',\E_k)^2
    \\
    &\le&P(\F)+C_0\,P(\E_k)\,|\vol(\F)-\vol(\E_k)|+\Big(\d(\F,\E_k)+C_0\,P(\E_k)\,|\vol(\F)-\vol(\E_k)|\Big)^2\,.
  \end{eqnarray*}
  Hence $P(\E_k)\le P(\F)+\Lambda\,\d(\E_k,\F)$ thanks to $|\vol(\F)-\vol(\E_k)|\le\d(\F,\E_k)$ and since, for $k$ large enough, $P(\E_k)\le 2\,P(\H)$.
  \end{proof}

\begin{proof}
  [Proof of Theorem \ref{thm main periodic}] We argue by contradiction. If the theorem is false, then $\k=0$ and thus by Lemma \ref{thm selection} there exists a sequence $\{\E_k\}_{k\in\N}$ of $(\Lambda,r_0)$-minimizing unit-area tilings of $\T$ such that $\a(\E_k)>0$, $\a(\E_k)\to 0$ as $k\to\infty$ and
  \[
  P(\E_k)=P(\H)+o(\a(\E_k)^2)\,,\qquad\mbox{as $k\to\infty$}\,.
  \]
  Up to translation we may assume that $\a(\E_k)=\d(\E_k,\H)\to 0$ as $k\to\infty$. Let $L$ and $\mu_*$ be the constants of Theorem \ref{thm improved convergence} (which depends on $\Lambda$ and $\H$) so that for every $\mu<\mu_*$ there exists $k(\mu)\in\N$ such that $\E_k$ is a $(\e_k,\mu,L)$-perturbation of $\H$ for every $k\ge k(\mu)$, with $\e_k\to 0$ as $k\to\infty$. Let $\e_0$ and $\mu_0$ be determined as in Theorem \ref{hexagonal honeycomb thm torus quantitative small} depending on $L$ and $|\T|$. If we set $\mu=\min\{\mu_*,\mu_0\}$ and increase $k(\mu)$ so that $\e_k\le\e_0$ for $k\ge k(\mu)$, then by Theorem \ref{hexagonal honeycomb thm torus quantitative small}, one finds $v_k\in\R^2$ with $|v_k|\le C\,\e_k$ such that
  \[
  P(\E_k)-P(\H)\ge c_0\,\hd(\pa\E_k,v_k+\pa\H)^2\ge c\,\d(\E_k,v_k+\H)^2\ge c\,\a(\E_k)^2\,,
  \]
  for some positive constant $c$. We have thus reached a contradiction, and proved the theorem.
\end{proof}

\begin{proof}[Proof of Theorem \ref{thm pertub volumes}]
   Let $\E_j=\E_{m^j}$ be minimizers in \eqref{variational problem volumes} for a sequence $\{m^j\}_{j\in\N}$ such that $\sum_{h=1}^Nm^j_h=N$, $m_h^j>0$ and $m_h^j\to 1$ as $j\to\infty$. By an explicit construction, for every $j$ large enough we can construct a small deformation $\H_j$ of $\H$ such that $|\H_j(h)|=m_h^j$ and $P(\H_j)\le P(\H)+C\,\max_{1\le h\le N}|m_h^j-1|$, with $C$ independent from $j$. (Alternatively, one can apply Lemma \ref{lemma volume fixing} with $\E_0=\E=\F=\H$, $\PHI=P$, $\psi\equiv1$ and $\eta_h=m_h^j-1$.) As a consequence, $\sup_{j\in\N} P(\E_j)<\infty$, and thus, up to extracting subsequences, $\d(\E_j,\E_0)\to0$ where $\E_0$ is a unit-area tiling of $\T$. In particular,
  \[
  P(\H)\le P(\E_0)\le\liminf_{j\to\infty}P(\E_j)\le\liminf_{j\to\infty}P(\H)+C\,\max_{1\le h\le N}|m_h^j-1|=P(\H)\,.
  \]
  By Hales's theorem, up to a relabeling of $\E_0$, $\E_0=v+\H$ for $v=(t\sqrt{3}\ell,s\ell)$ and $t,s\in[0,1]$. By performing the same relabeling on each $\E_j$, we have $\d(\E_j,v+\H)\to 0$. By exploiting Lemma \ref{lemma volume fixing} as in the proof of Lemma \ref{thm selection} one sees that each $\E_j$ is a $(\Lambda,r_0)$-minimizing tiling in $\T$, and then by arguing as in the proof of Theorem \ref{thm main periodic} we find a constant $L$ (depending on $\Lambda$ and $\H$) such that $\E_j-v$ is an $(\e_j,\mu_0,L)$-perturbation of $\H$ for $\mu_0$ as in Theorem \ref{hexagonal honeycomb thm torus quantitative small} and for $\e_j\to 0$ as $j\to\infty$. By Theorem \ref{hexagonal honeycomb thm torus quantitative small}, for $j$ large enough there exist $v_j\to 0$ and $C^{1,1}$-diffeomorphism $f_j$ between $v_j+\pa\H$ and $\pa\E_j-v$, with
  \begin{eqnarray*}
     C\,\max_{1\le h\le N}|m_h^j-1|\ge P(\E_j)-P(\H)
     \ge c\,\Big(\|f_j-\Id\|_{C^0(v_j+\pa\H)}^2+\|f_j-\Id\|_{C^1(v_j+\pa\H)}^4\Big)\,.
  \end{eqnarray*}
  Theorem \ref{thm pertub volumes} is then deduced by a contradiction argument.
\end{proof}

\begin{proof}[Proof of Theorem \ref{thm pertub metric}]
   In the following we denote by $\E_\de$ a minimizing in \eqref{finsler}, and set
   \[
   \de=\de(\Phi,\psi)=\|\Phi-1\|_{C^0(\T\times S^1)}+\|\psi-1\|_{C^0(\T)}\,,
   \]
   so that $\de<\de_0$. We notice that for every $E\subset\T$ of finite perimeter one has
   \begin{eqnarray}
    \label{bella}
      \Big|\int_E\psi -|E|\Big|&\le& C\,|E|\,\|\psi-1\|_{C^0(\T)}\,,
      \\
    \label{bella 2}
      |\PHI(E)-P(E)|&\le&C\,\min\{P(E)\,,\PHI(E)\}\,\|\Phi-1\|_{C^0(\T\times S^1)}\,,
   \end{eqnarray}
   where in \eqref{bella 2} we have also used the fact that $P(E)\le2\,\PHI(E)$ provided $\de_0\le 1$.

   \medskip

   \noindent {\it Step one}: We claim that, provided $\de_0$ is small enough, then
  \begin{eqnarray}\label{uno}
    \PHI(\E_\de)&\le&2\,P(\H)\,,
    \\\label{due}
    P(\E_\de)&\le&P(\H)+C\,\de\,.
  \end{eqnarray}
  Indeed, by considering an explicit small modification of $\H$ (or by applying Lemma \ref{lemma volume fixing} with $\E=\E_0=\F=\H$ and $\eta\ne 0$) we can construct a $N$-tiling $\H'$ of $\T$ such that $\int_{\H'(h)}\psi=N^{-1}\,\int_\T\psi$ for every $h=1,...,N$ and $\PHI(\H')\le \PHI(\H)+C\,\de$. By $\PHI(\E_\de)\le \PHI(\H')$ and by \eqref{bella 2}
  \begin{equation}
    \label{c}
      \PHI(\E_\de)\le \PHI(\H)+C\,\de\le P(\H)+C\,\de\,,
  \end{equation}
  which implies \eqref{uno}. Again by \eqref{bella 2}, $P(\E_\de)\le\PHI(\E_\de)+C\,\de$, and \eqref{c} gives \eqref{due}.

  \medskip

  \noindent {\it Step two}: We now show that if $\de_j=\de(\Phi_j,\psi_j)\to 0$ and $\E_j$ is a minimizer in \eqref{finsler} associated to $\Phi_j$ and $\psi_j$, then (and up to subsequences and to relabeling the chambers of $\E_j$) $\d(\E_j,v+\H)\to 0$ for some $v=(t\sqrt{3}\ell,s\ell)$, $s,t\in[0,1]$. By \eqref{uno} and since $\PHI_j(E)\ge P(E)/2$ for every $E\subset\T$ we find that $\sup_{j\in\N}P(\E_j)\le 4\,P(\H)$. By compactness, there exists a $N$-tiling $\E_*$ of $\T$ such that $\d(\E_j,\E_*)\to 0$ (up to subsequences). By \eqref{bella}, $\int_{\E_j(h)}\psi_j=N^{-1}\int_{\T}\psi_j$ implies $m_j(h)=|\E_j(h)|\to 1$ for every $h=1,...,N$. In particular, $\E_*$ is a unit-area tiling of $\T$, and thus by \eqref{hexagonal honeycomb thm torus}, by lower semicontinuity and by \eqref{due}
  \begin{equation}
    \label{ciao}
      P(\H)\le P(\E_*)\le\liminf_{j\to\infty}P(\E_j)\le P(\H)\,.
  \end{equation}
  By Hales's theorem, up a relabeling, $\E_*=v+\H$.

  \medskip

  \noindent {\it Step three}: Let $\e_0$, $r_0$, $\s_0$ and $C_0$ be the constants associated to $\E_0=\H$, $\Phi$ and $\psi$ by Lemma \ref{lemma volume fixing}. (Notice that the same constants will work on any translation of $\H$, and that these constants ultimately depend on $L$ and $\g$ only.) By step two we can assume that $\de_0$ is small enough to entail $\d(\E_\de,v_\de+\H)\le\e_0$ for some translation $v_\de$. We now claim that there exist positive constants $r_1\,,c_0>0$ such that
  \begin{equation}
    \label{lb}
    |\E_\de(h)\cap B_{r}(x)|\ge c_0\,r^2\,,\qquad\forall x\in\pa\E_\de(h)\,,r<r_1\,,h=1,...,N\,.
  \end{equation}
  This is a classical argument, see for example \cite[Lemma 30.2]{maggibook}, and we include some details just for the sake of completeness. Without loss of generality let us set $h=1$ and fix $x\in\pa\E_\de(1)$ and $r<r_1\le r_0$ such that $P(\E_\de;\pa B_{r}(x))=0$. There exists $j\in\{1,...,N\}$ such that
  \begin{equation}
    \label{questa}
      \H^1(\pa^*\E_\de(1)\cap\pa^*\E_\de(j)\cap B_{r}(x)\ge \H^1(\pa^*\E_\de(1)\cap\pa^*\E_\de(h)\cap B_{r}(x) )\,,\qquad\forall h\ne 1,j\,.
  \end{equation}
  If we set $\F(1)=\E_\de(1)\setminus B_r(x)$, $\F(j)=\E_\de(j)\cup(\E_\de(1)\cap B_r(x))$ and $\F(h)=\E_\de(h)$ for $h\ne 1,j$, then by applying Lemma \ref{lemma volume fixing} with $\E_0=v_\de+\H$, $\E=\E_\de$, and $\eta=0$ and setting $u(r)=|\E_\de(1)\cap B_r(x)|$, we find that, if $\e<r_0-r$, then
  \begin{eqnarray*}
    \PHI(\E_\de;B_{r+\e}(x))&\le&\PHI(\F;B_{r+\e}(x))+C_0\,P(\E_\de) \Big|\int_{\E_\de(1)\cap B_r(x)}\psi\Big|
    \\
    &\le&
    \PHI(\E_\de;B_{r+\e})+\hat{\PHI}(B_r(x);\E_\de(1))
    \\
    &&-\int_{\pa^*\E_\de(1)\cap\pa^*\E_\de(j)\cap B_r(x)}\hat{\Phi}(y,\nu_{\E_\de(1)}(y))\,d\H^1+C\,u(r)\,,
  \end{eqnarray*}
  where we have set $\hat\Phi(x,\nu)=(\Phi(x,\nu)+\Phi(x,-\nu))/2$. In particular, by \eqref{questa} and by $2\ge\Phi\ge 1/2$, for every $h\ne 1$ one finds
  \[
  \H^1(\pa^*\E_\de(1)\cap\pa^*\E_\de(h)\cap B_r(x))\le C(\H^1(\E_\de(1)\cap\pa B_r(x))+u(r))\,,
  \]
  i.e.
  \[
  P(\E_\de(1);B_r(x))\le C(u'(r)+u(r))\,,\qquad\mbox{for a.e. $r<r_1$}\,.
  \]
  By adding $u'(r)=\H^1(\E_\de(1)\cap\pa B_r(x))$ to both sides we find that
  \[
  C(u'(r)+u(r))\ge P(\E_\de(1)\cap B_r(x))\ge 2\sqrt{\pi\,u(r)}\,.
  \]
  In particular if $r_1$ is small enough to give $C\,u(r)\le C\sqrt{\pi r_1^2\,u(r)}\le \sqrt{\pi\,u(r)}$, then we find $\sqrt{u(r)}\le C\,u'(r)$ for a.e. $r<r_1$. This proves \eqref{lb}.

  \medskip

  \noindent {\it Step four}: We now conclude the proof. Again by step two and by Lemma \ref{lemma volume fixing}, one can find a unit-area tiling $\E_\de'$ of $\T$ such that $P(\E_\de')\le P(\E_\de)+C\,\de$ and $\d(\E_\de',\E_\de)\le C\,\de$. By Theorem \ref{thm main periodic}  and up to permutations of the chambers of $\E_\de$, we find a translation $v_\de$ such that
  \[
  c\,\d(\E_\de',v_\de+\H)^2\le P(\E_\de')-P(\H)\le P(\E_\de)-P(\H)+C\,\de\le C\,\de\,,
  \]
  where in the last inequality we have used \eqref{due}. Since $\d(\E_\de',v_\de+\H)\ge \d(\E_\de,v_\de+\H)-\d(\E_\de',\E_\de)$ we conclude
  \[
  \d(\E_\de,v_\de+\H)^2\le C\,\de\,.
  \]
  Setting for the sake of brevity $v_\de=0$, we now pick $x\in\pa\E_\de(1)$ such that $\dist(x,\pa\H(1))\ge\dist(y,\pa\H(1))$ for every $y\in\pa\E_\de(1)$. Let $r=\min\{r_1,\dist(x,\pa\H(1))\}$, so that either $B_r(x)\subset\T\setminus\H(1)$ or $B_r(x)\subset\H(1)$. In particular, provided $\de_0$ is small enough with respect to $c_0$, either
  \[
  \d(\E_\de,\H)\ge|\E_\de(1)\setminus\H(1)|\ge |\E_\de(1)\cap B_r(x)|\ge c_0\,r^2\ge c_0\,\dist(x,\pa\H(1))^2\,,
  \]
  or
  \begin{eqnarray*}
  \d(\E_\de,\H)&\ge&|\H(1)\setminus\E_\de(1)|\ge |B_r(x)\setminus\E_\de(1)|=\Big{|}\bigcup_{h=2}^NB_r(x)\cap\E_\de(h)\Big{|}
  \\
  &\ge& (N-1)c_0\,r^2\ge c_0\,\dist(x,\pa\H(1))^2\,;
  \end{eqnarray*}
  in both cases, $\pa\E_\de(1)\subset I_\e(\pa\H(1))$ for $\e=C\,\sqrt{\d(\E_\de,\H)}$. By the same argument (based on area density estimates for $\H$, which hold trivially) one finds that $\pa\H(1)\subset I_\e(\pa\E_\de(1))$.
\end{proof}

\chapter{Cheeger $N$-clusters}

\section{Introduction}\label{cpt 4 sct 1}
For a given open, bounded set $\Om$ and an integer $N\in \N$ we introduce the \textit{$N$-Cheeger constant of $\Om$} as:
\begin{equation}\label{N-cheeger constant 1} 
H_N(\Om)=\inf\left\{\sum_{i=1}^N\frac{P(\E(i))}{|\E(i)|}\ \Big{|}\ \E=\{\E(i)\}_{i=1}^{N} \subseteq \Om , \text{ is an $N$-cluster}\right\}.
\end{equation}
As shown below in Theorem \ref{existence}, the infimum in \eqref{N-cheeger constant 1} is always attained and we refer to the minimizers as the \textit{Cheeger $N$-clusters of $\Om$}.\\

\indent We focus on the quantity $H_N$ because it seems to represent the right object to study in order to provide some non trivial lower bound on the optimal partition functional
\begin{equation}\label{p-laplacian}
\Lambda_N^{(p)}(\Om)=\inf\left\{\sum_{i=1}^N\lambda_1^{(p)}(\E(i))  \right\},
\end{equation}
where $\lambda_1^{(p)}$ denotes the first Dirichlet eigenvalue of the p-Laplacian, defined as:
$$\lambda_1^{(p)}(E):=\inf\left\{\int_{E} |\nabla u|^{p} \d x \ \Big{|}\ u\in W^{1,p}_0(E), \ \|u\|_{L^{p}}=1 \right\}.$$
The infimum in \eqref{p-laplacian} is taken over all the $N$-clusters $\E$ whose chambers are \textit{quasi-open sets of $\Om$}. The family of quasi-open sets of an open bounded set $\Om$ is a suitable sub-class of the Borel's algebra of $\Om$  where the first Dirichlet eigenvalue of the $p$-Laplacian $\l_1^{(p)}$ can be defined. The definition of quasi-open set is related to the concept of $p$-capacitary measure in $\R^n$ that we do not need to recall in here (see \cite{EvGa91} for more details about it). For our purposes it is enough to recall that:
 \begin{center}
\textit{the quasi-open sets are the upper levels of $W^{1,p}$ functions as well as the open sets are the upper levels of continuous functions. Each open set of an open bounded set $\Om$ is also a quasi-open set of $\Om$.}
\end{center}

The importance of the partition problem \eqref{p-laplacian} relies in the fact that it provides a way to look at the asymptotic behavior in $N$ of the $N$-th Dirichlet eigenvalue of the classical Laplacian (the $2$-Laplacian), as Caffarelli and Lin show in \cite{CaLi07}. The  $N$-th Dirichlet eigenvalue of the Laplacian of an open set $\Om$ is recursively defined as 
\begin{align*}
\l_N^{(2)}(\Om)&=\inf_{u\in X_{N-1}}\left\{\frac{\int_{\Om}|\nabla u|^2\d x}{\int_{\Om}|u|^2\d x}\right\}\\
X_{N-1}&=\left\{u\in W^{1,2}_0(\Om) \ | \ \langle u,u_i\rangle_2=0, \ \ \text{for all $i=1,\ldots, N-1$}\right\}
\end{align*}
where $u_1,\ldots,u_{N-1}$ are the first $N-1$ eigenfunctions
	\[
	\l_i^{(2)}(\Om)=\frac{\int_{\Om}|\nabla u_i|^2\d x}{\int_{\Om}|u_i|^2\d x} \ \ \ \ \ \text{for all $i=1,\ldots, N-1$}
	\]
and $\langle \cdot ,\cdot\rangle_2$ denotes the standard scalar product of $L^2(\Om)$
	\[
	\langle u ,v\rangle_2=\int_{\Om} uv\d x \ \ \ \ \ \text{for all $u,v\in L^2(\Om)$}
	\]
(see \cite[Section 6.5]{EvGa91} for a detailed discussion about eigenvalues and eigenfunctions). In \cite{CaLi07}, Caffarelli and Lin prove that there exist two constants $C_1$ and $C_2$ depending only on the dimension such that
\begin{equation}\label{CaLi estimate}
C_1\frac{\Lambda_N^{(2)}(\Om)}{N}\leq \lambda_N^{(2)}(\Om) \leq C_2 \frac{\Lambda_N^{(2)}(\Om)}{N},
\end{equation}
where $\lambda_N^{(2)}$ is the $N$-th Dirichlet eigenvalue. The detailed study of $\lambda_N^{(2)}(\Om)$ for $N\geq 2$ seems to be an hard task (so far only the case $N=1,2$ are well known in details, see for instance \cite{H06}) and that is why the asymptotic approach suggested by Caffarelli and Lin could be a good way to look at the spectral problem. We also refer the reader to \cite{bucur2012minimization} where the existence of minimizers for $\l_N^{(2)}$ is proved. \\

\indent Caffarelli and Lin's conjecture (appearing in \cite{CaLi07}) about the asymptotic behavior of $\Lambda_N^{(2)}(\Om)$ in the planar case states that 
$$\Lambda_N^{(2)}(\Om)=\frac{N^{2}}{|\Om|}\lambda_1^{(2)}(H)+o(N^{2}),$$
where $H$ denotes a unit-area regular hexagon. So far, no progress has been made in proving the conjecture, anyway numerical simulations (see \cite{BouBucO09}) point out that the conjecture could be true. If the conjecture turns out to be true, relation \eqref{CaLi estimate} could be improved, in the planar case, as:
\begin{equation}\label{CaLi estimate 2.0}
C_1\frac{N\lambda_1^{(2)}(H)}{|\Om|}+o(N)\leq \lambda_N^{(2)}(\Om) \leq C_2 \frac{N\lambda_1^{(2)}(H)}{|\Om|}+o(N).
\end{equation}

In order to explain the connection between $H_N$ and $\Lambda_N^{(p)}$ we recall some well-known fact about the classical Cheeger constant of a Borel set $\Om$:
\begin{equation}\label{cheeger constant}
h(\Om):=\inf\left\{\frac{P(E)}{|E|} \ \Big{|} \ E \subseteq \Om \right\},
\end{equation}
(note that $h(\Om)=H_1(\Om)$). Given an open set $\Om$, each set $E\subseteq \Om$ such that $h(\Om)=\frac{P(E)}{|E|}$ is called \textit{Cheeger set for $\Om$}. It is possible to prove that each Cheeger set $E$ for $\Om$ is a $(\La,r_0)$-perimeter-minimizing inside $\Om$ and that $\pared E\cap \Om$ is a constant mean curvature analytic hypersurface relatively open inside $\pa E$. Furthermore, the mean curvature $C$ of the set $E$ in the open set $\Om$ is equal to $C=\frac{1}{n-1}h(E)$. We refer the reader to \cite{Pa11} and \cite{Leo15}: two exhaustive surveys on Cheeger sets and Cheeger constant.\\

The Cheeger constant was introduced by Jeff Cheeger in \cite{Ch70} and provides a lower bound on the first Dirichlet eigenvalue of the $p$-Laplacian of a domain $\Om$. By exploiting the coarea formula and Holdër' s inequality  it is possible to show that for every domain $\Om$ and for every $p>1$ it holds,
\begin{equation}\label{lb p-eig}
\lambda_1^{(p)}(\Om)\geq \left(\frac{h(\Om)}{p}\right)^p.
\end{equation}
The Cheeger constant is also called the {\it first Dirichlet eigenvalue of the 1-laplacian} since, thanks to \eqref{lb p-eig} and to a comparison argument
\begin{equation}\label{eqn limite per p che tende a uno}
\lim_{p\rightarrow 1} \lambda_1^{(p)}(\Om)=h(\Om).
\end{equation}
See, for example, \cite{KN08} for more details about the relation between the Cheeger constant and the first Dirichlet eigenvalue of the p-Laplacian  or \cite{BucBu05} and \cite{Bu10} for more details about the spectral problems and shape optimization problems. \\

We note here that the constant $H_N$ is the analogous of the Cheeger constant in the optimal partition problem for $p$-laplacian eigenvalues. We refer the reader to \cite{Pa09}, where a generalized type of Cheeger constant for the $2$-nd Dirichlet eigenvalue of the Laplacian is also studied. As we show in Proposition \ref{limit} below, we can always give a lower bound on $\Lambda_N^{(p)}$ by making use of \eqref{lb p-eig} and Jensen's inequality: 
\begin{equation}\label{Lp lowerbound}
\Lambda_N^{(p)}(\Om)\geq \frac{1}{N^{p-1}}\left(\frac{H_N(\Om)}{p}\right)^{p}.
\end{equation}
By combining \eqref{Lp lowerbound} with a comparison argument (see Theorem \ref{limite} below) we are also able to compute the limit as $p$ goes to $1$ and obtain
 	\begin{equation}\label{viva l italia}
 	\lim_{p\rightarrow 1}  \Lambda_N^{(p)}(\Om)=H_N(\Om).
 	\end{equation}
Thus, the constant $H_N$ seems to provide the suitable generalization of the Cheeger constant for the study of $\La_N^{(p)}$.\\

In this chapter we mainly focus on the general structure and regularity of Cheeger $N$-clusters in order to lay the basis for future investigations on $H_N$. In the final section, once we have proved \eqref{viva l italia}, we study the asymptotic behavior of $H_N$ in the planar case. The statements involving regularity are quite technical and we reserve to them the whole Section \ref{cpt 4 sct 1 sbsct 1} (Theorems \ref{mainthm1}, \ref{mainthm2} and \ref{mainthm3}),  we just point out here that if $\E$ is a Cheeger $N$-cluster of $\Om$ the following statement holds.\\

\textit{For every $i=1,\ldots,N$ the reduced boundary of each chambers $\pared \E(i)\cap \Om$ is a $C^{1,\a}-$h\-y\-per\-sur\-fa\-ces (for every $\a\in (0,1)$ ) that is relatively open inside $\pa \E(i)\cap \Om$. Furthermore it is possible to characterize the singular set of a Cheeger $N$-cluster $\E$ as a suitable collection of points with density zero for the external chamber 
	\[
	\E(0)=\Om\setminus \bigcup_{i=1}^N \E(i).
	\]
Moreover if the dimension is $n=2$ then the singular set is discrete and the chambers $\E(i)\cc \Om$ are indecomposable.}\\

Note that, in this context, the external chambers should be intended as $\Om\setminus (\cup_i \E(i) )$ instead of $\R^n\setminus \bigcup_i\E(i)$ as usual (that is because the ambient space is $\Om$ in place of $\R^n$). As we are pointing out below, also the definition of "singular set of a Cheeger $N$-cluster" must be given in a slightly different way (see \eqref{insieme singolare 1}) from the standard one $\pa \E \setminus \pared \E$, since this last set turns out to be too small. Let us postpone this discussion below to Section \ref{cpt 4 sct 1 sbsct 1} where precise statements are given, and let us, instead, briefly focus on the asymptotic properties of $H_N$ (to which Subsection \ref{asymptoyc of HN} is devoted).\\

We note that for $H_N$ it is reasonable to expect a behavior of the type 
\begin{equation}\label{come va}
H_N(\Om)= C(\Om) N^{\frac{3}{2}}+o(N^{\frac{3}{2}}), 
\end{equation}
for some constant $C(\Om)$. In Theorem \ref{asymptotic behavior} (Property 3) ) we provide some asymptotic estimate for $H_N$ showing that the exponent $\frac{3}{2}$ in \eqref{come va} is the correct one and proving that for any given bounded open set $\Om\subset \R^2$ it holds
\begin{equation}\label{asintotico in N per HN}
\frac{h(B)\sqrt{\pi}}{\sqrt{|\Om|}}\leq\liminf_{N\rightarrow +\infty} \frac{H_N(\Om)}{N^{\frac{3}{2}}}\leq  \limsup_{N\rightarrow +\infty} \frac{H_N(\Om)}{N^{\frac{3}{2}}}\leq  \frac{h(H)}{\sqrt{|\Om|}},
\end{equation}
We here conjecture that 
$$C(\Om)=\frac{h(H)}{\sqrt{|\Om|}},$$ 
which is nothing more than Caffarelli and Lin's conjecture for the case $p=1$. Note that, thanks to \eqref{Lp lowerbound} this would imply
\begin{equation}
\Lambda_N^{(2)}(\Om)\geq \frac{N^{2}}{|\Om|} \left(\frac{h(H)}{2}\right)^{2}+o(N^2),
\end{equation}
a "weak" version of Caffarelli and Lin's conjecture.  It seems coherent and natural to expect this kind of behavior for $H_N(\Om)$.\\

The chapter is organized as follows. In Section \ref{cpt 4 sct 1 sbsct 1} we present and comment the three main statements describing the regularity property and the structure of Cheeger $N$-clusters. Sections \ref{cpt 4 sct Existence and regularity}, \ref{cpt 4 sct The singular set of the Cheeger N-clusters in low dimension} and \ref{cpt 4 the planar case} are devoted to the proof of the Theorems introduced in Section \ref{cpt 4 sct 1 sbsct 1}. In the final section \ref{limite} we show the connection between $H_N$ and $\La^{(p)}_N$ and we establish the asymptotic trend of $H_N$ for $N$ large in the planar case.

\section{Basic definitions and regularity theorems for Cheeger $N$-clusters }\label{cpt 4 sct 1 sbsct 1}

We present three statements that we are going to prove in Section \ref{cpt 4 sct Existence and regularity} and in Subsections \ref{cpt 4 sbsct Proof of mainthm2} and \ref{cpt 4 sbsct proof of mainthm3}. \\

In Section \ref{cpt 4 sct Existence and regularity} after we have shown existence of Cheeger $N$-clusters for any given \textit{bounded} ambient space $\Om$ with finite perimeter (Theorem \ref{existence}) we provide the partial regularity Theorem \ref{mainthm1} in the spirit of Theorems \ref{rego planar cluster} and \ref{regularity}. Set, for a generic Borel set $F$ and for $i=1,\ldots,N$
\begin{align}
\S(\E(i);F)&:=[\pa \E(i)\setminus \pared\E(i)]\cap F\label{insieme singolare i},\\
\S(\E(i))&:=\S(\E(i);\R^n).\label{insieme singolare ii}
\end{align}
\begin{theorem}\label{mainthm1}
Let $n\geq 1, N\geq 2$. Let $\Om\subset \R^n$ be an open bounded set with finite perimeter and $\E$ be a Cheeger $N$-cluster of $\Om$.
Then for every $i=1,\ldots,N$ the following statements hold true: 
\begin{itemize}
\item[(i)] For every $\a\in \left(0,1\right)$ the set $\Om \cap \pared \E(i)$ is a $C^{1,\a}$-hypersurface that is relatively open in $\Om\cap \pa\E(i)$ and it is $\H^{n-1}$ equivalent to $\Om \cap \pared \E(i)$; 
\item[(ii)] For every $i=1,\ldots,N$ the set $\pa \E(i) \cap \Om$ can meet $\pared \Om$ only in a tangential way, that is: $\pa^*\Om \cap \pa \E(i) \subseteq \partial^*\E(i)$. Moreover for every $x\in \pa^*\Om \cap \pa\E(i)$ it holds:
$$\nu_{\E(i)}(x)=\nu_{\Om}(x).$$
Here $\nu_{\E(i)}$, $\nu_{\Om}$ denote, respectively, the measure theoretic outer unit normal to $\E(j)$ and to $\Om$;
\item[(iii)]  $\S(\E(i);\Om)$ is empty if $n\leq 7$;
\item[(iv)] $\S(\E(i);\Om)$ is discrete if $n=8$;
\item[(v)] if $n\geq 9$, then $\H^{s}(\S(\E(i);\Om))=0$ for every $s>n-8$.
\end{itemize}  
\end{theorem}

For proving $(i),(iii),(iv)$ and $(v)$ we simply show (in Theorem \ref{regolare}) that each chamber $\E(i)$ is a $(\Lambda,r_0)$-perimeter-minimizing in $\Om$ (see Definition \ref{Lambdarminimi}) and then we make use of the De Giorgi's regularity Theorem \ref{regularity}. We re-adapt an idea from \cite{BaMa82} based on the fact that a solution of an obstacle problem having bounded distributional mean curvature is regular. Assertion $(ii)$ follows as a consequence of \cite[Proposition 2.5, Assertion  (vii)]{LP14} retrieved below (Proposition \ref{leo}).

\begin{remark}
\rm{ We need to ask that $\Om$ is bounded otherwise no Cheeger $N$-clusters are attained. Indeed if $\Om$ is unbounded, by intersecting $\Om$ with $N$ suitable disjoint balls of radius approaching $+\infty$ we easily obtain $H_N(\Om)=0$.}
\end{remark}

\subsection{The role of the singular set $\S(\E)$}
Note that Theorem \ref{mainthm1} yields the inner regularity of \textit{all} the chambers, differently from Theorem \ref{rego planar cluster} which involves the reduced boundary of the cluster $\pared \E$ (which in general is smaller than the union of the reduced boundary of the chambers). This stronger regularity of the chambers somehow affect the behavior of the singular set.  For example consider the case $n\leq 7$. In this case, according to Theorem \ref{mainthm1}, for a Cheeger $N$-cluster it must hold that 
$$(\pa \E \setminus \pared \E)\cap \Om=\emptyset,$$
and this would lead us to say that the singular set of a Cheeger $N$-cluster is empty which is clearly not the case. Indeed let us highlights that there is somehow an "hidden chamber" that plays a key role and influences the behavior of the global structure of these objects, namely the \textit{external chamber}:
$$\E(0)=\Om\setminus \left(\bigcup_{i=1}^N \E(i)\right).$$
Note that this definition of external chamber is slightly different from the ones adopted in the previous chapters. Since our ambient space in this context is the open bounded set $\Om$ instead of $\R^n$ we found convenient and coherent to keep this notation throughout this chapter.\\

Even if Theorem \ref{mainthm1} provides a satisfactory description of $\Sigma(\E(i),\Omega)$, this does not exhaust the analysis of the singular set of $\E$. Indeed the chamber $\E(0)$ is not regular after all and there are points in $\pa \E(0)$ of cuspidal type. For a complete description of the singularity, the correct definition of \textit{singular set of a Cheeger $N$-cluster $\E$ in the Borel set $F$} must be given as
\begin{align}
\S(\E;F)&:=\S(\E(0);F) \cup \bigcup_{i=1}^N \S(\E(i);F),\label{insieme singolare 1}
\end{align}
where for $i\neq 0$ the set $\S(\E(i))$ are the ones defined in \eqref{insieme singolare i} and \eqref{insieme singolare ii}, while for $i=0$ we clearly set
	\[
	\S(\E(0);F)=[\pa \E(0)\setminus \pared\E(0) ]\cap F.
	\]
With this definitions, $(\pa \E\setminus \pared \E)\cap \Om \subseteq \S(\E;\Om)$. With a slight abuse of notation we denote by $\S(\E)$ the singular set of a Cheeger $N$-cluster, even if it is different from the one defined in \eqref{insieme singolare di un N-cluster}. Since Theorem \ref{mainthm1} do not provides information about $\S(\E(0))$, we focus our attention on it in Subsection \ref{cpt 4 sbsct Proof of mainthm2} where the following theorem is proved. 
\begin{theorem}\label{mainthm2}
Let $1\leq n\leq 7, N\geq 2$, Let $\Om\subset \R^n$ be an open, connected, bounded set with $C^1$ boundary and finite perimeter and $\E$ be a Cheeger $N$-cluster of $\Om$. Then the following statements hold true.
\begin{itemize}
\item[(i)] $\E(0)$ is not empty and $\H^{n-1}(\pa\E(0)\cap \pa\E(j))>0$ for all $i=1,\ldots,N$;
\item[(ii)]  $\S(\E(0);\Om)=\pa\E(0)\cap \E(0)^{\zero}$, $\S(\E(0);\Om)$ is closed and 
\begin{eqnarray}
\S(\E(0);\Om)  &=&\Om \cap \bigcup_{\substack{j,k=1, \\ k\neq j}}^N  (\pa \E(j)\cap \pa \E(k) \cap \pa \E(0) ) 
\end{eqnarray}
\end{itemize}
\end{theorem}

\begin{remark}\label{controesempioconne}
\rm{
Note that Assertion $(ii)$ of Theorem \ref{mainthm2}, stated as above, would be meaningless if we do not ensure that $|\E(0)|>0$ (that is Assertion $(i)$, proved in Proposition \ref{ognicameraconfinaconilvuoto}). The assumption on $\Om$ to be connected and with $C^1$-boundary are the necessary ones to ensure the validity of this fact. Probably, the theorem remains true also by replacing \textit{$C^1$ boundary} with \textit{Lipschitz boundary}. Anyway we prefer to state and prove it by taking advantage of this stronger regularity on $\pa \Om$ in order to avoid some technicality. Let us also point out that there are situations where $\Om$ is not connected or $\pa \Om$ is not Lipschitz and where $\E(0)$ turns out to be empty. For example, given a set $\Om$ and one of its Cheeger $N$-cluster $\E$, we provide a counterexample by defining the new open set 
$$\Om_0=\left(\bigcup_{j=1}^N \mathring{\E(j)}\right).$$
The $N$-cluster $\E$ will be a Cheeger $N$-clusters of $\Om_0$ also and, by construction, $|\E(0)|=0$ (see Figure \ref{controesempio}).  The reason is that $\Om_0$ has no regular boundary.  As a further example one may also consider the case when $\Om$ is the union of $N$ disjoint balls. Anyway, it is reasonable to expect that, no matter what kind of ambient space $\Om$ we choose, for $N$ sufficiently large the chamber $\E(0)$ will be not empty.}
\end{remark}
\begin{figure}
\begin{center}
 \includegraphics[scale=1]{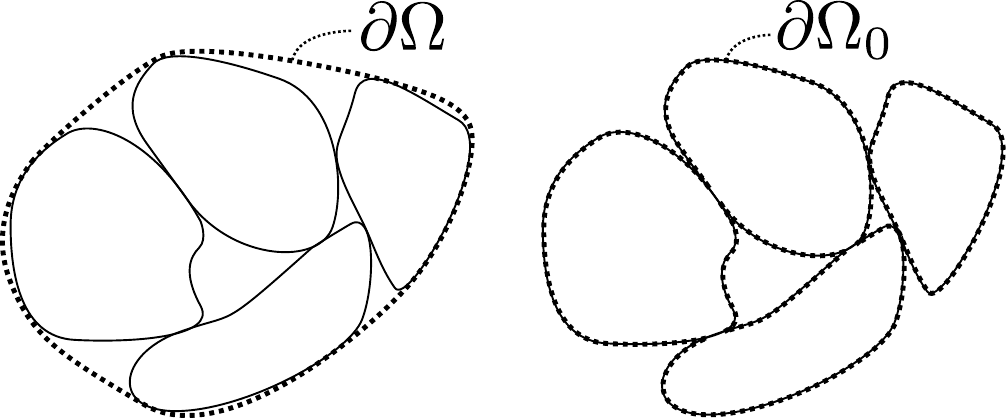}\caption{{\small The set $\Om_0$ built as the union of the interior of 
 the Cheeger $N$-cluster of an open set $\Om$. The external chamber of this Cheeger $N$-cluster of $\Om_0$ is empty because of the 
 cusps at the boundary of the open set.}}\label{controesempio}
 \end{center}
\end{figure}
\begin{remark}
\rm{Note that we ask for the dimension $n$ to be less than $7$. That is because, to prove Theorem \ref{mainthm2}, we exploit the regularity given by \ref{mainthm1} and we prefer to deal with the favorable case $n\leq 7$ where the singular set $\S(\E(i);\Om)=\emptyset$ for $i\neq 0$. Let us also point out that Assertion $(ii)$ remains true also in dimension bigger than $7$ up to replace $\Om$ with $\Om_0=\Om\setminus \cup_{i\neq 0} \S(\E(i);\Om)$. The interesting and not-trivial fact is that we actually do not know if assertion (i) remains true in dimension bigger than $7$ since, in the proof of Proposition \ref{ognicameraconfinaconilvuoto} (the crucial one in order to prove assertion $(i)$), we make a strong use of the fact $\S(\E(i);\Om)=\emptyset$. Roughly speaking in dimension bigger than $7$  it could happen that the chambers, by taking advantage of the possible presence of singular points $x\in \S(\E(i);\Om)$, can be combined in a way that kill $\E(0)$ even under a strong regularity assumption on $\Om$.}
\end{remark}

\begin{remark}\label{the singular set}
\rm{Somehow assertion $(ii)$ of Theorem \ref{mainthm2} is saying that the only singular points of $\E$ are the one where a cusp is attained.  Now we can give a complete description of the singular set $\S(\E;\Om)$ of a Cheeger $N$-cluster of an open, bounded, connected set $\Om$ with finite perimeter and $C^1$ boundary in dimension less than or equal to $7$. By combining Assertion $(iii)$ in Theorem \ref{mainthm1} and assertion $(i)$ in \ref{mainthm2} we can write
$$\S(\E;\Om)=\S(\E(0);\Om)=(\pa \E(0)\cap \E(0)^{(0)})\cap \Om.$$
}
\end{remark}

\subsection{The planar case}
Theorem \ref{mainthm2} gives us a precise structure of $\S(\E;\Om)$. We do not focus here on the singular set $\S(\E;\pa\Om)$ anyway, by exploiting the $C^1$-regularity assumption on $\pa \Om$, it is possible to prove a result in the spirit of Theorem \ref{mainthm2} also for the singular set $\S(\E(0);\pa \Om)$ (and thus characterize $\S(\E;\pa\Om)$). Let us point out that, at the present, the crucial information $\H^{n-1}(\S(\E(0);\Om))=0$ is missing. We are able to fill this gap when the ambient space dimension is $n=2$, together with some remarkable facts stated in the following theorem (proved in Subsection \ref{cpt 4 sbsct proof of mainthm3}). 
\begin{theorem}\label{mainthm3}
Let $n=2, N\geq 2$. Let $\Om\subset \R^2$ be an open, connected, bounded set with $C^1$ boundary and finite perimeter and $\E$ be a Cheeger $N$-cluster of $\Om$. Then the following statements hold true.
\begin{itemize}
\item[(i)] The singular set $\S(\E(0);\Om)$ is a finite union of points $\{x_j\}_{j=1}^k \subset \Om$. 
\item[(ii)] For every $j,k=0,\ldots,N$, $k \neq j$ the set
$$E_{j,k}:=[\pa \E(j)\cap \pa \E(k) \cap\Om]\setminus \S(\E(0);\Om)$$ 
is relatively open in $\pa \E(j)$ ($\pa \E(k)$) and is the finite union of segments and circular arcs. Moreover the set $\E(j)$ has constant curvature $C_{j,k}$ inside each open set $A$ such that $A\cap \pa \E(j) \subseteq E_{j,k}$. The constant $C_{j,k}$ is equal to:
\begin{equation}
C_{j,k}= \left\{
\begin{array}{cc}
\frac{|\E(k)|h(\E(j))-|\E(j)|h(\E(k))}{|\E(j)|+|\E(k)|}, & \text{if $k\ne 0$},\\
 & \\
h(\E(j)), & \text{if $k=0$}.
\end{array}
\right.
\end{equation}
As a consequence the set $\E(k)$ has constant curvature $C_{j,k}=-C_{k,j}$ inside each open set $A$ such that $A\cap \pa\E(k) \subseteq E_{k,j}$ ($=E_{j,k}$);
\item[(iii)] Each chamber $\E(j)\cc \Om$ is indecomposable.
\end{itemize} 
\end{theorem}
\begin{figure}
\centering
 \includegraphics[scale=3.0]{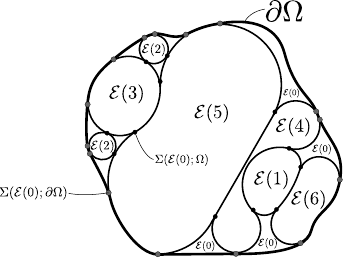}\caption{{\small An example of a possible Cheeger $6$-cluster in dimension $n=2$ suggested by Theorems \ref{mainthm1}, \ref{mainthm2} and \ref{mainthm3}}} \label{ese}
\end{figure}
\begin{remark}
\rm{
Theorems \ref{mainthm1}, \ref{mainthm2} and \ref{mainthm3} allow us to provide examples of planar Cheeger $N$-cluster. The one depicted in Figure \ref{ese} is a possible Cheeger $6$-clusters. Let us highlight that we do not want to suggest that the object in the figure is exactly the Cheeger $6$-cluster of the set $\Om$. We just want to point out the possible structure of such objects.
}
\end{remark}
\begin{remark}\label{remark controfinitezza}
\rm{Let us notice that Assertion $(i)$ of Theorem \ref{mainthm3} could fail when we replace  $\S(\E(0);\Om)$ with $\S(\E(0);\pa \Om)$. Indeed we can always modify $\Om$ at the boundary in order to produce a set $\Om_0$ having the same Cheeger $N$-clusters of $\Om$ and kissing the boundary of some $\pa \E(i)$ in a countable number of points (see Figure \ref{controfinitezza}). 
\begin{figure}
\begin{center}
 \includegraphics[scale=0.8]{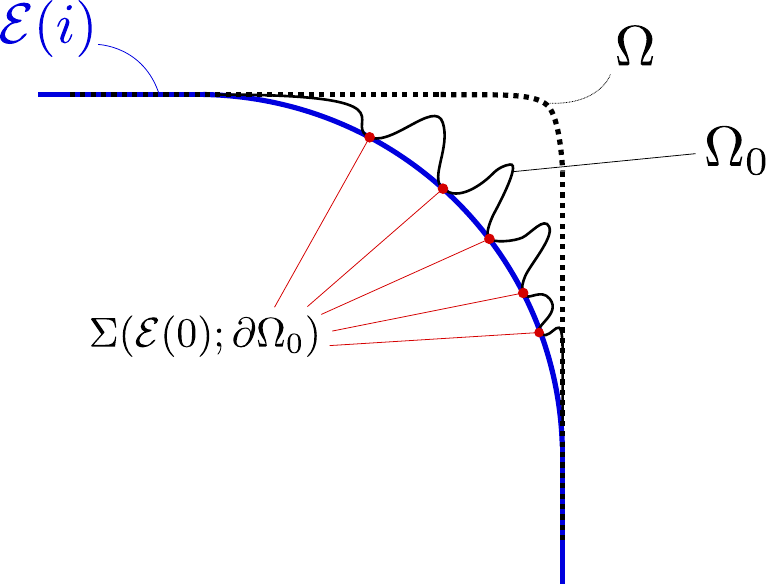}\caption{{\small By gently pushing $\pa \Om$ we can build as many contact points as we want. This proves that Assertion $(i)$ in Theorem \ref{struttura singolaritapiano} does not hold in general for $\S(\E(0);\pa \Om)$.} }\label{controfinitezza}
\end{center}
\end{figure}
}
\end{remark}

\begin{remark}
\rm{We speak of ``curvature of chambers" $\E(j),\E(k)$, instead of curvature of interfaces $\pa\E(j)\cap \pa\E(k)$ in order to point out that the sign of the constant $C_{j,k}$ depends on whether we are looking at $\pa\E(j)\cap \pa\E(k)$ as a piece of the boundary of $\E(j)$ or as a piece of the boundary of $\E(k)$ (namely it depends on the direction of the unit-normal vector to $\pa \E(j)\cap \pa\E(k)$ that we choose).
}
\end{remark}

\begin{remark}
\rm{Note that the set $E_{j,k}$ could be empty. For example, if $\pa\E(j) \cap \pa\E(k)\cap \Om=\{x\}$ consists of a single point, thanks to our characterization (assertion $(ii)$ Theorem \ref{mainthm2}) $x\in \S(\E(0);\Om)$. However, for some $k=0,\ldots, N$, $k\neq j$ it must clearly holds $\H^1 (E_{j,k})>0$. The natural question is whether there exists a chamber $\E(j)$ such that $E_{j,k}=\emptyset$ for all $k\neq 0$. We provide a lemma (Lemma \ref{nel piano!}) that excludes this possibility whenever $\E(j) \cc \Om$ and this will be our starting point for proving assertion $(iii)$ in Theorem \ref{mainthm3}. 
}
\end{remark}

\begin{remark}
\rm{ Since $E_{j,k}$ is relatively open in $\pa \E(j)\cap \pa\E(k)$ we can find an open set $A$ such that $A\cap \pa\E(k)\cap \pa \E(j)=E_{j,k}$ and conclude that $\pa \E(j)$ must have constant mean curvature in $A$ (that is, on $E_{j,k}$). In the sequel we sometimes refer to the distributional mean curvature of $\E(i)$ on $E\subset\pa \E(i)$, a relatively open subset of $\pa\E(i)$,  as the distributional mean curvature of $\E(i)$ inside the open set $A$ such that $A\cap \pa \E(i)=E$. 
}
\end{remark}

\begin{remark}
\rm{Assertions $(i)$ and $(ii)$ in Theorem \ref{mainthm3} tell us that a chamber $\E(j)$ has distributional curvature inside $\Om$ equal to
$$H_{\E(j)}(x)=\sum_{\substack{k=0\\ k,\neq j}}^{N} C_{j,k} \ca_{E_{j,k}}(x), \ \ \ \ \ \text{for $\H^1$-almost every $x\in \pa\E(j)\cap\Om$ }.$$ 
Indeed, since the set $\S(\E(0);\Om)$ is finite, $\pa\E(j)\cap \Om$ is $\H^1$-equivalent to $\bigcup_{k\neq j} E_{j,k}$. In particular, if $T\in C_c^{\infty}(\Om;\R^2)$ then
\begin{align*}
\int_{\pa \E(j)\cap \Om }\dive_{\E(j)}(T)\d \H^{1}(x)&= \sum_{\substack{k=0\\ k,\neq j} }^N \int_{\pa \E(j) \cap E_{j,k}\cap \Om }\dive_{\E(j)}(T)\d \H^{1}(x)\\
&=\sum_{\substack{k=0\\ k,\neq j}}^N \int_{\pa \E(j) \cap E_{j,k}\cap \Om }C_{j,k} (T\cdot \nu_{\E(j)})(x)\d \H^{1}(x)\\
&= \int_{\pa \E(j) \cap \Om } (T \cdot \nu_{\E(j)} )(x) \sum_{\substack{k=0\\ k,\neq j} }^N C_{j,k}\ca_{{E}_{j,k}}(x) \d \H^{1}(x).
\end{align*}
}
\end{remark}

We finally point out that, even if the indecomposability of the chambers is usually an hard task in the tessellation problems, in this case, thanks to a general fact for Cheeger sets (Proposition \ref{nel piano!}), we can easily achieve the proof of Assertion $(iii)$ in Theorem \ref{mainthm3}. This will be particularly useful when focusing our attention on the asymptotic behavior of $H_N$.

\section{Existence and regularity: proof of Theorem \ref{mainthm1}}\label{cpt 4 sct Existence and regularity}

We start by proving the existence and then, separately, we prove the regularity for Cheeger $N$-clusters. We present all the results needed to prove Theorem \ref{mainthm1}.

\begin{theorem}[Existence of Cheeger $N$-clusters.]\label{existence}
Let $\Omega$ be a bounded set with finite perimeter and $0<|\Om|$. For every $N\in \N$ there exists 
a Cheeger $N$-cluster of $\Om$, i.e. an $N$-cluster $\E\subseteq \Om$ such that:
$$H_N(\Om)=\sum_{i=1}^N\frac{P(\E(i))}{|\E(i)|}.$$
Moreover each Cheeger $N$-cluster of $\Om$ has the following properties:
\begin{equation}\label{nontrivialchambers}
|\E(i)|\geq \frac{n^n\om_n}{2^n H_N(\Omega)^n}  \ \ \ \text{for all } \ i=1,\ldots,N,
\end{equation}
\begin{equation}\label{selfcheeger}
 h(\E(i))=\frac{P(\E(i))}{|\E(i)|} \ \ \ \text{for all } \ i=1,\ldots,N.
\end{equation}
\end{theorem}
\begin{proof}
Clearly $H_N(\Omega)<+\infty$ since we can always choose, for example, $B_1,\ldots B_N$ disjoint balls
such that $|B_i\cap\Omega|>0$ and obtain
\begin{equation}\label{notinfinity}
H_N(\Omega)\leq \sum_{i=1}^N \frac{P( B_i \cap \Omega)}{|B_i\cap \Omega|}<+\infty \, .
\end{equation}
Moreover, thanks to the fact that $\Om$ is bounded we deduce $H_N(\Om)>0$. Indeed for every $N$-cluster $\E\subseteq \Om$, the isoperimetric inequality for sets of finite perimeter \eqref{eqn: chapter intro isoperimetric} implies
$$\sum_{i=1}^N\frac{P(\E(i))}{|\E(i)|}\geq nN\left(\frac{\om_n}{|\Om|}\right)^{1/n}$$
hence 
$$H_N(\Om)\geq nN\left(\frac{\om_n}{|\Om|}\right)^{1/n}>0.$$
Consider a minimizing sequence $\E^k=\{\E^k(i)\}_{i=1}^N$ of $N$-clusters such that
$$\lim_{k\rightarrow +\infty} \sum_{i=1}^N \frac{P(\E^k(i))}{|\E^k(i)|}=H_N(\Om).$$
Note that
\begin{align*}
 P(\E^k(i))&\leq|\Omega|\sum_{j=1}^N\frac{P(\E^k(i))}{|\E^k(i)}\leq 2|\Omega|H_N(\Omega).
 \end{align*}
Moreover, by exploiting again the isoperimetric inequality for sets of finite perimeter \eqref{eqn: chapter intro isoperimetric}, we provide the bound
\begin{align*}
n\left(\frac{\om_n}{|\E^k(i)|}\right)^{\frac{1}{n}} &\leq \frac{P(\E^k(i))}{|\E^k(i)|}\leq 2H_N(\Omega)
\end{align*}
and thus
\begin{eqnarray}
\sup_{k}\left\{\max_i \left\{P(\E^k(i))\right\}\right\}&\leq &2|\Omega|H_N(\Omega), \label{a} \\
\inf_{k}\left\{\min_i\left\{|\E^k(i)|\right\}\right\}&\geq & \frac{n^n\om_n}{2^n H_N(\Omega)^n}. \label{b} 
\end{eqnarray}
Thanks to the boundedness of $\Omega$ and to \eqref{a}, we can apply the compactness theorem 
for sets of finite perimeter (Theorem \ref{compactness theorem} in Subsection \ref{subsection Compactness and semicontinuity with respect to the $L^1$ topology}) and deduce that, up to a subsequence, each sequence of chambers $\E^k(i)$ is converging in $L^1$ to some 
$\E(i)\subseteq \Omega$ as $k\rightarrow +\infty$. Equation \eqref{b} implies the lower bound \eqref{nontrivialchambers} while 
the lower semicontinuity of distributional perimeter (Theorem \ref{semicontinuity theorem} in Subsection \ref{subsection Compactness and semicontinuity with respect to the $L^1$ topology}) yields:
\begin{align*}
H_N(\Omega)&\leq \sum_{i=1}^N\frac{P(\E(i))}{|\E(i)|}\leq \sum_{i=1}^N \liminf_{k\rightarrow \infty} \frac{P(\E^k(i))}{|\E^k(i)|}\\
&\leq \liminf_{k\rightarrow+\infty}\sum_{i=1}^N \frac{P(\E(i)^k)}{|\E^k(i)|}=H_N(\Omega).
\end{align*}
Property \eqref{selfcheeger} immediately follows from minimality.
\end{proof}
\begin{remark}\rm{
Thanks to  property \eqref{selfcheeger} $H_N$ can be equivalently defined as 
\begin{equation}\label{N-cheeger constant 2}
H_N(\Om)=\left\{\sum_{i=1}^{N}h(\E(i)) \ \Big{|} \ \E\subseteq \Om \text{ N-Cluster}\right\}.
\end{equation}
}
\end{remark}
We now show that every Cheeger $N$-cluster of a given open set is a $(\La,r_0)$-perimeter-minimizing inside $\Om$ that will implies immediately assertion $(i),(iii),(iv),(v)$ in Theorem \ref{mainthm1} by applying the regularity Theorem \ref{regularity}. \\

Note that, for proving regularity in the case of Cheeger $N$-clusters we have to deal with the possible non trivial components $\pa \E(i)\cap \pa \E(j)$.  Roughly speaking, property \eqref{selfcheeger}, implies that both $\E(i)$ and $\E(j)$ must have mean curvature bounded from above. This leads us to say that the mean curvature of $\E(i)$ ($\E(j)$) on $\pa\E(i)\cap\E(j)$ must be bounded from below as well and so neither outer nor inner cusps can be attained. This approach is based on an idea from \cite{BaMa82}, where the authors prove a regularity result for the solutions of some obstacle problems. 
\begin{theorem}\label{regolare}
Let $\Om$ be an open bounded set and $\E$ be a Cheeger $N$-cluster of $\Om$. Then there exists $\Lambda,r_0>0$ depending on $\E$ 
with $\Lambda r_0\leq 1$, such that each $\E(i)$ is a $(\Lambda,r_0)$-perimeter-minimizing in $\Om$. As a 
consequence, for every $i=1,\ldots,N$ the set $\Om\cap\pa \E(i)$ has the regularity of Theorem \ref{regularity}.
\end{theorem}

Before entering in the details of the proof of Theorem \ref{regolare} let us remark that, by exploiting formula \eqref{eqn: perimetro differenza} contained in Subsection \ref{sbst Union, intersection, differences}, it is possible to derive the inequality
\begin{equation}\label{diffe}
P(F\setminus E; A )+P(E\setminus F;A)\leq P(F;A)+P(E;A)
\end{equation}
holding for every couple of sets $E,F$ of locally finite perimeter and for open set $A$. \\

In order to prove Theorem \ref{regolare} we also recall the following definition. We say that a set of finite perimeter $M$ has \textit{distributional mean curvature less than $g\in L^1_{loc}(\Om)$ in $\Om$} if, there exists $r_0$ such that for every $B_r\cc \Om$ with $r<r_0$ and for every $L\subseteq M$ with $M\setminus L \cc B_r$, it holds
\begin{equation}
P(M;B_r)\leq P(L;B_r)+\int_{M\setminus L} g(x)\d x.
\end{equation}
\begin{proof}[Proof of theorem \ref{regolare}]
We start by fixing $i\in \{1,\ldots,N\}$ and by defining
$$M_i=\bigcup_{\substack{j=1, \\j\neq i }}^N\E(j).$$ 
We divide the proof in two steps.\\

\textit{Step one.} We prove that each $M_i$ has distributional mean curvature less than $H_N(\Om)$ in $\Om$. Let $B_r\subset\subset \Om$ be a ball and $L\subseteq M_i$ be a subset of finite perimeter of $M_i$ with $M_i\setminus L \subset\subset B_r$. What we need to prove is
\begin{equation}\label{limitiamolacurvatura}
P(M_i;B_r)\leq P(L;B_r)+H_N(\Om)|M_i\setminus L|. 
\end{equation}
Note that, up to choosing $r<r_0=\frac{n}{4H_N(\Om)}$ we can always assume $|\E(j)\cap L|>0$ for every $j\neq i$. Indeed $M_i\setminus L\cc B_r$ and, if by contradiction we assume $|\E(j)\cap L|=0$ for some $j\neq i$, this would mean $\E(j)\subset B_r$ up to a set of measure $0$ which implies (because of property \eqref{nontrivialchambers} and thanks to the choice of $r_0$) :
$$\frac{n^n\om_n}{2^nH_N(\Om)^n}<|\E(j)|<\om_n r^n<\frac{n^n\om_n}{4^nH_N(\Om)^n}$$
that is impossible.\\

By minimality it must hold:
\[
\frac{P(\E(j))}{|\E(j)|}\leq \frac{P(\E(j)\cap L)}{|\E(j)\cap L|} \ \ \ \ \text{ for every $j\neq i$,}
\]
that leads to:
\begin{align}
\frac{P(\E(j);B_r)+P(\E(j);B_r^c)}{|\E(j)|}&\leq \frac{P(\E(j)\cap L;B_r)+P(\E(j);B_r^c)}{|\E(j)|-|\E(j)\setminus L|}\nonumber\\
 P(\E(j);B_r)&\leq P(\E(j)\cap L;B_r)+|\E(j)\setminus L|h(\E(j)).\label{battle for falluja}
\end{align}
By exploiting \eqref{peri N-cluster} in Lemma \ref{tecnico} and \eqref{battle for falluja} above we obtain
\begin{align}
P(M_i;B_r)&=P\left(\cup_{j\neq i}\E(j);B_r\right)\nonumber\\
_{\text{\eqref{peri N-cluster} in Lemma \ref{tecnico}}}&=\sum_{j\neq i}P(\E(j);B_r)-\sum_{\substack{k,j\neq i,\ k\neq j}}\H^{n-1}(\partial^*\E(j)\cap \partial^* \E(k)\cap B_r)\nonumber\\
_{\eqref{battle for falluja}}&\leq \sum_{j\neq i}P(\E(j)\cap L;B_r) + |\E(j)\setminus L|h(\E(j))\nonumber\\
&-\sum_{\substack{k,j\neq i,\ k\neq j}}\H^{n-1}(\partial^*\E(j)\cap \partial^* \E(k)\cap B_r)\nonumber\\ 
&\leq \sum_{j\neq i}P(\E(j)\cap L;B_r) -\sum_{\substack{k,j\neq i,\ k\neq j}}\H^{n-1}(\partial^*\E(j)\cap \partial^* \E(k)\cap B_r)\nonumber\\
&+ H_N(\Om) |M_i\setminus L|  \label{il cacciatore},
\end{align} 
where in the last inequality we have used  the formulation of $H_N$ as in \eqref{N-cheeger constant 2}. By exploiting again Lemma \ref{tecnico} for $\{\E(j)\cap L\}_{j\neq i}$ we obtain
\begin{align}
P(M_i\cap L ; B_r)&=\sum_{j\neq i}P(\E(j)\cap L;B_r)\nonumber\\
&-\sum_{k,j\neq i \ \ k\neq j}\H^{n-1}(\pared (\E(j)\cap L) \cap \pared(\E(k)\cap L) \cap B_r).\label{full metal jacket}
\end{align}
After some quick computations, by exploiting formula \eqref{eqn: frontiera ridotta intersezione} for the reduced boundary of the intersections and the fact that the chambers $\E(j)$ are disjoint (up to a set of zero Lebesgue measure), we discover that
\[
\pared (\E(j)\cap L) \cap \pared(\E(k)\cap L) \cap B_r \approx L^{(1)}\cap \pared\E(k)\cap\pared \E(j) \cap B_r
\]
which plugged into \eqref{full metal jacket} leads to
\begin{equation}\label{platoon}
\sum_{j\neq i}P(\E(j)\cap L;B_r)-\sum_{k,j\neq i \ \ k\neq j}\H^{n-1}(\pared\E(k)\cap\pared \E(j) \cap B_r)\leq P(L; B_r),
\end{equation}
where we have exploited also $[(M_i\cap L)\Delta L] \cap B_r =\emptyset$. By combining \eqref{platoon} with \eqref{il cacciatore} we reach
\begin{align*}
P(M_i;B_r)&\leq P(L; B_r)+H_N(\Om) |M_i\setminus L|,
\end{align*} 
and we achieve the proof of Step one.\\

\textit{Step two.} We now prove that $\E(i)$ is a $(\Lambda,r_0)$-perimeter-minimizing for a suitable choice of $\Lambda$ 
and $r_0<\frac{n}{4H_N(\Om)}$ (according to Step one). Let $B_r\subset\subset \Om$ and $F$ be such that $F\Delta\E(i)\subset\subset B_r$. Define $E:=F\setminus M_i$ and observe, 
by minimality of $\E$ and by the relation $\E(i)\cap B_r^c=(F\setminus M_i)\cap B_r^c$, that:
\begin{align*}
\frac{P(\E(i))}{|\E(i)|} &\leq \frac{P(E)}{|E|}.\\
\end{align*}
Hence
\begin{align*}
\frac{P(\E(i);B_r)+P(\E(i);B_r^c)}{|\E(i)|} &\leq\frac{P(F\setminus M_i;B_r)+P(F\setminus M_i;B_r^c)}{|F|-|F \cap M_i|},\\
 &\leq\frac{P(F\setminus M_i;B_r)+P(\E(i);B_r^c)}{|\E(i)|+(|F\cap B_r|-|\E(i)\cap B_r|)-|F\cap M_i|}.
\end{align*}
If we expand the last inequality we get:
\begin{align*}
P(\E(i);B_r)|\E(i)|&\leq P(F\setminus M_i;B_r)|\E(i)|+P(\E(i))(|F\cap M_i|+|\E(i)\cap B_r|-|F\cap B_r|),
\end{align*}
which means (by observing that $F\cap M_i\subseteq F\setminus \E(i)$),
\begin{eqnarray}\label{quasifinita}
P(\E(i);B_r)&\leq& P(F\setminus M_i;B_r)+2h(\E(i))|\E(i)\Delta F|\, .
\end{eqnarray}
By making use of \eqref{diffe} we obtain
\begin{equation}\label{quasifinitachiave}
P(F\setminus M_i;B_r) \leq P(F;B_r)+P(M_i;B_r)-P(M_i\setminus F;B_r)
\end{equation}
Since $M_i\setminus F\subset M_i$ and $(M_i\setminus F)\Delta M_i\subset\subset B_r$ we can use step one (relation \eqref{limitiamolacurvatura}) with $L=M_i\setminus F$ for conclude that
\begin{eqnarray*}
P(M_i;B_r)&\leq & P(M_i\setminus F;B_r)+H_N(\Om)|M_i\setminus (M_i\setminus F)|\\
 &\leq & P(M_i\setminus F;B_r)+H_N(\Om)|M_i\cap F|.
 \end{eqnarray*}
 Hence
 \begin{equation}\label{e3}
 P(M_i;B_r)-P(M_i\setminus F;B_r)\leq  H_N(\Om)|F\setminus \E(i)|.
 \end{equation}
By plugging \eqref{e3} in \eqref{quasifinitachiave} we obtain
\begin{equation}\label{quasifinitachiave2}
P(F\setminus M_i;B_r)\leq P(F;B_r)+H_N(\Om)|\E(i)\Delta F|
\end{equation}
and by using \eqref{quasifinitachiave2} in \eqref{quasifinita} we find
$$P(\E(i);B_r)\leq P(F;B_r)+3H_N(\Om)|\E(i)\Delta F|.$$
By choosing $\Lambda=3H_N(\Om)$ and $r_0=\frac{1}{4 H_N(\Om)}$ we conclude that each $\E(i)$ is a 
$(\Lambda,r_0)$-perimeter-minimizing with $\Lambda r_0<1$ and we achieve the proof.
\end{proof}

Proof of assertion $(ii)$ can be viewed as a consequence of  \cite[Proposition 2.5, Assertion (vii)]{LP14} recalled below for the sake of clarity.
\begin{proposition}\label{leo}
Let $A$ be an open and bounded set and let $E$ be a Cheeger set $A$. Then
$$\pa^*A \cap \pa E \subseteq \partial^*E.$$
Moreover for every $x\in \pa^*A \cap \pa E$ 
it holds that 
$$\nu_{E}(x)=\nu_{A}(x),$$
where $\nu_{E}$, $\nu_{A}$ denotes the measure theoretic outer unit normal to $E$ and $A$ respectively.
\end{proposition}
The proof follows by combining the fact that each Cheeger set $E$ is a $(\Lambda,r_0)-$perimeter-minimizing in $A$ with the fact that the blow-ups of $\pa A$ at a point $x\in \pa^* A$ converge to an half plane.
\begin{proof}[Proof of Theorem \ref{mainthm1}]
Assertion $(i)$, $(iii),(iv), (v)$ follow by combining Theorems \ref{regolare} and \ref{regularity}. Assertion $(ii)$ is obtained by noticing that each chambers $\E(i)$ is a Cheeger set for 
	\[
	A=\Om \setminus \bigcup_{\substack{j=1,\\ j\neq i }}^N \E(j)
	\]
and then by applying Proposition \ref{leo}.
\end{proof}

\section{The singular set $\S(\E)$ of Cheeger $N$-clusters in low dimension}\label{cpt 4 sct The singular set of the Cheeger N-clusters in low dimension}
The following results are all stated and proved for open bounded and connected sets $\Om\subset \R^{n}$ having $C^{1}$ boundary and with the ambient space dimension less than $8$. We ask $\Om$ to be connected and with $C^1$ boundary because this is enough to avoid degenerate situations where $|\E(0)|=0$ (see Remark \ref{controesempioconne} where a Cheeger $N$-cluster with $|\E(0)|=0$ is provided).\\ 

We obtain the proof of Theorem \ref{mainthm2} by combining different results, sated and proved separately in Subsection \ref{cpt 4 sbsct Proof of mainthm2}. We premise some technical lemmas.

\subsection{Technical lemmas}

\begin{lemma}\label{complete regularity of cheeger N-clusters}
If  $n\leq 7$, $\Om$ is an open, bounded, connected sets with $C^1$ boundary and finite perimeter and $\E$ is a Cheeger $N$-cluster of $\Om$ it holds 
\[
\pared \E(i)=\pa \E(i) \ \ \ \ \ \text{for all $i\neq 0$}.
\]
 \end{lemma}
 \begin{proof}
 We decompose $\pa \E(i)$ as
 \[
 \pa\E(i)=(\pa\E(i)\cap \Om)\cup (\pa\E(i)\cap \pa\Om).
 \]
 and we note that
 \[
 \pa\E(i)\cap \Om =\pared \E(i)\cap \Om,
 \]
 because of Assertion (iii) of \ref{mainthm1}. Moreover, since $\Om$ has $C^1$ boundary $\pared \Om=\pa \Om$ and thus, thanks to Assertion (ii) we have also
 \[
 \pa \E(i)\cap \pa \Om = \pa \E(i)\cap \pared \Om \subseteq \pared \E(i).
 \]
 Hence
 \[
\pared\E(i)\subseteq \pa \E(i)=(\pa\E(i)\cap \Om)\cup (\pa \E(i)\cap \pa \Om)\subseteq \pared \E(i),
 \]
 and we achieve the proof.
 \end{proof}
 \begin{remark}\label{closed}
If $n\leq 7$, $\Om$ is an open, bounded, connected set with finite perimeter and $C^1$-boundary and $\E$ is a Cheeger $N$-cluster of $\Om$, by considering $\F(k)=\E(k)\cup \pa \E(k)$ for $k\neq 0$, thanks to Lemma \ref{complete regularity of cheeger N-clusters} we must have $|\F(k)\Delta \E(k)|\leq |\pa \E(k)|=0$ and thus
	\[
	P(\F(k))=P(\E(k)).
	\]
For this reason in the sequel we are always assuming that each chamber $\E(k)$ for $k\neq 0$ \textbf{is a closed set with $\pared \E(k)=\pa \E(k)$}.
 \end{remark}
\begin{lemma}\label{densita dei punti di omega}
Let $n\leq 7$, $N\geq 2$ and $\Om$ be an open, bounded and connected set with $C^{1}$ boundary and finite perimeter in $\R^{n}$. If $\E$ is a Cheeger $N$-cluster of $\Om$, then for every $x\in \R^{n}$ and every $k=1,\ldots,N$ there exists the $n$-dimensional density $\vt_n(x,\E(k))$ and it takes values:
\[
\vartheta_n(x,\E(k))=
\begin{sistema}
0 \ \ \text{if $x\notin \E(k)$};\\ 
\frac{1}{2}\ \  \text{if $x\in \pa\E(k)$};\\
1 \ \ \text{if $x\in \mathring{\E(k)}$}.
\end{sistema}
\]
\end{lemma}
\begin{proof} 
Each chamber $\E(k)$ for $k\neq 0$ is a closed set (see Remark \ref{closed}) and thus 
	\begin{align*}
		\E(k)^c&=\E(k)^{\zero},\\
		\mathring{\E(k)}&=\E(k)^{\uno}.
	\end{align*}
Lemma \ref{complete regularity of cheeger N-clusters} implies 
	\[
		\pa \E(k)=\pared \E(k)\subseteq \E(k)^{\mez}\subseteq \pa \E(k).
	\] 
\end{proof}

\subsection{Proof of theorem \ref{mainthm2}}\label{cpt 4 sbsct Proof of mainthm2}

We are now ready to prove two propositions that immediately imply Theorem \ref{mainthm2}. The following Proposition is needed in order to prove Assertion $(i)$ in Theorem \ref{mainthm2}.
\begin{proposition}\label{ognicameraconfinaconilvuoto}
Let $n\leq 7$, $N\geq 2$ and $\Om$ be an open, bounded and connected set with $C^{1}$ boundary and finite perimeter in $\R^{n}$. Let $\E$ be a Cheeger $N$-cluster of $\Om$. 
Then for every $i\in \{1,\ldots,N\}$ there exists $x\in \pa\E(i)$ such that $|B_s(x)\cap \E(0)|>0$ for all $s>0$.
\end{proposition}
\begin{proof}
Without loss of generality (and for the sake of clarity) we can assume $i=1$. We note that the proof of the lemma is a consequence of the following claim.\\

\textit{Claim.} $\pa \E(1)\setminus \left[\pa\Om \cup \bigcup_{k=2}^N \pa \E(k)\right]\neq \emptyset$.\\

Indeed, if the claim is in force then there exists $x\in \pa\E(1)\cap \Om$ and $x\notin \E(k)$ for all $k\neq 1$. Since the chambers are closed we can also find a small ball $B_s(x)\cc\Om$ such that $B_s(x)\cap \E(k)=\emptyset$ for all $k\neq 1$, implying (thanks to Lemma \ref{densita dei punti di omega})
\begin{align*}
|\E(0)\cap B_s(x)|&=|B_s|-\sum_{k=1}^n|\E(i)\cap B_s|=|B_s|-|\E(1)\cap B_s|>0
\end{align*}
(because $x\in \pa \E(1)=\E(1)^{\mez}$) and achieving the proof.\\

Let us focus on the proof of the claim. Thanks to the connectedness of $\Om$ it is easy to show that $\pa \E(1)\setminus \pa\Om \neq \emptyset$. If also $\pa \E(1)\cap \pa \E(k)=\emptyset$ for $k\neq 1$ the claim trivially holds. Otherwise it must exist at least an index $j\in\{2,\ldots,N\}$ such that $\pa\E(1)\cap \pa\E(j)\neq \emptyset.$ Assume without loss of generality $j=2$:
	\[
	\pa\E(1)\cap \pa\E(2)\neq \emptyset.
	\]
Choose $x\in \pa\E(1)\cap \pa\E(2)$ and let us denote by $M$ the connected component of $\pa \E(1)$ containing $x$. Note that $x\notin \pa \Om$. Otherwise we would have $x\in \pa\E(1)\cap \pa\E(2)\cap \pa\Om$ and because of the regularity of $\Om$ and thanks to Lemma \ref{densita dei punti di omega} this leads to a contradiction:
\begin{align*}
\frac{1}{2}&=\lim_{r\rightarrow 0^+} \frac{|\Om\cap B_r(x)|}{|B_r(x)|}=\lim_{r\rightarrow 0^+} \sum_{h=0}^N \frac{|\E(h)\cap B_r(x)|}{|B_r(x)|}\\
&\geq \lim_{r\rightarrow 0^+} \frac{|\E(1)\cap B_r(x)|}{|B_r(x)|}+\frac{|\E(2)\cap B_r(x)|}{|B_r(x)|}=1.
\end{align*}
Hence the following are in force:
\begin{equation}\label{leader maximo}
M\setminus \pa \Om\neq \emptyset, \ \ \ \ \ M\cap\pa\E(2)\neq \emptyset.
\end{equation}
We now note that, if 
\begin{equation}\label{non ne posso piu di questo lemma}
M\setminus \left[\pa \Om\cup \bigcup_{k=2}^N\pa \E(k)\right]=\emptyset
\end{equation} then, necessarily $M \subseteq\pa \E(2)\cap \Om$. Indeed considered 
	\[
	y\in \overline{M\cap \pa \E(2)}\cap \overline{(M\setminus \pa\E(2) )}=\bd_M(M\cap \pa \E(2)),
	\]
since \eqref{non ne posso piu di questo lemma} is in force (and since $y\in \bd_M(M\cap \pa \E(2))$) either there exists an index $k\neq 1,2$ such that $y\in \pa\E(k)$ or $y\in \pa \Om$. In both cases we reach a contradiction because $y$ would be a point of density $\frac{1}{2}$ for three disjoint sets ($\E(1),\E(2),\E(k)$ or $\E(1),\E(2),\Om^c $). Thus the only possibility is that $\bd_M(M\cap\pa\E(2))=\emptyset$ and since \eqref{leader maximo} is in force, by applying the following (topological) fact \eqref{topological fact} we conclude that $M=M\cap \pa\E(2)\subseteq\pa\E(2)$.
\begin{equation}\label{topological fact}
\begin{array}{c}
\text{If $M\subset \R^n$ is a closed connected set and $C\subseteq M$ is a non empty}\\
\text{subset of $M$, then $\bd_M(C):=\overline{C}\cap \overline{(M\setminus C)}=\emptyset$ if and only if $M=C.$}
\end{array}
\end{equation} 
As before $M\cap \pa\E(2)\cap \pa \Om=\emptyset$ otherwise we would have a point of density $\frac{1}{2}$ for three disjoint set $(\E(1),\E(2),\Om^c)$ and hence $M\subseteq \pa\E(2)\cap \Om$. This means that $M$ must be a closed $C^{1,\a}$ surface without boundary contained in $\Om$ and disjoint from the other sets $\E(k)$ and from $\pa \Om$, which means that one of the situation of Figure \ref{movimenti} has to be in force.
\begin{figure}
\centering
 \includegraphics[scale=0.8]{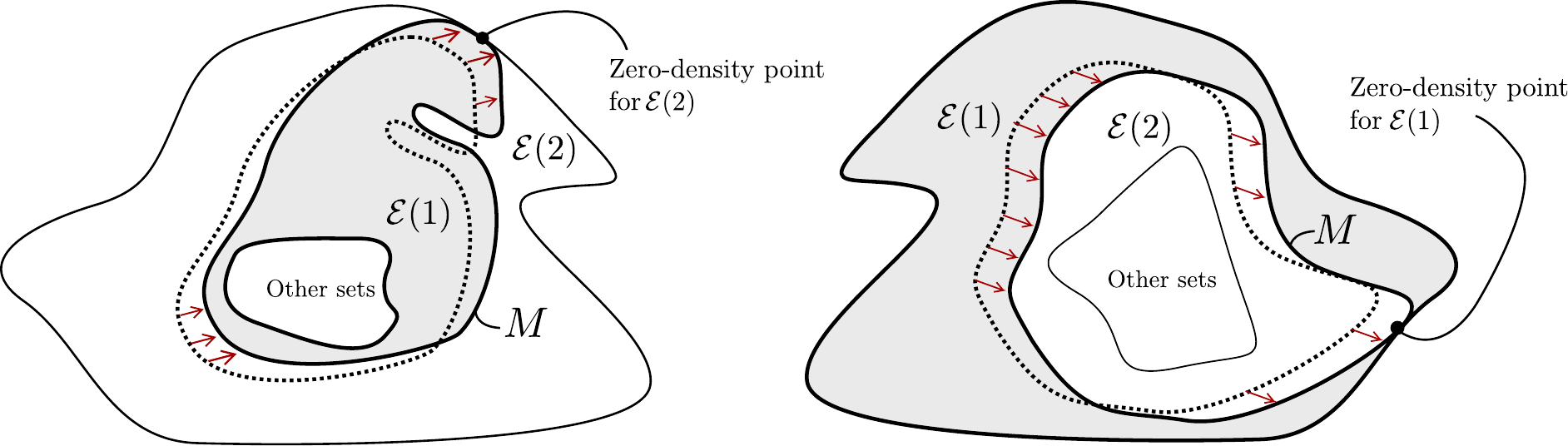}\caption{{\small If \eqref{non ne posso piu di questo lemma} holds, then one of these two situations must be in force and we can contradict regularity by simply translate $M$ until it kisses another part of the boundary yielding a not allowed point of density zero.}}\label{movimenti}
\end{figure}
We are thus able to move a little bit $M$, and whatever is bounded by $M$, inside $\Om$ as in Figure \ref{movimenti} until it kisses $\pa\E(2)$ or $\pa\E(1)$  (we easily exclude that $M$ bounds a hole of $\Om$ with a slight variation of this previous argument). In this way we produce a zero-density point for $\E(1)$ or for $\E(2)$ without changing $\sum_j\frac{P(\E(j))}{|\E(j)|}$ and this contradicts the regularity.\\

Hence \eqref{non ne posso piu di questo lemma} cannot holds and the claim is true.

\end{proof}

The next Proposition implies Assertion $(ii)$ in Theorem \ref{mainthm2}.
\begin{proposition}\label{prima caratterizazione}
Let $n\leq 7$, $N\geq 2$ and $\Om$ be an open, bounded and connected set with $C^{1}$ boundary and finite perimeter in $\R^{n}$. Let $\E$ be a Cheeger $N$-cluster of $\Om$. Then
\begin{align*}
\S(\E;\Om)&=\S(\E(0);\Om)\\
&=\{x\in \pa\E(0)\cap \Om \ | \ \vt_n(x,\E(0))=0\}\\
&=\Om\cap \bigcup_{\substack{i,j=1,\\ i\neq j}}^N \pa \E(i)\cap \pa\E(j)\cap \pa \E(0).
\end{align*}
\end{proposition}
\begin{proof}
Thanks to Proposition \ref{ognicameraconfinaconilvuoto} the set $\E(0)$ is not empty. As pointed out in Remark \ref{the singular set}, thanks to the regularity of each chambers, it is immediate that $\S(\E;\Om)=\S(\E(0);\Om)$. Let us denote (for the sake of brevity) by
	\begin{align*}
	\S_0&=\S(\E(0);\Om),\\
	A&=\{x\in \pa\E(0)\cap \Om \ | \ \vt_n(x,\E(0))=0 \}\\
	B&=\Om \cap \bigcup_{\substack{i,j=1,\\ k\neq j}}^N \pa \E(i)\cap \pa\E(j)\cap \pa \E(0).
	\end{align*}
We note that $B\subseteq A$ is immediate and also $A\subseteq \S_0$ is immediate, since if $x\in A$ then $x\notin \E(0)^{\mez}\supseteq \pared \E(0)$. We are left to show that $\S_0\subseteq B$. In order to do this we define the following family of subsets of $\Om$.
\begin{align}
E_i&:=\displaystyle \mathring{\E(i)} \ \ \ &\text{for all $0 \leq i\leq N$}\label{Partizione di omega in insiemi furbi}\\
F_{i,j}&:=\displaystyle \Om\cap \pa \E(i)\cap \pa\E(j) \setminus \left(\bigcup_{\substack{k=0, \\ k\neq i,j}}^N \pa\E(k)\right), \ \ \ &\text{for all $0\leq i<j \leq N$}\label{Partizione di omega in insiemi furbi1}\\
G_{i,j}&:=\displaystyle  \Om\cap \pa \E(i)\cap \pa\E(j) \cap \pa\E(0), \ \ \ &\text{for all $1\leq i<j \leq N$}.\label{Partizione di omega in insiemi furbi2}
\end{align}
It is easy to verify that the Borel sets defined in \eqref{Partizione di omega in insiemi furbi},\eqref{Partizione di omega in insiemi furbi1},\eqref{Partizione di omega in insiemi furbi2} form a partition of $\Om$. Now, for a given point $x\in \S_0$, clearly $x\notin E_i$ for all $i=0,\ldots,N$. Thus either $x\in F_{i,j}$  for some $0 \leq i < j\leq N$ or $x\in G_{i,j}$ for some $1 \leq i < j\leq N$. If $x\in F_{i,j}$, by closedness  there exists a small ball $B_s(x)$ such that $\pa \E(k)\cap B_s(x)=\emptyset$ for all $k\neq i,j$. This implies that either $i=0$ or $j=0$ (since we have chosen $x\in \S_0\subset \pa\E(0)$) and that $\pa \E(0)\cap B_s(x)=\pa\E(j)\cap B_s(x)$ leading to say that $\pa\E(0)$ must be regular in a small neighborhood of $x$ and contradicting $x\in \S_0$. Hence necessarily $x\in G_{i,j}$ for some  $1 \leq i < j\leq N$ and we achieve the proof: $\S_0\subseteq B$. 
\end{proof}
The following corollary is an easy consequence of Proposition \ref{prima caratterizazione}.
\begin{corollary}\label{egli e chiuso}
Let $n\leq 7$, $N\geq 2$ and $\Om$ be an open, bounded and connected set with $C^{1}$ boundary and finite perimeter in $\R^{n}$. If $\E$ is a Cheeger $N$-cluster for $\Om$, then $\S(\E(0);\Om)$ is closed.
\end{corollary}
\begin{proof}
Thanks to Proposition \ref{prima caratterizazione} we have that
	\begin{equation}\label{putin}
	\S(\E;\Om)=\Om\cap \bigcup_{\substack{i,j=1,\\ i\neq j}}^N \pa \E(i)\cap \pa\E(j)\cap \pa \E(0).
	\end{equation}
Let $\{x_k\}_{k\in \N}\subseteq \S(\E;\Om)$ such that $x_k\rightarrow x$. Up to extract a subsequence we have that $\{x_k\}_{k\in \N}\subset  \Om\cap \pa \E(i)\cap \pa\E(j)\cap \pa \E(0)$ for some $1\leq i<j\leq N$ (since \eqref{putin} is in force). By closedness we obtain $x\in \pa\E(i)\cap \pa\E(j)\cap \pa\E(0)$ and we need to prove that $x\in \Om$. If $x\in \pa\Om$ we have $x\in \pa\E(i)\cap \pa\E(j)\cap \pa\Om$ which is a contradiction since $x$ would be a point of density $\frac{1}{2}$ for three disjoint sets $\E(1),\E(2),\Om^c$). Hence $x\in \Om$ and thus $x\in \S(\E(0);\Om)$.
\end{proof}
The proof of Theorem \ref{mainthm2} is now obtained as an easy consequence of the previous results.
\begin{proof}[Proof of Theorem \ref{mainthm2}]
Follows by Propositions \ref{ognicameraconfinaconilvuoto}, \ref{prima caratterizazione} and by Corollary \ref{egli e chiuso}.
\end{proof}
\section{The planar case} \label{cpt 4 the planar case}
As in the previous sections, the proof of Theorem \ref{mainthm3} is attained by combining different results that we state and prove in Subsection \ref{cpt 4 sbsct proof of mainthm3}. We premise some technical lemmas.
\subsection{Technical lemmas}

\begin{lemma}\label{buone componenti connesse}
Let $n\leq 7$, $N\geq 2$ and $\Om$ be an open, bounded and connected set with $C^{1}$ boundary and finite perimeter in $\R^{n}$. Let $\E$ be a Cheeger $N$-cluster for $\Om$. If $E$ is an indecomposable component of $\E(0)$ such that $E\cc \Om$, then there exist at least three different indexes $i,j,k\neq 0$ such that $\pa E \cap \E(i)\neq \emptyset$, $\pa E\cap \E(j)\neq \emptyset$, and $\pa E\cap \E(k)\neq \emptyset$. In particular, $E$ shares boundary at least with three different chambers.
\end{lemma}
\begin{proof}
Let $E$ be a generic indecomposable component of $\E(0)$. Assume that $E$ shares its boundary with exactly 
two other different chambers $j,k\geq 1$ and $\partial E\cap \pa \Om=\emptyset$. Then either
$$a) \ \ \H^{n-1}(\pa^*E\cap\pa\E(j))\geq \H^{n-1}(\pa^*E\cap\pa\E(k)),$$
or
$$b) \ \ \H^{n-1}(\pa^*E\cap\pa\E(k))\geq \H^{n-1}(\pa^*E\cap\pa\E(j))$$
hold. Assume that $a)$ holds and define $\E_1(j):=\E(j)\cup E$, $\E_1(i)=\E(i)$ for $i\neq j$. Since 
$$P(\E_1(j))=P(\E(j))-\H^{n-1}(\pa^*E\cap\pa\E(j))+\H^{n-1}(\pa^*E\cap\pa\E(k))$$
we obtain:
\begin{eqnarray*}
H_N(\Om)&\leq& \sum_{i=1}^N\frac{P(\E_1(i))}{|\E_1(i)|}\\
&=&\frac{P(\E_1(j))}{|\E_1(j)|}+\sum_{i\neq j}\frac{P(\E(i))}{|\E(i)|}\\
&=&\frac{P(\E(j))-\H^{n-1}(\pa^*E\cap\pa\E(j))+\H^{n-1}(\pa^*E\cap\pa\E(k))}{|\E(j)|+|E|}+\sum_{i\neq j}\frac{P(\E(i))}{|\E(i)|}\\
&\leq&\frac{P(\E(j))}{|\E(j)|+|E|}+\sum_{i\neq j}\frac{P(\E(i))}{|\E(i)|}.
\end{eqnarray*}
If $|E|>0$ we are led to $H_N(\Om)<H_N(\Om)$ which is a contradiction, so $|E|=0$. Since $E$ is open (because $\E(0)$ is open), then $E=\emptyset$. If $E$ shares its boundary with exactly one chamber we argue in the same way. We have discovered that every decomposable component of $\E(0)$ that shares boundary with exactly one or two chambers is empty. That complete the proof.
\end{proof}

\begin{lemma}\label{nel piano!}
Let $E$ be a Cheeger set of an open bounded set $A\subset \R^{2}$. Assume that the following properties hold for $E$:
\begin{itemize}
\item[1)] $\#(\bd_{\pa A}(\pa A\cap\pa E))<+\infty$,
\item[2)] every $x\in \pa A\cap \pa E$ is a regular point for $\pa A$, namely $x\in \pa A\cap \pa E \subseteq \pared A$ ;
\end{itemize}
where
	\[
	\bd_{\pa A}(\pa A\cap\pa E)=[\pa A \cap\pa E]\cap\overline{[\pa A\setminus\pa E]}.
	\]
Then $\H^1(\pa E \cap \pa A)>0$.  
\end{lemma}

\begin{remark}
\rm{
It seems that it is possible to generalize Lemma \ref{nel piano!} to dimension $n\geq 2$ by making use of Alexandrov's Theorem \cite{aleksandrov1962uniqueness} for the characterization of the Constant Mean Curvature (CMC) embedded hypersurface in $\R^n$. In this (more technical) framework hypothesis 1) can be weakened. Anyway, since we do not have to deal (at least here) with $n\geq 2$ and since for our purposes Lemma \ref{nel piano!} is all we need to complete the proof of Theorem \ref{mainthm3} we decide to not add this generalization.
}
\end{remark}

\begin{proof}[Proof of Lemma \ref{nel piano!}]
Assume by contradiction that $\H^1(\pa E \cap \pa A)=0$. In this case 	
	\[
	\bd_{\pa A}(\pa A\cap\pa E)=\pa E \cap \pa A.
	\] 
Let $F$ be an indecomposable component of $E$ and note that, since $E$ is a Cheeger set for $A$ it must hold
	\begin{equation}\label{e pure lui cheeger}
	\frac{P(F)}{|F|}=h(A).
	\end{equation}
Set 
	\[
	M=\bd_{\pa A}(\pa A\cap\pa F)
	\]
and $\#(M)=k<+\infty$. The well-known regularity theory for Cheeger sets, combined with the fact that $k<+\infty$ tells us that 
 	\[
 	\pa F\cap A=\bigcup_{i=1}^k \a_i
 	\] 
where each $\a_i$ is a piece of the boundary of a suitable ball $B_i$ (relatively open inside $\pa B_i$) of radius $\frac{1}{h(A)}$. The finiteness of $M$ implies that for a suitably small $r$ it holds $B_r(x)\cap M=\{x\}$ for all $x\in M$ and this means that for every $x\in M$ there exists two arcs $\a_i,\a_j$ (with possibly $i=j$) such that $x\in \overline{\a_i}\cap \overline{\a_j}$. \\
Let $B_i, B_j$ the balls from which such arcs come from: $\a_i\in \pa B_i$, $\a_j \in \pa B_j$. Hypothesis $2)$ implies that the outer unit normal to $B_i$ and to $B_j$ at $x$ must coincide with $\nu_A(x)$ and hence the balls $B_i$ and $B_j$ must coincide as well. Since $k<+\infty$, by iterating this argument we conclude that there exists only one ball $B$ of radius $\frac{1}{h(A)}$ such that $M\subset \pa B$ and $\pa F\cap A= \pa B\cap A$. In particular $\pa F$ is equal to $\pa B$ and by exploiting \eqref{e pure lui cheeger} we bump into a contradiction
\begin{eqnarray*}
\frac{P(F)}{|F|}&=&h(A)\\
\frac{P(B)}{|B|}&=&h(A)\\
\frac{\frac{2\pi}{h(A)}}{\frac{\pi}{h(A)^2}}&=&h(A)\\
\frac{\frac{2}{h(A)}}{\frac{1}{h(A)^2}}&=&h(A)\\
2h(A)&=&h(A).
\end{eqnarray*}
The contradiction comes from the fact that we have assumed $\H^1(\pa E \cap \pa A)=0$, hence $\H^1(\pa E \cap \pa A)>0$ and the proof is complete.\\
\end{proof}

\subsection{Proof of Theorem \ref{mainthm3}} \label{cpt 4 sbsct proof of mainthm3}
\begin{proposition}\label{struttura singolaritapiano}
Let $N\geq 2$ and $\Om$ be an open, bounded and connected set with $C^{1}$ boundary and finite perimeter in the plane.  Let $\E$ be a Cheeger $N$-cluster of $\Om$. Then $\S(\E;\Om)=\S(\E(0);\Om)$ is a finite union of points.
\end{proposition}

\begin{proof}
We prove that $\S(\E(0);\Om)$ has no accumulation point. In this way we show that $\S(\E(0);\Om)$ is a closed (thanks to Theorem \ref{mainthm2}), bounded set of $\R^2$ (since $\Om$ is bounded) without accumulation points which means that $\S(\E(0);\Om)$ must be a finite union of points.\\

Set $\S_0=\S(\E(0);\Om)$ for the sake of brevity. Let $\xi\in \pa\E(i)\cap \pa\E(j)$ for some $1\leq i<j\leq N$ that without loss of generality we assume to be $i=1,j=2$. We can assume (up to a translation) also that $\xi=(0,0)$. Since $\pa\E(1),\pa\E(2)$ are regular up to a rotation we can find a small closed cube 
	\[
	Q_{\e}:=[-\e,\e]\times[-\e,\e]\cc\Omega
	\]
centered at $\xi=(0,0)$ and two $C^1$ functions $f_1,f_2:[-\e,\e]\rightarrow \R$ such that $f_1(x)\leq f_2(x)$ for all $x\in[-\e,\e]$ and:
\begin{align*}
\E(1)\cap Q_{\e}&=\{(x,y) \in Q_{\e} \ | \ -\e \leq y\leq  f_1(x) \},\\
\pa \E(1)\cap Q_{\e} &=\{(x,f_1(x)) \ | \ x\in [-\e,\e]\},\\
\E(2)\cap Q_{\e}&=\{(x,y) \in Q_{\e} \ | \ \  f_2(x)\leq y \leq \e  \},\\
\pa \E(2)\cap Q_{\e}&=\{(x,f_2(x)) \ | \ x\in [-\e,\e] \}\\
\pa \E(2)\cap \pa \E(1)\cap Q_{\e}&=\{(x,y) \in Q_{\e} \ | \ \ y=f_1(x)=f_2(x)\leq \e  \},\\
\E(0)\cap Q_{\e}&=\{(x,y)\in Q_{\e} \ | \ -\e\leq f_1(x)< y< f_2(x)\leq \e\},\\
\E(k)\cap Q_{\e}&=\emptyset \ \ \text{for all $k\geq 3$},
\end{align*}
(see Figure \ref{compoconne}). Since the blow-up of $\pa\E(1)\cap\pa\E(2)$ at $\xi=(0,0)$ is a line, up to further decrease $\e$, we can also assume that $\E(1)\cap Q_{\e}$ and $\E(2)\cap Q_{\e}$ are indecomposable, which is equivalent to say:
	\[
	|f_1(x)|<\e, \ \ \ |f_2(x)|<\e \ \ \ \ \forall \ x\in [-\e,\e].
	\]
We consider the set
	\begin{align*}
	E_0&:= \{ x\in [-\e,\e] \ | \ f_1(x)<f_2(x) \}.
	\end{align*}
which is relatively open inside $[-\e,\e]$ (is the counter-image of the open set $(-2\e,0)$ through the continuous function $f_1-f_2$).
\begin{figure}
\begin{center}
 \includegraphics[scale=1]{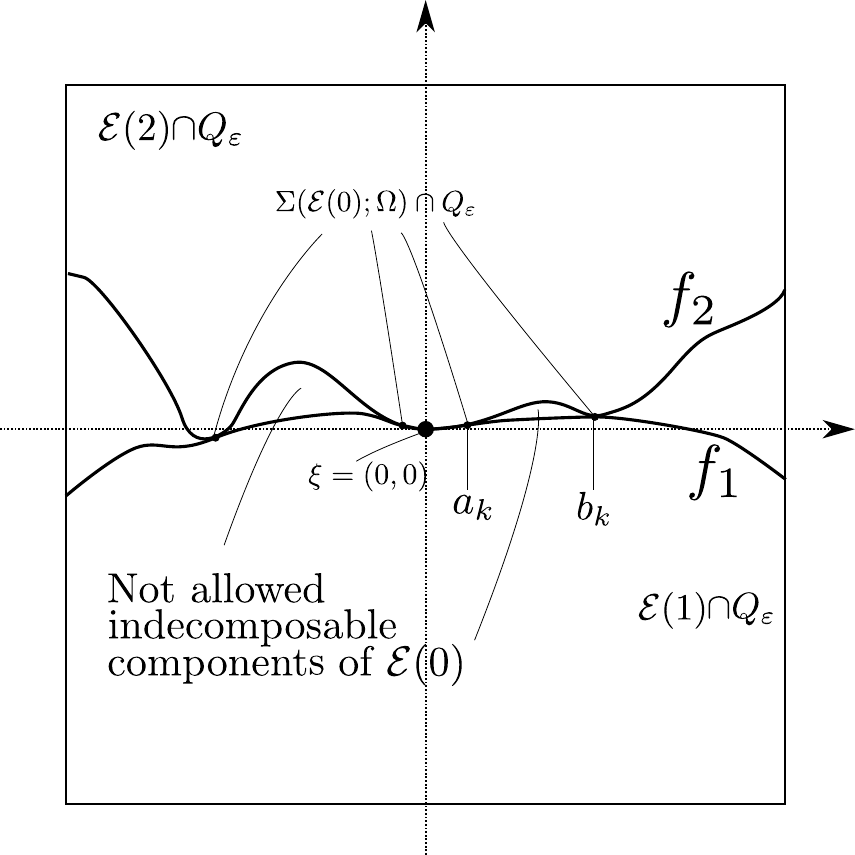}\caption{{\small This kind of behavior contradicts the minimality property of $\E$, in particular it contradicts Lemma \ref{buone componenti connesse}.}}\label{compoconne}
 \end{center}
\end{figure}
 Hence, $E_0$ must be the union of countably many disjoint (open) intervals:
	\[
	E_0=[-\e,a)\cup (b,\e] \cup \left(\bigcup_{k=2}^{+\infty}(a_k,b_k)\right)
	\]
for $\{a_k\}_{k=1}^+{\infty},\{b_k\}_{k=1}^+{\infty}\subset [-\e,\e]$  (with a slight abuse of notation we are allowing also the possible cases $a=-\e$, $b=\e$ or even $a=a_1$,$b_{\infty}=b$ as in Figure \ref{compoconne}). It is immediate that each
	\[
	A_k:=\{(x,y)\in Q_{\e} \ | \ a_k<x<b_k, \ f_1(x) < y <f_2(x) \}
	\]
is an indecomposable component of $\E(0)$. Observe that $E_k\cc Q_{\e} \cc \Om$ is an indecomposable component of $\E(0)$ that share its boundary with exactly two chambers ($\E(1), \E(2)$) and hence contradicts Lemma \ref{buone componenti connesse}. This means that the only possibility is
	\[
	E_0=[-\e,a)\cup (b,\e] 
	\]
for some $a,b\in [-\e,\e]$. By possibly decreasing $\e$ we can assume that $(-\e,f_1(-\e) ),(\e,f_1(\e) )\notin \S_0\cap Q_{\e}$. The only possibilities remained are
\begin{itemize}
\item[1)] $a=-\e$ and $b=\e$, thus $\S_0\cap Q_{\e}=\emptyset$;
\item[2)] $a\neq -\e$ and $b = \e$, thus $\S_0\cap Q_{\e}=\{(a,f_1(a))\}=\{(a,f_2(a))\}$;
\item[3)] $a=-\e$ and $b\neq \e$, thus $\S_0\cap Q_{\e}=\{(b,f_1(b))\}=\{(b,f_2(b))\}$;
\item[4)] $a\neq -\e$ and $b\neq \e$, thus $\S_0\cap Q_{\e}=\{(a,f_1(a)), (b,f_1(b))\}=\{(a,f_2(a)), (b,f_2(b))\}$.
\end{itemize} 
In particular $\#(\S_0\cap Q_{\e})\leq 2$ which means that $\S_0$ has no accumulation points.
\end{proof}

We now exploit the stationarity of Cheeger $N$-clusters in order to derive information on their structure. 

\begin{proposition} \label{curvature}
Let $N\geq 2$ and $\Om$ be an open, bounded and connected set with $C^{1}$ boundary and finite perimeter in the plane. Let $\E$ be a Cheeger $N$-cluster of $\Om$. For every $j,k=0,\ldots,N$, $k \neq j$ the set
$$E_{j,k}:=[\pa \E(j)\cap \pa \E(k) \cap\Om]\setminus \S(\E(0);\Om)$$ 
is relatively open in $\pa \E(j)\cap\pa\E(k)\cap \Om$ and is the finite union of segments and circular arcs. Moreover the set $\E(j)$ has constant curvature $C_{j,k}$ on each open set $A$ such that $A\cap \pa \E(j) \subseteq E_{j,k}$. The constant $C_{j,k}$ is equal to:
\begin{equation}
C_{j,k}= \left\{
\begin{array}{cc}
\frac{|\E(k)|h(\E(j))-|\E(j)|h(\E(k))}{|\E(j)|+|\E(k)|}, & \text{if $k\ne 0$}\\
 & \\
h(\E(j)), & \text{if $k=0$}.
\end{array}
\right.
\end{equation}
As a consequence the set $\E(k)$ has constant curvature $C_{k,j}=-C_{j,k}$ on each open set $A$ such that $A\cap \pa \E(k)\subseteq  E_{k,j}(=E_{j,k})$.
\end{proposition}

\begin{proof}
If $k=0$ (or $j=0$) we just notice that $\E(j)$ is a Cheeger set for 
	\[
	\Om_0=\Om\setminus \bigcup_{\substack{i=1,\\ i\neq j}}^{N} \E(i)
	\]
so the free boundary $E_{j,0}$ is the finite union of segments and circular arcs and $\E(j)$ has constant mean curvature $C_{j,0}=h(\E(j))$ on each open set $A$ such that $A\cap \pared \E(j)\subseteq E_{j,0}$.\\

Thus, we consider a couple $j,k\in \{1,\ldots, N\}$ such that 
	\[
	[\pa\E(i)\cap\pa\E(k)\cap \Om]\setminus  \S(\E(0);\Om)\neq \emptyset
	\]
(otherwise there is nothing to prove and the proposition is trivial). The set $\S(\E(0);\Om)$ is closed and is the finite union of points (thanks to Lemma \ref{struttura singolaritapiano}). Hence 
	\[
	E_{j,k}:=[\pa \E(j)\cap \pa \E(k) \cap\Om]\setminus \S(\E(0);\Om)
	\]
is relatively open in $\pa \E(j)\cap\pa\E(k) \cap \Om$. For every $x\in E_{j,k}$, by closedness, there exists a ball $B_r(x)$ such that 
	\[
	B_r(x)\cap  \E(i)=\emptyset \ \ \ \ \forall \ i=1,\ldots,N, \  \ i\neq j,k.
	\]
Note that, up to further decrease the value of $r$ it must hold as well
	\[
	B_r(x)\cap \E(0)=\emptyset.
	\]
Indeed if this is not the case, we would have that $x\in \pa \E(0)\cap \E(0)^{\zero}$ and thus (thanks to Proposition \ref{prima caratterizazione}) $x\in \S(\E(0);\Om)$ which is a contradiction since $x\in E_{j,k}$. Hence, because of the minimality of $\E$, the set $\pa\E(j)\cap \pa\E(k)\cap B_r(x)$ must solve an isoperimetric problem with volume constraint inside $B_r(x)$ and by exploiting stationarity it is possible to prove that each solution to an isoperimetric problem with volume constraint must be an analytic constant mean curvature hypersurface (\cite[Theorems 17.20, 24.4 ]{maggibook}). Set $C_{j,k}$ and $C_{k,j}$ to be respectively the value of the mean curvature of $\E(j)$ and of $\E(k)$ in $B_r(x)$. Observe that, since $\E(k)\cap B_r= \E(j)^c\cap B_r$ it holds trivially that $C_{k,j}=-C_{j,k}$. Let us compute the (constant) value of $C_{j,k}$.\\

Consider a map $T\in C^{\infty}_c(B_r;\R^2)$, define for all $|t|<\e$ the diffeomorphism $f_t(y)=y+t T(y)$ and the cluster $\E_t:=\{f_t(\E(i))\}_{i=1}^N$. Of course, for $t$ suitably small, $\E_t\Delta \E \subset\subset B_r $. Note that $\{f_t \ | \ -\e<t<\e \}$ is a local variation in $B_r$ and that $T$ is the initial velocity (according to the definitions given in Subsection \ref{Cpt 1 sbs: First variation of perimeter}). By minimality it holds: 
$$\frac{P(\E(j))}{|\E(j)|}+\frac{P(\E(k))}{|\E(k)|} \leq \frac{P(\E_t(j))}{|\E_t (j)|}+\frac{P(\E_t(k))}{|\E_t (k)|}, \ \ \ \forall \ |t|<\e.$$
Thus  
\begin{equation}\label{mi sparooo}
0\leq \frac{d}{dt}\Big{|}_{t=0}\frac{P(\E_t(j))}{|\E_t (j)|}+\frac{d}{dt}\Big{|}_{t=0} \frac{P(\E_t(k))}{|\E_t (k)|}.
\end{equation}
With some easy computations
\begin{eqnarray*}
\frac{d}{dt}\Big{|}_{t=0}\frac{P(\E_t(j))}{|\E_t (j)|} &=&\frac{|\E(j)|\frac{d}{dt}\Big{|}_{t=0}P(\E_t(j))-P(\E(j))\frac{d}{dt}\Big{|}_{t=0}|\E_t(j)|}{|\E (j)|^2},\\
\text{}\\
\frac{d}{dt}\Big{|}_{t=0}P(\E_t(j)) &=&C_{i,k} \int_{\pa \E(j)\cap \pa \E(k) \cap B_r}(T(y)\cdot \nu_{\E(j)}(y))  \d \H^1(y)\\
\frac{d}{dt}\Big{|}_{t=0}|\E_t(j)| &=& \int_{\pa \E(j)\cap \pa \E(k) \cap B_r}(T(y)\cdot \nu_{\E(j)}(y)) \d \H^1(y).
\end{eqnarray*}
where we have used formulas \eqref{eqn: chapter intro constant mean curvature variations} and \eqref{eqn: chpter intro first variaition of the lebesgue measure} combined with the fact that the mean curvature exists and it is constantly equal to $C_{j,k}$ in $B_r$ (and hence on $\pa\E(j)\cap \pa \E(k)\cap B_r$). By denoting with 
$$ f_i=\int_{\pa \E(j)\cap \pa \E(k) \cap B_r}(T(y)\cdot \nu_{\E(j)}(y))  \d \H^1(y)$$
$$ f_k=\int_{\pa \E(j)\cap \pa \E(k) \cap B_r}(T(y)\cdot \nu_{\E(k)}(y))  \d \H^1(y),$$
we can write:
\begin{eqnarray*}
\frac{d}{dt}\Big{|}_{t=0}\frac{P(\E_t(j))}{|\E_t (j)|} &=&\frac{|\E(j)| f_j C_{j,k}-P(\E(j))f_j}{|\E (j)|^2},\\
\frac{d}{dt}\Big{|}_{t=0}\frac{P(\E_t(k))}{|\E_t (k)|} &=&\frac{|\E(k)| f_k C_{k,j}-P(\E(k))f_k}{|\E (k)|^2},\\
\end{eqnarray*}
that plugged into \eqref{mi sparooo},by observing that $f_j=-f_k$, lead to the relation:
\begin{eqnarray*}
0&\leq &\frac{d}{dt}\Big{|}_{t=0}\frac{P(\E_t(j))}{|\E_t (j)|} +\frac{d}{dt}\Big{|}_{t=0}\frac{P(\E_t(k))}{|\E_t (k)|} \\
&=&\frac{|\E(i)| f_j C_{j,k}-P(\E(j))f_j}{|\E(j)|^2}+\frac{|\E(k)| f_k C_{k,j}-P(\E(k))f_k}{|\E(k)|^2}\\
&=&f_j\left[ \frac{|\E(j)|  C_{j,k}-P(\E(j))}{|\E (j)|^2}-\frac{|\E(k)|  C_{k,j}-P(\E(k))}{|\E (k)|^2}\right].
\end{eqnarray*}
By choosing a $T_1$ such that $f_j$ is positive and then a $T_2$ such that $f_j$ is negative we conclude that 
\begin{eqnarray*}
0&=&\frac{|\E(j)|  C_{j,k}-P(\E(j))}{|\E (j)|^2}-\frac{|\E(k)|  C_{k,j}-P(\E(k))}{|\E (k)|^2}.\\
\end{eqnarray*}
Finally, by exploiting $C_{j,k}=-C_{k,j}$ we rach
\begin{eqnarray*}
0&=&\frac{ C_{j,k}}{|\E (j)|}-\frac{P(\E(j))}{|\E (j)|^2}-\frac{  C_{k,j}}{|\E (k)|}+\frac{P(\E(k))}{|\E (k)|^2}\\
&=&\frac{ C_{j,k}}{|\E (j)|}-\frac{P(\E(j))}{|\E (j)|^2}+\frac{  C_{j,k}}{|\E (k)|}+\frac{P(\E(k))}{|\E (k)|^2},
\end{eqnarray*}
that can be re-arranged as:
\begin{eqnarray*}
C_{i,k}(|\E(j)|+|\E(k)|)&=&h(\E(j))|\E (k)|-h(\E(k))|\E (j)|,\\
C_{j,k}&=&\frac{h(\E(j))|\E (k)|-h(\E(k))|\E (j)|}{|\E(j)|+|\E(k)|}.
\end{eqnarray*}
In particular, since $C_{j,k}$ do not depend on $x\in E_{j,k}$ and since the ambient space dimension is $n=2$, $E_{j,k}$ must be a finite union of circular arcs or segments with curvature $|C_{j,k}|$.
\end{proof}

Our last proposition of the section put together Lemma \ref{nel piano!} Proposition \ref{struttura singolaritapiano} and Proposition \ref{curvature} and tells us that the interior chambers of a Cheeger $N$-cluster are always indecomposable. We are making strong use of Proposition \ref{struttura singolaritapiano} which does not holds on $\pa\Om$ (see Figure \ref{controfinitezza} and Remark \ref{remark controfinitezza}) and that is why we cannot extend the proof of the Proposition \ref{connessionecamere} to all the chambers. 

\begin{proposition}\label{connessionecamere}
Let $N\geq 2$ and $\Om$ be an open, bounded and connected set with $C^{1}$ boundary and finite perimeter in the plane. Let $\E$ be a Cheeger $N$-cluster for $\Om$. Then, every chamber $\E(i)\cc \Om$ for $i\neq 0$ is indecomposable. 
\end{proposition}

\begin{proof}
Assume, without loss of generality $i=1$ and let $E_1$ and $E_2$ be two different components of $\E(1)$. By minimality it must hold
	\begin{equation}\label{mi impiccoooooo}
	\frac{P(E_1)}{|E_1|}=\frac{P(E_2)}{|E_2|}=\frac{P(\E(1))}{|\E(1)|}.
	\end{equation}
The component $E_2$ is a Cheeger set for 
	\[
	A=\left(\bigcup_{j\neq 1} \Om\setminus \E(j)\right) \cup E_1
	\] 
and by Theorem \ref{mainthm1}, every $x\in \pa E_2\cap \pa A$ is a regular point for $\pa A$. Moreover $\bd_{\pa A}(\pa A \cap\pa E_2)\subseteq \S(\E(0);\Om)$ (since $\E(i)\cc \Om$) and thus, thanks to Proposition \ref{struttura singolaritapiano}, we have 
	\[
	\#(\bd_{\pa A}(\pa A \cap\pa E_2))\leq \#(\S(\E(0);\Om))<+\infty.
	\]
Therefore we can exploit Lemma \ref{nel piano!} on $E_2$ and conclude that $\H^1(\pa E_2 \cap \pa A)>0$. In particular we deduce that there exists an index $k\neq 0,1$ such that $\H^1(\pa E_2\cap \pa \E(k))>0$. Define the new cluster $\F(1)=E_1$, $\F(j)=\E(j)$ for $j\neq 1$ (see Figure \ref{lemmaindecompo}). Thanks to \eqref{mi impiccoooooo} it holds:
\begin{equation}
H_N(\Om)=\sum_{i=1}^N \frac{P(\F(i))}{|\F(i)|}.
\end{equation}
Hence $\F$ it is also a Cheeger $N$-cluster for $\Om$. Consider the piece of boundary 
	\[
	S=[\pa \E(k)\cap \pa E_2]\setminus \S(\E(0);\Om) \neq \emptyset
	\]
from the old cluster $\E$.  Proposition \ref{curvature} tells us that $S$ must be a circular arc and that $\E(k)$ must has constant mean curvature on $S$ equal to:
 $$C_{k,1}=\frac{|\E(1)|h(\E(k))-|\E(k)|h(\E(1))}{|\E(1)|+|\E(k)|}.$$
 \begin{figure}
\begin{center}
 \includegraphics[scale=2.2]{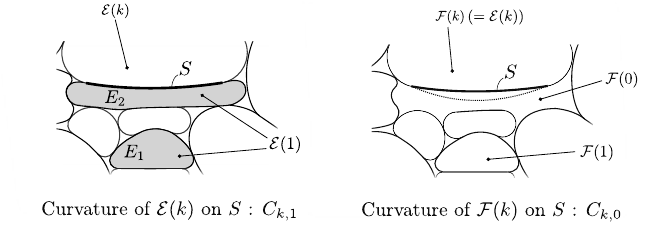}\caption{{\small }}\label{lemmaindecompo}
 \end{center}
\end{figure}
From the other side it holds $\F(k)=\E(k)$ and, since $S$ is now a part of the free boundary of $\F(k)$ (we have removed the component $E_2$), we have that $\F(k)=\E(k)$ must has constant mean curvature on $S$ also equal to:
$$C_{k,0}=h(\F(k))=h(\E(k)).$$
Thus equality $C_{k,1}=C_{k,0}$ must be in force, implying $(h(\E(k))+h(\E(1)))|\E(k)|=0$ which is impossible. 
\end{proof}

\begin{proof}[Proof of Theorem \ref{mainthm3}]
It follows from Propositions \ref{struttura singolaritapiano}, \ref{curvature} and \ref{connessionecamere}.
\end{proof}
\section{The limit of $\Lambda_N^{(p)}$ as $p$ goes to one }\label{limite}
We conclude this Chapter by focusing on the asymptotic trend of $H_N$. We first briefly state the following Theorem involving the existence of the optimal partition for problem \eqref{p-laplacian}.
\begin{theorem}
For every $1<p\leq n$ there exists an optimal partition for $\Om$ in quasi-open sets $\{\Om_i\}_{i=1}^N$ such that 
$$\Lambda_N^{(p)}(\Om)=\sum_{i=1}^N \lambda_1^{(p)}(\Om_i).$$
\end{theorem}
\begin{proof}
The existence of an optimal partition for $\Lambda_N^{(p)}(\Om) $ follow as a simple variation of the argument in \cite{CaLi07}, or 
as a consequence of more general results contained in \cite{BucBuH98}, \cite{BucVel13} or \cite{BuDM93} and thus we omit the details. 
\end{proof}
In the following Proposition we compute the limit  of $\Lambda_N^{(p)}$ as $p$ goes to one.
\begin{proposition}\label{limit}
If $\Om$ is an open bounded set with $C^1$ boundary then
$$\lim_{p\rightarrow 1} \Lambda_N^{(p)}(\Om)=H_N(\Om).$$
\end{proposition}
\begin{proof}
Let $\E$ be a Cheeger $N$-cluster for $\Om$. Since $\pa \E(i)$ is $C^1$, for every $i=1,\ldots,N$ there exists a sequence of open sets $\{\E_t(i)\}_{t>0}$ such that $\E_t(i)\cc \E(i)$ for all $t>0$ and
$$\E_{t}(i)\rightarrow \E(i) \text{  in $L^{1}$}, \ \ \ P(\E_{t}(i))\rightarrow P(\E(i)),$$ 
as $t\to 0$ (see \cite{Schmidt2015}). In this way, since $\E_t(i)$ are open sets (and thus quasi-open sets) strictly contained into $\Om$ and with disjoint closure, by exploiting \eqref{eqn limite per p che tende a uno} we reach: 
\begin{align*}
H_N(\Om)&= \lim_{t\rightarrow 0} \sum_{i=1}^{N}\frac{P(\E_{t}(i))}{|\E_{t}(i)|}\geq  \lim_{t\rightarrow 0} \sum_{i=1}^{N}h(\E_{t}(i))\\
&\geq  \lim_{t\rightarrow 0} \limsup_{p\rightarrow 1}\sum_{i=1}^{N}\lambda_N^{(p)}(\E_{t}(i))\geq \limsup_{p\rightarrow 1}\Lambda_N^{(p)}(\Om).
\end{align*}
On the other hand, thanks to \eqref{lb p-eig} and to Jensen's inequality we get \eqref{Lp lowerbound}:
\begin{align*}
\sum_{j=1}^N \lambda_1^{(p)}(\E(i))&\geq  \sum_{j=1}^N \left(\frac{h(\E(i))}{p}\right)^{p} \geq \frac{1}{N^{p-1}}\left(\sum_{j=1}^N \frac{h(\E(i))}{p}\right)^{p}\\
&\geq\frac{1}{N^{p-1}}\left( \frac{H_N(\Om)}{p}\right)^{p}
\end{align*}

which completes the proof.
\end{proof}
\subsection{On the asymptotic behavior of $H_N$ in dimension $n=2$}\label{asymptoyc of HN}

\begin{theorem}\label{asymptotic behavior}
Denote with $B$ a ball of unit radius and with $H$ a unit-area regular hexagon. Let $N\geq 2$ and $\Om$ be an open, bounded and connected set with $C^{1}$ boundary and finite perimeter in the plane. Then the following assertions hold true:
\begin{itemize}
\item[1)] If $\E$ is a Cheeger $(N+1)$-cluster for $\Om$ then:
	\[
	|\E(i)|\geq \frac{h(B)^2\pi}{(H_{N+1}(\Om)-H_N(\Om))^2}  \ \ \ \forall \ i=1,\ldots,N+1;
	\]
\item[2)] $\displaystyle H_N(\Omega)+\frac{h(B)\sqrt{\pi}}{\sqrt{|\Omega|}}\sqrt{N+1}\leq H_{N+1}(\Omega),$  for all $N\in \N$;
\item[3)] for every $0\leq \e<\frac{1}{2}$ there exists $N_0(\Om,\e)$ such that:
	\[
	\frac{\sqrt{\pi} h(B)}{\sqrt{|\Omega|}}N^{\frac{3}{2}}\leq H_N(\Omega)\leq \frac{h(H)}{\sqrt{|\Omega|}}N^{\frac{3}{2}}+ N^{\frac{3}{2}-\e} \ \ \ \  \text{for all $N\geq N_0(\Om,\e)$}.
	\]
\end{itemize}
\end{theorem}
\begin{proof}
Thanks to the planar Cheeger inequality (\ref{eqn: chapter intro cheeger inequality})
\begin{equation}\label{CheegerInequality}
h(E)\geq \sqrt{\pi}\frac{h(B)}{\sqrt{|E|}}
\end{equation}
we observe that, given $\E$ a Cheeger $(N+1)$-cluster of $\Om$, it holds:
\begin{align*}
H_{N+1}(\Om)&= \sum_{i=1}^{N+1} \frac{P(\E(i))}{|\E(i)|} \geq\frac{P(\E(j))}{|\E(j)|} +\sum_{i=1, i\neq j}^{N+1} \frac{P(\E(i))}{|\E(i)|}\\
&\geq h(\E(j))+H_{N}(\Om)\geq \frac{\sqrt{\pi}h(B)}{\sqrt{|\E(j)|}}+H_{N}(\Om)
\end{align*}
which, implies Property 1).\\

Property 2) follows from Property 1):
\begin{eqnarray*}
|\Om|-\frac{(N+1)h(B)^2\pi }{(H_{N+1}(\Om)-H_N({\Om}))^2} &\geq&|\Om|-\sum_{i=1}^{N+1}|\E(i)|\geq 0
\end{eqnarray*}
and so
\begin{eqnarray*}
|\Om|&\geq& \frac{(N+1)h(B)^2\pi}{(H_{N+1}(\Om)-H_N({\Om}))^2},
\end{eqnarray*}
which implies
\begin{eqnarray*}
(H_{N+1}(\Om)-H_N({\Om}))^2&\geq& \frac{(N+1)h(B)^2\pi}{|\Om|}
\end{eqnarray*}
and thus
\begin{eqnarray*}
H_{N+1}(\Om)&\geq& H_N({\Om})+ \sqrt{(N+1)}\frac{\sqrt{\pi}h(B)}{\sqrt{|\Om|}}.
\end{eqnarray*}
Let us prove Property 3). Let $\E$ be a Cheeger $N$-cluster for $\Om$. We exploit again  the Cheeger inequality \eqref{CheegerInequality} and
we obtain the lower bound
\begin{align*}
H_N(\Om)&=\sum_{i=1}^{N}h(\E(i))\geq  \sqrt{\pi} h(B)\sum_{i=1}^{N}\frac{1}{\sqrt{|\E(i)|}}\\
&\geq   \sqrt{\pi} h(B) N^{\frac{3}{2}}\left(\frac{1}{\sum_{i=1}^{N}|\E(i)|}\right)^{\frac{1}{2}}\geq   \sqrt{\pi} \frac{h(B)}{\sqrt{|\Om |}} N^{\frac{3}{2}}. 
\end{align*} 
Here we have used the inequality
$$\sum_{i=1}\frac{1}{x_i^{\frac{1}{n}}}\geq N^{\frac{n+1}{n}}\left(\frac{1}{\sum_{i=1}^Nx_i}\right)^{\frac{1}{n}}, \ \ \forall \ N,n\geq 2, \  x_i >0.$$
Let us focus on the upper bound. Let $\H_{\de}$ be the standard hexagonal grid of the plane, made by hexagons of area $\de$ (the one depicted in Figure \ref{fig reference} in Chapter 2). Define
	\begin{align*}
	I(\de)&:=\{i\in \N \ | \ \H_{\de}(i)\cc \Om\}.\\
	k(\de)&:=\#(I(\de) ).
	\end{align*}
Up to a relabeling, let us assume that $I(\de)=\{1,\ldots,k(\de)\}$. Note that since $\H_{\de}(i)\cc  \Om$ we get
	\[
	H_{k(\de)}(\Om)\leq\sum_{i=1}^{k(\de)} h(\sqrt{\delta} H)=\frac{k(\de)}{\sqrt{\de}}h(H).
	\]
From $\H_{\de}(i) \subset \Om$ for all $i=1,\ldots,k(\de)$ it follows
$$k(\de) \leq \frac{|\Om|}{\de }.$$
If we set $\de(N)=\frac{|\Om|}{N}-\frac{|\Om|}{N^{\a}}$ for some $\a>1$ to be chosen, we are led to 
	\begin{equation}\label{forse stasera finisco}
	H_{k(N)}(\Om)\leq \frac{N^{\frac{3}{2}}}{\sqrt{|\Om|}\left(1-N^{1-\a}\right)^{\frac{3}{2}}} h(H).
	\end{equation}
where $k(N)=k(\de(N))$.
Note that, by setting
	\[
	(\pa \Omega)_{r(N)}:=\pa\Om+B_{r(N)}
	\]
where $r(N)=\sqrt{\de(N)}\diam(H)$, it must hold 
$$\left(\Om\setminus \bigcup_{i=1}^{k(N)} \H(i)\right)\subseteq (\pa \Omega)_{r(N)}.$$
Since $\Om$ has Lipschitz boundary, for $N$ bigger than $N_0(\Om)$, it also holds that
	\[
	| (\pa \Omega)_{r(N)}|\leq 4 r(N) P(\Om)
	\]
and so:
\begin{align*}
|\Om|-\de(N) k(N) &\leq  |(\pa \Omega)_{r(N)}| \leq  4r(N)P(\Om) = 4\sqrt{\de(N)}\diam(H)P(\Om),\\
\end{align*}
that imply
\begin{align*}
 k(N) &\geq  \frac{N}{1-N^{1-\a}} - 4\sqrt{N}\frac{P(\Om)\diam(H)}{\sqrt{|\Om|} \sqrt{1-N^{1-\a}}} .
\end{align*}
For all $N$ bigger than some fixed $N_0$ depending only on $\Om$. It is easy to show that, for all $\a<\frac{3}{2}$, up to further increase $N_0$ in dependence only on $\Om$ and $\a$, it holds 
	\[
	\frac{N}{1-N^{1-\a}} - 4\sqrt{N}\diam(H)\frac{P(\Om)}{\sqrt{|\Om|} \sqrt{1-N^{1-\a}}} \geq N.
	\]
Hence by choosing $\a <\frac{3}{2}$ we obtain
	\[
	k(N)\geq N \ \ \ \ \ \forall \ N\geq N_0,
	\]
and, thanks to the monotonicity given by Property 2) and to \eqref{forse stasera finisco}, provided also $\a>1+\e$ we reach:
	\begin{align*}
	H_N(\Om) &\leq H_{k(N)}(\Om)\\
	&\leq \frac{N^{\frac{3}{2}}}{\sqrt{|\Om|}\left(1-N^{1-\a}\right)^{\frac{3}{2}}} h(H) \\
	&\leq \frac{h(H)}{\sqrt{|\Om|}}  (N^{\frac{3}{2}}+N^{\frac{3}{2}-\e})  \ \ \ \ \ \text{for all $N>N_0(\Om,\e)$}.
	\end{align*}
\end{proof}

\chapter{Final remarks and future research interests}

We conclude our discussion about the isoperimetric properties of $N$-clusters with some remarks and ideas for possible future investigations.\\

As a first remark we highlight that it could be interesting to explore whether the topics contained in Chapter Two can be extended to a dimension higher than $n=2$. A first apparent obstacle arises when one tries to generalize the arguments contained in the proofs of Theorems \ref{Equi dia} and \ref{Equi indeco}: so far there is no theorem in the spirit of Hales's Theorem \ref{teo: chapter intro Hales piano} that holds in dimension $n>2$. Another interesting direction could be to weaken the hypothesis involving indecomposability and boundedness of the chambers. For example we can try to implement the argument exposed in Chapter Two with the assumption that each chamber can split into a controlled number of pieces $p(N)\leq N^{a}$ for $a$ that varies in $[0,1)$. We can ask how to combine the result contained in Chapter Three, about the quantitative version of the hexagonal honeycomb theorem, with the energetic estimates of Chapter two. In particular: 
\begin{multline}\label{domanda fondamentale}
\text{\textit{Is there a way, starting from an energetic estimate,}}\\
\text{\textit{to obtain information about the shape of the chambers?}}
\end{multline}

As far as Chapter Three is concerned, we could investigate how the constant $\kappa(N)$ in Theorem \ref{thm main periodic} depends on $N$. This could provide an answer to \eqref{domanda fondamentale}. For example it is reasonable to expect that we can combine an energetic estimate of the type of the one appearing in Theorem \ref{Equi indeco} with the quantitative version of the Hexagonal Honeycomb Theorem for conclude that each chamber of an indecomposable minimizing $N$-cluster of $\Om$ converges in the $L^1$ sense to a regular hexagon. In order to do that we need to understand what is the dependence of the constant $\kappa(N)$ in Theorem \ref{thm main periodic} with respect to $N$.\\

Chapter Four instead left us with many open questions; for example:
\begin{multline}\label{domanda fondamentale 2}
\text{\textit{Is there a way to prove the Caffarelli and Lin's conjecture}}\\
\text{\textit{in the case $p=1$, by importing tools and instrument}}\\
\text{\textit{from the Hales's proof of Hexagonal honeycom Theorem?}}
\end{multline}
Is this easier than the Caffarelli and Lin's conjecture for $p=2$? It could be interesting to understand whether the approach proposed in Chapter One can be adapted to the context of Cheeger $N$-cluster in order to obtain a sort of periodicity (from an energetic point of view) in the asymptotic behavior of Cheeger $N$-clusters.\\

Let us conclude this work, that is the result of three years of efforts and study at the University of Pisa, by remarking that these are just a few questions, the more natural that occur when we deal with such problems, that will lead our research in the future, together with many (and we hope fruitful) others interesting queries.
\bibliography{referencescopia}
\bibliographystyle{alpha}

\addcontentsline{toc}{chapter}{Bibliography} 

\end{document}